\documentclass[11pt]{report}         

\usepackage{amsmath,amsfonts,amssymb}   

\usepackage{graphicx}
\usepackage{amssymb}
\usepackage{amsmath}
\usepackage{amsthm}
\usepackage{amscd}
\usepackage{xcolor}
\usepackage{enumerate}
\usepackage{array}
\usepackage{caption}
\usepackage{subcaption}
\usepackage{siunitx}
\usepackage{textcomp}
\usepackage[T1]{fontenc}
\usepackage{listings}

\usepackage{mathabx}

\newtheorem{thm}{Theorem}[section]
\newtheorem{cor}[thm]{Corollary}
\newtheorem{fact}[thm]{Fact}
\newtheorem{lem}[thm]{Lemma}
\newtheorem{exmp}[thm]{Example}
\newtheorem{prop}[thm]{Proposition}
\newtheorem{rem}[thm]{Remark}
\theoremstyle{definition}
\newtheorem{ex}[thm]{Example}
\newtheorem{defn}[thm]{Definition}

\lstdefinestyle{customc}{
  belowcaptionskip=1\baselineskip,
  breaklines=true,
  xleftmargin=\parindent,
  language=Python,
  showstringspaces=false,
  basicstyle=\footnotesize\ttfamily,
  keywordstyle=\bfseries\color{green!40!black},
  commentstyle=\itshape\color{purple!40!black},
  identifierstyle=\color{blue},
  stringstyle=\color{orange},
}

\title{
Monte Carlo methods and stochastic simulations: from integration to  approximation of SDEs\footnote{see also: https://arxiv.org/abs/2208.05531}}
\author{Pawe\l \ Przyby\l owicz\\
AGH University of Krakow,
Faculty of Applied Mathematics \\
pprzybyl@agh.edu.pl}
\date{}           

\begin{document}

\maketitle
\tableofcontents
\chapter{Introduction}
In recent years dynamical systems (of deterministic and stochastic nature), describing many models in mathematics, physics, engineering and finances, become more and more complex. Numerical analysis  narrowed only to deterministic algorithms seems to be insufficient for such systems, since, for example, curse of dimensionality affects deterministic methods. Therefore, we can observe increasing popularity of Monte Carlo algorithms and, closely related with them, stochastic simulations based on stochastic differential equations. In these lecture notes we present main ideas concerned with Monte Carlo methods and their theoretical properties. We apply them to such  problems as  integration and approximation of solutions of deterministic/stochastic differential equations. We also discuss implementation of exemplary algorithms in Python programming language and their application to option pricing.

Part of these notes has been used during lectures for PhD students at AGH University of Science and Technology, Krakow, Poland, at summer semesters in the years 2020, 2021, and 2023.
\chapter{Deterministic quadrature rules for approximation of multidimensional Lebesgue integrals}
\label{sec: det_met}
Let us start with considering the following model problem in numerical analysis: approximate   value of the following multidimensional integral
\begin{equation}
\label{I_1}
	I(f)=\int\limits_{[0,1]^d}f(x)dx.
\end{equation}
The integral above is understood in the Lebesgue sense with respect to $d$-dimensional Lebesgue measure $\lambda_d$ and $f:[0,1]^d\to\mathbb{R}$ is at least Borel measurable and integrable. 

In this chapter we will be assuming that $f$ is at least continuous function. Then the Lebesgue integral $I(f)$ is equivalent to the Riemann integral of $f$, see, for example, Theorem 4.7.3 in \cite{MAKPOD}.

The simplest quadrature that approximates \eqref{I_1} is the {\it rectangle quadrature}, defined as follows. Let $n\in\mathbb{N}$, $h=1/n$, and we take
\begin{equation}
	t_{i}=i\cdot h, \  i=0,1,\ldots,n,
\end{equation}
\begin{equation}
\label{def_Ais}
	A_{i_1,\ldots,i_d}=[t_{i_1},t_{i_1+1}]\times\ldots\times [t_{i_d},t_{i_d+1}], 
\end{equation}
for $i_k\in\{0,1,\ldots,n-1\}$, $k=1,2,\ldots,d$. Then
\begin{equation}
\label{div_prop_1}
	[0,1]^d=\bigcup_{i_1,\ldots,i_d=0,1,\ldots,n-1}A_{i_1,\ldots,i_d}.
\end{equation}
Since for all $(i_1,\ldots,i_d)\neq (j_1,\ldots,j_d)$ we have that
\begin{equation}
\label{div_prop_2}
	\lambda_d(A_{i_1,\ldots,i_d}\cap A_{j_1,\ldots,j_d})=0,
\end{equation}
we can write that
\begin{eqnarray}
\label{I_f_approx}
	&&I(f)=\sum\limits_{i_1,\ldots,i_d=0}^{n-1}\int\limits_{A_{i_1,\ldots,i_d}}f(x)dx\approx \sum\limits_{i_1,\ldots,i_d=0}^{n-1}\int\limits_{A_{i_1,\ldots,i_d}}f(\xi_{i_1,\ldots,i_d})dx\notag\\
	&&=\sum\limits_{i_1,\ldots,i_d=0}^{n-1}f(\xi_{i_1,\ldots,i_d})\cdot\lambda_d(A_{i_1,\ldots,i_d})=h^d\cdot\sum\limits_{i_1,\ldots,i_d=0}^{n-1}f(\xi_{i_1,\ldots,i_d})
\end{eqnarray}
with $\xi_{i_1,\ldots,i_d}=(t_{i_1},t_{i_2},\ldots,t_{i_d})$. Hence, the multidimensional rectangle rule is defined by
\begin{equation}
\label{mul_dim_Q_def}	Q^d_n(f)=h^d\cdot\sum\limits_{i_1,\ldots,i_d=0}^{n-1}f(\xi_{i_1,\ldots,i_d}).
\end{equation}
Note that the cost of $Q^d_n$, measured by the number of function evaluation, is $n^d$.

We investigate error of $Q^d_n$ for integrands belonging to the following class
\begin{equation}
	\mathcal{F}^d_L=\{f:[0,1]^d\to\mathbb{R} \ | \ |f(x)-f(y)|\leq L\cdot\|x-y\|_{\infty}, \ x,y\in [0,1]^d\},
\end{equation}
where $L\in (0,+\infty)$ is the class parameter and $\|\cdot\|_{\infty}$ is the maximum norm, i.e., $\|x\|_{\infty}=\max_{1\leq i\leq d}|x_i|$ for $x=(x_1,\ldots,x_d)\in\mathbb{R}^d$. 

\begin{thm} 
\label{RECT_MD_ERR_1}
For any $f\in\mathcal{F}^d_L$ and for all $n\in\mathbb{N}$ we have that $Q^d_n(f)$ uses $n^d$ values of $f$ and 
	\begin{equation}
		|I(f)-Q^d_n(f)|\leq Ln^{-1}.
	\end{equation}	
\end{thm}
{\bf Proof.} By the Fubini's theorem we have that
\begin{eqnarray}
	&&|I(f)-Q_n^d(f)|\leq \sum\limits_{i_1,\ldots,i_d=0}^{n-1}\int\limits_{A_{i_1,\ldots,i_d}}|f(x_1,\ldots,x_d)-f(t_{i_1},\ldots,t_{i_d})|dx_1\ldots dx_d\notag\\
	&&\leq L\sum\limits_{i_1,\ldots,i_d=0}^{n-1}\int\limits_{A_{i_1,\ldots,i_d}}\|(x_1,\ldots,x_d)-(t_{i_1},\ldots,t_{i_d})\|_{\infty}dx_1\ldots dx_d\notag\\
	&&= L\sum\limits_{i_1,\ldots,i_d=0}^{n-1}\int\limits_{t_{i_d}}^{t_{i_d+1}}\ldots\int\limits_{t_{i_1}}^{t_{i_1+1}}\max\{|x_1-t_{i_1}|,\ldots,|x_d-t_{i_d}|\}dx_1\ldots dx_d\notag\\
	&&\leq Lh \sum\limits_{i_1,\ldots,i_d=0}^{n-1}\int\limits_{t_{i_d}}^{t_{i_d+1}}\ldots\int\limits_{t_{i_1}}^{t_{i_1+1}}dx_1\ldots dx_d=Lh\cdot h^d\cdot\sum\limits_{i_1,\ldots,i_d=0}^{n-1}1=Lh^{d+1}\cdot n^d=Ln^{-1}.\notag
\end{eqnarray}
$\blacksquare$
\\
In the term of the cost the result above can be formulated as follows: for all $f\in\mathcal{F}^d_L$ and $N\in\mathbb{N}$
\begin{equation}
	|I(f)-Q^d_{N^{1/d}}(f)|\leq LN^{-1/d},
\end{equation}
where now $Q^d_{N^{1/d}}$ uses $N$ evaluations of $f$. Hence, for a fixed error-level $\varepsilon\in (0,+\infty)$ we need to take 
\begin{equation}
	N\geq\Big\lceil \Bigl(\frac{L}{\varepsilon}\Bigr)^d\Big\rceil=\Omega((1/\varepsilon)^d)
\end{equation}
function values in order to be sure that the error of the quadrature is at most $\varepsilon$. For example, if one take $\varepsilon=1/2$ then the number of evaluation needed depends exponentially in $d$. One can ask if this bad cost-error behavior is caused because of use of this particular quadrature rule or by assumptions imposed on $f$. The answer is negative, what follows from respective lower error bounds.
\section{Exercises}
\begin{itemize}
	\item [1.] Provide a version of  Theorem \ref{RECT_MD_ERR_1} in the case when we have non-uniform discretization of the cube $[0,1]^d$.
	\item [2.] Construct and investigate an error of  multidimensional version of the mid-point quadratic rule, approximating the integral \eqref{I_1}, which is defined by \eqref{mul_dim_Q_def} but this time with
 \begin{equation}
     \xi_{i_1,\ldots,i_d}=\Bigl((t_{i_1}+t_{i_1+1})/2,(t_{i_2}+t_{i_2+1})/2,\ldots,(t_{i_d}+t_{i_d+1})/2\Bigr).
 \end{equation}
 \item [3.] The Taylor-based composite quadrature rule for one-dimensional integrals is defined as follows:  let $n\in\mathbb{N}$, $h=(b-a)/n$, $t_{i}=a+ih$, $i=0,1,\ldots,n$, and
\begin{equation}
	l_n(t)=\sum\limits_{i=0}^{n-1}l_{i}(t)\cdot\mathbf{1}_{[t_{i},t_{i+1})}(t), \quad t\in [a,b],
\end{equation}
where
\begin{equation}
	l_{i}(t)=\sum\limits_{j=0}^{r-1}\frac{f^{(j)}(t_{i})}{j!}(t-t_{i})^j, \quad t\in [t_{i},t_{i+1}], \ i=0,1,\ldots,n-1.
\end{equation}
Then the  quadrature is defined by taking
\begin{equation}
	\mathcal{T}_n(f):=I(l_n)=\int\limits_a^b l_n(t)dt=\sum\limits_{i=0}^{n-1}\sum\limits_{j=0}^{r-1}\frac{f^{(j)}(t_{i})}{(j+1)!}(t_{i+1}-t_{i})^j.
\end{equation}
Show that for any $f\in C^{r}([a,b])$, $r\in\mathbb{N}$, it holds
	\begin{equation}
 \label{T_n_lim}
		\lim\limits_{n\to +\infty}n^r\cdot(I(f)-\mathcal{T}_n(f))=C_{\mathcal{T}}(f),
	\end{equation}
	where
	\begin{equation}
		C_{\mathcal{T}}(f)=\frac{(b-a)^r}{(r+1)!}\Bigl(f^{(r-1)}(b)-f^{(r-1)}(a)\Bigr).
	\end{equation}
 We call $C_{\mathcal{T}}(f)$ the {\it exact asymptotic constant}. Note that \eqref{T_n_lim} gives the exact asymptotic behavior of error of $\mathcal{T}_n$. 

  We can define the following corrected Taylor-based quadrature by
\begin{equation}
	\mathcal{\tilde T}_n(f):=\mathcal{T}_n(f)+\frac{C_{\mathcal{T}}(f)}{n^r}=\mathcal{T}_n(f)+\frac{h^r}{(r+1)!}\Bigl(f^{(r-1)}(b)-f^{(r-1)}(a)\Bigr).
\end{equation}
Note that for any $n\in\mathbb{N}$ the quadrature $\mathcal{\tilde T}_n$ uses only values of $f^{(j)}(t_{i})$ for $j=0,1,\ldots,r-1$, $i=0,1,\ldots,n-1$ (as $\mathcal{T}_n)$ with one additional evaluation of $f^{(r-1)}$ at $b$.

 Give a proof of the fact that for any $f\in C^{r}([a,b])$, $r\in\mathbb{N}$, it holds
	\begin{equation}
		\lim\limits_{n\to +\infty} n^{r} (I(f)-\mathcal{\tilde T}_n(f))=0.
	\end{equation}
Hence, the rate of convergence of the method $\mathcal{\tilde T}_n$ is greater than $n^{-r}$.
\end{itemize}

\chapter{Crude Monte Carlo method}
\label{sec: CMCM}
Throughout this chapter we denote by $(\Omega,\Sigma,\mathbb{P})$ a complete probability space. We denote by $\|\cdot\|_2$ the Euclidean norm in $\mathbb{R}^d$. 

We can roughly say that Monte Carlo method is a direct consequence of the strong law of large numbers which, in its simpler version, goes as follows (see, for example, \cite{GRTAL}, \cite{cohn}).
\begin{thm} (Strong Law of Large Numbers, SLLN)
\label{THM_SLLN}
Let $(X_j)_{j\in\mathbb{N}}$ be a sequence of independent and identically distributed random variables with values in $\mathbb{R}^d$. If $\mathbb{E}\|X_1\|_2<+\infty$, then
\begin{equation}
	\frac{1}{N}\sum\limits_{j=1}^NX_j\to \mathbb{E}(X_1) \ \hbox{as} \ N\to +\infty \ a.s.
\end{equation}
If $\mathbb{E}\|X_1\|_2=+\infty$, then
\begin{equation}
	\limsup\limits_{N\to +\infty}\frac{1}{N}\sum\limits_{j=1}^NX_j=+\infty \ a.s.
\end{equation}
\end{thm} 
We start with the well-known example of Monte Carlo simulation of $\pi$ number.
\section{Monte Carlo simulation of $\pi$}
Let us consider circle $O(C,r)$ with the centre $C=(c_1,c_2)\in\mathbb{R}^2$ and the radius $r\in (0,+\infty)$ inscribed in the square $S$ with the side lenght $2r$. We randomly pick $N$ points $X_j$ from the squre $S$ and we count  the number $K$ of points $X_j$ that lie in the circle $O(C,r)$, i.e., $(X_j)_{j=1,2,\ldots,N}$ is the iid sequence of random variables with the common law $U([s_1-r,s_1+r]\times [s_2-r,s_2+r])$, and
\begin{equation}
	K=\sum\limits_{j=1}^N\mathbf{1}_{O(C,r)}(X_j).
\end{equation}
Note that the ratio of the area of the circle $O$ to the area of the square $S$ is $(\pi r^2)/(2r)^2=\pi/4$ and it is intuitively clear that we have  
\begin{equation}
\label{ratio_pi}
	\frac{\pi}{4}\approx\frac{K}{N},
\end{equation}
as $N\to +\infty$. Hence, we can take
\begin{equation}
	\hat\pi_N=4K/N=\frac{4}{N}\sum\limits_{j=1}^N\mathbf{1}_{O(C,r)}(X_j)
\end{equation}
as the Monte Carlo approximation of $\pi$. Let us now establish rigorously properties of the estimator $\hat\pi_N$. 

Since  value of areas' ratio in \eqref{ratio_pi} is independent of $r$ and the circle's center $C$, we take for simplicity that $c_1=c_2=0$ and $r=1$. Moreover,
\begin{eqnarray}
	&&\mathbb{E}(\mathbf{1}_{O(C,r)}(X_j))=\mathbb{P}(X_j\in O(C,r))=\int\limits_{\mathbb{R}^2}\mathbf{1}_{O(C,r)}(x,y)\cdot\frac{1}{2}\mathbf{1}_{[-1,1]}(x)\cdot\frac{1}{2}\mathbf{1}_{[-1,1]}(y)dxdy\notag\\
	&&=\frac{1}{4}\int\limits_{O(C,r)}dxdy=\frac{1}{4}\lambda_2(O(C,r))=\frac{\pi}{4},
\end{eqnarray}
since $X_j=[X_j^1,X_j^2]$, and $X_j^1,X_j^2$ are independent and uniformly distributed random variables in $[-1,1]$. Hence, by Theorem \ref{THM_SLLN} we get that
\begin{equation}
	\hat\pi_N\to 4\cdot\mathbb{E}(\mathbf{1}_{O(C,r)}(X_1))=\pi
\end{equation}
almost surely as $N\to+\infty$. Furthermore, for all $N\in\mathbb{N}$
\begin{equation}
	\mathbb{E}(\hat\pi_N)=\frac{4}{N}\sum\limits_{j=1}^N \mathbb{P}(X_j\in O(C,r))=4\cdot\mathbb{P}(X_1\in O(C,r))=\pi,
\end{equation}
since $(X_j)_{j=1,2,\ldots,N}$ are identically distributed. Hence, $\hat\pi_N$ is the unbiased estimator of $\pi$. We also compute the mean square error of the estimatior $\hat\pi_N$,
\begin{equation}
	(\mathbb{E}|\pi-\hat\pi_N|^2)^{1/2}=(\mathbb{E}(\hat\pi_N^2)-\pi^2)^{1/2},
\end{equation}
where, by the independence of $X_i$ and $X_j$ for $i\neq j$, it holds that
\begin{eqnarray}
	&&\mathbb{E}(\hat\pi_N^2)=\frac{16}{N^2}\mathbb{E}\Bigl(\sum\limits_{j=1}^N\mathbf{1}_{O(C,r)}(X_j)\Bigr)^2\notag\\
	&&=\frac{16}{N^2}\Bigl(\sum\limits_{j=1}^N \mathbb{P}(X_j\in O(C,r))+\sum\limits_{i\neq j} \mathbb{P}(X_i\in O(C,r))\cdot\mathbb{P}(X_j\in O(C,r))\Bigr)\notag\\
	&&=\frac{16}{N^2}\Bigl(N\cdot\frac{\pi}{4}+(N^2-N)\frac{\pi^2}{16}\Bigr)=\frac{4\pi}{N}+\pi^2(1-\frac{1}{N}).
\end{eqnarray}
Therefore, for all $N\in\mathbb{N}$
\begin{equation}
	(\mathbb{E}|\pi-\hat\pi_N|^2)^{1/2}=(\pi (4-\pi))^{1/2}\cdot N^{-1/2},
\end{equation}
and the mean square error is inversely proportional to the square root of number of samples $N$. 
\section{Definition and basic properties of the crude Monte-Carlo method}
Again we consider the problem of approximating the integral \eqref{I_1} but now we assume even less then before for the integrand $f$. Namely, we assume that $f\in L^2([0,1]^d,\mathcal{B}([0,1]^d),\lambda_d)$. (Formally $L^2([0,1]^d,\mathcal{B}([0,1]^d),\lambda_d)$ consists of equivalence classes and by writing $f\in L^2([0,1]^d,\mathcal{B}([0,1]^d),\lambda_d)$ we mean that we take any representation of the equivalence class. As we will see later this do not cause any problem.) 
By the Schwarz inequality and the fact that $\lambda_d([0,1]^d)=1$ we have $I(|f|)<+\infty$. 
For $N\in\mathbb{N}$ we set
\begin{equation}
\label{crude_MC}
	MC_N(f)=\frac{1}{N}\sum\limits_{j=1}^N f(\xi_j),
\end{equation}
where $(\xi_j)_{j\in\mathbb{N}}$ is a sequence of independent and identically distributed random variables on $(\Omega,\Sigma,\mathbb{P})$, and $\xi_j\sim U([0,1]^d)$, $j\in\mathbb{N}$. Note that the cost of $MC_N(f)$ is equal to $N$ evaluations of $f$.

By the SLLN we have that
\begin{equation}
	MC_N(f)\to \mathbb{E}(f(\xi_1))=I(f) \ \Bigl(=\int\limits_{[0,1]^d}f(x)dx\Bigr) \ a.s.,
\end{equation}
and we can ask how good (in the terms of suitably defined error) the approximation is?
\begin{thm}
\label{mean_sq_MC}
For any $f\in L^2([0,1]^d,\mathcal{B}([0,1]^d),\lambda_d)$ the following holds:
\begin{itemize}
\item [(i)] For all $N\in\mathbb{N}$
\begin{equation}
	\mathbb{E}(MC_N(f))=I(f).
\end{equation}
\item [(ii)] For all $N\in\mathbb{N}$
\begin{equation}
		\Bigl(\mathbb{E}|I(f)-MC_N(f)|^2\Bigr)^{1/2}=(Var(MC_N(f)))^{1/2}=\frac{\sigma(f)}{\sqrt{N}},
\end{equation}
	where
	\begin{equation}
		\sigma^2(f):=I(f^2)-(I(f))^2.
	\end{equation}
\item [(iii)]  $MC_N(f)$ is independent of the choice of the representation of
the equivalence class $f\in  L^2([0,1]^d,\mathcal{B}([0,1]^d),\lambda_d)$ in the following sense: let $f,g$ be two representations of the same equivalence class in $L^2([0,1]^d,\mathcal{B}([0,1]^d),\lambda_d)$. Then 
\begin{equation}
\label{equiv_indep}
	\mathbb{P}\Bigl(\bigcap_{N\in\mathbb{N}} \{MC_N(f)=MC_N(g)\}\Bigr)=1.
\end{equation}
\item [(iv)] $\forall_{f\in L^2([0,1]^d,\mathcal{B}([0,1]^d),\lambda_d)}\Bigl[\sigma(f)=0\Longleftrightarrow \exists_{c\in\mathbb{R}}\lambda_d\Bigl(\{x\in [0,1]^d \ | \ f(x)=c\}\Bigr)=1\Bigr]$.
\end{itemize}
\end{thm}
{\bf Proof.} Note that $(f(\xi_j))_{j\in\mathbb{N}}$ is the sequence of identically distributed and independent random variables. Hence
\begin{equation}
\mathbb{E}(MC_N(f))=\frac{1}{N}\sum\limits_{j=1}^N\mathbb{E}(f(\xi_j))=\frac{1}{N}\cdot N\cdot\mathbb{E}(f(\xi_1))=\frac{1}{\lambda_d([0,1]^d)}\int\limits_{[0,1]^d}f(x)dx=I(f).
\end{equation}
Therefore, by the properties of variance we get
\begin{eqnarray}
	&&\mathbb{E}|I(f)-MC_N(f)|^2=\mathbb{E}\Bigl(|\mathbb{E}(MC_N(f))-MC_N(f)|^2\Bigr)\notag\\
	&&=Var(MC_N(f))=\frac{1}{N^2}\sum\limits_{j=1}^N Var(f(\xi_j))=\frac{1}{N^2}\cdot N\cdot Var(f(\xi_1))=\frac{\sigma^2(f)}{N},\notag
\end{eqnarray}
since
\begin{equation}
	Var(f(\xi_1))=\sigma^2(f).
\end{equation}
Let $f,g$ be two representations of the same equivalence class in $L^2([0,1]^d,\mathcal{B}([0,1]^d),\lambda_d)$ and set
\begin{equation}
	A=\{x\in [0,1]^d \ | \ f(x)\neq g(x)\}.
\end{equation}
Of course $\lambda_d(A)=0$. Hence, for all $N\in\mathbb{N}$ we get by the triangle inequality
\begin{eqnarray}
	&&\|MC_N(f)-MC_N(g)\|_{L^1(\Omega;\mathbb{R})}\leq\frac{1}{N}\sum\limits_{j=1}^N\|f(\xi_j)-g(\xi_j)\|_{L^1(\Omega;\mathbb{R})}\notag\\
	&&=\|(f-g)(\xi_1)\|_{L^1(\Omega;\mathbb{R})}=I(|f-g|)\notag\\
	&&=\int\limits_{[0,1]^d\setminus A}|f(x)-g(x)|dx+\int\limits_A|f(x)-g(x)|dx=0.
\end{eqnarray}
Thereby, for all $N\in\mathbb{N}$
\begin{equation}
	\mathbb{P}\Bigl(\{MC_N(f)=MC_N(g)\}\Bigr)=1,
\end{equation}
which implies \eqref{equiv_indep}.  The proof of $(iv)$ is left as an exercise. \ \ \ $\blacksquare$ \\ \\
\begin{rem}\rm Note that if we assume that $f\in L^2([0,1]^d,\mathcal{B}([0,1]^d),\lambda_d)$,  then any deterministic quadrature rule based on function evaluations is not well-defined. This is due to the fact that the evaluation of an (in fact) equivalence class of a function at a point $x\in [0,1]^d$ is mathematically senseless - in particular the value of such quadrature will depend on the choice of representation of the equivalence class. Having in mind Theorem \ref{mean_sq_MC} (iii), this is an another issue where Monte Carlo quadratures overcome the deterministic ones.
$\square$
\end{rem}
\begin{rem} \rm Formally in Theorem \ref{mean_sq_MC} we have shown that for any equivalence class $[f]\in L^2([0,1]^d,\mathcal{B}([0,1]^d),\lambda_d)$ it mathematically correct to take
\begin{equation}
	MC_N([f]):=MC_N(f).
\end{equation}
$\square$
\end{rem}
\begin{rem}\rm  From (iv) in the Theorem above we see that the method $MC_N$ is exact  only in the trivial case when $f$ is (almost surely) constant function. Therefore, in contrary to deterministic quadrature rules, it does not make sense to provide a definition of  degree of the quadrature $MC_N$. $\square$
\end{rem}
Theorem above tells that the mean-square error of $MC_N(f)$ is $O(N^{-1/2})$ with the cost of $N$ evaluations of $f$ and independently of the dimension $d$. But how good is our approximation for fixed $\omega\in\Omega$, when we compute $MC_N(f)(\omega)$ based on the realizations $\xi_1(\omega),\ldots,\xi_N(\omega)$? The question is very important and natural, since such a situation occurs in practice on a computer. To answer this question we will investigate empirical confidence intervals and the error of the Monte-Carlo method in the almost sure sense. 

\subsection{Theoretical and empirical asymptotic confidence intervals}
\label{subs_emp_conf_int}
Recall that by the Central Limit Theorem \ref{CLT} and \eqref{CLT_F1} it holds
\begin{equation}
\label{ctg_theor_int}
	\lim\limits_{N\to +\infty}\mathbb{P}\Bigl(\Bigl\{\sqrt{N}\cdot\frac{MC_N(f)-I(f)}{\sigma(f)}\in [a,b]\Bigr\}\Bigr)=\Phi(b)-\Phi(a),
\end{equation}
for all $-\infty<a<b<+\infty$, where $\displaystyle{\Phi(t)=\int\limits_{-\infty}^t \frac{1}{\sqrt{2\pi}}e^{-x^2/2}dx}$. Let us fix confidence level $\alpha\in (0,1)$ and let $q_{\alpha}$ be the two-sided $\alpha$-quantile defined as the (unique) solution of the equation
\begin{equation}
	\mathbb{P}(\{|Z|\leq q_{\alpha}\})=\alpha,
\end{equation}
where $Z\sim N(0,1)$. We also have that $\displaystyle{\mathbb{P}(\{|Z|\leq q_{\alpha}\})=\Phi(q_{\alpha})-\Phi(-q_{\alpha})=2\Phi(q_{\alpha})-1}$. For example $q_{\alpha}=2$ corresponds to $\alpha\simeq 0,95$. We set $a:=-q_{\alpha}$, $b=q_{\alpha}$ and we define the following  confidence interval 
\begin{equation}
	\mathcal{J}^{\alpha}_N(f)=\Bigl[MC_N(f)-q_{\alpha}\frac{\sigma(f)}{\sqrt{N}},MC_N(f)+q_{\alpha}\frac{\sigma(f)}{\sqrt{N}}\Bigr].
\end{equation}
By \eqref{ctg_theor_int} we have that
\begin{equation}
	\mathbb{P}(\{I(f)\in \mathcal{J}^{\alpha}_N(f)\})=\mathbb{P}\Bigl(\Bigl\{\sqrt{N}\cdot\frac{|MC_N(f)-I(f)|}{\sigma(f)}\leq q_{\alpha}\Bigr\}\Bigr)\to \Phi(q_{\alpha})-\Phi(-q_{\alpha})=\alpha,
\end{equation}
as $N\to +\infty$.

We call $\mathcal{J}^{\alpha}_N$ the {\it theoretical confidence interval}, since it uses the value of $\sigma(f)$ which, in practice, is usually not known. And now comes the {\it magic trick}: {the variance $\sigma^2(f)$ can be simulated in parallel with $MC_N(f)$}. Namely, we can use the following estimator of $\sigma^2(f)$
\begin{eqnarray}
	\widehat{\sigma^2}_N(f)&&=\frac{1}{N-1}\sum\limits_{j=1}^N\Bigl(f(\xi_j)-MC_N(f)\Bigr)^2\\
	\label{var_est_sl}
	&&=\frac{1}{N-1}\Bigl(\sum\limits_{j=1}^N\Bigl(f(\xi_j)\Bigr)^2-N\cdot\Bigl(MC_N(f)\Bigr)^2\Bigr).
\end{eqnarray}
(Although single-loop implementation of the equation \eqref{var_est_sl} is straightforward, it should not be used in practice because of the danger of cancellations. Numerically stable recursive one-pass  formula will be presented in the sequel.)
Again by the SLLN we have that
\begin{equation}
\label{slln_emp_var}
	\widehat{\sigma^2_N}(f)\to \sigma^2(f) 
\end{equation}
almost surely as $N\to+\infty$, and for all $N\in\mathbb{N}$
\begin{equation}
\label{exp_emp_var}
	\mathbb{E}(\widehat{\sigma^2_N}(f))=\sigma^2(f).
\end{equation}
By using Central Limit Theorem together with Slutsky's lemma we get that
\begin{equation}
\label{ctg_emp_var}
	\lim\limits_{N\to +\infty}\mathbb{P}\Bigl(\Bigl\{\sqrt{N}\cdot\frac{MC_N(f)-I(f)}{\sqrt{\widehat{\sigma^2_N}(f)}}\in [a,b]\Bigr\}\Bigr)=\Phi(b)-\Phi(a) 
\end{equation}
for all $-\infty<a<b<+\infty$. The {\it empirical confidence interval} is thus given by
\begin{equation}
	\mathcal{\hat J}^{\alpha}_N(f)=\Bigl[MC_N(f)-q_{\alpha}\sqrt{\widehat{\sigma^2_N}(f) / N},MC_N(f)+q_{\alpha}\sqrt{\widehat{\sigma^2_N}(f) / N}\Bigr],
\end{equation}
and by \eqref{ctg_emp_var} we have that
\begin{equation}
	\mathbb{P}(\{I(f)\in \mathcal{\hat J}^{\alpha}_N(f)\})\to\alpha, \ \hbox{as} \ N\to+\infty. 
\end{equation}
\subsection{Hoeffding inequality, probability of large deviations and non-asymptotic confidence intervals}
In the case when $f$ is a bounded Borel function it is possible to show an exponential decay of the
probability of the so-called {\it exceptional set}. Namely, we have the following result that follows from the  Hoeffding inequality (see Theorems \ref{H_ineq}, \ref{HA_ineq}).
\begin{prop}  
\label{app_HA1}
Let $(\xi_j)_{j\in\mathbb{N}}$ be an i.i.d sequence of random vectors and $f:\mathbb{R}^d\to\mathbb{R}$ a bounded Borel function. Then for all $\varepsilon\in (0,+\infty)$ and $N\in\mathbb{N}$
\begin{equation}
	\mathbb{P}\Bigl(\Bigl\{\Bigl|\frac{1}{N}\sum\limits_{j=1}^Nf(\xi_j)-\mathbb{E}f(\xi_1)\Bigl|>\varepsilon \Bigr\}\Bigr)\leq 2\exp\Bigl(-\frac{\varepsilon^2 N}{8\|f\|^2_{\infty}}\Bigr),
\end{equation}	
where $\|f\|_{\infty}=\sup\limits_{x\in \mathbb{R}^d}|f(x)|$.
\end{prop}
(We left proof of  Proposition \ref{app_HA1} as an exercise.)

For \eqref{crude_MC} we directly get what follows.
\begin{cor} 
\label{exp_dec_1}
Let $f:[0,1]^d\to\mathbb{R}$ be a bounded Borel function. Then for all $C\in (0,+\infty)$, $N\in\mathbb{N}$
\begin{equation}
	\mathbb{P}\Bigl(\Bigl\{|MC_N(f)-I(f)|>C\cdot N^{-1/2} \Bigr\}\Bigr)\leq 2\exp\Bigl(-\frac{C^2}{8\|f\|^2_{\infty}}\Bigr),
\end{equation}
where $\displaystyle{\|f\|_{\infty}=\sup\limits_{x\in [0,1]^d}|f(x)|}$.
\end{cor}
We can see that the probability of the exceptional set tends to zero exponentially fast with $C\to +\infty$.

The second nice thing that we can get from the Hoeffding inequality are the {\it non-asymptotic confidence intervals}. Namely, for any $\delta\in (0,1)$ let us take 
\begin{equation}
	C(\delta)=\sqrt{8\ln(2/\delta)}\cdot \|f\|_{\infty}, 
\end{equation} 
provided that $\|f\|_{\infty}$ (or its 'good' approximation) is known. Then, by Corollary \ref{exp_dec_1} we obtain for all $\delta\in (0,1)$, $N\in\mathbb{N}$
\begin{equation}
\label{H_int_1}
	\mathbb{P}\Bigl(\Bigl\{|MC_N(f)-I(f)|> C(\delta)\cdot N^{-1/2} \Bigr\}\Bigr)\leq \delta.
\end{equation}
Hence, for  all $\delta\in (0,1)$, $N\in\mathbb{N}$
\begin{equation}
\label{H_int_2}
	I(f)\in \Bigl[MC_N(f)-\frac{C(\delta)}{\sqrt{N}}, MC_N(f)+\frac{C(\delta)}{\sqrt{N}}\Bigr]
\end{equation}
with probability at least $1-\delta$. This probability of success does not depend on $N$, which means that we can design confidence intervals for Monte Carlo simulation uniformly in the number of samples $N$.

Let us now fix $\varepsilon,\delta\in (0,1)$. From \eqref{H_int_1} we have that  if we take
\begin{equation}
	N=\Theta\Biggl(\Bigl(\frac{1}{\varepsilon}\Bigr)^2\ln(\frac{2}{\delta})\Biggr)
\end{equation}
then the probability that the error $|MC_N(f)-I(f)|$ is $O(\varepsilon)$ is at least $1-\delta$. Moreover, this error is achieved by $\displaystyle{O\Biggl(\Bigl(\frac{1}{\varepsilon}\Bigr)^2\ln(\frac{2}{\delta})\Biggr)}$ evaluations of $f$.
\subsection{Almost sure convergence}
In this section we will investigate convergence of the crude Monte Carlo method in the almost sure sense. For this we need to prove the following auxiliary result.
\begin{thm} 
\label{MC_Lp_est}
Let $f\in L^p([0,1]^d,\mathcal{B}([0,1]^d),\lambda_d)$ for some $p\in [2,+\infty)$. Then there exists $C_p(f)\in (0,+\infty)$ such that  for all $N\in\mathbb{N}$ it holds
\begin{equation}
\label{MC_p_1}
	\Bigl(\mathbb{E}|I(f)-MC_N(f)|^p\Bigr)^{1/p}\leq C_p(f)\cdot N^{-1/2},
\end{equation}
and
\begin{equation}
\label{MC_q_1}
	\Bigl(\mathbb{E}|I(f)-MC_N(f)|^q\Bigr)^{1/q}\leq C_p(f)\cdot N^{-1/2}
\end{equation}
for all $q\in (0,p]$.
\end{thm}
{\bf Proof.}
Let
\begin{equation}
	\mathcal{G}_j=\sigma(\xi_1,\ldots,\xi_j), \ j\geq 1,
\end{equation}
with $\mathcal{G}_0=\{\emptyset,\Omega\}$. Moreover, we set
\begin{equation}
\label{DT_mart_1}
	Z_j=\sum\limits_{k=1}^j\Bigl(f(\xi_k)-I(f)\Bigr), \ 1\leq j\leq N,
\end{equation}
and $Z_0:=0$. The sequence $(Z_j,\mathcal{G}_j)_{j=0,1,\ldots,N}$ is a discrete-time martingale such that $\displaystyle{\max\limits_{1\leq j\leq N}\mathbb{E}|Z_j|^p<+\infty}$ (see exercise 11.). Moreover, its quadratic variation is given by
\begin{equation}
	[Z]_j=\sum\limits_{k=1}^j |Z_k-Z_{k-1}|^2=\sum\limits_{k=1}^j\delta_k^2, \ 1\leq j\leq N.
\end{equation} 
where $\delta_k:=f(\xi_k)-I(f)$. Therefore, by the Burkholder-Davis-Gundy inequality (Theorem \ref{BDG_DISC}) we get
\begin{equation}
	\mathbb{E}\Bigl(\max\limits_{0\leq j\leq N}|Z_j|^p\Bigr)\leq C^p_p\cdot\mathbb{E}([Z]^{p/2}_N)=C^p_p\cdot\mathbb{E}\Bigl(\sum\limits_{k=1}^N\delta_k^2\Bigr)^{p/2}.
\end{equation} 
Since $p\in [2,+\infty)$, the function $[0,+\infty)\ni x\to x^{p/2}$ is convex and by the Jensen inequality we get
\begin{equation}
	\Bigl(\sum\limits_{k=1}^N\delta_k^2\Bigr)^{p/2}\leq N^{\frac{p}{2}-1}\sum\limits_{k=1}^N |\delta_k|^p.
\end{equation}
Hence,
\begin{equation}
	\mathbb{E}([Z]^{p/2}_N)\leq N^{\frac{p}{2}-1}\sum\limits_{k=1}^N \mathbb{E}|\delta_k|^p= N^{\frac{p}{2}}\cdot\int\limits_{[0,1]^d}|f(x)-I(f)|^p dx,
\end{equation}
and
\begin{equation}
	\mathbb{E}|I(f)-MC_N(f)|^p=\frac{1}{N^p}\cdot\mathbb{E}|Z_N|^p\leq  C^p_p\cdot N^{-\frac{p}{2}}\cdot\int\limits_{[0,1]^d}|f(x)-I(f)|^p dx,
\end{equation}
which ends the proof \eqref{MC_p_1} with $\displaystyle{C_p(f)= C_p\cdot\Bigl(\int\limits_{[0,1]^d}|f(x)-I(f)|^p dx \Bigr)^{1/p} <+\infty}$. For $q\in (0,p]$ we have from the Lyapunov inequality that
\begin{equation}
	\Bigl(\mathbb{E}|I(f)-MC_N(f)|^q\Bigr)^{1/q}\leq  \Bigl(\mathbb{E}|I(f)-MC_N(f)|^p\Bigr)^{1/p}\leq C_p(f)\cdot N^{-1/2}.
\end{equation}
This ends the proof of \eqref{MC_q_1}. \ \ \ $\blacksquare$ \\ \\
As a corollary we obtain the following main result of the section.
\begin{thm}
\label{AS_CONV_mc}
 Let $\displaystyle{f\in \bigcap\limits_{p\in [2,+\infty)} L^p([0,1]^d,\mathcal{B}([0,1]^d),\lambda_d)}$. Then for all $\varepsilon\in (0,1/2)$ there exists $\displaystyle{\eta_{\varepsilon}(f)\in \bigcap\limits_{p\in [1,+\infty)} L^p(\Omega)}$ such that
\begin{equation}
	\mathbb{P}\Bigl(\bigcap_{N\in\mathbb{N}}\Bigl\{|I(f)-MC_N(f)|\leq \eta_{\varepsilon}(f)\cdot N^{-\frac{1}{2}+\varepsilon}\Bigr\}\Bigr)=1.
\end{equation} 
\end{thm}
\noindent
{\bf Proof.} From Theorem \ref{MC_Lp_est} we have that for all $p\in [1,+\infty)$ and all $N\in\mathbb{N}$
\begin{equation}
	\Bigl(\mathbb{E}|I(f)-MC_N(f)|^p\Bigr)^{1/p}\leq C_{\max\{2,p\}}(f)\cdot N^{-1/2}.
\end{equation}
Applying Theorem \ref{BC_asconv} we get the thesis. \ \ \ $\blacksquare$\\ \\
Under the assumptions of Theorem \ref{AS_CONV_mc} we only know that $\mathbb{E}|\eta_{\varepsilon}(f)|^p<+\infty$ and currently there are no further (more accurate) estimates available in the literature. 
\section{Variance reduction}
In this section we discuss possibilities of variance reduction by two (seemingly the most common) approaches:
\begin{itemize}
    \item the use of control variates,
    \item stratified sampling.
\end{itemize}
\subsection{Control variates}
Let us assume that we are able to decompose function $f\in L^2([0,1]^d, \mathcal{B}([0,1]^d),\lambda_d)$ in the following way
\begin{equation}
\label{decom_CV}
	f=g+(f-g),
\end{equation}
where $g$ is such that we can compute the integral $I(g)$ explicitly. (For example, $g$ can be  a piecewise polynomial that interpolates $f$. However, in general it can be any function that is in $L^2([0,1]^d, \mathcal{B}([0,1]^d),\lambda_d)$.) We now take
\begin{equation}
	\widetilde {MC}_N(f):=I(g)+MC_N(f-g).
\end{equation}
Applying Theorem \ref{mean_sq_MC} we get 
\begin{equation}
	\Bigl(\mathbb{E}|I(f)-\widetilde {MC}_N(f)|^2\Bigr)^{1/2}=(Var(MC_N(f-g)))^{1/2}=\sigma(f-g)/\sqrt{N}.
\end{equation}
Hence, we look for  $g$, satisfying \eqref{decom_CV}, such that $I(g)$ is known and $\sigma(f-g)<\sigma(f)$. Then $g$ is called the {\it control variate} for $f$. 
\\
We now present a good choice of $g$ in the scalar case (i.e., $\underline{d=1}$)  for $f\in C^r ([0,1])$, $r\in \mathbb{N}=\{1,2,\ldots\}$, with  H\"older continuous $r$th derivative. Firstly, we prove the following useful lemma, see \cite{daun} for an alternative method of proof.
\begin{lem} 
\label{err_Lag_Hold}
Let $f\in C^{r}([a,b])$ with $r\in \mathbb{N}=\{1,2,\ldots\}$, $-\infty<a<b<+\infty$, and let us assume that there exist $H\in (0,+\infty)$, $\varrho\in (0,1]$ such that for all $x,y\in [a,b]$
\begin{equation}
	|f^{(r)}(x)-f^{(r)}(y)|\leq H|x-y|^{\varrho}.
\end{equation}
Let $w_r$ be the Lagrange interpolation polynomial of degree at most $r$, that interpolates $f$ at $r+1$ points $x_0<x_1<\ldots<x_r$ such that $x_i\in [a,b]$ for $i=0,1,\ldots,r$. Then for all $x\in [a,b]$
\begin{equation}
	|f(x)-w_r(x)|\leq\frac{H}{r!}(b-a)^{r+\varrho}.
\end{equation}
\end{lem}
{\bf Proof.} If $x=x_j$ for some $j$ then $|f(x)-w_r(x)|=0$. So let us assume that $x\in [a,b] \setminus\{x_0,x_1,\ldots,x_r\}$. Then by the classical error formula for the Lagrange interpolation (see, for example, \cite{jankow}) we get for all $x\in [a,b]$
\begin{equation}
\label{est_Lag_1}
	f(x)-w_r(x)=f[x_0,\ldots,x_r,x]\cdot\prod_{k=0}^{r}(x-x_k),
\end{equation}
with
\begin{equation}
\label{est_Lag_2}
	w_r(x)=\sum\limits_{j=0}^r f[x_0,x_1,\ldots,x_j]\cdot\prod_{k=0}^{j-1}(x-x_k)=\sum\limits_{j=0}^r f(x_j)\cdot\prod\limits_{i=0, i\neq j}^r\frac{x-x_i}{x_j-x_i},
\end{equation}
where $f[x_0,\ldots,x_k,x]$ is the divided difference based on the nodes $x_0,\ldots,x_k,x$. Recall that for pairwise distinct points $x_0,\ldots,x_r,x$ we also have
\begin{equation}
\label{est_Lag_3}
	f[x_0,\ldots,x_r,x]=\frac{f[x_1,\ldots,x_r,x]-f[x_0,\ldots,x_{r-1},x_r]}{x-x_0}.
\end{equation}
From \eqref{est_Lag_1} and \eqref{est_Lag_3} we obtain
\begin{equation}
\label{est_Lag_4}
	f(x)-w_r(x)=\Bigl(f[x_1,\ldots,x_r,x]-f[x_0,\ldots,x_{r-1},x_r]\Bigr)\cdot\prod\limits_{j=1}^r(x-x_j).
\end{equation}
Since $f^{(r)}$ is continuous on $[a,b]$ we have that for any $x\in [a,b]$ there exist $\xi_x,\eta\in [a,b]$ such that
\begin{eqnarray}
\label{est_Lag_5}
	&&f[x_1,\ldots,x_r,x]=\frac{f^{(r)}(\xi_x)}{r!},\notag\\
	&&f[x_0,\ldots,x_{r-1},x_r]=\frac{f^{(r)}(\eta)}{r!}.
\end{eqnarray}
Hence, by \eqref{est_Lag_4} and \eqref{est_Lag_5} we get for any $x\in [a,b]$
\begin{eqnarray}
	|f(x)-w_r(x)|&=&\frac{1}{r!}\cdot |f^{(r)}(\xi_x)-f^{(r)}(\eta)|\cdot\prod\limits_{j=1}^r|x-x_j|\notag\\
			&\leq&\frac{H}{r!}|\xi_x-\eta|^{\varrho}\cdot (b-a)^{r}\leq \frac{H}{r!}\cdot (b-a)^{r+\varrho},
\end{eqnarray}
and the proof is finished. \ \ \ $\blacksquare$ \\ \\
We now take $a=0,b=1$, $x_i=ih$, $i=0,1,\ldots,n$, where $h=1/n$, $n\in\mathbb{N}$. On each subinterval $[x_i,x_{i+1}]$, $i=0,1,\ldots,n-1$ we interpolate $f$ at the equidistant points  $x_i^{(j)}=x_i+\frac{h}{r}j$, $j=0,1,\ldots,r$, by the Lagrange polynomial  $w_{i,n,r}$. The control variate $g_n:[0,1]\to\mathbb{R}$ is defined as follows :
\begin{equation}
	 g_n(x)=w_{i,n,r}(x)=\sum\limits_{j=0}^r f(x_i^{(j)})\cdot\prod\limits_{k=0, k\neq j}^r\frac{x-x^{(k)}_i}{x^{(j)}_i-x^{(k)}_i}, \ x\in [x_i,x_{i+1}], \ i=0,1,\ldots,n-1.
\end{equation}
Of course $g_n\in C([0,1])$, $I(g_n)$ can be computed explicitly (since it is a piecewise polynomial function), and by Lemma \ref{err_Lag_Hold} we have that
\begin{equation}
	\|f-g_n\|_{\infty}=\sup\limits_{x\in [0,1]}|f(x)-g_n(x)|=\max\limits_{0\leq i\leq n-1}\sup\limits_{x_i\leq x\leq x_{i+1}}|f(x)-w_{i,n,r}(x)|\leq \frac{H}{r!}\cdot n^{-(r+\varrho)}.
\end{equation}
In addition
\begin{equation}
\label{est_lag_sup_1}
	\sigma^2(f-g_n)\leq I(|f-g_n|^2)\leq \|f-g_n\|_{\infty}^2\leq\Bigl(\frac{H}{r!}\Bigr)^2\cdot n^{-2(r+\varrho)},
\end{equation}
and $\sigma(f-g_n)<\sigma(f)$ for sufficiently large $n$. Hence
\begin{equation}
	\Bigl(\mathbb{E}|I(f)-\widetilde {MC}_{N,n}(f)|^2\Bigr)^{1/2}=\sigma(f-g_n)/\sqrt{N}\leq \frac{H}{r!}\cdot n^{-(r+\varrho)}\cdot N^{-1/2},
\end{equation}
where
\begin{equation}
\label{wt_mc_1}
	\widetilde {MC}_{N,n}(f)=I(g_n)+\frac{1}{N}\sum\limits_{j=1}^N (f-g_n)(\xi_j)
\end{equation}
and $(\xi_j)_{j\in\mathbb{N}}$ is an i.i.d sequence of random variables, uniformly distributed on $[0,1]$. The informational cost of $\widetilde {MC}_{N,n}(f)$ is $O(n+N)$ evaluations of $f$. For $n=N$ we arrive at the method $\widetilde {MC}_{N,N}$ that has the error
\begin{equation}
\Bigl(\mathbb{E}|I(f)-\widetilde {MC}_{N,N}(f)|^2\Bigr)^{1/2}\leq \frac{H}{r!}\cdot N^{-(r+\varrho+\frac{1}{2})},
\end{equation}
with the cost of $O(N)$ evaluations of $f$. Note that $\widetilde {MC}_{N,N}$ evaluates $f$ at points chosen both in deterministic and randomized way. As we can see, in this case we can significantly decrease the error from $O(N^{-1/2})$ to $O(N^{-(r+\varrho+\frac{1}{2})})$, and, essentially, the cost $O(N)$ remains the same. 
\begin{rem}
    For a definition of the control variate $g_n$ in the case when $f\in C([0,1])$ (i.e., $r=0$) and $f$ is H\"older continuous with the exponent  $\varrho\in (0,1]$ see \eqref{Taylor_cv_def} in Exercise 15.
\end{rem}
\begin{rem}
    The rate $O(N^{-(r+\varrho+\frac{1}{2})})$ is optimal, see Proposition 1.3.9, page 34 in \cite{Nov1}.
\end{rem}
Having the results above we now investigate the probability of the exceptional set for the method $\widetilde{MC}_{N,n}$.

Note that for $h_n=f-g_n$
\begin{equation}
I(f)-\widetilde {MC}_{N,n}(f)=I(h_n)-MC_N(h_n)=-\frac{1}{N}Z^n_N,
\end{equation}
where
\begin{eqnarray}
	&& Z_j^n=\sum\limits_{k=1}^j \Bigl(h_n(\xi_j)-I(h_n)\Bigr), \ 1\leq j\leq N,\\
	&&Z^n_0=0.
\end{eqnarray}
Let $\mathcal{G}_0=\{\emptyset,\Omega\}$ and $\mathcal{G}_j=\sigma(\xi_1,\ldots,\xi_j)$ for $j\geq 1$. Since
\begin{equation}
	\max_{1\leq j\leq N}|Z_j^n|\leq 2N\|f-g_n\|_{\infty}<+\infty,
\end{equation}
and for $j=0,1,\ldots,N-1$
\begin{equation}
	\mathbb{E}\Bigl( Z^n_{j+1} - Z_j^n \ | \ \mathcal{G}_j\Bigr) = \mathbb{E}(h_n(\xi_{j+1}))-I(h_n)=0, 
\end{equation}
we get that for all $n\in\mathbb{N}$ the process $(Z_j^n,\mathcal{G}_j)_{1\leq j\leq N}$ is a discrete-time martingale. Moreover, for all $j=1,2,\ldots,N$ we have 
\begin{equation}
	|Z^n_j-Z^n_{j-1}|\leq |h_n(\xi_j)|+|I(h_n)|\leq 2\|f-g_n\|_{\infty},
\end{equation}
hence from the Hoeffding inequality (see Theorem \ref{HA_ineq}) we get for any $\varepsilon>0$, $N,n\in\mathbb{N}$ that
\begin{eqnarray}
	&&\mathbb{P}\Bigl(|I(f)-\widetilde {MC}_{N,n}(f)|>\varepsilon\Bigr)=\mathbb{P}\Bigl(|I(h_n)- {MC}_{N}(h_n)|>\varepsilon\Bigr)\notag\\
	&&=\mathbb{P}\Bigl(|Z_N^n|>N\cdot\varepsilon\Bigr)\leq2\exp\Bigl(-\frac{\varepsilon^2 N}{8\|f-g_n\|^2_{\infty}}\Bigr),
\end{eqnarray}
and by \eqref{est_lag_sup_1} we obtain
\begin{equation}
\mathbb{P}\Bigl(|I(f)-\widetilde {MC}_{N,n}(f)|>\varepsilon\Bigr)\leq 2\exp\Bigl(-\frac{\varepsilon^2\cdot N \cdot n^{2(r+\varrho)}}{8C_1}\Bigr),
\end{equation}
where $C_1=(H/r!)^2$. Therefore, for any $C\in (0,+\infty)$, $N,n\in\mathbb{N}$
\begin{equation}
\mathbb{P}\Bigl(|I(f)-\widetilde {MC}_{N,n}(f)|>C \cdot n^{-(r+\varrho)}\cdot N^{-1/2}\Bigr)\leq 2\exp\Bigl(-\frac{C^2}{8C_1}\Bigr).
\end{equation}
For $\delta\in (0,1)$ set
\begin{equation}
	C(\delta)=\sqrt{8\ln (2/\delta)} \cdot (H/r!).
\end{equation}
Then for any $\delta\in (0,1)$, $N,n\in\mathbb{N}$ we have
\begin{equation}
\label{vr_H_int_1}
	\mathbb{P}\Bigl(|I(f)-\widetilde {MC}_{N,n}(f)|>C(\delta)\cdot n^{-(r+\varrho)}\cdot N^{-1/2}\Bigr)\leq \delta
\end{equation}
and
\begin{equation}
	I(f)\in\Bigl[\widetilde {MC}_{N,n}(f)-\frac{C(\delta)}{n^{r+\varrho
	}\cdot N^{1/2}},\widetilde {MC}_{N,n}(f)+\frac{C(\delta)}{n^{r+\varrho
	}\cdot N^{1/2}}\Bigr]
\end{equation}
with the probability at least $1-\delta$.

Let us now fix $\varepsilon,\delta\in (0,1)$ and take $n=N$. From \eqref{vr_H_int_1} we have that  if we set
\begin{equation}
	N=\Theta\Biggl(\Bigl(\frac{1}{\varepsilon}\Bigr)^{\frac{1}{r+\varrho+1/2}}\Bigl(\ln(\frac{2}{\delta})\Bigr)^{\frac{1}{2(r+\varrho)+1}}\Biggr)
\end{equation}
then the probability that the error $|\widetilde MC_{N,N}(f)-I(f)|$ is $O(\varepsilon)$ is at least $1-\delta$. Moreover, this error is achieved by $\displaystyle{O\Biggl(\Bigl(\frac{1}{\varepsilon}\Bigr)^{\frac{1}{r+\varrho+1/2}}\Bigl(\ln(\frac{2}{\delta})\Bigr)^{\frac{1}{2(r+\varrho)+1}}\Biggr)}$ evaluations of $f$.

We have shown the possible definition of control variate $g$ for scalar integration of functions from $C^r([0,1])$ with H\"older $r$th derivative $f^{(r)}$. We now show how to decrease the variance by choosing suitable $g$ for multidimensional integration of functions from the H\"older class
\begin{equation}
	\mathcal{F}^{\varrho,d}_L=\{f:[0,1]^d\to\mathbb{R} \ | \ |f(x)-f(y)|\leq L\cdot\|x-y\|^{\varrho}_{\infty}, \ x,y\in [0,1]^d\},
\end{equation}
where $L\in (0,+\infty)$, $\varrho\in (0,1]$ are the class parameters and $\|\cdot\|_{\infty}$ is the maximum norm. For a fixed $N\in\mathbb{N}$, and assuming (without the loss of generality) that $N^{1/d}\in\mathbb{N}$, we take
\begin{equation}
g(x)=\sum_{i_1,\ldots,i_d=0}^{N^{1/d}-1}f(\xi_{i_1,\ldots,i_d})\cdot\mathbf{1}_{A_{i_1,\ldots,i_d}}(x),
\end{equation}
 where $\xi_{i_1,\ldots,i_d}$, $A_{i_1,\ldots,i_d}$ are defined as in \eqref{I_f_approx} and \eqref{def_Ais}, respectively. Note that $g$ uses $N$ values of $f$ and, in fact, we have divided $[0,1]^d$ into $N$ sub-boxes $A_{i_1,\ldots,i_d}$ such that \eqref{div_prop_1} and \eqref{div_prop_2} hold. Then, proceeding analogously as in the proof of Theorem \ref{RECT_MD_ERR_1}, we have
 \begin{equation}
     \sigma^2(f-g)\leq I[(f-g)^2]=\sum_{i_1,\ldots,i_d=0}^{N^{1/d}-1}\int\limits_{A_{i_1,\ldots,i_d}}|f(x)-f(\xi_{i_1,\ldots,i_d})|^2 dx\leq L^2 \cdot N^{-2\varrho/d}.
 \end{equation}
 Moreover
 \begin{equation}
     I(g)=\frac{1}{N}\sum_{i_1,\ldots,i_d=0}^{N^{1/d}-1}f(\xi_{i_1,\ldots,i_d})=Q^d_{N^{1/d}}(f),
 \end{equation}
 with $Q^d_{N^{1/d}}(f)$ defined by \eqref{mul_dim_Q_def}. Hence, for
 \begin{equation}
    \widetilde {MC}_N(f)=Q^d_{N^{1/d}}(f)+\frac{1}{N}\sum_{j=1}^N (f-g)(\xi_j),
 \end{equation}
 where $(\xi_j)_{j=1,\ldots,N}$ is an iid sequence sampled from $U([0,1]^d)$, we have that
 \begin{equation}
     \Bigl(\mathbb{E}|I(f)-\widetilde {MC}_N(f)|^2\Bigr)^{1/2}=\sigma(f-g)/\sqrt{N}\leq L\cdot N^{-(\frac{\varrho}{d}+\frac{1}{2})}.
 \end{equation}
 Note also that $ \widetilde {MC}_N(f)$ uses $2N$ values of $f$.
\subsection{Stratified sampling}
In this approach we divide the box $D=[0,1]^d$ (or in the more general setting, the integration domain) into $K$ sub-boxes $D_i$, $i=1,2,\ldots,K$, in such a way that
\begin{equation}
    \bigcup_{i=1}^K D_i = [0,1]^d,
\end{equation}
and for all $i\neq j$
\begin{equation}
    \lambda_d(D_i\cap D_j)=0. 
\end{equation}
We then have
\begin{eqnarray}
    I(f)=\sum_{j=1}^K I_i(f),
\end{eqnarray}
where $\displaystyle{I_i(f)=I(\mathbf{1}_{D_i}f)=\int\limits_{D_i}f(x)dx}$. Let us take
\begin{equation}
\label{sts_def_1}
    \widebar{MC}_N(f)=\sum_{i=1}^K \widebar{MC}^{(i)}_{N_i}(f),
\end{equation}
with $N=\sum_{i=1}^K N_i$ and 
\begin{equation}
    \widebar{MC}^{(i)}_{N_i}(f)=\frac{vol(D_i)}{N_i}\sum\limits_{j=1}^{N_i}f(\xi_j^i),
\end{equation}
and $(\xi_j^i)_{j=1,\ldots,N_i,i=1,\ldots,K}$ are independent random vectors, where for each $i$ the random vectors $(\xi^i_j)_{j=1,\ldots,N_i}$ are uniformly distributed in $D_i$. Then for all $i=1,2,\ldots,K$  
\begin{eqnarray}
\label{unb_sts_1}
    \mathbb{E}(\widebar{MC}^{(i)}_{N_i}(f))=\frac{vol(D_i)}{N_i}\sum\limits_{j=1}^{N_i}\mathbb{E}(f(\xi_j^i))=I_i(f),
\end{eqnarray}
and hence
\begin{eqnarray}
    \mathbb{E}(\widebar{MC}_N(f))=\sum_{i=1}^K I_i(f)=I(f).
\end{eqnarray}
Since $\displaystyle\{\widebar{MC}^{(1)}_{N_1}(f),\ldots,\widebar{MC}^{(K)}_{N_K}(f)\}$ is a family of independent random variables, we get  from \eqref{unb_sts_1} and the properties of variance
\begin{equation}
    \mathbb{E}|I(f)-\widebar{MC}_N(f)|^2=Var(\widebar{MC}_N(f))=\sum_{i=1}^K Var(\widebar{MC}^{(i)}_{N_i}(f)),
\end{equation}
where, by the independence of $\{f(\xi_1^i),\ldots,f(\xi_{N_i}^i)\}$ for given $i$, we have
\begin{equation}
    Var(\widebar{MC}^{(i)}_{N_i}(f))=\frac{vol^2(D_i)}{N^2_i}\sum\limits_{j=1}^{N_i}Var(f(\xi_j^i)).
\end{equation}
Note that
\begin{eqnarray}
    Var(f(\xi_j^i))=\mathbb{E}(f(\xi_j^i))^2-(\mathbb{E}(f(\xi_j^i)))^2=\frac{1}{vol(D_i)}I_i(f^2)-\frac{1}{vol^2(D_i)}(I_i(f))^2
\end{eqnarray}
for all $j=1,\ldots,N_i$. Therefore
\begin{equation}
     Var(\widebar{MC}^{(i)}_{N_i}(f))=\frac{1}{N_i}\Bigl(vol(D_i)\cdot I_i(f^2)-(I_i(f))^2\Bigr),
\end{equation}
and
\begin{equation}
    \mathbb{E}|I(f)-\widebar{MC}_N(f)|^2=\sum_{i=1}^K\frac{1}{N_i}\Bigl(vol(D_i)\cdot I_i(f^2)-(I_i(f))^2\Bigr).
\end{equation}
Assuming  for $i=1,\ldots,K$ that
\begin{equation}
    N_i=vol(D_i)\cdot N\in\mathbb{N},
\end{equation}
we obtain
\begin{equation}
\label{sts_err_2}
    \mathbb{E}|I(f)-\widebar{MC}_N(f)|^2=\frac{1}{N}\sum_{i=1}^K\Bigl(I_i(f^2)-\frac{(I_i(f))^2}{vol(D_i)}\Bigr)=\frac{1}{N}\Bigl(I(f^2)-\sum_{i=1}^K\frac{(I_i(f))^2}{vol(D_i)}\Bigr).
\end{equation}
We compare the error \eqref{sts_err_2} of the stratified sampling method \eqref{sts_def_1} with the error of the crude Monte Carlo method
\begin{eqnarray}
    MC_N(f)=\frac{1}{N}\sum_{j=1}^N f(\xi_j),
\end{eqnarray}
where $(\xi_j)_{j=1,\ldots,N}$ is an iid sequence sampled from $U([0,1]^d)$, and $N=\sum_{i=1}^K N_i$. Note that each algorithm $\widebar{MC}_N(f)$, $MC_N(f)$ uses $N$ (random) evaluations of $f$. From the Cauchy-Schwarz inequality
\begin{equation}
    (I(f))^2=\Bigl(\sum_{i=1}^K (vol(D_i))^{1/2}\cdot\frac{I_i(f)}{(vol(D_i))^{1/2}}\Bigr)^2\leq \sum_{i=1}^K vol(D_i)\cdot\sum_{i=1}^K\frac{(I_i(f))^2}{vol(D_i)}=\sum_{i=1}^K\frac{(I_i(f))^2}{vol(D_i)},
\end{equation}
since $\sum_{i=1}^K vol(D_i)=vol(D)=1$. Hence
\begin{equation}
\label{comp_errs}
    \|I(f)-MC_N(f)\|_{L^2(\Omega)}=\frac{(I(f^2)-(I(f))^2)^{1/2}}{\sqrt{N}}\geq \frac{\Bigl(I(f^2)-\sum_{i=1}^K\frac{(I_i(f))^2}{vol(D_i)}\Bigr)^{1/2}}{\sqrt{N}}=\|I(f)-\widebar{MC}_N(f)\|_{L^2(\Omega)},
\end{equation}
and the error of $\widebar{MC}_N(f)$ is not bigger than the error of $MC_N(f)$. Moreover, by the properties of Cauchy-Schwarz inequality, the equality in \eqref{comp_errs} holds if and only if  $I_i(f)/vol(D_i)=const$ for $i=1,\ldots,K$.

We now consider the following (exemplary) class of H\"older continuous functions $f:[0,1]^d\to\mathbb{R}$ in order to see how much we can gain by using stratified sampling. Namely, let
\begin{equation}
	\mathcal{F}^{\varrho,d}_L=\{f:[0,1]^d\to\mathbb{R} \ | \ |f(x)-f(y)|\leq L\cdot\|x-y\|^{\varrho}_{\infty}, \ x,y\in [0,1]^d\},
\end{equation}
where $L\in (0,+\infty)$, $\varrho\in (0,1]$ are the class parameters and $\|\cdot\|_{\infty}$ is the maximum norm. 

We have
\begin{equation}
    Var(f(\xi_j^i))=\mathbb{E}\Bigl(f(\xi_j^i)-\mathbb{E}(f(\xi_j^i))\Bigr)^2=\mathbb{E}\Bigl(f(\xi_j^i)-\frac{I_i(f)}{vol(D_i)}\Bigr)^2=\frac{1}{vol(D_i)}I_i\Bigl[(f-A_i)^2\Bigr],
\end{equation}
where $A_i=\frac{I_i(f)}{vol(D_i)}$. By the mean value theorem \ref{MVT} we have that for all $i=1,2,\ldots,K$ there exist $\bar x_i\in D_i$ such that $A_i=f(\bar x_i)$. Hence
\begin{equation}
    I_i\Bigl[(f-A_i)^2\Bigr]=\int\limits_{D_i}|f(x)-f(\bar x_i)|^2 dx\leq L^2\int\limits_{D_i} \|x-\bar x_i\|^{2\varrho}_{\infty} dx\leq L^2 \cdot vol(D_i)\cdot\sup\limits_{x\in D_i} \|x-\bar x_i\|^{2\varrho}_{\infty},
\end{equation}
which in turn implies that
\begin{equation}
     Var(f(\xi_j^i))\leq L^2\cdot\sup\limits_{x\in D_i} \|x-\bar x_i\|^{2\varrho}_{\infty},
\end{equation}
\begin{equation}
    Var(\widebar{MC}^{(i)}_{N_i}(f))=\frac{vol^2(D_i)}{N^2_i}\sum\limits_{j=1}^{N_i}Var(f(\xi_j^i))\leq \frac{L^2 vol^2(D_i)}{N_i} \cdot\sup\limits_{x\in D_i} \|x-\bar x_i\|^{2\varrho}_{\infty},
\end{equation}
and finally, since $N_i=vol(D_i)\cdot N$,
\begin{equation}
    \mathbb{E}|I(f)-\widebar{MC}_N(f)|^2=\sum_{i=1}^K Var(\widebar{MC}^{(i)}_{N_i}(f))\leq\frac{L^2}{N}\sum_{i=1}^K vol(D_i)\cdot\sup\limits_{x\in D_i} \|x-\bar x_i\|^{2\varrho}_{\infty}.
\end{equation}
Let us divide $D$ into $K=N$ sub-boxes $D_i$, each of the edge length equal to $N^{-1/d}$, and where we assume that $N^{1/d}\in\mathbb{N}$. Hence $vol(D_i)=(N^{-1/d})^d=1/N$ and we have that $N_i=1$ for $i=1,2,\ldots,N$. It means that we are using only one random sample on each $D_i$. Thus in this case the stratified sampling Monte Carlo method \eqref{sts_def_1} has the form
\begin{equation}
    \widebar{MC}_N(f)=\frac{1}{N}\sum_{i=1}^Nf(\xi_1^i)
\end{equation}
where $\xi_1^1,\ldots,\xi_1^N$ are independent random vectors and $\xi_1^i\sim U(D_i)$.
Moreover for $i=1,2,\ldots,N$
\begin{equation}
    \sup\limits_{x\in D_i} \|x-\bar x_i\|^{2\varrho}_{\infty}\leq N^{-2\varrho/d},
\end{equation}
and
\begin{equation}
    \Bigl(\mathbb{E}|I(f)-\widebar{MC}_N(f)|^2\Bigr)^{1/2}\leq L N^{-(\frac{\varrho}{d}+\frac{1}{2})},
\end{equation}
where $\widebar{MC}_N(f)$ uses $N$ values of $f$.
\begin{rem}
    The most visible gain of using the stratified sampling Monte Carlo algorithm is for $d=1$, since in the class $\mathcal{F}^{\varrho,d}_L$ its $L^2(\Omega)$-error  is $O(N^{-(\varrho
    +1/2)})$, the error of the crude Monte Carlo method is $O(N^{-1/2})$ while the error of the deterministic rectangle quadrature rule is $O(N^{-\varrho})$.
\end{rem}
\begin{rem}
    In the H\"older class $\mathcal{F}^{\varrho,d}_L$ the upper bound on the $L^2(\Omega)$-error for the both methods $\widebar{MC}_N(f)$ and $\widetilde{MC}_N(f)$ is the same, however, the algorithm $\widetilde{MC}_N(f)$ uses two times more function evaluations than $\widebar{MC}_N(f)$.
\end{rem}
\section{There exists deterministic quadrature rule as good as crude Monte Carlo method}
There is an interesting consequence of Theorem \ref{mean_sq_MC} and mean-value theorem for integrals. Namely, combining those two we can show that for every $f\in C([0,1]^d)$ and $N\in\mathbb{N}$ there exist nodes 
\begin{equation}
\label{det_nodes}
	x^*_1,x^*_2,\ldots,x^*_N\in [0,1]^d
\end{equation}
such that
\begin{equation}
\label{det_QR}
	\Bigl|I(f)-\frac{1}{N}\sum\limits_{j=1}^N f(x^*_j)\Bigl|=\sigma(f)/\sqrt{N}.
\end{equation}
This means that the deterministic quadrature rule $\displaystyle{Q^{det}_N(f)=\frac{1}{N}\sum\limits_{j=1}^N f(x^*_j)}$ achieves the error $O(N^{-1/2})$, the same as the crude Monte Carlo method $MC_N$, but by using only deterministic evaluations of $f$.

To show \eqref{det_QR} note the joint law $\displaystyle{\mu_{(\xi_1,\ldots,\xi_N)}}$ of independent identically distributed random variables $\xi_1,\ldots,\xi_N$ ($\xi_i\sim U([0,1]^d)$) is the product measure $\displaystyle{\mu_{(\xi_1,\ldots,\xi_N)}=\mu_{\xi_1}\times\ldots\times\mu_{\xi_N}=\lambda_d\times\ldots\times\lambda_d}$. Hence, by Theorems \ref{mean_sq_MC}, \ref{COM} we get
\begin{equation}
	\sigma^2(f)/N=\mathbb{E}|I(f)-MC_N(f)|^2=\int\limits_{\underset{N-times}{[0,1]^d\times\ldots\times [0,1]^d}}\Bigl|I(f)-\frac{1}{N}\sum\limits_{j=1}^N f(x_j)\Bigl|^2 \ (\lambda_d\times\ldots\times\lambda_d) (dx_1,\ldots,dx_N).
\end{equation}
Since the function
\begin{equation}
	[0,1]^d\times\ldots\times [0,1]^d \ni (x_1,\ldots,x_N)\to \Bigl|I(f)-\frac{1}{N}\sum\limits_{j=1}^N f(x_j)\Bigl|^2\in [0,+\infty)
\end{equation}
is continuous, we get by Theorem \ref{MVT}  that there exist points \eqref{det_nodes} such that
\begin{equation}
	\int\limits_{\underset{N-times}{[0,1]^d\times\ldots\times [0,1]^d}}\Bigl|I(f)-\frac{1}{N}\sum\limits_{j=1}^N f(x_j)\Bigl|^2 \ (\lambda_d\times\ldots\times\lambda_d) (dx_1,\ldots,dx_N)=\Bigl|I(f)-\frac{1}{N}\sum\limits_{j=1}^N f(x^*_j)\Bigl|^2.
\end{equation} 
(Note that $(\lambda_d\times\ldots\times\lambda_d)([0,1]^d\times\ldots\times [0,1]^d)=1$.)

The error of $Q^{det}_N$ the same as for $MC_N$. However, the nodes \eqref{det_nodes} (that depend on $f$) are not known (we have only established the existence of such points) and the method $Q^{det}_N$ cannot be used in practice. The result \eqref{det_QR} is of purely theoretical nature. Nevertheless, the technique described above is often used for nonconstructive proving of existence of  deterministic quadrature rules and it is called {\it probabilistic method}.
\section{Some generalizations}
We have discuses in details the crude Monte Carlo method for \eqref{I_1}. Of course this is not the most general problem that we can address. 
\subsection{Weighted integration}
We consider the problem of approximating the following weighted integral
\begin{equation}
\label{I_2w}
	I_w(f)=\int\limits_{\mathbb{R}^d}w(x)f(x)dx,
\end{equation}
where we assume that $f:\mathbb{R}^d\to\mathbb{R}$, $w:\mathbb{R}^d\to [0,+\infty)$ are Borel measurable and
\begin{eqnarray}
    && W:=\int\limits_{\mathbb{R}^d} w(x)dx<+\infty,\\
    && \max\Bigl\{\int\limits_{\mathbb{R}^d}w(x)|f(x)|dx,\int\limits_{\mathbb{R}^d}w(x)|f(x)|^2dx\Bigr\}<+\infty.
\end{eqnarray}
In this case the function $\mathbb{R}^d\ni x \to g(x):= w(x)/W$ is a probability density function, that is absolutely continuous with respect to the $d$-dimensional Lebesgue measure $\lambda_d$. Such integrals often appear in mathematical finances and option pricing.

Assuming that we can sample from $g$ we define
\begin{equation}
    MC_N(f)=\frac{W}{N}\sum_{j=1}^N f(\xi_j),
\end{equation}
where $(\xi_j)_{j\in\mathbb{N}}$ is an iid sequence of random vectors with the common density $g$. It is easy to see that
\begin{equation}
    \mathbb{E}(MC_N(f))=I_w(f),
\end{equation}
and
\begin{equation}
    Var(MC_N(f))=\frac{1}{N}\sigma^2_{w}(f),
\end{equation}
where
\begin{equation}
    \sigma_w^2(f)=W\cdot I_w(f^2)-(I_w(f))^2.
\end{equation}
This implies that
\begin{equation}
    \|I_w(f)-MC_N(f)\|_{L^2(\Omega)}=\frac{\sigma_w(f)}{\sqrt{N}},
\end{equation}
for all $N\in\mathbb{N}$.
\begin{exmp}
    For the weighted integral
    \begin{equation}
        I_w(f)=\int\limits_{\mathbb{R}} e^{-x^2} f(x) dx,
    \end{equation}
    the crude Monte Carlo method has the form
    \begin{equation}
        MC_N(f)=\frac{\sqrt{2\pi}}{N}\sum_{j=1}^Nf(\xi_j),
    \end{equation}
    where $(\xi_j)_{j\in\mathbb{N}}$ are independently sampled from the standard normal distribution $N(0,1)$.
\end{exmp}
\begin{exmp}
    We apply weighted integration to  estimation of Lebesgue measure of compact subsets in $\mathbb{R}^d$. For a compact subset $D$ of $\mathbb{R}^d$ let $M$ be  a $d$-dimensional cube such that $D\subset M$. We assume that the volume $vol(M)=\lambda_d(M)$ is known. We take $f=\mathbf{1}_D$ and $w=\mathbf{1}_M$. Then  
    \begin{equation}
        I_w(f)=\lambda_d(D).
    \end{equation}
    Hence, if we take $(\xi_j)_{j\in\mathbb{N}}$ as an iid sequence of random vectors that are uniformly distributed on the cube $M$, then we have with probability one
    \begin{equation}
        MC_N(f)=\frac{vol(M)}{N}\sum_{j=1}^N \mathbf{1}_D(\xi_j)\to \lambda_d(D),
    \end{equation}
    as $N\to +\infty$. 
\end{exmp}
\subsection{Abstract setting}
Let $\mu$ be a probabilistic measure on $(\mathbb{R}^d,\mathcal{B}(\mathbb{R}^d))$, i.e., $\mu(\mathbb{R}^d)=1$. 
For $f\in L^2(\mathbb{R}^d,\mathcal{B}(\mathbb{R}^d),\mu)$  we consider the problem of Monte Carlo approximation of the following Lebesgue integral
\begin{equation}
\label{gen_leb_int}
	I(f)=\int\limits_{\mathbb{R}^d}f(x)\mu(dx).
\end{equation}
The crude Monte Carlo estimator of \eqref{gen_leb_int} is analogous to  \eqref{crude_MC}, namely
\begin{equation}
\label{gen_crude_MC}
	MC_N(f)=\frac{1}{N}\sum\limits_{j=1}^Nf(\xi_j),
\end{equation}
where this time $(\xi_j)_{j\in\mathbb{N}}$ is an i.i.d sequence of random variables with distribution $\mu$. We stress that we can use \eqref{gen_crude_MC} only in the case when we can sample from the distribution $\mu$. Note that the weighted integration problem, described in the previous section, is a special case of \eqref{gen_leb_int} with $\mu(dx)=\frac{1}{W}w(x)\lambda_d(dx)$, since
\begin{equation}
    \int\limits_{\mathbb{R}^d}f(x)\mu(dx)=\int\limits_{\mathbb{R}^d}\frac{w(x)}{W}f(x)dx.
\end{equation}
We can go even further. Let us switch from $(\mathbb{R}^d,\mathcal{B}(\mathbb{R}^d))$ into arbitrary measurable space $(X,\mathcal{A})$ and let us consider the random variable $\xi$ with values in $X$, i.e., $\Sigma$-to-$\mathcal{A}$ measurable mapping
\begin{equation}
	\xi:\Omega\to X.
\end{equation}
 The law of $\xi$ is given by
\begin{equation}
	\mu(A)=\mathbb{P}(\{\xi \in A\})=\mathbb{P}(\xi^{-1}(A)),
\end{equation}
for all $A\in\mathcal{A}$. Then the space $(X,\mathcal{A},\mu)$ is a measure space equipped with the probability measure $\mu$ ($\mu(X)=1$). Let $f:X\to\mathbb{R}$ be a function from $L^2(X,\mathcal{A},\mu)$. Note that by Theorem \ref{COM} we get
\begin{equation}
	\mathbb{E}(f(\xi))=\int\limits_{\Omega}(f\circ \xi)(\omega)\mathbb{P}(d\omega)=I(f),
\end{equation}
where this time
\begin{equation}
\label{abs_int_X}
	I(f)=\int\limits_X f(x)\mu(dx).
\end{equation}
(Since $\mu(X)=1$, we get by the H\"older inequality $I(|f|)\leq (I(f^2))^{1/2}<+\infty$.) The Monte Carlo method approximating \eqref{abs_int_X} is again defined analogously as before, i.e.,
\begin{equation}
\label{MC_abstr}
	MC_N(f)=\frac{1}{N}\sum\limits_{j=1}^N f(\xi_j),
\end{equation}
where $(\xi_j)_{j\in\mathbb{N}}$ is an i.i.d sequence of $X$-valued random variables with distribution $\mu$. Again $(f(\xi_j))_{j\in\mathbb{N}}$ are identically distributed and independent real random variables. Hence,  by Theorem \ref{THM_SLLN}, $MC_N(f)\to I(f)$ almost surely as $N\to+\infty$, and by using the same proof technique as for Theorem \ref{mean_sq_MC} we get for all $N\in\mathbb{N}$
\begin{equation}
	\Bigl(\mathbb{E}|I(f)-MC_N(f)|^2\Bigr)^{1/2}=\sigma(f)/\sqrt{N},
\end{equation}
where $\displaystyle{\sigma^2(f)=I(f^2)-(I(f))^2}$. Moreover, it is easy to see that the following generalizations of Theorems \ref{MC_Lp_est}, \ref{AS_CONV_mc} are true.
\begin{thm} 
\label{MC_Lp_est_gen}
Let $f\in L^p(X,\mathcal{A},\mu)$ for some $p\in [2,+\infty)$. Then for
\begin{equation}
	K_p(f)=C_p\cdot\Bigl(\int\limits_X|f(x)-I(f)|^p\mu(dx)\Bigr)^{1/p}<+\infty,
\end{equation}
 where $C_p$ is the constant from the BDG inequality, and  for all $N\in\mathbb{N}$ it holds
\begin{equation}
\label{MC_p_11}
	\Bigl(\mathbb{E}|I(f)-MC_N(f)|^p\Bigr)^{1/p}\leq K_p(f)\cdot N^{-1/2},
\end{equation}
and
\begin{equation}
\label{MC_q_11}
	\Bigl(\mathbb{E}|I(f)-MC_N(f)|^q\Bigr)^{1/q}\leq K_p(f)\cdot N^{-1/2}
\end{equation}
for all $q\in (0,p]$.
\end{thm}
\begin{thm}
\label{AS_CONV_mc_gen}
 Let $\displaystyle{f\in \bigcap\limits_{p\in [2,+\infty)} L^p(X,\mathcal{A},\mu)}$. Then for all $\varepsilon\in (0,1/2)$ there exists $\displaystyle{\eta_{\varepsilon}(f)\in \bigcap\limits_{p\in [1,+\infty)} L^p(\Omega)}$ such that
\begin{equation}
	\mathbb{P}\Bigl(\bigcap_{N\in\mathbb{N}}\Bigl\{|I(f)-MC_N(f)|\leq \eta_{\varepsilon}(f)\cdot N^{-\frac{1}{2}+\varepsilon}\Bigr\}\Bigr)=1.
\end{equation} 
\end{thm}
However, in order to use \eqref{MC_abstr} we have to be able to generate realizations $\xi(\omega)$ of the random variable $\xi$ with values in the abstract set $X$. This is much harder task than sampling from the finite dimensional distribution as it was for \eqref{gen_crude_MC}. One of possible solution is to construct a random variable
\begin{equation}
	\hat \xi:\Omega\to X,
\end{equation}
with simpler structure than $\xi$, that approximates $\xi$ in suitable way, and (what is the most important) for which we have access to the values (samples) $\hat\xi(\omega)\in X$. Such a construction might be highly nontrivial and involves additional computational cost. We will face with such a situation in future lectures, where $\xi$ will be a solution of the underlying stochastic differential equation, while $\hat \xi$ will be a suitable approximation of $\xi$. (For example, $\hat\xi$ might be output of the Euler or Milstein method.) Then approximate computation of $\mathbb{E}(f(\xi))$ corresponds to the weak approximation of a solution of SDE. 

 Now, let
\begin{equation}
\label{wh_mcn}
	\widehat{MC}_N(f)=\frac{1}{N}\sum\limits_{j=1}^N f(\hat\xi_j),
\end{equation}
where $(\hat\xi_j)_{j\in\mathbb{N}}$ is an i.i.d sequence of $X$-valued random variables with distribution $\hat\mu$, where $\hat\mu$ is a law of $\hat\xi$. Let us now additionally assume that 
\begin{equation}	
	\mathbb{E}|f(\hat\xi)|^2=\int\limits_X|f(x)|^2\hat\mu(dx)<+\infty.
\end{equation}
Then \eqref{wh_mcn} is a crude Monte Carlo method that approximates
\begin{equation}
	\mathbb{E}(f(\hat\xi))=\int\limits_X f(x)\hat\mu(dx),
\end{equation}
and
\begin{equation}
\label{wh_mcn_err}
		\Bigl(\mathbb{E}|\mathbb{E}(f(\hat\xi))-\widehat{MC}_N(f)|^2\Bigr)^{1/2}=\hat\sigma(f)/\sqrt{N},
\end{equation}
where
\begin{equation}
	\hat\sigma^2(f)=\int\limits_X|f(x)|^2\hat\mu(dx)-\Bigl(\int\limits_X f(x)\hat\mu(dx)\Bigr)^2.
\end{equation}
Furthermore
\begin{eqnarray}
	\Bigl(I(f)-\widehat{MC}_N(f)\Bigr)^2 &=& \Biggl(\Bigl(\mathbb{E}(f(\xi))-\mathbb{E}(f(\hat\xi))\Bigr)+\Bigl(\mathbb{E}(f(\hat\xi))-\widehat{MC}_N(f)\Bigr)\Biggr)^2\notag\\
	&=&\Bigl(\mathbb{E}(f(\xi))-\mathbb{E}(f(\hat\xi))\Bigr)^2+\Bigl(\mathbb{E}(f(\hat\xi))-\widehat{MC}_N(f)\Bigr)^2\notag\\
	&& +2 \Bigl(\mathbb{E}(f(\xi))-\mathbb{E}(f(\hat\xi))\Bigr)\cdot \Bigl(\mathbb{E}(f(\hat\xi))-\widehat{MC}_N(f)\Bigr).
\end{eqnarray}
Since $\displaystyle{\mathbb{E}(\widehat{MC}_N(f))=\mathbb{E}(f(\hat\xi))}$, by \eqref{wh_mcn_err} we arrive at
\begin{equation}
\label{mnsqrt_err_2}
	\mathbb{E}\Bigl(I(f)-\widehat{MC}_N(f)\Bigr)^2 = |\mathbb{E}(f(\xi))-\mathbb{E}(f(\hat\xi))|^2+\frac{\hat\sigma^2(f)}{N}.
\end{equation}
The term 
\begin{equation}
	|\mathbb{E}(f(\xi))-\mathbb{E}(f(\hat\xi))|=\Bigl|\int\limits_X f(x)\mu(dx)-\int\limits_X f(x)\hat\mu(dx)\Bigl|
\end{equation}
is the so called {\it weak error} and, roughly speaking, it measures how accurate the law $\hat\mu$ approximates $\mu$. Hence, the mean square error \eqref{mnsqrt_err_2} consists of two types of errors: the weak error and the error of Monte Carlo simulation.
\\
We always have
\begin{equation}
\label{weak_str_est}
	|\mathbb{E}(f(\xi))-\mathbb{E}(f(\hat\xi))|\leq \mathbb{E}|f(\xi)-f(\hat\xi)|,
\end{equation}
and sometimes this inequality might be used in order to provide further upper bound on the weak error.
For example, in the case when $(X,d)$ is a separable metric space (what will be the most common case) equipped with the Borel $\sigma$-field $\mathcal{A}=\mathcal{B}(X)$, let us assume that $f$ is Lipschitz functional, i.e., there exists $L\in (0,+\infty)$ such that for all $x,y\in X$  
\begin{equation}
	|f(x)-f(y)|\leq L d(x,y).
\end{equation}
(For instance, $f$ might be the payoff of the European option.)
Then we  have that
\begin{equation}
	\mathbb{E}|f(\xi)-f(\hat\xi)|\leq L\mathbb{E}(d(\xi,\hat\xi)).
\end{equation}
(Note that $d(\xi,\hat\xi)$ is $\Sigma$-to-$\mathcal{B}([0,+\infty))$ measurable non-negative function, see Exercises \ref{ex:ex_meas_d} at the end of this chapter.) The value of $\mathbb{E}(d(\xi,\hat\xi))$ is the average distance between trajectories of $\xi$ and $\hat\xi$, and it is called {\it strong error}. (Comparing to the weak error, it relies on the joint law $\mu_{(\xi,\hat \xi)}$ of $(\xi,\hat \xi)$.) Hence in this case we get
\begin{equation}
	\Bigl(\mathbb{E}|I(f)-\widehat{MC}_N(f)|^2\Bigr)^{1/2}\leq L\mathbb{E}(d(\xi,\hat\xi))+\hat\sigma(f)/\sqrt{N}.
\end{equation}
Note however, that for more smooth functions $f$ it is possible to get better upper bounds  not by using \eqref{weak_str_est} but by estimating directly the weak error.

The setting described above for Monte Carlo simulations is quite general and might look a bit abstract. However, in future lectures it will turn out that this setting is a natural and convenient framework for a problem  of weak approximation of solutions of SDEs, and, as a result, for option pricing. 
\section{There is always enough randomness}
For the Monte Carlo methods we always use sequences of independent random variables or random elements with values in metric spaces. Then the natural question arises if there exists a probability space large enough to carry such infinite sequences. 

For the proof of this fundamental results see \cite{KALLEN} or page 243. in \cite{loeve}.
\begin{thm} (Ulam - \L omnicki theorem)
\label{UL_THM}
	Let $T$ be an arbitrary set. Let for every $t\in T$ a measurable space $(E_t,\mathcal{E}_t)$ and probability measure $\mathbb{P}_t$ on it be given. Then there exists a probability space $(\Omega,\Sigma,\mathbb{P})$ and independent random elements $(X_t)_{t\in T}$ such that each $X_t:\Omega\to E_t$ is $\Sigma / \mathcal{E}_t$-measurable and $\mathbb{P}(\{X_t\in B\})=\mathbb{P}_t(B)$ for all $B\in\mathcal{E}_t$.
\end{thm}
\section{Implementation issues}

We present simple implementation of the crude Monte Carlo method for the following integral 
\begin{equation}
	\int\limits_a^b f(x)dx.
\end{equation}
First we will make a suitable change of variable and we get that
\begin{equation}
	\int\limits_a^b f(x)dx=\int\limits_0^1 g(x)dx,
\end{equation}
with $\displaystyle{g(x)=(b-a)f(a+(b-a)x)}$. In the Python implementation we use NumPy module and, so called, {\it vectorization}.

\begin{lstlisting}[language=Python]
import numpy as np

def quad_MC(f, a, b, N):
    random_samples_0 = np.random.uniform(0, 1, N).astype('float32')
    random_samples_0 = a + (b - a) * random_samples_0
    g_values = (b-a) * f(random_samples_0)

    MC_g = np.mean(g_values)

    random_samples_1 = np.power(g_values - MC_g, 2)
    Var_g = (1.0 / (N - 1)) * np.sum(random_samples_1)

    L_CI = MC_g - 2.0 * np.sqrt(Var_g / N)
    R_CI = MC_g + 2.0 * np.sqrt(Var_g / N)

    return MC_g, Var_g, L_CI, R_CI
\end{lstlisting}

\section{Concluding remarks}

As summary, we recall after \cite{GPage} \underline{three mandatory steps} while performing  Monte Carlo simulation to approximate $\mathbb{E}(f(\xi))$ for a given $f$:
\begin{itemize}
	\item [1.] Specification of a confidence level $\alpha\in (0,1)$, that should be close to $1$.
	\item [2.] Simulation of an $N$-sample $\xi_1,\ldots,\xi_N$ of i.i.d random vectors having the same distribution $\mu$ as $\xi$ and (possibly recursive) computation of both its empirical mean $MC_N(f)$ and its empirical variance $\widehat{\sigma^2}_N(f)$.
	\item [3.] Computation of the resulting empirical confidence interval $\mathcal{\hat J}_N^{\alpha}(f)$ at the confidence level $\alpha$. This is the  only  trustworthy output of the performed Monte Carlo simulation.
\end{itemize}
\section{Exercises}
\begin{itemize}
	\item [1.] Show that $(f(\xi_j))_{j\in\mathbb{N}}$ is the sequence of identically distributed and independent random variables.
	\item [2.] Show that $Var(f(\xi_1))=\sigma^2(f)$.
	\item [3.] Let $(A_j)_{j\in\mathbb{N}}\subset\Sigma$ be a sequence of events such that $\mathbb{P}(A_j)=1$ for all $j\in\mathbb{N}$. Show that $\displaystyle{\mathbb{P}\Bigl(\bigcap_{j\in\mathbb{N}}A_j\Bigr)=1}$. 
	\item [4.] Construct a crude Monte-Carlo method that approximates the integral 
	\begin{equation}
		I(f)=\int\limits_{\mathbb{R}^d}f(x)e^{-\|x\|_2^2/2}dx,
	\end{equation}
	for a Borel measurable functions $f:\mathbb{R}^d\to\mathbb{R}$ that belong to the weighted $L^2(\mathbb{R}^d;e^{-\|x\|_2^2/2})$ space. (Such integrals often appear in mathematical finance.)\\
	Hint. Note that
	\begin{equation}
		I(f)=(2\pi)^{d/2}\cdot\int\limits_{\mathbb{R}^d}f(x)\mu(dx), 
	\end{equation}
	where $\displaystyle{\mu(A)=\int\limits_{A}\frac{1}{(2\pi)^{d/2}} e^{-\|x\|_2^2/2}\lambda_d(dx)}$, $A\in\mathcal{B}(\mathbb{R}^d)$, is the Gaussian measure on $(\mathbb{R}^d,\mathcal{B}(\mathbb{R}^d))$.
	\item [5.] ({\it Option pricing in the Black-Scholes model}) Construct Monte Carlo estimator of the following expected value
		\begin{equation}
			\mathbb{E}((S(T)-K)^+)
		\end{equation}
		for $K\in [0,+\infty)$ and $\displaystyle{S(t)=S(0)\exp\Bigl((\mu-\frac{1}{2}\sigma
^2)t+\sigma W(t)\Bigr)}$, $t\in [0,T]$, with $S(0)>0$, $\mu\in\mathbb{R}$, $\sigma>0$. Recall that $(W(t))_{t\in [0,+\infty)}$ is a one-dimensional Wiener process, see Chapter \ref{three_processes}.
	\item [6.] Give a proof of Theorem \ref{mean_sq_MC} (iv).
	\item [7.] Show \eqref{slln_emp_var}, \eqref{exp_emp_var}, and \eqref{ctg_emp_var} in full details.
	\item [8.] Show that value of \eqref{MC_abstr} is independent of the choice of the representation of
the equivalence class $f\in  L^2(X,\mathcal{A},\mu)$ in the following sense: let $f,g$ be two representations of the same
equivalence class in $L^2(X,\mathcal{A},\mu)$. Then 
\begin{equation}
\label{equiv_indep2}
	\mathbb{P}\Bigl(\bigcap_{N\in\mathbb{N}} \{MC_N(f)=MC_N(g)\}\Bigr)=1.
\end{equation}
	\item [9.] \label{ex:ex_meas_d}
Justify that the mapping 
	\begin{equation}
		\Omega\ni \omega\to d(\xi(\omega),\hat \xi(\omega)) \in [0,+\infty)
	\end{equation}
	is $\Sigma$-to-$\mathcal{B}(\mathbb{R})$ measurable.\\
	Hint. Firstly show that the function $X\times X \ni (x,y)\to d(x,y) \in [0,+\infty)$ is continuous, and therefore $\mathcal{B}(X\times X) / \mathcal{B}(\mathbb{R})$-measurable. Then use the fact that for a separable metric space $X$ we have $\mathcal{B}(X\times X)=\mathcal{B}(X)\otimes\mathcal{B}(X)$.
	\item [10.] Give a proof of the following lemma (see Lemma 5.1 in \cite{higham2000}).
	\begin{lem} 
\label{as_stab_lem_1}
    Given a sequence of real-valued, independent and identically distributed random variables $\{Z_n\}_{n\in\mathbb{N}_0}$ with $\displaystyle{\mathbb{P}(Z_1>0)=1}$, consider the sequence of random variables $\{Y_n\}_{n\in\mathbb{N}}$ defined by
    \begin{equation}
        Y_n=\Bigl(\prod\limits_{i=0}^{n-1}Z_i\Bigr)Y_0,
    \end{equation}
    where $\displaystyle{\mathbb{P}(Y_0>0)=1}$. If $\ln(Z_1)$ is integrable, then 
    \begin{equation}
        \mathbb{E}(\ln (Z_1))< 0 \Rightarrow\lim\limits_{n\to +\infty}Y_n=0, \hbox{with probability} \ 1 \Rightarrow \mathbb{E}(\ln (Z_1))\leq 0.
    \end{equation}
	\end{lem}
	Hint. Use the fact that $\displaystyle{\ln (Y_n)=\ln (Y_0)+S_n}$ with $\displaystyle{S_n=\sum\limits_{i=0}^{n-1}\ln Z_i}$.
	\item [11.] Show that $(Z_j,\mathcal{G}_j)_{j=0,1,\ldots,N}$, defined in \eqref{DT_mart_1},  is a discrete-time martingale such that $\displaystyle{\max\limits_{1\leq j\leq N}\mathbb{E}|Z_j|^p<+\infty}$.
	\item [12.] Investigate the $L^p(\Omega)$-error, $p\in [2,+\infty)$, for $\widetilde {MC}_{N,N}$ defined in \eqref{wt_mc_1}.
	\item [13.] Construct empirical confidence intervals for $\widetilde {MC}_{N,N}$.
	\item [14.] Construct empirical confidence intervals for $MC_N(f)$ defined in \eqref{MC_abstr}.
 	\item [15.] $f\in C^{r}([0,1])$ with $\varrho$-H\"older continuous $r$th derivative $f^{(r)}$, $r\in\mathbb{N}_0$, $\varrho\in (0,1]$. Investigate the mean square error of $\widetilde {MC}_{N,n}$ when the control variate $g_n:[0,1]\to\mathbb{R}$ is a piecewise Taylor polynomial 
	\begin{equation}
 \label{Taylor_cv_def}
		g_n(x)=\sum\limits_{i=0}^{n-1}\mathbf{1}_{[x_i,x_{i+1})}(x)\cdot\sum\limits_{j=0}^r \frac{f^{(j)}(x_i)}{j!}(x-x_i)^j,
	\end{equation}
for $x_i=i/n$, $i=0,1,\ldots,n$.
\item [16.] Let $f\in C^r([0,1])$, $r\in\mathbb{N}$. As a control variate we take the following piecewise Taylor polynomial
\begin{equation}
		g_n(x)=\sum\limits_{i=0}^{n-1}\mathbf{1}_{[x_i,x_{i+1})}(x)\cdot\sum\limits_{j=0}^{r-1} \frac{f^{(j)}(x_i)}{j!}(x-x_i)^j, \quad x\in [a,b]
	\end{equation}
for $x_i=i/n$, $i=0,1,\ldots,n$. For $\widetilde {MC}_{n,n}$ show that 
\begin{equation}
		\lim\limits_{n\to +\infty} n^{r+\frac{1}{2}} \cdot \|I(f)-\widetilde {MC}_{n,n}(f)\|_{L^2(\Omega)}=C_{MC}(f),
	\end{equation}
	where
	\begin{equation}
		C_{MC}(f)=\Bigl[\frac{1}{(r!)^2\cdot (2r+1)}\int\limits_{0}^{1}(f^{(r)}(t))^2dt-\frac{1}{((r+1)!)^2}(f^{(r-1)}(1)-f^{(r-1)}(0))^2\Bigr]^{1/2}.
	\end{equation}
\end{itemize}
\chapter{Deterministic vs randomized algorithms for ODEs}
\label{MC_ODEs}
In this chapter we show how to define randomized algorithms for approximate solving of ordinary differential equations. We consider initial-value problems of the following form
\begin{equation}
	\label{ODE_PROBLEM1}
		\left\{ \begin{array}{ll}
			z'(t)=f(t,z(t)), &t\in [a,b], \\
			z(a)=\xi, 
		\end{array}\right.
\end{equation}
where $d\in\mathbb{N}$, $\xi\in\mathbb{R}^d$, $-\infty<a<b<+\infty$. In $\mathbb{R}^d$ we consider the Euclidean norm $\|\cdot\|$. We assume that $f\in C([a,b]\times\mathbb{R}^d;\mathbb{R}^d)$ is of at most linear growth and satisfies local Lipschitz condition. We consider explicit randomized Euler and Runge-Kutta schemes.

We impose the following assumptions on the right-hand side function $f$:
\begin{itemize}
	\item [(A1)] $f\in C([a,b]\times\mathbb{R}^d;\mathbb{R}^d)$,
	\item [(A2)] there exists $K\in [0,+\infty)$ such that for all $t\in [a,b]$, $y\in \mathbb{R}^d$
	\begin{equation}
		\|f(t,y)\|\leq K(1+\|y\|),
	\end{equation}
	\item [(A3)] for every $R>0$ there exists $L_R\in [0,+\infty)$ such that for all $t\in [a,b]$, and all $x,y\in \mathbb{R}^d$ satisfying $\|x\|\leq R$, $\|y\|\leq R$ it holds
	\begin{equation}
		\|f(t,x)-f(t,y)\|\leq L_R \|x-y\|,
	\end{equation}
	\item [(A4)]  there exists $\varrho\in (0,1]$ such that for every $H>0$ there exists $L_H\in [0,+\infty)$ such that  for all $t,s\in [a,b]$, and all $x\in \mathbb{R}^d$ satisfying $\|x\|\leq H$  it holds
	\begin{equation}
		\|f(t,x)-f(s,x)\|\leq L_H |t-s|^{\varrho}.
		\end{equation}
\end{itemize}
Therefore $f=f(t,x)$ is globally continuous, of at most linear growth, it satisfies local Lipschitz assumption with respect to $x$, and it is globally H\"older continuous with respect to $t$. 
\begin{ex} The exemplary function satisfying (A1)-(A4) is
\begin{equation}
\label{f_a1a4}
	f(t,x)=x\cdot\sin(x^2\cdot |t|^{\varrho}), \quad (t,x)\in [a,b]\times\mathbb{R}.
\end{equation}
\end{ex}
In this section we will use the following version of Peano's theorem, see, for example, Theorem 70.4, page 292. in \cite{gring} for its proof. (For further comments on existence of solutions of \eqref{ODE_PROBLEM1} check Remarks \ref{uni_ex_1}, \ref{uni_ex_2}).
\begin{thm} (\cite{gring}) \label{peano}
If the function $f$ satisfies Assumptions (A1) and (A2), then the initial value problem \eqref{ODE_PROBLEM1} has at least one solution on $[a,b]$.
\end{thm}
\begin{lem}
\label{lemma_ode_1}
Let $\xi\in\mathbb{R}^d$.
\begin{itemize}
	\item [(i)] Let us assume that $f$ satisfies the assumptions (A1)-(A3). Then there exists a unique solution $z:[a,b]\to\mathbb{R}^d$, $z\in C^1([a,b]\times\mathbb{R}^d;\mathbb{R}^d)$, of \eqref{ODE_PROBLEM1}. Moreover
\begin{equation}
\label{bound_ode_1}
	\sup\limits_{a\leq t\leq b}\|z(t)\|\leq \max\{1,K(b-a)\}e^{K(b-a)}(1+\|\xi\|),
\end{equation}
and for all $s,t\in [a,b]$
\begin{equation}
\label{bound_ode_2}
	\|z(t)-z(s)\|\leq K\Bigl(1+\max\{1,K(b-a)\}e^{K(b-a)}\Bigr)(1+\|\xi\|)|t-s|. 
\end{equation}
	\item [(ii)] Let us assume that $f$ satisfies the assumptions (A1)-(A4). Then there exists $L_1=L_1(a,b,K,\xi)\in [0,+\infty)$ such that for all $t,s,\in [a,b]$
\begin{equation}
\label{bound_ode_3}
	\|z'(t)-z'(s)\|\leq L_1|t-s|^{\varrho}.
\end{equation}
\end{itemize} 
\end{lem}
\begin{proof} Let us assume that $f$ satisfies the assumptions (A1)-(A3).  According to Theorem \ref{peano} under the  conditions (A1), (A2) the equation \eqref{ODE_PROBLEM1} has at least one solution $z\in C^1([a,b]\times\mathbb{R}^d;\mathbb{R}^d)$. For all $t\in [a,b]$ we have
\begin{equation}
	\|z(t)\|\leq \|\xi\|+K(b-a)+\int\limits_a^t \|z(s)\|ds,
\end{equation}
and by Gronwall's lemma \ref{GLCW} we get \eqref{bound_ode_1}. Moreover, for all $t,s\in [a,b]$
\begin{equation}
	\|z(t)-z(s)\|=\|z(t\vee s)-z(t\wedge s)\|\leq \int\limits_{t\wedge s}^{t\vee s}\|f(u,z(u))\|du\leq K\Bigl(1+\sup\limits_{a\leq t\leq b}\|z(t)\|\Bigr)|t-s|,
\end{equation}
and from \eqref{bound_ode_1} we get \eqref{bound_ode_2}.

Let $z$ and $\tilde z$ be two solutions of \eqref{ODE_PROBLEM1}. From \eqref{bound_ode_1} we have that $\max\{\|z(t)\|,\|\tilde z(t)\|\}\leq R$ for all $t\in [a,b]$, where $R=\max\{1,K(b-a)\}e^{K(b-a)}(1+\|\xi\|)$. Then, by the assumption (A3) we have that there exists $L_R\in [0,+\infty)$ such that for all $t\in [a,b]$
\begin{equation}
    \|z(t)-\tilde z(t)\|\leq \int\limits_a^t \|f(s,z(s))-f(s,\tilde z(s))\| ds\leq L_R\int\limits_a^t \|z(s)-\tilde z(s)\| ds.
\end{equation}
Again, by  Gronwall's lemma  \ref{GLCW} for all $t\in [a,b]$ we have $z(t) = \tilde z(t)$, and the uniqueness of the solution $z$ follows. Finally, let us assume that $f$ satisfies the assumptions (A1)-(A4). From \eqref{bound_ode_1} we have that $\|z(t)\|\leq R$ for all $t\in [a,b]$, where $R=H=\max\{1,K(b-a)\}e^{K(b-a)}(1+\|\xi\|)$. Hence, by (A3), (A4), \eqref{bound_ode_1}, \eqref{bound_ode_2} we have that there exist $L_R,L_H\in [0,+\infty)$ such that for all $t,s\in [a,b]$
\begin{eqnarray}
&&\|z'(t)-z'(s)\|=\|f(t,z(t))-f(s,z(s))\|\leq \|f(t,z(t))-f(t,z(s))\|+\|f(t,z(s))-f(s,z(s))\|\notag\\
&&\leq L_R\|z(t)-z(s)\|+L_H  |t-s|^{\varrho}.
\end{eqnarray}
This and the fact that $|t-s|^{1-\varrho}\leq 1+|t-s|\leq 1+b-a$ implies \eqref{bound_ode_3}.
\end{proof}
The following result states that the solution $z$ depends on the initial-value $\xi$ in continuous-like way. Note that the $L_2$ in the proposition below depends on the initial-value, which is  due to the local Lipschitz condition for $f$. This is not the case  when $f$ satisfies the Lipschitz condition globally.
\begin{prop} Let us assume that $f$ satisfies (A1)-(A3). Let $z=z(\xi,f)$ denotes the solution of \eqref{ODE_PROBLEM1} under the initial value $\xi\in\mathbb{R}^d$, while $y=y(\eta,f)$ the solution of \eqref{ODE_PROBLEM1} with the initial-value $\eta\in\mathbb{R}^d$. Then there exists $L_2=L_2(a,b,K,\xi,\eta)\in [0,+\infty)$ such that for all $t\in [a,b]$
\begin{equation}
	\|z(t)-y(t)\|\leq e^{L_2(t-a)}\|\eta-\xi\|.
\end{equation}
\end{prop}
\begin{proof}
	By Lemma \ref{lemma_ode_1} (i) we have that
	\begin{equation}
	\label{bound_ode_11}
	\sup\limits_{a\leq t\leq b}\|z(t)\|\leq R_1:=\max\{1,K(b-a)\}e^{K(b-a)}(1+\|\xi\|),
\end{equation}
and
\begin{equation}
\label{bound_ode_12}
	\sup\limits_{a\leq t\leq b}\|y(t)\|\leq R_2:= \max\{1,K(b-a)\}e^{K(b-a)}(1+\|\eta\|).
\end{equation}
Hence, for $R=R_1\vee R_2$ there exists $L_R\in [0,+\infty)$  such that for all $s\in [a,b]$
\begin{equation}
	\|f(s,z(s))-f(s,y(s))\|\leq L_R \|z(s)-y(s)\|.
\end{equation}
Therefore, for all $t\in [a,b]$
\begin{equation}
	\|z(t)-y(t)\|\leq \|\xi-\eta\|+L_R\int\limits_a^t \|z(s)-y(s)\|ds
\end{equation}
and by the Gronwall's lemma \ref{GLCW} we get the thesis.
\end{proof}
\begin{rem}
\label{uni_ex_1}
We refer to Theorem (2.12) at page 252 in \cite{andresgorniewicz1} where  even more general version of Peano's theorem was established than Theorem \ref{peano}.
\end{rem}
\begin{rem}
\label{uni_ex_2}
Existence and uniqueness of solution to \eqref{ODE_PROBLEM1} under the assumptions (A1)-(A3) can also be derived from Theorem 2.2 at pages 104-105 in \cite{fried1}. 
\end{rem}
\section{Deterministic Euler scheme}
Let $n\in\mathbb{N}$, $h=(b-a)/n$, $t_k=a+kh$ for $k=0,1,\ldots,n$. We define the Euler scheme as follows. Take
\begin{equation}
	y_0=\xi,
\end{equation}
and
\begin{equation}
	y_{k+1}=y_k+hf(t_k,y_k),
\end{equation}
for $k=0,1,\ldots,n-1$. Moreover, as a time-continuous approximation of $z$ in the whole interval $[a,b]$ we set
\begin{equation}
	l_n(t):=l_{k,n}(t) \ \hbox{for} \  t\in [t_k,t_{k+1}], \ k=0,1,\ldots,n-1,
\end{equation}
where
\begin{equation}
	l_{k,n}(t)=y_k+(t-t_k)f(t_k,y_k).
\end{equation}
Of course $l_n\in C([a,b];\mathbb{R}^d)$ and it is piecewise linear function. In order to construct the function $l_n$ we use $n$ evaluations of $f$.
\begin{thm} Let $\xi\in\mathbb{R}^d$ and let $f$ satisfy the assumptions (A1)-(A4). Then there exists $C\in [0,+\infty)$ such that for all $n\in\mathbb{N}$ the following holds
\begin{equation}
	\sup\limits_{a\leq t\leq b}\|z(t)-l_n(t)\|\leq Ch^{\varrho}.
\end{equation}
\end{thm}
\begin{proof}
	Note that for all $t\in [a,b]$ we can write
	\begin{equation}
		l_n(t)=\xi+\int\limits_a^t\sum\limits_{k=0}^{n-1}\mathbf{1}_{[t_k,t_{k+1})}(s)f(t_k,y_k)ds,
	\end{equation}
	and
	\begin{equation}
		z(t)=\xi+\int\limits_a^t\sum\limits_{k=0}^{n-1}\mathbf{1}_{[t_k,t_{k+1})}(s)f(s,z(s))ds.
	\end{equation}
Hence, for all $t\in [a,b]$
\begin{eqnarray}
\label{est_loc_eul_0}
	&&\|z(t)-l_n(t)\|\leq \int\limits_a^t\sum\limits_{k=0}^{n-1}
	\mathbf{1}_{[t_k,t_{k+1})}(s)\|f(s,z(s))-f(t_k,y_k)\|ds\notag\\
	&&\leq \int\limits_a^t\sum\limits_{k=0}^{n-1}
	\mathbf{1}_{[t_k,t_{k+1})}(s)\|f(s,z(s))-f(s,y_k)\|ds+\int\limits_a^t\sum\limits_{k=0}^{n-1}
	\mathbf{1}_{[t_k,t_{k+1})}(s)\|f(s,y_k)-f(t_k,y_k)\|ds.
\end{eqnarray}	
Moreover, from Lemma \ref{lemma_ode_1} (i) we have that
	\begin{equation}
	\label{bound_ode_111}
	\sup\limits_{a\leq t\leq b}\|z(t)\|\leq R_1:=\max\{1,K(b-a)\}e^{K(b-a)}(1+\|\xi\|),
\end{equation}
and for $k=0,1,\ldots,n-1$
\begin{equation}
	\|y_{k+1}\|\leq (1+Kh)\|y_k\|+Kh
\end{equation}
so by the discrete version of the Gronwall's lemma we have for all $n\in\mathbb{N}$
\begin{equation}
	\max\limits_{0\leq k\leq n}\|y_k\|\leq R_2:=e^{K(b-a)}(1+\|\xi\|).
\end{equation}
By (A3), (A4) for $H=R=R_1\vee R_2$ there exist $L_R,L_H\in [0,+\infty)$ such that for all $s\in [t_k,t_{k+1}]$, $k=0,1,\ldots,n$, and all $n\in\mathbb{N}$
\begin{eqnarray}
\label{est_loc_eul_1}
	&&\|f(s,z(s))-f(s,y_k)\|\leq L_R \|z(s)-y_k\|\leq L_R\|z(s)-z(t_k)\|+L_R \|z(t_k)-l_n(t_k)\|\notag\\
	&&\leq L_R C_1h+L_R \|z(t_k)-l_n(t_k)\|,
\end{eqnarray}
where $C_1=K\Bigl(1+\max\{1,K(b-a)\}e^{K(b-a)}\Bigr)(1+\|\xi\|)$ and
\begin{equation}
\label{est_loc_eul_2}
	\|f(s,y_k)-f(t_k,y_k)\|\leq L_H |s-t_k|^{\varrho}\leq L_H h^{\varrho}.
\end{equation}
Hence, by \eqref{est_loc_eul_0}, \eqref{est_loc_eul_1}, \eqref{est_loc_eul_2}, \eqref{est_dsc_phi} we get for all $t\in [a,b]$
\begin{eqnarray}
	&&\|z(t)-l_n(t)\|\leq (b-a)L_R C_1h+L_R\int\limits_a^t\sum\limits_{k=0}^{n-1} 	\mathbf{1}_{[t_k,t_{k+1})}(s)\|z(t_k)-l_n(t_k)\|ds\notag\\
	&&+(b-a)L_Hh^{\varrho}\notag\\
	&&\leq (b-a)\Bigl(L_RC_1(1+b-a)+L_H\Bigr)h^{\varrho}+L_R\int\limits_a^t\sup\limits_{a\leq u\leq s}\|z(u)-l_n(u)\|ds,
\end{eqnarray}
which in turn implies that for all $t\in [a,b]$
\begin{equation}
	\sup\limits_{a\leq u\leq t}\|z(u)-l_n(u)\|\leq (b-a)\Bigl(L_RC_1(1+b-a)+L_H\Bigr)h^{\varrho}+L_R\int\limits_a^t\sup\limits_{a\leq u\leq s}\|z(u)-l_n(u)\|ds.
\end{equation}
The function $[a,b]\ni t\to \sup\limits_{a\leq u\leq t}\|z(u)-l_n(u)\|$ is Borel measurable (as a nondecreasing function, see Proposition \ref{meas_fun_examples}) and bounded, since $z,l_n\in C([a,b];\mathbb{R}^d)$ and therefore 
\begin{equation}
	\sup\limits_{a\leq t\leq b}\sup\limits_{a\leq u\leq t}\|z(u)-l_n(u)\|\leq \sup\limits_{a\leq u\leq b}\|z(u)-l_n(u)\|\leq \sup\limits_{a\leq u\leq b}\|z(u)\|+\sup\limits_{a\leq u\leq b}\|l_n(u)\|<+\infty.
\end{equation}
Hence, by Gronwall's lemma \ref{GLCW} we get 
\begin{equation}
	\sup\limits_{a\leq t\leq b}\|z(t)-l_n(t)\|\leq e^{L_R(b-a)}(b-a)\Bigl(L_RC_1(1+b-a)+L_H\Bigr)h^{\varrho},
\end{equation}
from which the thesis follows.
\end{proof}
\section{Randomized Euler scheme}
Let $(\Omega,\Sigma,\mathbb{P})$ be a complete probability space. Let $(\tau_j)_{j\in\mathbb{N}_0}$ be a sequence of independent random variables on $(\Omega,\Sigma,\mathbb{P})$ that are identically uniformly distributed on $[0,1]$. Let $n\in\mathbb{N}$, $h=(b-a)/n$, $t_k=a+kh$ for $k=0,1,\ldots,n$.  We define the randomized Euler scheme as follows. Take
\begin{equation}
	y_0=\xi,
\end{equation}
and
\begin{equation}
	y_{k+1}=y_k+hf(\theta_k,y_k),
\end{equation}
for $k=0,1,\ldots,n-1$, where $\theta_k=t_k+h\tau_k$ is uniformly distributed in $[t_k,t_{k+1}]$. Moreover, as a time-continuous approximation of $z$ in the whole interval $[a,b]$ we set
\begin{equation}
	l_n(t):=l_{k,n}(t) \ \hbox{for} \  t\in [t_k,t_{k+1}], \ k=0,1,\ldots,n-1,
\end{equation}
where
\begin{equation}
	l_{k,n}(t)=y_k+(t-t_k)f(\theta_k,y_k).
\end{equation}
Note that this time $\{l_n(t)\}_{t \in [a,b]}$ is a stochastic process with continuous trajectories. In order to construct the process $l_n$ we use $n$ (random) evaluations of $f$.
\begin{lem} Let $p\in [2,+\infty)$, $\xi\in\mathbb{R}^d$ and let $f$ satisfy the assumptions (A1)-(A4). Then for all $n\in\mathbb{N}$ the function
\begin{equation}
\label{error_f_raneu}
	[a,b]\ni t\to \mathbb{E}\Bigl(\sup\limits_{a\leq u\leq t}\|z(u)-l_n(u)\|^p\Bigr)
\end{equation}
is Borel measurable and bounded.
\end{lem}
\begin{proof} Note that for all $u\in [a,b]$ the mapping $\Omega\ni\omega\to \|z(u)-l_n(u,\omega)\|^p$ is a nonnegative random variable. Moreover, trajectories of the stochastic process $\{\|z(u)-l_n(u)\|^p\}_{u\in [a,b]}$ are continuous and therefore for all $t\in [a,b]$, $c\in\mathbb{R}$
 \begin{eqnarray}
 	&&\Bigl\{\omega\in\Omega \ | \ \sup\limits_{a\leq u\leq t}\|z(u)-l_n(u,\omega)\|^p\leq c\Bigr\}=\Bigl\{\omega\in\Omega \ | \ \sup\limits_{u\in [a,t]\cap\mathbb{Q}}\|z(u)-l_n(u,\omega)\|^p\leq c\Bigr\}\notag\\
 	&&=\bigcap_{u\in [a,t]\cap\mathbb{Q}}\Bigl\{\omega\in\Omega \ | \ \|z(u)-l_n(u,\omega)\|^p\leq c\Bigr\}\in\Sigma.
 \end{eqnarray}
 This means that for every $t\in [a,b]$ the function $\Omega\ni\omega\to \sup\limits_{a\leq u \leq t}\|z(u)-l_n(u,\omega)\|^p$ is a nonnegative random variable. Hence the function \eqref{error_f_raneu} is well-defined. The function \eqref{error_f_raneu} is Borel measurable as a nondecreasing function. Boundedness is left as an exercise.
\end{proof}
\begin{thm} 
\label{ran_Eu_lp}
Let $p\in [2,+\infty)$, $\xi\in\mathbb{R}^d$ and let $f$ satisfy the assumptions (A1)-(A4). Then there exists $C\in [0,+\infty)$ such that for all $n\in\mathbb{N}$ the following holds
\begin{equation}
	\Bigl\|\sup\limits_{a\leq t\leq b}\|z(t)-l_n(t)\|\Bigl\|_p\leq Ch^{\min\{\varrho+\frac{1}{2},1\}}.
\end{equation}
\end{thm}
\begin{proof}
 We have that for all $t\in [a,b]$
	\begin{equation}
		l_n(t)=\xi+\int\limits_a^t\sum\limits_{k=0}^{n-1}\mathbf{1}_{[t_k,t_{k+1})}(s)f(\theta_k,y_k)ds,
	\end{equation}
	and, as before,
	\begin{equation}
		z(t)=\xi+\int\limits_a^t\sum\limits_{k=0}^{n-1}\mathbf{1}_{[t_k,t_{k+1})}(s)f(s,z(s))ds.
	\end{equation}
Hence, for all $t\in [a,b]$
\begin{eqnarray}
\label{est_loc_raneul_0}
	&&\|z(t)-l_n(t)\|\leq \Bigl\|\int\limits_a^t\sum\limits_{k=0}^{n-1}
	\mathbf{1}_{[t_k,t_{k+1})}(s)\Bigl(f(s,z(s))-f(s,z(t_k))\Bigr)ds\Bigl\|\notag\\
	&&+ \Bigl\|\int\limits_a^t\sum\limits_{k=0}^{n-1}
	\mathbf{1}_{[t_k,t_{k+1})}(s)\Bigl(f(s,z(t_k))-f(\theta_k,z(t_k))\Bigr)ds\Bigl\|\notag\\
	&&+\Bigl\|\int\limits_a^t\sum\limits_{k=0}^{n-1}
	\mathbf{1}_{[t_k,t_{k+1})}(s)\Bigl(f(\theta_k,z(t_k))-f(\theta_k,y_k)\Bigr)ds\Bigl\|.
\end{eqnarray}	
Again from Lemma \ref{lemma_ode_1} (i) we have that
	\begin{equation}
	\label{bound_ode_111}
	\sup\limits_{a\leq t\leq b}\|z(t)\|\leq R_1:=\max\{1,K(b-a)\}e^{K(b-a)}(1+\|\xi\|),
\end{equation}
and for $k=0,1,\ldots,n-1$
\begin{equation}
	\|y_{k+1}\|\leq (1+Kh)\|y_k\|+Kh,
\end{equation}
so by the discrete version of the Gronwall's lemma
\begin{equation}
	\max\limits_{0\leq k\leq n}\|y_k\|\leq R_2:=e^{K(b-a)}(1+\|\xi\|)
\end{equation}
for all $n\in\mathbb{N}$ with probability one. By (A3), (A4) for $H=R=R_1\vee R_2$ there exist $L_R,L_H\in [0,+\infty)$ such that for all $s\in [t_k,t_{k+1}]$, $k=0,1,\ldots,n$, and all $n\in\mathbb{N}$ 
\begin{eqnarray}
\label{est_loc_raneul_1}
	&&\|f(s,z(s))-f(s,z(t_k))\|\leq L_R \|z(s)-z(t_k)\|\leq L_R C_1h,
\end{eqnarray}
where $C_1=K\Bigl(1+\max\{1,K(b-a)\}e^{K(b-a)}\Bigr)(1+\|\xi\|)$,
\begin{equation}
\label{est_loc_raneul_2}
	\|f(s,z(t_k))-f(\theta_k,z(t_k))\|\leq L_H |s-\theta_k|^{\varrho}\leq L_H h^{\varrho},
\end{equation}
\begin{equation}
\label{est_loc_raneul_3}
	\|f(\theta_k,z(t_k))-f(\theta_k,y_k)\|\leq L_R \|z(t_k)-l_n(t_k)\|,
\end{equation}
with probability one. Hence, by the estimates above we get for all $t\in [a,b]$
\begin{equation}
\label{est_err_raneu2}
\Bigl\|\int\limits_a^t\sum\limits_{k=0}^{n-1}
	\mathbf{1}_{[t_k,t_{k+1})}(s)\Bigl(f(s,z(s))-f(s,z(t_k))\Bigr)ds\Bigl\|\leq (b-a)L_RC_1h,
\end{equation}
\begin{equation}
\label{est_err_raneu3}
	\Bigl\|\int\limits_a^t\sum\limits_{k=0}^{n-1}
	\mathbf{1}_{[t_k,t_{k+1})}(s)\Bigl(f(\theta_k,z(t_k))-f(\theta_k,y_k)\Bigr)ds\Bigl\|\leq L_R\int\limits_a^t\sum\limits_{k=0}^{n-1}
	\mathbf{1}_{[t_k,t_{k+1})}(s)\|z(t_k)-l_n(t_k)\|ds.
\end{equation}
We define
\begin{equation}
	i(t)=\sup\{i=0,1,\ldots,n \ | \ t_i\leq t\}, 
\end{equation}
\begin{equation}
	\zeta(t)=t_{i(t)},
\end{equation}
for $t\in [a,b]$. Then
\begin{equation}
\label{est_a12n}
	\Bigl\|\int\limits_a^t\sum\limits_{k=0}^{n-1}
	\mathbf{1}_{[t_k,t_{k+1})}(s)\Bigl(f(s,z(t_k))-f(\theta_k,z(t_k))\Bigr)ds\Bigl\|\leq A^{RE}_{1,n}(t)+A^{RE}_{2,n}(t),
\end{equation}
with
\begin{equation}
	A^{RE}_{1,n}(t)=\Biggl\|\sum_{k=0}^{i(t)-1}\int\limits_{t_k}^{t_{k+1}}\Bigl(f(s,z(t_k))-f(\theta_k,z(t_k))\Bigr)ds\Biggl\|
\end{equation}
and
\begin{equation}
	A^{RE}_{2,n}(t)=\Biggl\|\;\int\limits_{\zeta(t)}^t\Bigl(f(s,z(\zeta(t)))-f(\theta_{i(t)},z(\zeta(t)))\Bigr)ds\Biggl\|,
\end{equation}
where $A^{RE}_{2,n}(b)=0$. By (A4) and \eqref{bound_ode_111} we have for all $t\in [a,b]$ that $0\leq t-\zeta(t)\leq h$ and hence
\begin{equation}
\label{est_a2n}
	A^{RE}_{2,n}(t)\leq L_H\int\limits_{\zeta(t)}^t|s-\theta_{i(t)}|^{\varrho}ds\leq L_H h^{\varrho+1}.
\end{equation}
Let for $k=0,1,\ldots,n-1$
\begin{equation}
	\label{DEF_T_Y_K}
	Y_k=\int\limits_{t_k}^{t_{k+1}}\Bigl(f(s,z(t_k))-f(\theta_k,z(t_k))\Bigr)ds, 
\end{equation}
then by \eqref{est_loc_raneul_2}
\begin{equation}
	\|Y_k\|\leq L_Hh^{\varrho+1}
\end{equation}
with probability one. Moreover, let
\begin{equation}
	Z_j=\sum_{k=0}^j Y_k, \ j=0,1,\ldots,n-1,
\end{equation}
and $Z_{-1}:=0$.  
Then we can write that for $t\in [a,b]$ 
\begin{equation}
	 A^{RE}_{1,n}(t)=\|Z_{i(t)-1}\|=\Bigl\|\sum_{k=0}^{n-1}Z_{k-1}\cdot\mathbf{1}_{[t_k,t_{k+1})}(t)+Z_{n-1}\cdot\mathbf{1}_{\{b\}}(t)\Bigl\|.
\end{equation}
 Therefore, we have for all $t\in [a,b]$ that
\begin{equation}
 \mathbb{E}\Biggl(\sup\limits_{t\in [a,b]}(A^{RE}_{1,n}(t))^p\Biggr)=\mathbb{E}\Biggl(\max\limits_{0\leq j\leq n-1}\|Z_j\|^p\Biggr).
\end{equation}
Note that for all $n\in\mathbb{N}$ the process
\begin{equation}
	\Bigl(Z_j, \ \sigma(\theta_0,\theta_1,\ldots,\theta_{j})\Bigr)_{j=0,1,\ldots,n-1}
\end{equation}
is a discrete-time martingale, since for $j=0,1,\ldots,n-1$
\begin{equation}
	\|Z_j\|\leq (b-a)L_Hh^{\varrho}<+\infty,
\end{equation}
and for $j=0,1,\ldots,n-2$
\begin{equation}
\label{mart_zj}
	\mathbb{E}\Bigl(Z_{j+1}-Z_j \ | \  \sigma(\theta_0,\theta_1,\ldots,\theta_{j})\Bigr) = \mathbb{E}(Y_{j+1} \ | \ \sigma(\theta_0,\theta_1,\ldots,\theta_{j})) = \mathbb{E}(Y_{j+1})=0.
\end{equation}
In \eqref{mart_zj} we used the following facts:
\begin{itemize}
 	\item $\sigma(Y_{j+1})$ and $\ \sigma(\theta_0,\theta_1,\ldots,\theta_{j})$ are independet $\sigma$-fields,
 	\item since $\theta_{j+1}\sim U[t_{j+1},t_{j+2}]$, it holds
 	\begin{equation}
	\mathbb{E}\Bigl(f(\theta_{j+1},z(t_{j+1}))\Bigr)=\frac{1}{h}\int\limits_{t_{j+1}}^{t_{j+2}}f(s,z(t_{j+1}))ds.
\end{equation}
\end{itemize}
By using the Burkholder inequality (see Theorem \ref{BDG_DISC}) and the estimate \eqref{BDG_UPP_EST} we obtain 
\begin{equation}
\label{est_a1n}
	\mathbb{E}\Biggl(\max\limits_{0\leq j\leq n-1}\|Z_j\|^p\Biggr)\leq C_p^p\mathbb{E}\Bigl[\Bigl(\sum\limits_{k=0}^{n-1}\|Y_k\|^2\Bigr)^{p/2}\Bigr]\leq C_p^p n^{\frac{p}{2}-1}\sum\limits_{k=0}^{n-1}\mathbb{E}\|Y_k\|^p\leq C_p^pL_H^p(b-a)^{p/2}h^{p(\varrho+\frac{1}{2})}.
\end{equation}
Combining \eqref{est_a12n}, \eqref{est_a2n}, \eqref{est_a1n} we get
\begin{equation}
\label{est_raneu_bdg}
	\mathbb{E}\Biggl(\sup\limits_{a\leq t \leq b}\Bigl\|\int\limits_a^t\sum\limits_{k=0}^{n-1}
	\mathbf{1}_{[t_k,t_{k+1})}(s)\Bigl(f(s,z(t_k))-f(\theta_k,z(t_k))\Bigr)ds\Bigl\|^p\Biggr)\leq ch^{p(\varrho+\frac{1}{2})},
\end{equation}
where $c>0$ does not depend on $n$. From \eqref{est_loc_raneul_0}, \eqref{est_err_raneu2}, \eqref{est_err_raneu3}, by H\"older and Jensen inequalities we get for all $t\in [a,b]$
\begin{eqnarray}
	&& \sup\limits_{a\leq u\leq t}\|z(u)-l_n(u)\|^p\leq c_1h^p+c_2\int\limits_a^t\sum\limits_{k=0}^{n-1}
	\mathbf{1}_{[t_k,t_{k+1})}(s)\|z(t_k)-l_n(t_k)\|^p ds\notag\\
	&&+c_3\sup\limits_{a\leq t \leq b}\Bigl\|\int\limits_a^t\sum\limits_{k=0}^{n-1}
	\mathbf{1}_{[t_k,t_{k+1})}(s)\Bigl(f(s,z(t_k))-f(\theta_k,z(t_k))\Bigr)ds\Bigl\|^p\notag\\
	&&\leq  c_1h^p+c_2\int\limits_a^t\sup\limits_{a\leq u\leq s}\|z(u)-l_n(u)\|^p ds\notag\\
	&&+c_3\sup\limits_{a\leq t \leq b}\Bigl\|\int\limits_a^t\sum\limits_{k=0}^{n-1}
	\mathbf{1}_{[t_k,t_{k+1})}(s)\Bigl(f(s,z(t_k))-f(\theta_k,z(t_k))\Bigr)ds\Bigl\|^p
\end{eqnarray}
hence by Fubini's theorem (see exercise 5.) and \eqref{est_raneu_bdg}
\begin{equation}
	\mathbb{E}\Bigl(\sup\limits_{a\leq u\leq t}\|z(u)-l_n(u)\|^p\Bigr)\leq c_4h^{p\min\{\varrho+\frac{1}{2},1\}}+c_2\int\limits_a^t\mathbb{E}\Bigl(\sup\limits_{a\leq u\leq s}\|z(u)-l_n(u)\|^p\Bigr) ds.
\end{equation}
By applying Gronwall's lemma we get the thesis.
\end{proof}
\begin{rem}
    When $\varrho=1$ then the both algorithms have the same rate of convergence equal to $1$. However, for $\varrho\in (0,1)$ the randomized Euler algorithm outperforms its deterministic counterpart. Moreover, in the limiting case when $\varrho=0$ and $f$ is only continuous (or even only Borel measurable) with respect to the time variable $t$ then the randomized Euler scheme has the error $O(h^{1/2})$, while the error of its deterministic version is $O(1)$ and there is lack of convergence, see \cite{KruseWu_1}, \cite{PMPP14}.
\end{rem}
Note that in the special case, when $f=f(t)$, does not depend on the space variable $y$,  we get the following result.
\begin{prop} 
\label{rand_Riemann_Holder}
Let us assume that $\xi=0$ and let there exist $L_H\in [0,+\infty), \varrho\in (0,1]$ such that for all $t,s\in [a,b]$
\begin{equation}
	\|f(t)-f(s)\|\leq L_H |t-s|^{\varrho}.
\end{equation}
Then there exists $C\in [0,+\infty)$ such that for all $n\in\mathbb{N}$ the following holds
\begin{equation}
	\Bigl\|\sup\limits_{a\leq t\leq b}\|z(t)-l_n(t)\|\Bigl\|_p\leq Ch^{\varrho+\frac{1}{2}},
\end{equation}
the functions $[a,b]\ni t\to\mathbb{E}(l_n(t))\in\mathbb{R}$, $[a,b]\ni t\to Var(l_n(t))\in [0,+\infty)$ are continuous, and for all $t\in [a,b]$ we have
\begin{equation}
\label{exp_est_ln}
	\mathbb{E}(l_n(t))=\sum\limits_{k=0}^{n-1}\Bigl(z(t_k)+\frac{t-t_k}{h}(z(t_{k+1})-z(t_k))\Bigr)\cdot\mathbf{1}_{[t_k,t_{k+1}]}(t)+z(b)\cdot\mathbf{1}_{\{b\}}(t),
\end{equation}
while for $t\in [t_k,t_{k+1}]$, $k=0,1,\ldots,n-1$
\begin{eqnarray}
\label{var_est_ln}	
	&&Var(l_n(t))=h\cdot\Biggl[\int\limits_a^{t_k}(f(s))^2ds+\Bigl(\frac{t-t_k}{h}\Bigr)^2\int\limits_{t_k}^{t_{k+1}}(f(s))^2ds\Biggr]\notag\\
	&&\quad\quad-\sum\limits_{j=0}^{k-1}\Bigl(\int\limits_{t_j}^{t_{j+1}}f(s)ds\Bigr)^2-\Bigl(\frac{t-t_k}{h}\Bigr)^2\cdot\Bigl(\int\limits_{t_k}^{t_{k+1}}f(s)ds\Bigr)^2.
\end{eqnarray}
\end{prop}
\begin{proof}
	 We have that for all $t\in [a,b]$
	\begin{equation}
		l_n(t)=\int\limits_a^t\sum\limits_{k=0}^{n-1}\mathbf{1}_{[t_k,t_{k+1})}(s)f(\theta_k)ds,
	\end{equation}
	and, as before,
	\begin{equation}
		z(t)=\int\limits_a^t\sum\limits_{k=0}^{n-1}\mathbf{1}_{[t_k,t_{k+1})}(s)f(s)ds=\int\limits_a^t f(s)ds.
	\end{equation}
Hence, by \eqref{est_raneu_bdg} we get for all $t\in [a,b]$ that
\begin{equation}
\label{est_loc_raneul_01}
	\mathbb{E}\Bigl(\sup\limits_{a\leq t\leq b}\|z(t)-l_n(t)\|^p\Bigr)=\mathbb{E}\Biggl( \sup\limits_{a\leq t\leq b}\Bigl\|\int\limits_a^t\sum\limits_{k=0}^{n-1}
	\mathbf{1}_{[t_k,t_{k+1})}(s)\Bigl(f(s)-f(\theta_k)\Bigr)ds\Bigl\|^p\Biggr)\leq ch^{p(\varrho+\frac{1}{2})}.
\end{equation}	
The equalities \eqref{exp_est_ln}, \eqref{var_est_ln} follow from direct computations and their proofs are left as an exercise.
\end{proof}
In this case $l_n$ (which can be seen as a randomized Riemann sum) approximates in $[a,b]$ the function $z$, which is the anti-derivative of $f$. Moreover, $l_n(t)$ for $t\in\{t_0,t_1,\ldots,t_n\}$ is the unbiased estimator of $z(t)$. We can see that when $f=f(t)$ does not depend on $y$ then the order of convergence is $\varrho+1/2$ which might be faster than $\min\{\varrho+1/2,1\}$.

We will take a  closer look at variance of the randomized Riemann sum in the next subsection.
\subsection{Coming back to integration: variance of the randomized Riemann sum vs variance of the crude Monte Carlo method}
We assume that $f\in L^2[a,b]$ and let us  consider the problem of approximation of the Lebesgue integral
\begin{equation}
	I(f)=\int\limits_a^b f(t)dt \ (=z(b)).
\end{equation}
The crude Monte Carlo method is defined as
\begin{equation}
	MC_n(f)=h\sum\limits_{k=0}^{n-1} f(\xi_k),
\end{equation}
where $(\xi_k)_{k=0,1,\ldots,n-1}$ is the iid sequence of random variables with the law $U([a,b])$, while the randomized Riemann sum is defined as
\begin{equation}
	\mathcal{R}_n(f)=h\sum\limits_{k=0}^{n-1}f(\theta_k),
\end{equation}
where $(\theta_k)_{k=0,1,\ldots,n-1}$ is the sequence of independent random variables such that each $\theta_k$ has the law $U[t_k,t_{k+1}]$, $t_k=a+kh$, $k=0,1,\ldots,n$, $h=(b-a)/n$. Note that the both quadratures use $n$ random evaluations of $f$.

It holds that $\mathbb{E}(MC_n(f))=I(f)$ and
\begin{equation}
	\mathbb{E}|I(f)-MC_n(f)|^2=Var(MC_n(f))=(b-a)\cdot s^2(f)\cdot n^{-1},
\end{equation}
where $$s^2(f)=I(f^2)-\frac{1}{b-a}(I(f))^2.$$ For $\mathcal{R}_n$ we have that $\mathbb{E}(\mathcal{R}(f))=I(f)$ and
\begin{equation}
	\mathbb{E}|I(f)-\mathcal{R}_n(f)|^2=Var(\mathcal{R}_n(f))=(b-a)\cdot s_n^2(f)\cdot n^{-1},
\end{equation}
with
\begin{equation}
	s_n^2(f)=I(f^2)-\frac{1}{h}\sum\limits_{k=0}^{n-1}\Bigl(\int\limits_{t_k}^{t_{k+1}}f(s)ds\Bigr)^2,
\end{equation}
where $0\leq s^2_n(f)\leq s^2(f)$. Moreover, even stronger result holds, see the thesis of the following lemma.
\begin{lem}
\label{conv_sn}
	For any $f\in L^2[a,b]$ we have
	\begin{equation}
	\label{conv_sn_1}
		\lim\limits_{n\to +\infty}s_n^2(f)=0.
	\end{equation}
\end{lem}
\begin{proof}
	Let us assume first that $f\in C[a,b]$. Then by the mean value theorem for integrals we get that there exist $\xi_k\in [t_k,t_{k+1}]$ for $k=0,1,\ldots,n-1$ and therefore
	\begin{equation}
	\frac{1}{h}\sum\limits_{k=0}^{n-1}\Bigl(\int\limits_{t_k}^{t_{k+1}}f(s)ds\Bigr)^2=\sum\limits_{k=0}^{n-1}(f(\xi_k))^2\cdot h\to \int\limits_{a}^b (f(s))^2ds
	\end{equation}
	as $n\to +\infty$. Hence, \eqref{conv_sn_1} holds for any $f\in C[a,b]$.
	
	Recall that $C[a,b]$ is a dense subset of $L^2[a,b]$, see, for example, Proposition 3.4.4., page 101 in \cite{cohn}. Hence, for $f\in L^2[a,b]$ and all $\varepsilon>0$ there exists $f_{\varepsilon}\in C[a,b]$ such that 
	\begin{equation}
	\label{dense_f}
		\|f-f_{\varepsilon}\|_{L^2[a,b]}<\varepsilon.
\end{equation}
In particular, by the first part of the proof we have for all $\varepsilon>0$ that
\begin{equation}
\label{lim_sne}
	\lim\limits_{n\to +\infty}s^2_n(f_\varepsilon)=0.
\end{equation}	 
Moreover, we can write that
\begin{equation}
	s_n^2(f)=\|f\|^2_{L^2[a,b]}-\frac{1}{h}\sum\limits_{k=0}^{n-1}\Bigl(\int\limits_{t_k}^{t_{k+1}}f(s)ds\Bigr)^2,
\end{equation}
and 
\begin{equation}
\label{est_sn_1}
	|s_n^2(f)-s_n^2(f_{\varepsilon})|\leq E_{1,n,\varepsilon}(f)+\frac{1}{h}\cdot 
	E_{2,n,\varepsilon}(f),
\end{equation}
where
\begin{equation}
	E_{1,n,\varepsilon}(f)=\Bigl|\|f\|^2_{L^2[a,b]}-\|f_{\varepsilon}\|^2_{L^2[a,b]}\Bigl|
\end{equation}
\begin{equation}
	E_{2,n,\varepsilon}(f)=\Biggl|\sum\limits_{k=0}^{n-1}\Bigl(\int\limits_{t_k}^{t_{k+1}}f(s)ds\Bigr)^2-\sum\limits_{k=0}^{n-1}\Bigl(\int\limits_{t_k}^{t_{k+1}}f_{\varepsilon}(s)ds\Bigr)^2\Biggl|.
\end{equation}
By \eqref{dense_f} we have for all $n\in\mathbb{N}$ and all $\varepsilon>0$ that
\begin{eqnarray}
\label{est_sn_2}
 &&E_{1,n,\varepsilon}(f)=\Bigl|\|f\|_{L^2[a,b]}-\|f_{\varepsilon}\|_{L^2[a,b]}\Bigl|\cdot \Bigl(\|f\|_{L^2[a,b]}+\|f_{\varepsilon}\|_{L^2[a,b]}\Bigr)\notag\\
 &&\leq \|f-f_{\varepsilon}\|_{L^2[a,b]}\cdot \Bigl(2\|f\|_{L^2[a,b]}+\|f-f_{\varepsilon}\|_{L^2[a,b]}\Bigr)\leq \varepsilon\cdot (2\|f\|_{L^2[a,b]}+\varepsilon),
\end{eqnarray}
and, by the Schwarz inequality,
\begin{eqnarray}
\label{est_sn_3}
	&&E_{2,n,\varepsilon}(f)\leq\sum\limits_{k=0}^{n-1}\Biggl|\Bigl(\int\limits_{t_k}^{t_{k+1}}f(s)ds\Bigr)^2-\Bigl(\int\limits_{t_k}^{t_{k+1}}f_{\varepsilon}(s)ds\Bigr)^2\Biggl|\notag\\
	&&=\sum\limits_{k=0}^{n-1}\Biggl|\int\limits_{t_k}^{t_{k+1}}(f(s)-f_{\varepsilon}(s))ds\Biggl|\cdot\Biggl|\int\limits_{t_k}^{t_{k+1}}(f(s)+f_{\varepsilon}(s))ds\Biggl|\notag\\
	&&\leq \Biggl(\sum\limits_{k=0}^{n-1}\Bigl|\int\limits_{t_k}^{t_{k+1}}(f(s)-f_{\varepsilon}(s))ds\Bigl|^2\Biggr)^{1/2}\times \Biggl(\sum\limits_{k=0}^{n-1}\Bigl|\int\limits_{t_k}^{t_{k+1}}(f(s)+f_{\varepsilon}(s))ds\Bigl|^2\Biggr)^{1/2}\notag\\
	&&\leq \Biggl(\sum\limits_{k=0}^{n-1}h\cdot\int\limits_{t_k}^{t_{k+1}}|f(s)-f_{\varepsilon}(s)|^2 ds\Biggr)^{1/2}\times \Biggl(\sum\limits_{k=0}^{n-1}h\cdot\int\limits_{t_k}^{t_{k+1}}|f(s)+f_{\varepsilon}(s)|^2 ds\Biggr)^{1/2}\notag\\
	&&=h \cdot \|f-f_{\varepsilon}\|_{L^2[a,b]}\cdot \|f+f_{\varepsilon}\|_{L^2[a,b]}\leq h\cdot \|f-f_{\varepsilon}\|_{L^2[a,b]}\cdot \Bigl(\|f\|_{L^2[a,b]}+\|f_{\varepsilon}\|_{L^2[a,b]}\Bigr)\notag\\
	&&\leq h\cdot \|f-f_{\varepsilon}\|_{L^2[a,b]}\cdot \Bigl(2\|f\|_{L^2[a,b]}+\|f-f_{\varepsilon}\|_{L^2[a,b]}\Bigr)\leq h \cdot\varepsilon\cdot (2\|f\|_{L^2[a,b]}+\varepsilon).
\end{eqnarray}
Hence, by \eqref{est_sn_1}, \eqref{est_sn_2}, \eqref{est_sn_3} we get for all $n\in\mathbb{N}$ and all $\varepsilon>0$
\begin{equation}
	|s_n^2(f)-s_n^2(f_{\varepsilon})|\leq 2\varepsilon\cdot (2\|f\|_{L^2[a,b]}+\varepsilon),
\end{equation}
which, by \eqref{lim_sne} and the fact that $s_n^2(f)\geq 0$, gives for all $\varepsilon>0$
\begin{equation}
	0\leq\liminf\limits_{n\to +\infty}s_n^2(f)\leq \limsup\limits_{n\to +\infty} s^2_n(f)\leq 2\varepsilon\cdot (2\|f\|_{L^2[a,b]}+\varepsilon),
\end{equation}
and this implies the thesis.
\end{proof}
By Lemma \ref{conv_sn} we have that $s^2_n(f)<s^2(f)$ for sufficiently large $n$. Hence, the randomized Riemann quadrature $\mathcal{R}_n$ has asymptotically smaller variance than the classical Monte Carlo method $MC_n$.  Moreover, due to low smoothness of $f\in L^2[a,b]$, for approximation of $I(f)$ there are not too many other options than the randomized Riemann sum $\mathcal{R}_n$.

Note that we cannot use CLT in order to construct empirical confidence intervals for $\mathcal{R}_n$, as it was for $MC_n$, since the random variables $f(\theta_k)$ do not have the same distribution. However, in the case when the Borel measurable function $f:[a,b]\to\mathbb{R}$ is bounded (i.e., $\|f\|_{\infty}<+\infty$) we can use the Hoeffding-Azuma inequality in order to design non-asymptotic confidence intervals, see Theorem \ref{HA_ineq}.
\subsection{Convergence with probability one}
By using Lemma \ref{BC_asconv} and Theorem \ref{ran_Eu_lp} we get the following result on convergence with probability one of the randomized Euler scheme.
\begin{cor} Let  $\xi\in\mathbb{R}^d$ and let $f$ satisfy the assumptions (A1)-(A4) for some $\varrho\in (0,1]$. Then for all $\varepsilon\in (0,\min\{\varrho+\frac{1}{2},1\})$ there exists a random variable $\eta_{\varepsilon}$ such that
\begin{equation}
\sup\limits_{a\leq u\leq b}\|z(u)-l_n(u)\|\leq \eta_{\varepsilon}\cdot n^{-\min\{\varrho+\frac{1}{2},1\}+\varepsilon} \ \hbox{almost surely} 
\end{equation}
for all $n\in\mathbb{N}$.
\end{cor}
\section{Deterministic two-stage Runge-Kutta method}

 The definition of the deterministic two-stage Runge-Kutta method (also known as an explicit mid-point method) goes as follows
\begin{equation}
	y^0=\xi,
\end{equation}
and
\begin{eqnarray}
    &&	y^{j}_{1/2}=y^{j-1}+\frac{h}{2} f(t_{j-1},y^{j-1}),\\
    && 	y^{j}=y^{j-1}+hf(t_{j-1}+\frac{h}{2},y^{j}_{1/2}),
\end{eqnarray}
for $j=1,2,\ldots,n$. As a time-continuous approximation of $z$ in the whole interval $[a,b]$ we set
\begin{equation}
	l_n(t):=l_{j,n}(t) \ \hbox{for} \  t\in [t_j,t_{j+1}], \ j=0,1,\ldots,n-1,
\end{equation}
where
\begin{equation}
	l_{k,n}(t)=\frac{y^{j+1}-y^{j}}{t_{j+1}-t_j}(t-t_j)+y^j=y^j+(t-t_j)f(t_{j}+\frac{h}{2},y^{j+1}_{1/2}).
\end{equation}
\section{Randomized two-stage Runge-Kutta method}
As in the case of randomized Euler scheme, let $(\tau_j)_{j\in\mathbb{N}_0}$ be a sequence of independent random variables on $(\Omega,\Sigma,\mathbb{P})$ that are identically uniformly distributed on $[0,1]$. Let $n\in\mathbb{N}$, $h=(b-a)/n$, $t_j=a+jh$ for $j=0,1,\ldots,n$.  We define the randomized two-stage Runge-Kutta method as follows
\begin{equation}
	y^0=\xi,
\end{equation}
and
\begin{eqnarray}
    &&	y^{j}_{\tau}=y^{j-1}+h\tau_j f(t_{j-1},y^{j-1}),\\
    && 	y^{j}=y^{j-1}+hf(\theta_j,y^{j}_{\tau}),
\end{eqnarray}
for $j=1,2,\ldots,n$, where $\theta_{j}=t_{j-1}+h\tau_j$ is uniformly distributed in $[t_{j-1},t_{j}]$. Moreover, as a time-continuous approximation of $z$ in the whole interval $[a,b]$ we set
\begin{equation}
	l_n(t):=l_{j,n}(t) \ \hbox{for} \  t\in [t_j,t_{j+1}], \ j=0,1,\ldots,n-1,
\end{equation}
where
\begin{equation}
	l_{k,n}(t)=\frac{y^{j+1}-y^{j}}{t_{j+1}-t_j}(t-t_j)+y^j=y^j+(t-t_j)f(\theta_{j+1},y^{j+1}_{\tau}).
\end{equation}
In particular, if we take $\tau_j=1/2$ for all $j$ then we get the deterministic mid-point method.
\begin{rem}
    Randomized algorithms for ODEs and their errors have been investigated, for example,  in \cite{JenNeuen}, \cite{KruseWu_1}, \cite{stengle1}, \cite{stengle2}, \cite{TBPhD}, \cite{randRK}, \cite{HeinMilla}.
\end{rem}
\section{On connection between Euler/RK2 and gradient descent algorithms}
Let us consider the following optimization problem
\begin{equation}
\label{opt_prob}
    \min\limits_{x\in \mathbb{R}^d}f(x)
\end{equation}
for a given $f\in C^1(\mathbb{R}^d;\mathbb{R})$. We associate with \eqref{opt_prob} the following ODE
\begin{equation}
	\label{grad_ODE_PROBLEM1}
		\left\{ \begin{array}{ll}
			x'(t)=-\nabla f(x(t)), &t\in [0,+\infty), \\
			x(0)=x_0.
		\end{array}\right.
\end{equation}
Such a systems are called {\it gradient systems}. We assume that the gradient $\mathbb{R}^d\ni x\to\nabla f(x)\in\mathbb{R}^d$ satisfies the assumptions (A1), (A2), (A3).  Then the equation \eqref{grad_ODE_PROBLEM1} has a unique solution $x$ that is defined on the whole $[0,+\infty)$, see Remark \ref{grad_sys_ex_1}.
Note that we have
\begin{equation}
    \frac{d}{dt} f(x(t))=\langle\nabla f(x(t)),x'(t)\rangle=-\|\nabla f(x(t))\|^2\leq 0.
\end{equation}
Hence, $f$ is non-increasing along the trajectory of $x$. Furthermore, we have the following result  about convergence of $x$ to the local minimum of $f$. The proof can be found in \cite{VB1}, page 142.
\begin{thm}
    If $x^*$ is an isolated critical point of $f$ which is also a local minimum, then the stationary solution $x=x^*$ is an asymptotically stable solution of \eqref{grad_ODE_PROBLEM1}.
\end{thm}
Theorem above implies that if the initial-value $x_0$ is sufficiently close to $x^*$ then $\lim\limits_{t\to +\infty}x(t)=x^*$. Approaching the local minimum $x^*$ of $f$ by observing the trajectory $x=x(t)$ while $t\to +\infty$ is sometimes called {\it continuous learning}. 

For $h>0$ and  $t_k=kh$, $k=0,1,\ldots$, the Euler scheme for \eqref{grad_ODE_PROBLEM1} takes the form
\begin{eqnarray}
\label{EU_GD}
    &&y_0=x_0,\notag\\
    &&y_{k+1}=y_k-h\cdot\nabla f(y_k), \ k=0,1,\ldots
\end{eqnarray}
This is the well-known gradient descent (GD) algorithm for searching a (local) minimum of $f$, when starting from $x_0$. In this context the stepsize $h$ can be interpreted as the learning rate hyperparameter. Using (GD) method in order to minimize $f$ is also called {\it discrete learning}. For discrete learning, when $h$ is sufficiently small and $\nabla f$ is globally Lipschitz, we can only be sure that either $f(x_k)\to -\infty$ either $\nabla f(x_k)\to 0$ when $k\to+\infty$. In that sense we can say that discrete learning is harder than continuous one. However, since in general we do not know the exact solution to \eqref{grad_ODE_PROBLEM1}, we can provide on computer only the discrete one.

We discussed the origins of the gradient descent method (GD) from the perspective of gradient ODEs \eqref{EU_GD}. We now show how to obtain its stochastic counterpart called {\it stochastic gradient descent} (SGD). In ML many loss functions $f$ that we need to minimize have the following special form
\begin{equation}
 \label{anova_f}
     f(x)=\frac{1}{n}\sum_{j=1}^n f_j(x).
 \end{equation}
For such a function the (full) GD method has the form
\begin{eqnarray}
 \label{EU_GD_2}
    &&y_0=x_0,\notag\\
    &&y_{k+1}=y_k-h\cdot\frac{1}{n}\sum_{j=1}^n\nabla f_j(y_k), \ k=0,1,\ldots.
 \end{eqnarray}
 For large $n$ (which is mostly the size of the training data) computation of the next $y_{k+1}$ may be very computationally demanding and, therefore, time consuming. Note that we have 
 \begin{equation}
     f(x)=\frac{1}{n}\sum_{j=1}^n f_j(x)=\mathbb{E}(f_{\xi}(x))
\end{equation}
where $\xi\sim U(\{1,2,\ldots,n\})$. The idea of constructing the SGD is now to use Monte Carlo approximation to $\mathbb{E}(f_{\xi}(x))$. This can be done in the following way:
\begin{itemize}
    \item [(i)] we approximate  $\mathbb{E}(f_{\xi}(x))$ by using only one sample per step, i.e., we take $(\xi_j)_{j\in\mathbb{N}}$ an iid sequence from $U(\{1,2,\ldots,n\})$ and
    \begin{eqnarray}
     && y_0=x_0, \notag\\
     && y_{k+1} =  y_k-h\nabla f_{\xi_{k+1}}(y_k), \quad k= 0,1,\ldots
 \end{eqnarray}
    \item  [(ii)] we fix $M$, the size of the mini-batch, at each $(k+1)$th-step we sample $M$ ($M\ll n$) independent random numbers $\xi^{(1)}_{k+1},\xi^{(2)}_{k+1},\ldots,\xi^{(M)}_{k+1}$ from $U(\{1,2,\ldots,n\})$ and then we set
     \begin{eqnarray}
     && y_0=x_0, \notag\\
     && y_{k+1} =  y_k-h\cdot\frac{1}{M}\cdot\sum_{j=1}^M\nabla f_{\xi^{(j)}_{k+1}}(y_k), \quad k= 0,1,\ldots
 \end{eqnarray}
 (We also assume that all random variables $(\xi_k^{(j)})_{j=1,2,\ldots,M,k\in\mathbb{N}}$ are independent.)
\end{itemize}
 Hence, in the SGD version we do not compute the full gradient but, instead, we randomly pick an index $j$ and function $f_j$ for which we are computing gradient and then using it to establish the next step $y_{k+1}$.
 
There are of course other options of discretizing \eqref{grad_ODE_PROBLEM1}. For example, we can also apply randomized RK2 method to \eqref{grad_ODE_PROBLEM1} and obtain the approximation sequence:
\begin{eqnarray}
    &&y^0=x_0,\notag\\
    &&y^{k+1}_{\tau}=y^k-h\tau_{k+1}\cdot \nabla f(y^k)\notag\\
    &&y^{k+1}=y^k-h\cdot\nabla f(y^{k+1}_{\tau}), \ k=0,1,\ldots,
\end{eqnarray}
where $(\tau_k)_{k\in\mathbb{N}}$ is an iid sequence sampled from $U(0,1)$. Analogously, we can define the following {\it randomized two-stage stochastic gradient descent method} for the functions \eqref{anova_f}:
 \begin{eqnarray}
     && \bar y^0=x_0, \notag\\
     && \bar y^j_{\tau} = \bar y^{j-1}-h\tau_j \nabla f_{\xi_j}(\bar y^{j-1}),\notag\\
     && \bar y^j=\bar y^{j-1}-h\nabla f_{\zeta_j}(\bar y^j_{\tau}), \quad j=1,2,\ldots,
 \end{eqnarray}
where  $(\tau_k)_{k\in\mathbb{N}}$ is an iid sequence sampled from $U(0,1)$, $(\xi_j)_{j\geq 1}$ and $(\zeta_j)_{j\geq 1}$ are iid sequences sampled from $U(\{1,2,\ldots,n\})$, and finally $(\tau_j)_{j\geq 1}$, $(\xi_j)_{j\geq 1}$, $(\zeta_j)_{j\geq 1}$ are independent. 

We were not able to find any references or analysis of this method in the literature, however, we conjecture that this method posses better convergence properties than gradient descent method \eqref{EU_GD}.
\begin{rem}
    \label{grad_sys_ex_1} Since the gradient $\nabla f$ satisfies (A1), (A2), (A3), we have the existence and uniqueness of solution $x$ for \eqref{grad_ODE_PROBLEM1} on the interval $[0,T]$ for any $T>0$. Moreover, this solution is bounded on $[0,T]$ (which follows from the fact that $\nabla f$ is of at most linear growth) and therefore $x$ can be continued on the whole $[0,+\infty)$, see, for example, Theorem 2.11 and Example 2.1, pages 46-47 in \cite{VB1}.
\end{rem}
\section{Exercises}
\begin{itemize}
	\item [1.] Show that $x\vee y-x\wedge y=|x-y|$ 	for all $x,y\in\mathbb{R}$.
	\item [2.] Show that the function \eqref{f_a1a4} satisfies the assumptions (A1)-(A4).
	\item [3.] Justify the boundedness of the function \eqref{error_f_raneu}.
	\item [4.] Implement deterministic and randomized Euler scheme in Python and perform numerical experiments for the function \eqref{f_a1a4}.
	\item [5.] Prove that the stochastic process $\displaystyle{\{\sup\limits_{a\leq u\leq t}\|z(u)-l_n(u)\|^p\}_{t\in [a,b]}}$ is product measurable.
	\item [6.] Give a proof of \eqref{exp_est_ln}, \eqref{var_est_ln}.
	\item [7.] Show that $0\leq s^2_n(f)\leq s^2(f)$ for all $n\in\mathbb{N}$, $f\in L^2[a,b]$.
	\item [8.] Let us assume that $f$ is Borel measurable and bounded. Show that for any $n\in\mathbb{N}$, $\delta\in (0,1)$ we have that
	\begin{equation}
		I(f)\in \Bigl[\mathcal{R}_n(f)-\frac{C(\delta)}{\sqrt{n}},\mathcal{R}_n(f)+\frac{C(\delta)}{\sqrt{n}}\Bigr]
	\end{equation}
	with the probability at least $1-\delta$, where $C(\delta)=(8\ln (2/\delta))^{1/2}\cdot (b-a)\cdot\|f\|_{\infty}$.
 \item [9.] Show that for any discretization $a=t_0<t_1<\ldots t_n=b$ and any function $\phi:[a,b]\to [0,+\infty)$ the following holds 
 \begin{equation}
    \label{est_dsc_phi}
     \sum\limits_{k=0}^{n-1}\phi(t_k)\cdot\mathbf{1}_{[t_k,t_{k+1})}(s)\leq\sup\limits_{a\leq u\leq s}\phi(u)
 \end{equation}
 for any $s\in [a,b]$.
\end{itemize}
\chapter{Three canonical processes}
\label{three_processes}
In this chapter we assume that $(\Omega,\Sigma,(\Sigma_t)_{t\ [0,+\infty)},\mathbb{P})$ is a filtered probability space. We start with recalling definitions of  Wiener and Poisson processes (non-homogeneous and compound).
\section{Building blocks--Wiener and Poisson processes}
\begin{defn} An $\mathbb{R}^m$-valued stochastic process $W=(W(t))_{t\in [0,+\infty)}$ is called an {\it $m$-dimensional  $(\Sigma_t)_{t\in [0,+\infty)}$-Wiener process} on $(\Omega,\Sigma,\mathbb{P})$ (or a {\it $m$-dimensional Wiener process with respect to the filtration $(\Sigma_t)_{t\in [0,+\infty)}$}) if 
\begin{itemize}
	\item [(i)] $\mathbb{P}(W(0)=0)=1$,
	\item [(ii)] $\sigma(W(t))\subset\Sigma_t$ for all $t\in [0,+\infty)$,
	\item [(iii)] almost all trajectories of $W$ are continuous,
	\item [(iv)] $\sigma(W(t)-W(s))$ is independent of $\Sigma_s$ for all $0\leq s<t$,
	\item [(v)] for $0\leq s<t$ the increment $W(t)-W(s)$ is normally distributed with mean zero and covariance matrix equal to $(t-s)I_m$, where $I_m$ is the $\mathbb{R}^{m\times m}$-identity matrix.
\end{itemize}
\end{defn}
\begin{defn} An $\mathbb{R}^m$-valued stochastic process $N=(N(t))_{t\in [0,+\infty)}$ is called an {$m$-dimensional non-homogeneous $(\Sigma_t)_{t\in [0,+\infty)}$-Poisson process} on $(\Omega,\Sigma,\mathbb{P})$ (or an {\it $m$-dimensional non-homogeneous Poisson process with respect to the filtration $(\Sigma_t)_{t\in [0,+\infty)}$}) if 
\begin{itemize}
	\item [(i)] $\mathbb{P}(N(0)=0)=1$,
	\item [(ii)] $\sigma(N(t))\subset\Sigma_t$ for all $t\in [0,+\infty)$,
	\item [(iii)] almost all trajectories of $N$ are c\`adl\`ag,
	\item [(iv)] $\sigma(N(t)-N(s))$ is independent of $\Sigma_s$ for all $0\leq s<t$,
	\item [(v)] for $j\in\{1,2,\ldots,m\}$, $0\leq s<t$ the increment $N^{j}(t)-N^{j}(s)$ has Poisson law with the intensity parameter $\displaystyle{\int\limits_s^t\lambda_j(u)du}$,  where the intensity function $\lambda_j:[0,+\infty)\to (0,+\infty)$ is Borel and locally integrable,
	\item [(vi)] $\displaystyle\{\Sigma^{N^1}_{\infty},\ldots,\Sigma^{N^m}_{\infty}\}$ is a family of independent $\sigma$-fields.
\end{itemize}
\end{defn}
We also denote by $\displaystyle{m_j(t)=\int\limits_0^t\lambda_j(s)ds}$ and $\Lambda(t,s)=m(t)-m(s)$ for $s,t\in [0,T]$ and $j=1,2,\ldots,m_N$. If $\lambda_j$ is a constant function then we say that $N^j$ is a homogeneous Poisson process.
\begin{defn} Let $N=(N_t)_{t\in [0,+\infty)}$ be a one-dimensional homogeneous Poisson process, with respect to a given filtration $(\Sigma_t)_{t\in [0,+\infty)}$, and let $(\xi_j)_{j\in\mathbb{N}}$ be an i.i.d sequence of  real random variables with distribution $\mu$, such that the $\sigma$-fields $\Sigma^N_{\infty}$ and $\displaystyle{\sigma\Bigl(\bigcup_{j\in\mathbb{N}}\sigma(\xi_j)\Bigr)}$ are independent. The process
\begin{equation}
	J(t)=\sum\limits_{j=1}^{N(t)}\xi_j, \ t\in [0,+\infty),
\end{equation}
is called the {\it compound Poisson process}.
\end{defn}
Analogously we can define non-homogeneous compound Poisson process.

Due to Definition 13.9 and Theorem 13.11 in \cite{YEH} the condition $(iv)$ in the above definitions implies that $Z\in\{N,W\}$ is a process with {\it independent increments}, i.e., for all $ \in\mathbb{N}$, $0\leq t_1<\ldots<t_n$ the family
\begin{equation}
\label{indep_increm}
	\{\sigma(Z(t_1)),\sigma(Z(t_2)-Z(t_1)),\ldots, \sigma(Z(t_{n})-Z(t_{n-1}))\}
\end{equation} 
is a collection of independent $\sigma$-fields. The compound Poisson process $Y$ has also independent increments in the sense that the collection \eqref{indep_increm} for $Z=J$ consists of independent $\sigma$-fields. Hence, by Theorem 13.10 in \cite{YEH} for all $0\leq s<t$ the $\sigma$-fields $\sigma(J(t)-J(s))$ and $\Sigma_s^J$ are independent.

The independence-of-increments property is crucial both from theoretical and practical point of view. In the next section we provide, basing on this property, simple methods of simulating the processes above.

\section{Simulation of Poisson and Wiener processes}
We make use of the fact that both processes $N$ and $W$ has independent increments.

Firstly, we describe how to generating values of $Z\in\{N,W\}$ at fixed grid points 
\begin{equation}
	0=t_0<t_1<\ldots<t_n=T, \ n\in\mathbb{N}, \ T\in (0,+\infty).
\end{equation}
We have that
\begin{equation}
	Z(t_i)=\sum\limits_{j=0}^{i-1}\Delta Z_j, \ i=0,1,\ldots,n,
\end{equation}
where $\displaystyle{\Delta Z_j=Z(t_{j+1})-Z(t_j)}$ and the random variables $(\Delta Z_j)_{j=0}^{n-1}$ are independent. Hence, let us take $(V_j)_{j=0}^{n-1}$, an i.i.d sequence of random variables, where $V_1\sim N(0,1)$, and let $(U_j)_{j=0}^{n-1}$ be a sequence of independent random variables, such that $\displaystyle{U_j\sim \hbox{Poiss}\Bigl(\int\limits_{t_j}^{t_{j+1}}\lambda(s)ds\Bigr)}$. (We also assume that $(V_j)_{j=0}^{n-1}$ and $(U_j)_{j=0}^{n-1}$ are independent, i.e., the $\sigma$-fields $\displaystyle{\sigma\Bigl(\bigcup_{j=0}^{n-1}\sigma(U_j)\Bigr)}$ and $\displaystyle{\sigma\Bigl(\bigcup_{j=0}^{n-1}\sigma(V_j)\Bigr)}$ are independent.) Then
\begin{itemize}
	\item if $Z=W$, then $\Delta Z_j\sim N(0,t_{j+1}-t_j)$  has the same law as $(t_{j+1}-t_j)^{1/2}\cdot V_j$, and we set
	\begin{equation}
		Z(t_i)=\sum\limits_{j=0}^{i-1}(t_{j+1}-t_j)^{1/2}\cdot V_j,
	\end{equation}
	\item if $Z=N$, then $\Delta Z_j$ and $U_j$ has the same law $\displaystyle{\hbox{Poiss}\Bigl(\int\limits_{t_j}^{t_{j+1}}\lambda(s)ds\Bigr)}$, and
	\begin{equation}
		Z(t_i)=\sum\limits_{j=0}^{i-1}U_j.
	\end{equation}
	\item for $Z=J$ let us assume that the values of the Poisson process $N(t_1),\ldots,N(t_n)$ has already been generated by the method described above. Then sample an i.i.d sequence $\xi_1,\ldots,\xi_{N(t_n)}$ according to the distribution $\mu$ and take
	\begin{equation}
		Z(t_i)=\sum\limits_{k=1}^{N(t_i)}\xi_k
	\end{equation}
	for $i=1,\ldots,n$.
\end{itemize}
\section{Brownian and Poisson bridges}
Recall that conditional variance of a square integrable random variable $X$, given sub-$\sigma$-field $\mathcal{G}$, is defined as
\begin{equation}
	Var(X \ | \ \mathcal{G})=\mathbb{E}\Bigl((X-\mathbb{E}(X \ | \ \mathcal{G}))^2 \ |  \ \mathcal{G}\Bigr).
\end{equation}
\begin{lem} (\cite{hertl}, \cite{MGHAB}, \cite{GPage},  \cite{PP8} )
\label{WN_bridge}
 Let us fix $n\in\mathbb{N}$ and $0=t_0<t_1<\ldots<t_n=T$. Then for $t \in (t_i,t_{i+1})$ and $i=0,1,\ldots,n-1$
\begin{itemize}
	\item [(i)] conditioned on $(W(t_1),\ldots,W(t_n))$ the random variable $W(t)$ has the law $N(\mu_W(t),\sigma_W^2(t))$ with
	\begin{equation}
		\mu_W(t)=\mathbb{E}(W(t) \ | \ W(t_1),\ldots,W(t_n))= \frac{(t-t_i)\cdot W(t_{i+1})+(t_{i+1}-t)\cdot W(t_i)}{t_{i+1}-t_i},
	\end{equation}
	\begin{equation}
		\sigma_W^2(t)=Var(W(t) \ | \ W(t_1),\ldots,W(t_n))=\frac{(t_{i+1}-t)\cdot (t-t_i)}{t_{i+1}-t_i}.
	\end{equation}
	\item [(ii)] conditioned on $(N(t_1),\ldots,N(t_n))$ the increment $N(t)-N(t_i)$ has  a
binomial distribution with the number of trials $N(t_{i+1}) - N(t_i)$ and with the probability of success
in each trial equal to $\displaystyle{\frac{\Lambda(t, t_i)}{\Lambda(t_{i+1}, t_i)}}$. Hence
\begin{equation}
		\mu_N(t)=\mathbb{E}(N(t) \ | \ N(t_1),\ldots,N(t_n))= \frac{\Lambda(t,t_i)\cdot N(t_{i+1})+\Lambda(t_{i+1},t)\cdot N(t_i)}{\Lambda(t_{i+1},t_i)},
	\end{equation}
	\begin{equation}
		\sigma_N^2(t)=Var(N(t) \ | \ N(t_1),\ldots,N(t_n))=(N(t_{i+1})-N(t_i))\cdot\frac{\Lambda(t_{i+1},t)\cdot\Lambda(t,t_i)}{(\Lambda(t_{i+1},t_i))^2},
	\end{equation}
	for $t\in [t_i,t_{i+1}]$, $i=0,1,\ldots,n-1$.
\end{itemize}
\end{lem}
In particular, since $W,N$ have independent increments we have that for all $t\in [t_i,t_{i+1}]$, $i=0,1,\ldots,n-1$, $A\in\mathcal{B}(\mathbb{R})$, $k=0,1,\ldots$
\begin{equation}
	\mathbb{P}\Bigl( W(t)\in A \ | \ W(t_1),\ldots, W(t_n)\Bigr) = \mathbb{P}\Bigl( W(t)\in A \ | \ W(t_i), W(t_{i+1})\Bigr),
\end{equation}
and 
\begin{equation}
	\mathbb{P}\Bigl(N(t)-N(t_i)=k \ | \ N(t_1),\ldots, N(t_n)\Bigr) = \mathbb{P}\Bigl( N(t)-N(t_i)=k \ | \ N(t_i), N(t_{i+1})\Bigr).
\end{equation}
This, together with Lemma \ref{WN_bridge}, allows us to generate in easy way additional values of $W$ and $N$ in between previously generated values $Z(t_1),\ldots,Z(t_n)$ of respective process $Z\in\{N,W\}$.

For $Z\in\{N,W\}$ let
\begin{equation}
	\bar Z_n =(\bar Z_n(t))_{t\in [0,T]}=\Bigl(\mathbb{E}(Z(t) \ | \ Z(t_1),\ldots, Z(t_n))\Bigr)_{t\in [0,T]}=\Bigl(\sum\limits_{i=0}^{n-1}\mu_Z(t)\cdot\mathbf{1}_{[t_i,t_{i+1}]}(t)\Bigr)_{t\in [0,T]},
\end{equation}
be the piecewise linear interpolation of $\{(t_i,Z(t_i))\}_{i=0,1,\ldots,n}$
and define the stochastic process
\begin{equation}
\label{def_Z_bridges_1}
	\hat Z_n=(\hat Z_n(t))_{t\in [0,T]}=\Bigl(Z(t)-\bar Z_n(t)\Bigr)_{t\in [0,T]}.
\end{equation}
Then the process $\hat Z_n$ consists of independent Brownian bridges (for $Z=W$) or Poisson bridges (when $Z=N$) corresponding to the respective (disjoint) subintervals. Moreover, $\hat Z_n(t_j)=0$ for all $j=0,1,\ldots,n$, and $\mathbb{E}(\hat Z_n(t))=0$ for all $t\in [0,T]$. This processes often appear in the theory of optimal approximation of SDEs. It is easy to see that $\sigma(\hat Z_n(t))\subset\sigma(\mathcal{N}_n(Z))$ for all $t \in [0,T]$, however, $\hat Z_n$ is not adapted to $(\Sigma_t)_{t\in [0,T]}$.
\section{Optimal reconstruction of trajectories of the Wiener processes}
In this section we show how to approximate optimally (in the Information-Based Complexity sense) trajectories of the Wiener  process. Firstly, we rigorously define what we mean by an optimal algorithm.

We consider the vector space $C([0,T])$  of all continuous functions $f:[0,T]\to\mathbb{R}$, equipped with the supremum norm $\|\cdot\|_{\infty}$. Then $(C([0,T]),\|\cdot\|_{\infty})$ is a separable Banach space. In $C([0,T])$ we consider the Borel $\sigma$-field $\mathcal{B}(C([0,T]))$.

Any approximation method $\bar X=(\bar X_n)_{n\in\mathbb{N}}$  is defined by two sequences $\bar\varphi=(\varphi_n)_{n\in\mathbb{N}}$, $\bar\Delta=(\Delta_n)_{n\in\mathbb{N}}$, where every 
\begin{equation}
\label{alg_phi}
	\varphi_n:\mathbb{R}^n\to C([0,T])
\end{equation}
is Borel function (i.e., $\mathcal{B}(\mathbb{R}^n)$-to-$\mathcal{B}(C([0,T]))$ measurable) and
\begin{equation}
	\Delta_n=\{t_{0,n},t_{1,n},\ldots,t_{n,n}\}
\end{equation}
is a (possibly) non-uniform partition of $[0,T]$ such that
\begin{equation}
	0=t_{0,n}<t_{1,n}<\ldots<t_{n,n}=T.
\end{equation}
By
\begin{equation}
	\mathcal{\bar N}(W)=(\mathcal{N}_n(W))_{n\in\mathbb{N}}
\end{equation}
we mean a sequence of vectors $\mathcal{N}_n(W)$ which provides {\it standard information} with $n$
evaluations about  the Wiener process $W$ at the discrete points from $\Delta_n$, i.e.,
\begin{equation}
	\mathcal{N}_n(W)=[W(t_{1,n}),W(t_{2,n}),\ldots,W(t_{n,n})].
\end{equation}
(Recall that $W(0)=0$ almost surely.) After computing the information $\mathcal{N}_n(W)$, we apply the mapping $\varphi_n$ in order to obtain
the $n$th approximation $\bar X_n=(\bar X_n(t))_{t\in [0,T]}$ in the following way
\begin{equation}
	\bar X_n=\varphi_n(\mathcal{N}_n(W)).
\end{equation}
The informational cost of each $\bar X_n$ is equal to $n$ evaluations of $W$. The set of all methods $\bar X=(\bar X_n)_{n\in\mathbb{N}}$,  defined as above, is denoted by $\chi^{\rm noneq}$. We also consider the following subclass of $\chi^{\rm noneq}$
\begin{equation}
	\chi^{\rm eq}=\{\bar X\in \chi^{\rm noneq} \ | \ \forall_{n\in\mathbb{N}} \Delta_n=\Delta_n^{\rm eq}:=\{iT/n \ | \ i=0,1,\ldots,n\}\}.
\end{equation}
Before we give a definition o error of $\bar X$, we discuss some analytical properties of $\bar X_n$. Note that for each $n\in\mathbb{N}$  the function 
\begin{equation}
	\bar X_n:\Omega\to C([0,T])
\end{equation}
is $\sigma(\mathcal{N}_n(W)) / \mathcal{B}(C([0,T]))$-measurable as composition of measurable mappings. Hence, it is a random element in $C([0,T])$. By taking
\begin{equation}
	\tilde X_n(\omega,t) =(\bar X_n(\omega))(t), \ (\omega,t)\in\Omega\times [0,T],
\end{equation}
we obtain $\sigma(\mathcal{N}_n(W))\otimes\mathcal{B}([0,T])$-to-$\mathcal{B}(\mathbb{R})$ measurable stochastic process, see Theorem \ref{from_B2prod}. Hence, $\bar X_n$ and $\tilde X_n$ are two representation of the same $n$th approximation. In the rest of this chapter  we shall abuse notation and write $\bar X_n$ instead of $\tilde X_n$.

The $n$th error of a method $\bar X$ is defined as
\begin{equation}
	e_n(\bar X)=\|W-\bar X_n\|_{L^2(\Omega\times [0,T])}=\Biggl(\int\limits_{\Omega\times [0,T]}|W(\omega,t)-(\bar X_n(\omega))(t)|^2 \ \mathbb{P}\times\lambda_1(d\omega,dt)\Biggr)^{1/2}.
\end{equation}
We implicitly assume that for all $n\in\mathbb{N}$
\begin{equation}
\label{finite_norm_Xn}
	\mathbb{E}\int\limits_0^T |\bar X_n(t)|^2dt<+\infty,
\end{equation}
since otherwise (by the Minkowski inequality) the error $e_n(\bar X)$ would be infinite for some $n$. Note that by the Tonelli's Theorem the error $e_n(\bar X)$ can be equivalently written as
\begin{equation}
	e_n(\bar X)=\Bigl(\mathbb{E}\|W-\bar X_n\|^2_{L^2([0,T])}\Bigr)^{1/2}=\Bigl(\mathbb{E}\int\limits_0^T |W(t)-\bar X_n(t)|^2 dt\Bigr)^{1/2}.
\end{equation}
Finally, the $n$th minimal error, in respective class of algorithms, is given by
\begin{equation}
\label{nth_min_err1}
	e^\diamond(n) = \inf\limits_{\bar X\in\chi^\diamond }e_n(\bar X), \ \diamond \in\{\rm eq, \rm noneq\}.
\end{equation}
We will investigate (possibly) sharp bounds on $n$th minimal errors. Moreover, we  determine  optimal
methods $\bar X^\diamond $, $\diamond \in\{\rm eq, \rm noneq\}$ that attain the infimum in \eqref{nth_min_err1}.

At first, we establish lower bound for the error of an arbitrary method from $\chi^{\rm noneq}$. 
\begin{lem}
\label{W_asympt_lb} 
For any $\bar X\in\chi^{\rm noneq}$ and for all $n\in\mathbb{N}$ we have that
\begin{equation}
\label{low_b_X_W}
  e_n(\bar X)\geq\frac{T}{\sqrt{6}}\cdot n^{-1/2}.
\end{equation}
\end{lem}
{\bf Proof.} Let us fix any method $\bar X\in\chi^{\rm noneq}$ based on a sequence of discretizations $(\Delta_n)_{n\in\mathbb{N}}$.  Hence, by \eqref{finite_norm_Xn}, Tonelli's theorem and basic properties of Lebesgue integral we get that
\begin{equation}
	\mathbb{E}|\bar X_n(t)|^2<+\infty
\end{equation}
for all $n\in\mathbb{N}$ and almost all $t\in [0,T]$.  Then by the projection property for conditional expectation (see Theorem \ref{proj_wwo}) we have for all $n\in\mathbb{N}$ and almost all $t\in [0,T]$ that
\begin{equation}	
\label{ineq_wwo_W}
	\mathbb{E}|W(t)-\bar X_n(t)|^2\geq \mathbb{E}|W(t)-\mathbb{E}(W(t) \ | \ \mathcal{N}_n(W))|^2,
\end{equation}
since $\mathbb{E}|W(t)|^2<+\infty$, for all $t\in [0,T]$, and $\sigma(\bar X_n(t))\subset\sigma(\mathcal{N}_n(W))$. Due to Tonelli's theorem and product measurability of the underlying stochastic processes, both functions $[0,T]\ni t\to \mathbb{E}|W(t)-\bar X_n(t)|^2$ and $[0,T]\ni t\to \mathbb{E}|W(t)-\mathbb{E}(W(t) \ | \ \mathcal{N}_n(W))|^2$ are Borel. Therefore, \eqref{ineq_wwo_W} and Lemma \ref{WN_bridge} imply for all $n\in\mathbb{N}$  that
\begin{eqnarray}
	&&(e_n(\bar X))^2=\int\limits_0^T\mathbb{E}|W(t)-\bar X_n(t)|^2dt\geq \int\limits_0^T\mathbb{E}|W(t)-\mathbb{E}(W(t) \ | \ \mathcal{N}_n(W))|^2dt\notag\\
	&&=\int\limits_0^T\mathbb{E}\Bigl(Var(W(t) \ | \ \mathcal{N}_n(W))\Bigr) dt=\sum\limits_{i=0}^{n-1}\int\limits_{t_i}^{t_{i+1}}\mathbb{E}|\hat W_n(t)|^2 dt=\frac{1}{6}\sum\limits_{i=0}^{n-1}(t_{i+1}-t_i)^2,
\end{eqnarray}
and by using the Jensen inequality we arrive at
\begin{equation}
	(e_n(\bar X))^2\geq \frac{1}{6}\sum\limits_{i=0}^{n-1}(t_{i+1}-t_i)^2\geq \frac{1}{6n}\Bigl(\sum\limits_{i=0}^{n-1} (t_{i+1}-t_i)\Bigr)^2=\frac{T^2}{6n}.
\end{equation}
This implies \eqref{low_b_X_W}. \ \ \ $\blacksquare$ \\ \\
Now we show that the method $\bar W^{\rm eq}$ based on the sequence of equidistant discretizations $\bar\Delta^{\rm eq}$ achieves the error \eqref{low_b_X_W}.
\begin{thm} 
\label{opt_rec_W}
For the method  
\begin{equation}
	\bar W^{\rm eq}_n =(\bar W^{\rm eq}_n(t))_{t\in [0,T]}=\Bigl(\sum\limits_{i=0}^{n-1}\frac{(t-t_i)\cdot W(t_{i+1})+(t_{i+1}-t)\cdot W(t_i)}{t_{i+1}-t_i}\cdot\mathbf{1}_{[t_i,t_{i+1}]}(t)\Bigr)_{t\in [0,T]}
\end{equation}
	where $t_i=iT/n$, $i=0,1,\ldots,n$, we have
	\begin{equation}
		e_n(\bar W^{\rm eq})=\frac{T}{\sqrt{6}}\cdot n^{-1/2}.
	\end{equation}
	Therefore, the method $\bar W^{\rm eq}=(\bar W^{\rm eq}_n)_{n\in\mathbb{N}}$ is optimal in the class $\chi^{\rm noneq}$, i.e., for all $n\in\mathbb{N}$
	\begin{equation}
	\label{n_th_opt_b_1}
		e_n(\bar W^{\rm eq})=\inf_{\bar X\in\chi^{\rm noneq}}e_n(\bar X)=\frac{T}{\sqrt{6}}\cdot n^{-1/2}.
	\end{equation}
\end{thm}
{\bf Proof.} By \eqref{low_b_X_W} we have that
\begin{equation}
	(e_n(\bar W^{\rm eq}))^2=\sum\limits_{i=0}^{n-1}\int\limits_{t_i}^{t_{i+1}}\mathbb{E}|W(t)-\bar W^{\rm eq}(t)|^2dt=\frac{1}{6}\sum\limits_{i=0}^{n-1}(t_{i+1}-t_i)^2=\frac{T^2}{6n}\leq \Bigl(\inf\limits_{\bar X\in\chi^{\rm noneq}}e_n(\bar X)\Bigr)^2,
\end{equation}
hence we get \eqref{n_th_opt_b_1}. \ \ \ $\blacksquare$
\begin{rem} With the help of Theorems \ref{gen_rand_el_1} (iii), \ref{from_B2prod} (iii) it turns out that  the thesis of Theorem \ref{opt_rec_W} holds true even if we consider the mappings \eqref{alg_phi} as $\mathcal{B}(\mathbb{R}^n)$-to-$\mathcal{B}(L^2([0,T]))$ measurable functions
\begin{equation}
	\varphi_n:\mathbb{R}^n\to L^2([0,T]).
\end{equation} 
\end{rem}
\section{Approximation of Lebesgue integral of Wiener process}
\label{Opt_approx_Leb_W}
In this section we investigate the minimal error for the problem of approximation of the following integral
\begin{equation}
\label{Leb_int_W1}
	\mathcal{I}(W)=\int\limits_0^T W(t)dt.
\end{equation}
The integrals of the form \eqref{Leb_int_W1} appears, for example, in the Wagner-Platen scheme that is used for approximation of solutions of SDEs. See also Chapter \ref{opt_one_points_sdes} where more general problem of weighted integration of Wiener process is considered.

\begin{rem}
Due to Fubini's Theorem, the function $\displaystyle{\Omega\ni \omega\to \mathcal{I}(W)(\omega)=\int\limits_0^T W(\omega,t)dt\in\mathbb{R}}$ is $\Sigma$-to-$\mathcal{B}(\mathbb{R})$ measurable. Moreover, it is  Gaussian random variable. 
\end{rem}

An approximation method $\bar X=(\bar X_n)_{n\in\mathbb{N}}$  is defined by two sequences $\bar\varphi=(\varphi_n)_{n\in\mathbb{N}}$, $\bar\Delta=(\Delta_n)_{n\in\mathbb{N}}$, where every 
\begin{equation}
	\varphi_n:\mathbb{R}^n\to \mathbb{R}
\end{equation}
is a Borel function and
\begin{equation}
	\Delta_n=\{t_{0,n},t_{1,n},\ldots,t_{n,n}\}
\end{equation}
is a (possibly) non-uniform partition of $[0,T]$ such that
\begin{equation}
	0=t_{0,n}<t_{1,n}<\ldots<t_{n,n}=T.
\end{equation}
By
\begin{equation}
	\mathcal{\bar N}(W)=(\mathcal{N}_n(W))_{n\in\mathbb{N}}
\end{equation}
we mean a sequence of vectors $\mathcal{N}_n(W)$ which provides {\it standard information} with $n$
evaluations about  the Wiener process $W$ at the discrete points from $\Delta_n$, i.e.,
\begin{equation}
	\mathcal{N}_n(W)=[W(t_{1,n}),W(t_{2,n}),\ldots,W(t_{n,n})].
\end{equation}
(Recall that $W(0)=0$ almost surely.) After computing the information $\mathcal{N}_n(W)$, we apply the mapping $\varphi_n$ in order to obtain
the $n$th approximation $\bar X_n$ in the following way
\begin{equation}
	\bar X_n=\varphi_n(\mathcal{N}_n(W)).
\end{equation}
The informational cost of each $\bar X_n$ is equal to $n$ evaluations of $W$. The set of all methods $\bar X=(\bar X_n)_{n\in\mathbb{N}}$,  defined as above, is denoted by $\chi^{\rm noneq}$. We also consider the following subclass of $\chi^{\rm noneq}$
\begin{equation}
	\chi^{\rm eq}=\{\bar X\in \chi^{\rm noneq} \ | \ \forall_{n\in\mathbb{N}} \Delta_n=\Delta_n^{\rm eq}:=\{iT/n \ | \ i=0,1,\ldots,n\}\}.
\end{equation} 
Note that for each $n\in\mathbb{N}$  the function 
\begin{equation}
	\bar X_n:\Omega\to \mathbb{R}
\end{equation}
is $\sigma(\mathcal{N}_n(W)) / \mathcal{B}(\mathbb{R})$-measurable as composition of measurable mappings. The $n$th error of a method $\bar X$ is defined as
\begin{equation}
	e_n(\bar X)=\|\mathcal{I}(W)-\bar X_n\|_{L^2(\Omega)}=\Bigl(\mathbb{E}|\mathcal{I}(W)-\bar X_n|^2\Bigr)^{1/2}.
\end{equation}
We implicitly assume that for all $n\in\mathbb{N}$
\begin{equation}
\label{finite_norm_int_Xn}
	\mathbb{E}|\bar X_n|^2<+\infty,
\end{equation}
since otherwise the error $e_n(\bar X)$ would be infinite for some $n$. 

Finally, the $n$th minimal error, in respective class of algorithms, is given by
\begin{equation}
	\label{nth_min_err2}
	e^\diamond(n) = \inf\limits_{\bar X\in\chi^\diamond }e_n(\bar X), \ \diamond \in\{\rm eq, \rm noneq\}.
\end{equation}
We will investigate (possibly) sharp bounds on $n$th minimal errors. Moreover, we  determine  optimal
methods $\bar X^\diamond $, $\diamond \in\{\rm eq, \rm noneq\}$ that attain the infimum in \eqref{nth_min_err2}.

The following useful technical fact can be shown by direct calculations, so we left it as an exercise. 
\begin{lem} (\cite{hertl}, \cite{GPage})
\label{covar_bbW} 
For any $s,t\in [t_i,t_{i+1}]$, $i=0,1,\ldots,n-1$ and $n\in\mathbb{N}$ we have that
\begin{equation}
\label{covar_W}
	Cov(\hat W_n(t),\hat W_n(s))=\frac{(t_{i+1}-\max\{t,s\})(\min\{t,s\}-t_i)}{t_{i+1}-t_i},
\end{equation}
and
\begin{equation}
\label{int_covar}
	\int\limits_{t_i}^{t_{i+1}}\int\limits_{t_i}^{t_{i+1}}Cov(\hat W_n(t),\hat W_n(s))dsdt=\frac{(t_{i+1}-t_i)^3}{12}.
\end{equation}
Moreover, for any $s\in [t_j,t_{j+1}]$, $t\in [t_i,t_{i+1}]$, $i,j=0,1,\ldots,n-1$, $i\neq j$
\begin{equation}
\label{covar_indep_bb}
	Cov(\hat W_n(t),\hat W_n(s))=0.
\end{equation}
\end{lem}
We now prove the following lower bound.
\begin{lem}
\label{int_W_asympt_lb} 
For any $\bar X\in\chi^{\rm noneq}$ and for all $n\in\mathbb{N}$ we have that
\begin{equation}
\label{low_b_X_intW}
  e_n(\bar X)\geq\frac{T^{3/2}}{\sqrt{12}}\cdot n^{-1}.
\end{equation}
\end{lem}
{\bf Proof.} Let us fix any method $\bar X\in\chi^{\rm noneq}$ based on a sequence of discretizations $(\Delta_n)_{n\in\mathbb{N}}$.  Then by the projection property for conditional expectation (Theorem \ref{proj_wwo}) we have that for all $n\in\mathbb{N}$ that
\begin{equation}	
\label{ineq_wwo_intW}
	\mathbb{E}|\mathcal{I}(W)-\bar X_n|^2\geq \mathbb{E}|\mathcal{I}(W)-\mathbb{E}(\mathcal{I}(W) \ | \ \mathcal{N}_n(W))|^2,
\end{equation}
since $\mathbb{E}|\mathcal{I}(W)|^2<+\infty$ and $\sigma(\bar X_n)\subset\sigma(\mathcal{N}_n(W))$. By the conditional version of Fubini's Theorem (Theorem \ref{cond_fubini}) we get
\begin{equation}
	\mathbb{E}(\mathcal{I}(W) \ | \ \mathcal{N}_n(W))=\int\limits_0^T\mathbb{E}(W(t) \ | \ \mathcal{N}_n(W))dt \ a.s.
\end{equation}
Hence
\begin{equation}
	(e_n(\bar X))^2=\mathbb{E}|\mathcal{I}(W)-\bar X_n|^2\geq\mathbb{E}\Biggl|\sum\limits_{j=0}^{n-1}\int\limits_{t_j}^{t_{j+1}} \hat W_n(t)dt \ \Biggl|^2=\sum\limits_{j=0}^{n-1}\mathbb{E}(S_j^2)+\sum\limits_{i\neq j}\mathbb{E}(S_iS_j),
\end{equation}
where $\displaystyle{S_i=\int\limits_{t_i}^{t_{i+1}}\hat W_n(t)dt}$. From Lemma \ref{covar_bbW} we obtain
\begin{equation}
	\mathbb{E}(S_j^2)=\int\limits_{t_j}^{t_{j+1}}\int\limits_{t_j}^{t_{j+1}}\mathbb{E}(\hat W_n(t)\cdot\hat W_n(s))dsdt=\frac{(t_{j+1}-t_j)^3}{12},
\end{equation}
\begin{equation}
	\mathbb{E}(S_iS_j)=\int\limits_{t_i}^{t_{i+1}}\int\limits_{t_j}^{t_{j+1}}\mathbb{E}(\hat W_n(t)\cdot\hat W_n(s))dsdt=0.
\end{equation}
Therefore, by the  Jensen inequality we get
\begin{equation}
	(e_n(\bar X))^2\geq \frac{1}{12}\sum\limits_{i=0}^{n-1}(t_{i+1}-t_i)^3\geq\frac{T^3}{12n^2}.
\end{equation}
This implies \eqref{low_b_X_intW}. \ \ \ $\blacksquare$ \\ \\
We now show that the trapezoidal quadrature rule is the optimal one.
\begin{thm} 
\label{opt_int_W_approx}
For the trapezoidal quadrature rule 
\begin{equation}
	\bar X^{\rm eq}_n =\frac{T}{2n}\sum\limits_{i=0}^{n-1}(W(t_i)+W(t_{i+1}))=\frac{T}{n}\sum\limits_{i=0}^{n-1}W(t_i)+\frac{T}{2n}W(t_n),
\end{equation}
	with $t_i=iT/n$, $i=0,1,\ldots,n$, we have
	\begin{equation}
		e_n(\bar X^{\rm eq})=\frac{T^{3/2}}{\sqrt{12}}\cdot n^{-1}.
	\end{equation}
	Therefore, the method $\bar X^{\rm eq}=(\bar X^{\rm eq}_n)_{n\in\mathbb{N}}$ is optimal in the class $\chi^{\rm noneq}$, i.e., for all $n\in\mathbb{N}$
	\begin{equation}
	\label{opt_trapez}
		e_n(\bar X^{\rm eq})=\inf_{\bar X\in\chi^{\rm noneq}}e_n(\bar X)=\frac{T^{3/2}}{\sqrt{12}}\cdot n^{-1}.
	\end{equation}
\end{thm}
{\bf Proof.} The following holds
\begin{equation}
	\bar X^{\rm eq}_n=\mathbb{E}(\mathcal{I}(W) \ | \ W(iT/n), \  i=1,\ldots,n)=\int\limits_{0}^{T}\bar W^{\rm eq}_n(t)dt.
\end{equation}
Thereby
\begin{equation}
	(e_n(\bar X^{\rm eq}))^2=\mathbb{E}\Bigl|\sum\limits_{j=0}^{n-1}\int\limits_{t_j}^{t_{j+1}}(W(t)-\bar W_n^{\rm eq}(t))dt\Bigl|^2=\frac{1}{12}\sum\limits_{j=0}^{n-1}(t_{j+1}-t_j)^3=\frac{T^3}{12n^2}.
\end{equation}
This and \eqref{low_b_X_intW} imply \eqref{opt_trapez}. \ \ \ $\blacksquare$
\section{Implementation issues}

Drawing trajectories of Wiener process.

\begin{lstlisting}[language=Python]
import numpy as np

def Wiener_path( T, n ):
    h = float( T / n )
    sq_h = np.sqrt( h )

    W = 0.0
    x = [0.0]
    y = [W]

    Normals = sq_h * np.random.normal(0, 1, n)

    for i in range(1, n+1):
        W += Normals[i-1]

        x.append(i * h)
        y.append(W)

    return x, y
\end{lstlisting}

Drawing trajectories of Poisson process.

\begin{lstlisting}[language=Python]
import numpy as np

def Poisson_path(lam, T, n):
    h = float(T / n)
    lh = lam * h

    P = 0.0
    x = [0.0]
    y = [P]

    Poiss = np.random.poisson(lh, n)

    for i in range(1, n+1):
        P += Poiss[i - 1]

        x.append(i*h)
        y.append(P)

    return x, y
\end{lstlisting}

Drawing trajectories of compound Poisson process.

\begin{lstlisting}[language=Python]
def Compound_Poisson(T, n, lam):
    h = float(T / n)
    lh = lam * h

    dN = np.random.poisson(lh, n)
    N = np.cumsum(dN)

    N = np.insert(N, 0, 0.0)

    t = [0.0]
    Y = [0.0]
    num_T = len(N)-1

    xi = np.random.normal(0.0, 3.0, N[num_T])
    for i in range(1, num_T+1):
        p = 0.0
        for k in range(1, N[i]):
            p += xi[k]
        t.append(i * h)
        Y.append(p)

    return t, Y
\end{lstlisting}

\section{Concluding remarks}
Due to Theorem \ref{opt_int_W_approx} we have for all $n\in\mathbb{N}$
\begin{eqnarray}
	&&\inf\limits_{\varphi_n:\mathbb{R}^n\to\mathbb{R}-Borel, \ 0<t_1<t_2<\ldots<t_n=T}\mathbb{E}\Bigl|\int\limits_0^T W(t)dt-\varphi_n(\mathcal{N}_n(W))\Bigl|^2\notag\\
&&=\inf\limits_{\varphi_n:\mathbb{R}^n\to\mathbb{R}-Borel, \  0<t_1<t_2<\ldots<t_n=T}\int\limits_{C([0,T])}\Bigl|\int\limits_0^T f(t)dt-\varphi_n(f(t_1),\ldots,f(t_n))\Bigl|^2 w(df)=\frac{T^{3/2}}{\sqrt{12}}n^{-1},
\end{eqnarray}
where $w$ is the Wiener measure on $\Bigl(C([0,T]),\mathcal{B}(C([0,T]))\Bigr)$. Hence, Theorem \ref{opt_int_W_approx} answers the question: {\it in average, what is the error when approximating the Lebesgue integral of continuous function via deterministic quadrature rule based on finite number of  function evaluations?} From the point of view of Information-Based Complexity such model of computation is called the {\it average case setting}, see \cite{RIT}, \cite{TWW88}. See also \cite{hertl}, where more general problem of nonlinear Lebesgue integration in the average-case setting was considered. Finally, note 
 that the problem of (nonlinear) Lebesgue integration is connected with complexity of stochastic integration, see \cite{WW01}, \cite{hertl}.
\section{Exercises}
\begin{itemize}
	\item [1.] Let $X$ be a square integrable random variable and let $\mathcal{G}$ be a sub-$\sigma$-sigma field of $\Sigma$. Moreover, let $Y$ be a $\mathcal{G}$-measurable random variable.
		\begin{itemize}
			\item [(i)] In addition, let $\mathbb{E}Y^2<+\infty$. Show that
				\begin{equation}
					Var(X+Y\ | \ \mathcal{G})=Var(X \ | \ \mathcal{G}).
				\end{equation}
			\item [(ii)] Additionally assume that $\mathbb{E}|XY|^2<+\infty$. Justify that
				\begin{equation}
					Var(XY \ | \ \mathcal{G})=Y^2\cdot Var (X \ | \ \mathcal{G}).
				\end{equation}
		\end{itemize}
  \item [2.] Let $X$, $Y$ be square integrable random variables and let  $\mathcal{F}$ be a sub-$\sigma$-field of $\Sigma$. We define the conditional covariance of the random variables $X$, $Y$ as follows
\begin{equation}
	\label{WAR_COV}
	Cov(X,Y \ | \ \mathcal{F})=\mathbb{E}\Bigl((X-\mathbb{E}(X \ | \ \mathcal{F}))\cdot (Y-\mathbb{E}(Y \ | \ \mathcal{F})) \ | \ \mathcal{F}\Bigr).
\end{equation}
\begin{itemize}
	\item [(a)] Prove that the  conditional covariance (\ref{WAR_COV})  is well-defined, i.e., give a proof that the conditional expectation
		\begin{equation}
			\mathbb{E}\Bigl((X-\mathbb{E}(X \ | \ \mathcal{F}))\cdot (Y-\mathbb{E}(Y \ | \ \mathcal{F})) \ | \ \mathcal{F}\Bigr)
		\end{equation}
		exists.
	\item [(b)] Prove that the following equality holds
			\begin{equation}
				\label{WAR_COV2}
					Cov(X,Y \ | \ \mathcal{F})=\mathbb{E}(XY \ | \ \mathcal{F})-\mathbb{E}(X \ | \ \mathcal{F})\cdot\mathbb{E}(Y \ | \ \mathcal{F}).
			\end{equation}
	\item [(c)] Prove the following
			\begin{equation}
				\label{WAR_COV3}
					Cov(X,Y)=\mathbb{E}(Cov(X,Y \ | \ \mathcal{F}))+Cov(\mathbb{E}(X \ | \ \mathcal{F}),\mathbb{E}(Y \ | \ \mathcal{F})).
			\end{equation}
\end{itemize} 
	\item [3.] Let $(\mathcal{G}_t)_{t\in [0,+\infty)}$ be an arbitrary filtration on $(\Omega,\Sigma,\mathbb{P})$. Show that
	\begin{equation}
		\bigcap_{n\in\mathbb{N}}\mathcal{G}_{t+\frac{1}{n}}=\bigcap_{s>t}\mathcal{G}_s.
	\end{equation}
	\item [4.] Let $(\Sigma_t)_{t\in [0,+\infty)}$ be a filtration on $(\Omega,\Sigma,\mathbb{P})$, satisfying the usual conditions, and let $\mathcal{G}$ be a sub-$\sigma$-field of $\Sigma$ that is independent of $\Sigma_{\infty}$. Show that $(\mathcal{H}_t)_{t\in [0,+\infty)}$, where
	\begin{equation}
	\label{H_filtr}
		\mathcal{H}_t=\sigma\Bigl(\Sigma_t\cup\mathcal{G}\Bigr),
	\end{equation}
	is a filtration that also satisfies the usual conditions.\\
	Hint. Chamont, Yor, "Exercises...."
	\item [5.] Let $W=(W_t)_{t\in [0,+\infty)}$ be a $(\Sigma_t)_{t\in [0,+\infty)}$-Wiener process on $(\Omega,\Sigma,\mathbb{P})$, and let $\mathcal{G}$ be a sub-$\sigma$-field of $\Sigma$ that is independent of $\Sigma_{\infty}$. Show that $W=(W_t)_{t\in [0,+\infty)}$ is a Wiener process with respect to  the filtration $(\mathcal{H}_t)_{t\in [0,+\infty)}$, defined as in \eqref{H_filtr}.
	\item [6.] What is the law of $\mathcal{I}(W)$?
	\item [7.] Show that $\hat W_n$ is a Gaussian process.
	\item [8.] Provide more efficient implementation of the procedure  that generates trajectories of the compound Poisson process.
	\item [9.] Write a function that simulates trajectories of the ruin process with perturbations which is defined as follows
	\begin{equation}
		R(t)=x_0+ct+\sigma W(t)+J(t), t\in [0,T],
	\end{equation}
	where $x_0,c,\sigma\in (0,+\infty)0$ and $W$, $J$ are Wiener and compound Poisson process, respectively. Approximate, by using a suitably constructed Monte Carlo method, the probability of ruin
	\begin{equation}
		\mathbb{P}(R(T)<0).
	\end{equation}
	Do not forget to compute  empirical confidence intervals.
	\item [10.] Give a proof of Lemma \ref{covar_bbW}.
\end{itemize}
\chapter{Stochastic integration with respect to semimartingales}
Let $(\Omega,\Sigma,\mathbb{P})$ be a complete probability space equipped with filtration $(\Sigma_t)_{t\in [0,+\infty)}$ that satisfies the usual conditions. In this chapter we mainly base on the results from \cite{Applb}, \cite{MEDVEG}, \cite{prott}, \cite{situ} and \cite{weizwin}. We focus only on stochastic integration of processes from the space $\mathbb{L}$ (see \eqref{def_proc_L}) which, as we will see, is sufficient in our case.

A process $X$ is a {\it semimartingale} if it admits a decomposition
\begin{equation}
\label{semi_dec1}
	X = M+A,
\end{equation}
where $M$ is a right-continuous local $L^2$-martingale and $A$ is right-continuous, adapted and of locally finite variation in the sense that almost surely the paths of $A$ are of finite variation on each compact subinterval of $[0,+\infty)$. This decomposition is not unique, unless we restrict to the case when $A$ is continuous, see, for example, page 150. in \cite{weizwin}. Below we provide main examples of semimartingales. 

\begin{itemize}
	\item The Wiener process  $W$ is a (continuous) semimartingale, since it is a martingale, so $A=0$,
	\item The process $(W^2(t))_{t\geq 0}$ is a semimartingale with the following decomposition
		\begin{equation}
		\label{decomp_1}
			W^2(t)=(W^2(t)-t)+t,
		\end{equation}
		$(W^2(t)-t,\bar\Sigma^W_t)_{t\geq 0}$ is a c\`adl\`ag $L^2$-martingale.
	\item The non-homogeneous Poisson process is a semimartingale due to the fact that
	\begin{equation}
	\label{decomp_2}
		N(t)=(N(t)-m(t))+m(t),
	\end{equation}
	and the process $(N(t)-m(t),\bar\Sigma^N_t)_{t\geq 0}$ is a c\`adl\`ag $L^2$-martingale. It is denoted by $\tilde N$ and called the {\it compensated Poisson process}.
	\item The process $(\tilde N^2(t))_{t\geq 0}$ is a semimartingale with the decomposition
	\begin{equation}
	\label{decomp_3}
		\tilde N^2(t)=(\tilde N^2(t)-m(t))+m(t), 
	\end{equation}
	where $(\tilde N^2(t)-m(t),\bar\Sigma^N_t)_{t\geq 0}$ is a c\`adl\`ag $L^2$-martingale.
	\item Provided that $\mathbb{E}(\xi_1^2)<+\infty$, the compound Poisson process is a semimartingale since
	\begin{equation}
		J(t)=(J(t)-t\lambda \mathbb{E}(\xi_1))+t\lambda \mathbb{E}(\xi_1)
	\end{equation}
	and $(J(t)-t\lambda \mathbb{E}(\xi_1),\bar\Sigma^J_t)_{t\geq 0}$ is a  c\`adl\`ag $L^2$-martingale called {\it compensated compound Poisson process}. We denote it by $\tilde J$.
	\item If $\mathbb{E}(\xi_1^2)<+\infty$, then the process $(\tilde J^2(t))_{t\geq 0}$ is a semimartingale with the decomposition
	\begin{equation}
		\tilde J^2(t)=(\tilde J^2(t)-t\lambda \mathbb{E}(\xi_1^2))+t\lambda \mathbb{E}(\xi_1^2)
	\end{equation}
	and $(\tilde J^2(t)-t\lambda \mathbb{E}(\xi_1^2),\bar\Sigma^J_t)_{t\geq 0}$ is a  c\`adl\`ag $L^2$-martingale.
\end{itemize}
Note that the processes $N$, $\tilde N^2$, $J$, $\tilde J^2$ are itself processes with locally finite variation. 

Suppose that $Y\in\mathbb{L}$ and $X$ is a semimartingale with the decomposition \eqref{semi_dec1}. Then the stochastic integral of $Y$ with respect to $X$ is defined as
\begin{equation}
\label{int_semimart}
	\int\limits_0^t Y(s)dX(s)=\int\limits_0^t Y(s)dM(s)+\int\limits_0^t Y(s)dA(s).
\end{equation}
Note that the value of stochastic integral is independent of the semimaringale decomposition \eqref{semi_dec1}, see \cite{weizwin}. In \eqref{int_semimart} the integral
\begin{itemize}
	\item $\displaystyle{\int\limits_0^t Y(s)dA(s)}$ is a Lebesgue-Stieltjes integral, computed path-by-path,
	\item $\displaystyle{\int\limits_0^t Y(s)dM(s)}$ is a stochastic integral with respect to the c\`adl\`ag $L^2$-martingale $M$.
\end{itemize}
The precise construction of stochastic integrals with respect to martingales is given, for example, in \cite{situ}. Here we only state  main properties of such integrals that we will use in the context of approximation of solutions of jump-diffusion stochastic differential equations.

The fundamental theorem that allows to define stochastic integral with respect to c\`adl\`ag $L^2$-martingales is the following theorem, called the Doob-Meyer decomposition.
\begin{thm} (\cite{situ})
	\label{dm_decomp}
	Let $(M(t),\Sigma_t)_{t\geq 0}$ be a square integrable c\`adl\`ag martingale. Then there is a unique decomposition
	\begin{equation}
		M^2(t)=Y(t)+A(t), \ t\geq 0,
	\end{equation}
	where $(Y(t),\Sigma_t)_{t\geq 0}$ is c\`adl\`ag martingale and $A=(A(t))_{t\geq 0}$ is a predictable, right continuous, and increasing process such that $A(0)=0$ and $\mathbb{E}(A(t))<+\infty$ for all $t\geq 0$. From now on the process $A=(A(t))_{t\geq 0}$ is denoted by $\langle M\rangle=(\langle M\rangle_t)_{t\geq 0}$ and it is called the {\bf predictable quadratic variation} of $M$. 
\end{thm}
It is easy to see that \eqref{decomp_1} and \eqref{decomp_3} are two (very important)  examples of Doob-Meyer decomposition. 

From Theorem \eqref{dm_decomp} we have what follows.
\begin{cor}
\label{cross_var1}
	Let $(M_i(t),\Sigma_t)_{t\geq 0}$, $i=1,2$, be  square integrable c\`adl\`ag martingales. Then the process $(M_1(t)\cdot M_2(t))_{t\geq 0}$ has a uniqe decomposition
	\begin{equation}
		M_1(t)\cdot M_2(t)=Y(t)+A(t), \ t\geq 0,
	\end{equation}
	where $(Y(t),\Sigma_t)_{t\geq 0}$ is c\`adl\`ag martingale and $A=(A(t))_{t\geq 0}$ is a difference of two processes $A_i=(A_i(t))_{t\geq 0}$ where $A_i$  is a predictable, right continuous, and increasing process such that $A_i(0)=0$ and $\mathbb{E}(A_i(t))<+\infty$ for all $t\geq 0$. From now on the process $A=(A(t))_{t\geq 0}$ is denoted by $\langle M_1,M_2\rangle=(\langle M_1,M_2\rangle_t)_{t\geq 0}$ and it is called the {\bf predictable cross quadratic variation} of $(M_1(t)\cdot M_2(t))_{t\geq 0}$. 
\end{cor}
The following result is a direct consequnence of Corollary \ref{cross_var1}.
\begin{cor}
\label{cross_var2}
Suppose that $(M_i(t),\Sigma_t)_{t\geq 0}$, $i=1,2$, are  square integrable c\`adl\`ag martingales such that the process $(M_1(t)\cdot M_2(t), \Sigma_t)_{t\geq 0}$ is a martingale. Then $\langle M_1,M_2\rangle_t=0$ for all $t\geq 0$.
\end{cor}
\begin{prop} 
\label{qv_NW}
Let $W=(W(t))_{t\in [0,+\infty)}$ be a $m_W$-dimensional $(\Sigma_t)_{t\in [0,+\infty)}$-Wiener process, while $N=(N(t))_{t\in [0,+\infty)}$ is a $m_N$-dimensional non-homogeneous $(\Sigma_t)_{t\in [0,+\infty)}$-Poisson process, both defined on  $(\Omega,\Sigma,\mathbb{P})$, and such that the $\sigma$-fields $\Sigma^W_{\infty}$, $\Sigma^N_{\infty}$ are independent. (i.e. $W$ and $N$ are independent). We assume that each intensity $\lambda_j:[0,T]\to (0,+\infty)$ of $(N^j(t))_{t\in [0,+\infty)}$ is Borel and integrable, $j=1,2,\ldots,m_N$.  If $i\neq j$ then for any $t\geq 0$
\begin{eqnarray}
\label{qv_1}
	&& \langle W^i,W^j\rangle_t = 0,\\
\label{qv_2}	
	&& \langle \tilde N^i,\tilde N^j\rangle_t = 0,
\end{eqnarray}
while for $i=j$
\begin{eqnarray}
\label{qv_21}
	&& \langle W^i\rangle_t=\langle W^i, W^i\rangle_t=t,\\
	&& \langle \tilde N^i\rangle_t=\langle \tilde N^i,\tilde N^i\rangle_t = m_i(t).
\end{eqnarray}
If $\sigma([W(t)-W(s),N(t)-N(s)])$ is independent of $\Sigma_s$ for any $0\leq s<t$ then for any $i\in\{1,\ldots,m_W\}$, $j\in\{1,\ldots,m_N\}$ and $t\geq 0$
\begin{equation}
\label{qv_3}
 \langle W^i,\tilde N^j\rangle_t = 0,
\end{equation}
\end{prop}
{\bf Proof.} Equalities \eqref{qv_21} are direct consequences of \eqref{decomp_1} and \eqref{decomp_3}. 

 We only show \eqref{qv_3}, since the proofs of \eqref{qv_1},\eqref{qv_2} are analogous.

For any $i,j$ and $0\leq s<t$ we have
\begin{eqnarray}
&&	\mathbb{E}(W^i(t)\cdot\tilde N^j(t) \ | \ \Sigma_s)=\notag\\
&&=\mathbb{E}\Bigl((W^i(t)-W^i(s)+W^i(s))\cdot (\tilde N^j(t)-\tilde N^j(s)+\tilde N^j(s)) \ | \ \Sigma_s\Bigr)\notag\\
&&=\mathbb{E}\Bigl((W^i(t)-W^i(s))\cdot (\tilde N^j(t)-\tilde N^j(s)) \ | \ \Sigma_s\Bigr)\notag\\
&&\quad+\tilde N^j(s)\cdot \mathbb{E}\Bigl( W^i(t)-W^i(s) \ | \ \Sigma_s\Bigr)+W^i(s)\cdot \mathbb{E}\Bigl( \tilde N^j(t)-\tilde N^j(s) \ | \ \Sigma_s\Bigr)\notag\\
&&\quad+\mathbb{E}\Bigl(W^i(s)\cdot\tilde N^j(s) \ | \ \Sigma_s\Bigr)\notag\\
&&=\mathbb{E}\Bigl((W^i(t)-W^i(s))\cdot (\tilde N^j(t)-\tilde N^j(s))\Bigr)+W^i(s)\cdot\tilde N^j(s)\notag\\
&&=\mathbb{E}\Bigl(W^i(t)-W^i(s)\Bigr)\cdot \mathbb{E}\Bigl(\tilde N^j(t)-\tilde N^j(s)\Bigr)+W^i(s)\cdot\tilde N^j(s)\notag\\
&&=W^i(s)\cdot\tilde N^j(s).
\end{eqnarray} 
Therefore, the process $(W^i(t)\cdot\tilde N^j(t),\Sigma_t)_{t\geq 0}$ is a martingale and by Corollary \ref{cross_var2} we get \eqref{qv_3}. \ \ \ $\blacksquare$ \\ \\
We now recall basic properties of stochastic integrals. Let $X$ be a semimartingale with respect to $(\Sigma_t)_{t\in [0,+\infty)}$. Moreover, let $f,f_1,f_2\in \mathbb{L}$ and $\alpha,\beta\in\mathbb{R}$ and suppose that  $\rho,\tau$ are $(\Sigma_t)_{t\in [0,+\infty)}$-stopping times, such that $\rho\leq\tau$ a.s. Recall that the stochastic interval $(\rho,\tau]$ is defined as follows
\begin{equation}
	(\rho,\tau]=\{(\omega,t)\in\Omega\times [0,+\infty) \ | \ \rho(\omega)<t\leq\tau(\omega)\}.
\end{equation}
Moreover the process $\mathbf{1}_{(\rho,\tau]}$ belongs to $\mathbb{L}$. The following holds.
\begin{itemize}
	\item  For all $t\in [0,+\infty)$
	\begin{equation}
		\int\limits_0^t (\alpha f_1(s)+\beta f_2(s))dX(s)=\alpha\int\limits_0^t f_1(s)dX(s)+\beta\int\limits_0^t f_2(s)dX(s) \ a.s.
	\end{equation}
	\item The stochastic integral process $\displaystyle{Y=(Y(t))_{t\geq 0}=\Bigl(\int\limits_0^t f_1(s)dX(s)\Bigr)_{t\in [0,+\infty)}}$ is itself a semimartingale, and
	\begin{equation}
		\int\limits_0^t f_2(s)dY(s)=\int\limits_0^t f_2(s)f_1(s)dX(s).
	\end{equation}
	\item For all $t\in [0,+\infty)$
	\begin{eqnarray}
	&&\int\limits_0^t\mathbf{1}_{(\rho,\tau]}(s)f(s)dX(s)=\int\limits_0^t\mathbf{1}_{(0,\tau]}(s)f(s)dX(s)-\int\limits_0^t\mathbf{1}_{(0,\rho]}(s)f(s)dX(s)\notag\\
	&&=\int\limits_0^{t\wedge \tau}f(s)dX(s)-\int\limits_0^{t\wedge \rho}f(s)dX(s)=\int\limits_{t\wedge \rho}^{t\wedge \tau}f(s)dX(s)
	\end{eqnarray}
	\item If $Z$ is a $\Sigma_{\rho}$-measurable random variable then
	\begin{equation}
		\int\limits_{\rho}^{\tau}Zf(s)dX(s)=Z\int\limits_{\rho}^{\tau}f(s)dX(s) \ a.s.
	\end{equation}
	\item Let $f=(f(t))_{t\geq 0}$ be a {\bf predictable simple process}, i.e., there is a sequence of stopping times
	\begin{equation}
		0=\tau_0<\tau_1<\ldots<\tau_n<\ldots
\end{equation}		
such that for all $t\geq 0$
	\begin{equation}
		f(t)=\xi_0\mathbf{1}_{\{0\}}(t)+\sum_{i=0}^{+\infty}\xi_i\mathbf{1}_{(\tau_i,\tau_{i+1}]}(t),
	\end{equation}
	where $\xi_0$ is $\Sigma_0$-measurable and $\xi_i$ are $\Sigma_{\tau_i}$-measurable random variables. Then for all $t\geq 0$
	\begin{eqnarray}
		&&\int\limits_0^t f(s)dX(s)=\sum_{i=0}^{+\infty}\xi_i\cdot\Bigl(X(\tau_{i+1}\wedge t)-X(\tau_{i}\wedge t)\Bigr)\notag\\
		&&=\sum_{i=0}^{+\infty}\xi_i\cdot\Bigl( M(\tau_{i+1}\wedge t)-M(\tau_{i}\wedge t)\Bigr)+\sum_{i=0}^{+\infty}\xi_i\cdot\Bigl(A(\tau_{i+1}\wedge t)-A(\tau_{i}\wedge t)\Bigr) \ a.s.
	\end{eqnarray}
	\item If the semimartingale $X$ has paths of finite variation on compacts, then the stochastic integral $\displaystyle{\int\limits_0^t f(s)dX(s)}$ is indistinguishable from the Lebesgue-Stieltjes integral, computed path-by-path.
\end{itemize}
If $Y\in\mathbb{L}$ then 
\begin{equation}
\label{qv_31}
	\int\limits_0^t Y(s)dN(s)=\int\limits_0^t Y(s)d\tilde N(s)+\int\limits_0^t Y(s)dm(s)=\int\limits_0^t Y(s)d\tilde N(s)+\int\limits_0^t Y(s)\lambda(s)ds,
\end{equation}
but, on the other hand, the stochastic integrals $\displaystyle{\int\limits_0^t Y(s)dN(s)}$, $\displaystyle{\int\limits_0^t Y(s)d\tilde N(s)}$ are equivalent to the respective Lebesgue-Stieltjes integrals. Moreover, since $N$ is a pure jump process, we get
\begin{equation}
	\int\limits_0^t Y(s)dN(s)=\sum\limits_{0<s\leq t} Y(s)\Delta N(s)=\sum\limits_{k=1}^{N(t)}Y(\tau_k),
\end{equation} 
while
\begin{equation}
	\int\limits_0^t Y(s)dJ(s)=\sum\limits_{0<s\leq t} Y(s)\Delta J(s)=\sum\limits_{k=1}^{N(t)}Y(\tau_k)\xi_k,
\end{equation}
where $(\tau_k)_{k\in\mathbb{N}}$ are the jump times of $N$. 

Below we present main properties of stochastic integrals with respect to  cadlag $L^2$-martingales. 

For a given $L^2$-martingale $M$ let us consider the following subset of $\mathbb{L}$
\begin{equation}
 	\mathbb{L}^2(M)=\Bigl\{X\in\mathbb{L} \ | \ 
 	\forall_{T>0} \ \mathbb{E}\int\limits_0^T |X(s)|^2d\langle M\rangle_s<+\infty\Bigr\}.
\end{equation}
Suppose that $(M(t),\Sigma_t)_{t\in [0,+\infty)}$, $(M_1(t),\Sigma_t)_{t\in [0,+\infty)}$, $(M_2(t),\Sigma_t)_{t\in [0,+\infty)}$ are square integrable c\`adl\`ag martingales. Let $f\in\mathbb{L}^2(M)$.
\begin{itemize}
	\item For all $0\leq s<t$ we have
		\begin{equation}
			\mathbb{E}\Bigl(\int\limits_0^t f(u)dM(u) \ \Bigl| \ \Sigma_s\Bigr)=\int\limits_0^s f(u)dM(u) \ a.s.,
		\end{equation}
		and, in particular,
		\begin{equation}
			\mathbb{E}\Bigl(\int\limits_0^t f(u)dM(u) \Bigr)=0.		
		\end{equation}	
	\item
	\begin{equation}
			\mathbb{E}\Bigl(\Bigl(\int\limits_s^t f(u)dM(u)\Bigr)^2 \ \Bigl| \ \Sigma_s\Bigr)=\mathbb{E}\Bigl(\int\limits_s^t |f(u)|^2 d\langle M\rangle_u \ \Bigl| \ \Sigma_s\Bigr) \ a.s.
	\end{equation}		
	\item If $f_i\in\mathbb{L}^2(M_i)$, $i=1,2$, then for all $0\leq s<t$
		\begin{equation}
		\label{covar_war_int}
			\mathbb{E}\Bigl(\int\limits_s^t f_1(u)dM_1(u)\cdot\int\limits_s^t f_2(u)dM_2(u)\ \Bigl| \ \Sigma_s\Bigr)=\mathbb{E}\Bigl(\int\limits_s^t f_1(u)f_2(u)d\langle M_1,M_2\rangle_u\ \Bigl| \ \Sigma_s\Bigr)		\ a.s.	
		\end{equation}
\end{itemize}
If $X$ is an $\mathbb{R}^d$-valued process such that each coordinate $X^i$, $i=1,2,\ldots,d$, is an $\mathbb{R}$-valued semimartingale, then we call $X$ the vector valued semimartingale. For an $\mathbb{R}^{d\times m}$-valued stochastic process $Y\in\mathbb{L}$ we define
\begin{equation}
	\int\limits_0^t Y(s)dX(s)=\Bigl(\sum\limits_{j=1}^m\int\limits_0^t Y^{ij}(s)dX^j(s)\Bigr)_{i=1,2,\ldots,d}.
\end{equation} 
For example in the case of multidimensional non-homogeneous Poisson process $N$ we have that
\begin{equation}
	\int\limits_0^t Y(s)dN(s)=\Bigl(\sum\limits_{j=1}^{m_N}\int\limits_0^t Y^{ij}(s)dN^j(s)\Bigr)_{i=1,2,\ldots,d}=\int\limits_0^t Y(s)d\tilde N(s)+\int\limits_0^t Y(s)dm(s),
\end{equation}
where
\begin{equation}
	\int\limits_0^t Y(s)d\tilde N(s)=\Bigl(\sum\limits_{j=1}^{m_N}\int\limits_0^t Y^{ij}(s)d\tilde N^j(s)\Bigr)_{i=1,2,\ldots,d},
\end{equation}
\begin{equation}
	\int\limits_0^t Y(s)dm(s)=\Bigl(\sum\limits_{j=1}^{m_N}\int\limits_0^t Y^{ij}(s)dm_j(s)\Bigr)_{i=1,2,\ldots,d}=\Bigl(\sum\limits_{j=1}^{m_N}\int\limits_0^t Y^{ij}(s)\lambda_j(s)ds\Bigr)_{i=1,2,\ldots,d}.
\end{equation} 
Let us now assume that $Y^{ij}\in\mathbb{L}^2(\tilde N^j)$ for $i=1,2,\ldots,d$, $j=1,2,\ldots,m_N$. From the definition above, Proposition \ref{qv_NW} and \eqref{covar_war_int} we arrive at the following version of the multidimensional It\^o isometry
\begin{eqnarray}
	&&\mathbb{E}\Bigl\|\int\limits_0^t Y(s)d\tilde{N}(s)\Big\|^2=\mathbb{E}\sum_{k=1}^d\Bigl|\sum_{j=1}^{m_N}\int\limits_0^t Y^{kj}(s)d\tilde{N}^j(s)\Bigl|^2\notag\\
	&&=\sum_{k=1}^d\sum_{j_1,j_2=1}^{m_N}\mathbb{E}\Bigl(\int\limits_0^t Y^{kj_1}(s)d\tilde{N}^{j_1}(s)\cdot\int\limits_0^t Y^{k j_2}(s)d\tilde{N}^{j_2}(s)\Bigr)\notag\\
	&&=\sum_{k=1}^d\sum_{j_1,j_2=1}^{m_N}\mathbb{E}\Bigl(\int\limits_0^t Y^{kj_1}(s)Y^{kj_2}(s)d\langle\tilde N^{j_1},\tilde N^{j_2}\rangle_s\Bigr)\notag\\
	&&=\sum_{k=1}^d\sum_{j=1}^{m_N}\mathbb{E}\int\limits_0^t|Y^{kj}(s)|^2\lambda_j(s)ds=\sum_{j=1}^{m_N}\mathbb{E}\int\limits_0^t\|Y^{(j)}(s)\|^2\lambda_j(s)ds,
\end{eqnarray}
where $Y^{(j)}$ is the $j$th column of the matrix $Y=[Y^{kj}]_{1\leq k\leq d, 1\leq j\leq m_N}$. In similar way, if $Y^{ij}\in\mathbb{L}^2(W^j)$ for $i=1,2,\ldots,d$, $j=1,2,\ldots,m_W$ we get
\begin{equation}
	\mathbb{E}\Bigl\|\int\limits_0^t Y(s)dW(s)\Big\|^2=\mathbb{E}\int\limits_0^t\|Y(s)\|^2ds.
\end{equation}
\section{Exercises}
\begin{itemize}
	\item [1.] Let $-\infty<\alpha<\beta<+\infty$ and suppose $Z:\Omega\to\mathbb{R}^{d\times m_W}$ to be $\Sigma_{\alpha}$-measurable random vector. Show that
\begin{equation}
	\int\limits_{\alpha}^{\beta} Z dW(t)=Z\cdot (W(\beta)-W(\alpha)).
\end{equation}
	\item [2.] Let $-\infty<\alpha<\beta<+\infty$ and suppose $Z:\Omega\to\mathbb{R}^{d\times m_N}$ to be $\Sigma_{\alpha}$-measurable random vector. Show that
\begin{equation}
	\int\limits_{\alpha}^{\beta} Z dN(t)=Z\cdot (N(\beta)-N(\alpha)).
\end{equation}
\end{itemize}
\chapter{Jump-diffusion stochastic differential equations}
In this section by $\|\cdot\|$ we denote either the Euclidean or the Frobenius norm, and the meaning is clear from the context.

Let $(\Omega,\Sigma,\mathbb{P})$ be a complete probability space equipped with filtration $(\Sigma_t)_{t\in [0,+\infty)}$ that satisfies the usual conditions.
Let $W=(W(t))_{t\in [0,+\infty)}$ be a $m_W$-dimensional $(\Sigma_t)_{t\in [0,+\infty)}$-Wiener process, while $N=(N(t))_{t\in [0,+\infty)}$ is a $m_N$-dimensional non-homogeneous $(\Sigma_t)_{t\in [0,+\infty)}$-Poisson process, both defined on  $(\Omega,\Sigma,\mathbb{P})$, and such that the $\sigma$-fields $\Sigma^W_{\infty}$, $\Sigma^N_{\infty}$ are independent. (i.e. $W$ and $N$ are independent). We assume that each intensity $\lambda_j:[0,T]\to (0,+\infty)$ of $(N^j(t))_{t\in [0,+\infty)}$ is Borel and integrable, $j=1,2,\ldots,m_N$.  
\begin{itemize}
	\item [(A0)] Let $\xi:\Omega\to\mathbb{R}^d$ be a $\Sigma_0$-measurable random vector, such that $\mathbb{E}\|\xi\|^2<+\infty$. 
\end{itemize}
Let us fix functions
\begin{eqnarray}
	&&a:[0,T]\times\mathbb{R}^d\to \mathbb{R}^d,\\
	&&b:[0,T]\times\mathbb{R}^d\to \mathbb{R}^{d\times m_W},\\
	&&c:[0,T]\times\mathbb{R}^d\to \mathbb{R}^{d\times m_N},
\end{eqnarray}
such that 
\begin{itemize}
	\item [(A1)] $a,b,c\in C([0,T]\times\mathbb{R}^d)$
	\item [(A2)]there exists $K\in (0,+\infty)$ such that for $f\in\{a,b,c\}$ and for all $t\in [0,T],x,y\in\mathbb{R}^d$
	\begin{equation}
		\|f(t,x)-f(t,y)\|\leq K\|x-y\|.
	\end{equation}
\end{itemize}
 We consider he following {\it jump-diffusion stochastic differential equation}
\begin{equation}
	\label{SDE_PROBLEM}
		\left\{ \begin{array}{ll}
			dX(t)=a(t,X(t-))dt+ b(t,X(t-))dW(t)+c(t,X(t-))dN(t), &t\in [0,T], \\
			X(0)=\xi, 
		\end{array}\right.
\end{equation}
where we denote $X(t-)=\lim\limits_{s\to t-}X(s)$, we take $X(0-):=X(0)$ (no jump at $t=0$) and $X(T+):=X(T)$. The SDE \eqref{SDE_PROBLEM} can also be rewritten as an SDE driven by the multidimensional  semimartingale $Y=[t,W,N]=[t,W^1,\ldots,W^{m_W},N^1,\ldots,N^{m_N}]$
\begin{equation}
	\label{SDE_PROBLEM2}
		\left\{ \begin{array}{ll}
			dX(t)=F(t,X(t-))dY(t), &t\in [0,T], \\
			X(0)=\xi, 
		\end{array}\right.
\end{equation}
where
\begin{equation}
F(t,y)=[F_{ij}(t,y)]_{1\leq i\leq d,1\leq j\leq 1+m_W+m_N}=
	\begin{bmatrix}
a_{1} & b_{12} & \cdots & b_{1 m_W} & c_{11} & \cdots & c_{1 m_N} \\
a_{2} & b_{22} & \cdots & b_{2 m_W} & c_{21} & \cdots & c_{2 m_N} \\
\vdots  & \vdots  & \ddots & \vdots & \vdots  & \ddots & \vdots \\
a_{d} & b_{d,1} & \cdots & b_{d m_W} & c_{d1} & \cdots & c_{d m_N} 
\end{bmatrix}(t,y).
\end{equation}
Under the assumptions $(A1)$, $(A2)$ we have that $F\in C([0,T]\times\mathbb{R}^d)$ and for all $t\in [0,T],x,y\in\mathbb{R}^d$
\begin{equation}
\label{cond_F_LIP}
	\|F(t,x)-F(t,y)\|\leq K_1\|x-y\|.
\end{equation}
Of course the stochastic differential equation \eqref{SDE_PROBLEM2} must be understood as an integral equation
\begin{equation}
	X(t)=\xi+\int\limits_0^t F(s,X(s-))dY(s), \ t\in [0,T].
\end{equation}
It follows from Theorem 12.8 at page 390 in \cite{KRAO} that under $(A0)$ and \eqref{cond_F_LIP} the SDE \eqref{SDE_PROBLEM2} has a unique strong solution $X=(X(t))_{t\in [0,T]}\in\mathbb{D}$.

A classical example is a following linear jump-diffusion SDEs
\begin{equation}
	\label{MERTON_SDE}
		\left\{ \begin{array}{ll}
			dX(t)=\mu X(t)dt+\sigma X(t) dW(t)+cX(t-)dN(t), &t\in [0,T], \\
			X(0)=x_0, 
		\end{array}\right.
\end{equation}
that models the stock price in the Merton's model, see \cite{PBL}. We assume  that $\mu\in\mathbb{R},\sigma>0, x_0>0$  and $c>-1$. The solution of (\ref{MERTON_SDE}) is
\begin{equation}
\label{gBm_jumps_def}
	X(t)=x_0 \exp\Bigl((\mu-\frac{1}{2}\sigma^2)t+\sigma W(t)\Bigr)\cdot (1+c)^{N(t)}.
\end{equation}

The results below are known for general SDEs driven by semimartingales.
\begin{lem} 
\label{mean_sq_reg_X}
Under the assumptions (A0), (A1), (A2)  there exists $C\in (0,+\infty)$ such that
\begin{equation}
	\mathbb{E}\Bigl(\sup\limits_{0\leq t\leq T}\|X(t)\|^2\Bigr)\leq C (1+\mathbb{E}\|\xi\|^2),
\end{equation}
and for all $t,s\in [0,T]$
\begin{equation}
	\mathbb{E}\|X(t)-X(s)\|^2\leq C\cdot (1+\mathbb{E}\|\xi\|^2)\cdot |t-s|.
\end{equation}
Moreover, if there exists $p\in [2,+\infty)$ such that $\mathbb{E}\|\xi\|^p<+\infty$ then there exists $C_p\in (0,+\infty)$ such that
\begin{equation}
\label{mom_b_lp_X}
    \mathbb{E}\Bigl(\sup\limits_{0\leq t\leq T}\|X(t)\|^p\Bigr)\leq C_p (1+\mathbb{E}\|\xi\|^p),
\end{equation}
and for all $t,s\in [0,T]$
\begin{equation}
\label{lp_holder_reg_X}
	\Bigl(\mathbb{E}\|X(t)-X(s)\|^p\Bigr)^{1/p}\leq C_p\cdot (1+(\mathbb{E}\|\xi\|^p)^{1/p})\cdot |t-s|^{1/2},
\end{equation}
\end{lem}
\begin{rem} Note that if $c=0$ then for any $p\in [2,+\infty)$ there exists $C_p\in (0,+\infty)$ such that for all $t,s\in [0,T]$
\begin{equation}
	\Bigl(\mathbb{E}\|X(t)-X(s)\|^p\Bigr)^{1/p}\leq C_p\cdot (1+(\mathbb{E}\|\xi\|^p)^{1/p})\cdot |t-s|^{1/2},
\end{equation}
provided that $\mathbb{E}\|\xi\|^p<+\infty$. This is no longer true if $c\neq 0$. Due to the properties of the Poisson process we have in that case
\begin{equation}
	\Bigl(\mathbb{E}\|X(t)-X(s)\|^p\Bigr)^{1/p}\leq C_p\cdot (1+(\mathbb{E}\|\xi\|^p)^{1/p})\cdot |t-s|^{1/p}.
\end{equation}
So in the jump-diffusion case the $L^p(\Omega)$-regularity of $X$ is the same as in the pure Gaussian case only for $p=2$. This, rather low for large $p>2$, $L^p$-regularity of $X$ in the jump case affects the rates of convergence for approximating algorithms.
\end{rem}
We also prove to following useful fact on properties of the left continuous modification of $X$, i.e. $(X(t-))_{t\in [0,T]}$.
\begin{prop}
\label{unsigned_Lp}
    Let us assume (A0), (A1), (A2) and let there exists $p\in [2,+\infty)$ such that $\mathbb{E}\|\xi\|^p<+\infty$. Then we have that
    \begin{itemize}
        \item [(i)] for all $t\in [0,T]$
        \begin{equation}
        \label{X_rl_eq_1}
            \mathbb{P}( X(t)=X(t-)) =1,
        \end{equation}
        \item [(ii)] for all $t,s\in [0,T]$, $\diamond, *\in\{-,+\}$
        \begin{equation}
        \label{Mean_dist_Xst}
            \|X(t^{\diamond})-X(s^{*})\|_{L^p(\Omega)}=\|X(t)-X(s)\|_{L^p(\Omega)},
        \end{equation}
        \item [(iii)] the function
        \begin{equation}
            [0,T]\ni t\to \|X(t)\|_{L^p(\Omega)} \in [0,+\infty)
        \end{equation}
        is continuous.
    \end{itemize}
\end{prop}
\begin{proof}
    We start by showing (i). By \eqref{lp_holder_reg_X} we have for all $t\in [0,T]$ that
    \begin{equation}
    \label{lim_e_st_1}
        \lim\limits_{s\to t-}\mathbb{E}\|X(t)-X(s)\|^p=0.
    \end{equation}
    Since $\mathbb{R}^d\ni x\to \|x\|^p$ is a continuous function, we get for all $t\in [0,T]$ 
    \begin{equation}
    \label{lim_e_st_2}
        \lim\limits_{s\to t-}\|X(t)-X(s)\|^p=\|\Delta X(t)\|^p
    \end{equation}
    and, for all $s\in [0,T]$,
    \begin{equation}
        \|X(t)-X(s)\|^p\leq 2c_p\cdot \sup\limits_{0\leq v\leq T}\|X(v)\|^p.
    \end{equation}
    where, by \eqref{mom_b_lp_X}, we have that $\mathbb{E}\Bigl(\sup\limits_{0\leq v\leq T}\|X(v)\|^p\Bigr)<+\infty$.
 Hence, by the Lebesgue’s dominated convergence theorem, \eqref{lim_e_st_1}, and \eqref{lim_e_st_2} we get for all $t\in [0,T]$
 \begin{equation}
 \label{lim_e_st_3}
     0=\lim\limits_{s\to t-}\mathbb{E}\|X(t)-X(s)\|^p=\mathbb{E}\|\Delta X(t)\|^p,
 \end{equation}
 which implies \eqref{X_rl_eq_1}.
 
 In order to proof (ii) note that for any $t\in [0,T]$, $*\in\{-,+\}$ we have
 \begin{equation}
     \|X(t^*)-X(t)\|_{L^p(\Omega)}\in\{0,\|\Delta X(t)\|_{L^p(\Omega)}\},
 \end{equation}
 hence, by \eqref{lim_e_st_3}, we have 
 \begin{equation}
 \label{Xtt_0}
      \|X(t^*)-X(t)\|_{L^p(\Omega)}=0.
 \end{equation}
 Applying \eqref{Xtt_0}, triangle inequality and reverse triangle inequality we get 
 \begin{eqnarray}
 \label{X_UP}
       &&\|X(t^{\diamond})-X(s^{*})\|_{L^p(\Omega)}\leq \|X(t^{\diamond})-X(t)\|_{L^p(\Omega)}+\|X(t)-X(s)\|_{L^p(\Omega)}+\|X(s)-X(s^{*})\|_{L^p(\Omega)}\notag\\
       &&=\|X(t)-X(s)\|_{L^p(\Omega)},
 \end{eqnarray}
 and
 \begin{eqnarray}
 \label{X_DOWN}
    && \|X(t^\diamond)-X(s^*)\|_{L^p(\Omega)}\geq \|X(t)-X(s)\|_{L^p(\Omega)}\notag\\
    &&-\|(X(t^{\diamond})-X(t))+(X(s)-X(s^*))\|_{L^p(\Omega)}=\|X(t)-X(s)\|_{L^p(\Omega)},
 \end{eqnarray}
 since, again by \eqref{Xtt_0},
 \begin{equation}
     0\leq \|(X(t^{\diamond})-X(t))+(X(s)-X(s^*))\|_{L^p(\Omega)}\leq \|X(t^{\diamond})-X(t)\|_{L^p(\Omega)}+\|X(s)-X(s^*)\|_{L^p(\Omega)}=0.
\end{equation}
  Combining \eqref{X_UP} and \eqref{X_DOWN} we get \eqref{Mean_dist_Xst}.

 Let us denote by $g(t)=\|X(t)\|_{L^p(\Omega)}$. We have for all $t\in [0,T]$ that
 \begin{equation}
     \lim\limits_{s\to t-}\|X(s)\|^p=\|X(t-)\|^p,
 \end{equation}
 and for all $s\in [0,T]$
 \begin{eqnarray}
     \|X(s)\|^p\leq \sup\limits_{0\leq v\leq T}\|X(v)\|^p,
 \end{eqnarray}
  where $\mathbb{E}\Bigl(\sup\limits_{0\leq v\leq T}\|X(v)\|^p\Bigr)<+\infty$.
 Hence, by the Lebesgue’s dominated convergence theorem
 \begin{equation}
 \label{g_left_lim}
     g(t-)=\lim\limits_{s\to t-}g(s)=\|X(t-)\|_{L^p(\Omega)}.
 \end{equation}
 From \eqref{lim_e_st_3}, \eqref{g_left_lim} we have for all $t\in [0,T]$
 \begin{equation}
     0\leq |g(t)-g(t-)|\leq \|X(t)-X(t-)\|_{L^p(\Omega)}=\|\Delta X(t)\|_{L^p(\Omega)}=0,
 \end{equation}
 and hence
 \begin{equation}
 \label{g_lcont}
     g(t-)=g(t).
 \end{equation}
 Moreover, since $X$ is ca\'dl\'ag, we have for all $t\in [0,T]$
 \begin{equation}
     \lim\limits_{s\to t+}\|X(s)\|^p=\|X(t+)\|^p=\|X(t)\|^p,
 \end{equation}
 and by using again the Lebesgue’s dominated convergence theorem we get for all $t\in [0,T]$ that
 \begin{equation}
 \label{g_rcont}
     g(t+)=\lim\limits_{s\to t+}g(s)=g(t).
 \end{equation}
 From \eqref{g_lcont}, \eqref{g_rcont} we have that the function $g$ is continuous on $[0,T]$.
\end{proof}
\section{*SDEs driven by the Wiener process and compound Poisson process}
In the scalar case we briefly comment the case when the jump part of the underlying SDE is driven by the compound Poisson process $Y$, i.e.  
\begin{equation}
	\label{SDE_PROBLEM_CP1}
		\left\{ \begin{array}{ll}
			dX(t)=a(t,X(t-))dt+ b(t,X(t-))dW(t)+c(t,X(t-))dJ(t), &t\in [0,T], \\
			X(0)=\xi. 
		\end{array}\right.
\end{equation}
In this case
\begin{equation}
	\int\limits_0^t c(s,X(s-))dJ(s)=\sum\limits_{0<s\leq t} c(s,X(s-))\Delta J(s)=\sum\limits_{k=1}^{N(t)}c(\tau_k,X(\tau_k-))\xi_k.
\end{equation}
Note that more general equations are considered, for example, when the noise comes from Levy processes, Poisson random measures etc.

Merton model, driven by compound Poisson process, has the following form
\begin{equation}
	\label{MERTON_SDE1}
		\left\{ \begin{array}{ll}
			dX(t)=\mu X(t)dt+\sigma X(t) dW(t)+X(t-)dJ(t), &t\in [0,T], \\
			X(0)=x_0, 
		\end{array}\right.
\end{equation}
with the explicit solution
\begin{equation}
	X(t)=x_0\exp\Bigl((\mu-\frac{1}{2}\sigma^2)t+\sigma W(t)\Bigr)\prod\limits_{k=1}^{N(t)}(\xi_k+1).
\end{equation}
\chapter{Euler scheme and its convergence}
One of the most basic, but still fundamental, algorithm for approximation of solutions of SDEs is the Euler-Maruyama algorithm. 

For the functions
\begin{eqnarray}
	&&a:[0,T]\times\mathbb{R}^d\to \mathbb{R}^d,\\
	&&b:[0,T]\times\mathbb{R}^d\to \mathbb{R}^{d\times m_W},\\
	&&c:[0,T]\times\mathbb{R}^d\to \mathbb{R}^{d\times m_N},
\end{eqnarray}
we assume that
\begin{itemize}
	\item [(B)] there exists $K\in (0,+\infty)$ such that for $f\in\{a,b,c\}$ and for all $s,t\in [0,T]$, $x,y\in\mathbb{R}^d$
	\begin{equation}
		\|f(t,x)-f(s,y)\|\leq K(|t-s|+\|x-y\|),
	\end{equation}
		\item [(C)] the intensity functions $\lambda_j:[0,T]\to (0,+\infty)$, $j=1,2,\ldots,m_N$, are Borel measurable and bounded.
\end{itemize}
For $f=a$ the norm $\|\cdot\|$ is the euclidean norm $\|\cdot\|_2$, while for $f\in\{b,c\}$ the norm $\|\cdot\|$ is the Frobenius norm $\|\cdot\|_F$. From now we will denote the both norms by $\|\cdot\|$ and the notion will be clear from the context.
(Note that if $f$ satisfies $(B)$ then it also satisfies $(A1)$ and $(A2)$.)

The following facts are left as  simple exercises.
\begin{fact} 
\label{f_lin_gr}
If $f$ satisfies $(B)$ then there exists $K_1\in (0,+\infty)$ such that for all $(t,x)\in [0,T]\times\mathbb{R}^d$
	\begin{equation}
		\|f(t,x)\|\leq K_1(1+\|x\|).
	\end{equation}
\end{fact}
\begin{fact}
\label{f_np}
Suppose that the random vectors $X:\Omega\to\mathbb{R}^{d\times m}$ and  $Y:\Omega\to\mathbb{R}^m$ are independent. Then
	\begin{equation}
	\label{submult_1}
		\mathbb{E}\|X\cdot Y\|^2\leq \mathbb{E}\|X\|^2\cdot\mathbb{E}\|Y\|^2.
	\end{equation}
\end{fact}
The classical Euler scheme is defined as follows. 

Let $n\in\mathbb{N}$ and $0=t_0<t_1<\ldots<t_{n}=T$ be (possibly) non-uniform discretization of $[0,T]$. The classical Euler scheme for the jump-diffusion SDE \eqref{SDE_PROBLEM} is defined as follows:
\begin{equation}
	X^E_n(0)=\xi,
\end{equation}
and for $k=0,1,\ldots,n-1$
\begin{equation}
	X^E_n(t_{k+1})=X^E_n(t_{k})+a(U_{k,n}^E)\cdot\Delta t_k+b(U_{k,n}^E)\cdot\Delta W_k+c(U_{k,n}^E)\cdot\Delta N_k,
\end{equation}
where
\begin{equation} 
	U_{k,n}^E=(t_k,X^E_n(t_{k})),
\end{equation}	
\begin{eqnarray}
	&&\Delta t_k = t_{k+1}-t_k,\notag\\
	&&\Delta W_k=[\Delta W^1_k,\ldots,\Delta W^{m_W}_k]^T,\notag\\
	&&\Delta N_k=[\Delta N^1_k,\ldots,\Delta N^{m_N}_k]^T,
\end{eqnarray}
and
\begin{equation}
	\Delta Z_k^j=Z^j(t_{k+1})-Z^j(t_{k}), \ Z\in\{N,W\}.
\end{equation}
Since
\begin{equation}
	b(U_{k,n}^E)\cdot\Delta W_k=\Bigl(\sum\limits_{j=1}^{m_W}b_{ij}(U_{k,n}^E)\cdot\Delta W^j_k\Bigr)_{i=1,2,\ldots,d},
\end{equation}
\begin{equation}
	c(U_{k,n}^E)\cdot\Delta N_k=\Bigl(\sum\limits_{j=1}^{m_N}c_{ij}(U_{k,n}^E)\cdot\Delta N^j_k\Bigr)_{i=1,2,\ldots,d},
\end{equation}
we can write each component of the Euler scheme as follows
\begin{equation}
	X^E_{n,i}(0)=\xi_{i},
\end{equation}
\begin{eqnarray}
	X^E_{n,i}(t_{k+1})=X^E_{n,i}(t_{k})&+&a_i(U_{k,n}^E)\cdot\Delta t_k\notag\\
	&+&\sum\limits_{j=1}^{m_W}b_{ij}(U_{k,n}^E)\cdot\Delta W^j_k\notag\\
	&+&\sum\limits_{j=1}^{m_N}c_{ij}(U_{k,n}^E)\cdot\Delta N^j_k,
\end{eqnarray}
for $k=0,1,\ldots,n-1$, $i=1,2,\ldots,d$.

\begin{exmp}
For the scalar equation \eqref{MERTON_SDE} we have that
    \begin{equation}
        X_n^E(t_{k+1})=X_n^E(t_{k})(1+\mu\Delta t_k+\sigma\Delta W_k+c\Delta N_k)=x_0\cdot\prod_{j=0}^{k}(1+\mu \Delta t_j+\sigma\Delta W_j+c\Delta N_j),
    \end{equation}
for $k=0,1,\ldots,n-1$.    
\end{exmp}
The aim of this chapter is to prove the following result.
\begin{thm}
\label{err_euler_1}
Let us assume that $a,b,c$ satisfy $(B)$, $(C)$ and $\xi$ satisfies $(A0)$. Then there exists $C\in (0,+\infty)$ such that for all $n\in\mathbb{N}$ and $0=t_0<t_1<\ldots<t_n=T$ we have
\begin{equation}
	\max\limits_{0\leq k\leq n}\Bigl(\mathbb{E}\|X(t_k)-X_n^E(t_k)\|^2\Bigr)^{1/2}\leq C\cdot \Bigl(1+(\mathbb{E}\|\xi\|^2)^{1/2}\Bigr)\cdot\max\limits_{0\leq i\leq n-1}(t_{i+1}-t_i)^{1/2}.
\end{equation}
\end{thm}
Before we prove this fact we need several auxiliary lemmas concerning the so called {\it time-continuous Euler approximation} $\tilde X^E_n=(\tilde X^E_n(t))_{t\in [0,T]}$. It is defined as follows:
\begin{equation}
	\tilde X^E_n(0)=\xi,
\end{equation}
and for $t\in (t_k,t_{k+1}]$, $k=0,1,\ldots,n-1$ we set
\begin{equation}
\label{euler_proc_def}
	\tilde X^E_n(t)=X^E_n(t_k)+a(U_{k,n}^E)\cdot (t-t_k)+b(U_{k,n}^E)\cdot (W(t)-W(t_k))+c(U_{k,n}^E)\cdot (N(t)-N(t_k)).
\end{equation}
The process $\tilde X^E_n$ can be seen as a specific interpolation between the points $X^E(t_k)$, $k=0,1,\ldots,n$, with the coefficients $a,b,c$ 'frozen' at $U_{k,n}^E$. The approximation $\tilde X^E_n$ is not implementable, since it requires the complete knowledge of trajectories of $N$ and $W$. 

By induction we have that
\begin{equation}
\label{interp_E_proc}
	\tilde X^E_n(t_k)=X^E_n(t_k), \ k=0,1,\ldots,n.
\end{equation}
Since the trajectories of $[N,W]$ are right-continuous we get 
\begin{equation}
	\lim\limits_{t\to t_k+}\tilde X^E_n(t)=X^E_n(t_k)=\tilde X^E_n(t_k),
\end{equation}
and
\begin{eqnarray}
	&&\lim\limits_{t\to t_{k+1}-}\tilde X^E_n(t)=X^E_n(t_k)+a(U_{k,n}^E)\cdot (t_{k+1}-t_k)+b(U_{k,n}^E)\cdot (W(t_{k+1})-W(t_k))\notag\\
&&\quad\quad +c(U_{k,n}^E)\cdot (N(t_{k+1}-)-N(t_k))\neq \tilde X^{E}_n(t_{k+1}) \ (=X^E_n(t_{k+1}))
\end{eqnarray}
for $k=0,1,\ldots,n-1$. Moreover, it is easy to see that $\tilde X^E_n=(\tilde X^E_n(t))_{t\in [0,T]}$ is adapted to $(\Sigma_t)_{t\in [0,+\infty)}$ and has c\`adl\`ag paths. Hence, $\tilde X^E_n\in\mathbb{D}$.
\begin{lem}
\label{err_euler_2}
Let us assume that $a,b,c$ satisfy $(B)$, $(C)$ and $\xi$ satisfies $(A0)$. Then there exists $C\in (0,+\infty)$ such that for all $n\in\mathbb{N}$ and $0=t_0<t_1<\ldots<t_n=T$ we have
\begin{equation}
\label{err_euler_21}
	\sup\limits_{t\in [0,T]}\mathbb{E}\|\tilde X^E_n(t)\|^2\leq C(1+\mathbb{E}\|\xi\|^2).
\end{equation}
\end{lem}
{\bf Proof.} Firstly we show that 
\begin{equation}
\label{bound_mom_E_1}
	\max\limits_{0\leq k \leq n}\mathbb{E}\|X^E_n(t_k)\|^2<+\infty.
\end{equation}
We proceed by induction and let us assume that there exists $s\in \{0,1,\ldots,n-1\}$ with the property that 
\begin{equation}
	\max\limits_{0\leq k\leq s}\mathbb{E}\|X^E_n(t_k)\|^2<+\infty.
\end{equation}
(This is of course satisfied for $s=0$.) Since $\{\sigma(b(U_{s,n}^E)),\sigma(\Delta W_s)\}$ and $\{\sigma(c(U_{s,n}^E)),\sigma(\Delta N_s)\}$ are two families of independent $\sigma$-fields, we get by the Fact \ref{f_np}
\begin{eqnarray}
	\mathbb{E}\|X^E_n(t_{s+1})\|^2 &\leq& C_1\Bigl(\mathbb{E}\|X^E_n(t_s)\|^2+(\Delta t_s)^2\cdot\mathbb{E}\|a(U_{s,n}^E)\|^2\notag\\
&&\quad +\mathbb{E}\|b(U_{s,n}^E)\|^2 \cdot\mathbb{E}\|\Delta W_s\|^2+\mathbb{E}\|c(U_{s,n}^E)\|^2 \cdot\mathbb{E}\|\Delta N_s\|^2\Bigr).
\end{eqnarray}
Since there exist $c_W,c_N\in (0,+\infty)$ such that for all $t\in [t_k,t_{k+1}]$, $k=0,1,\ldots,n-1$ it holds
\begin{eqnarray}
	&&\mathbb{E}\|W(t)-W(t_k)\|^2\leq c_W (t-t_k),\notag\\
	&&\mathbb{E}\|N(t)-N(t_k)\|^2\leq c_N (t-t_k),
\end{eqnarray}
we get by Fact \ref{f_lin_gr} that
\begin{eqnarray}
	\mathbb{E}\|X^E_n(t_{s+1})\|^2 &\leq& C_2\Bigl(\mathbb{E}\|X^E_n(t_s)\|^2+(\Delta t_s)^2\cdot (1+\mathbb{E}\|X^E_n(t_s)\|^2)\notag\\
&&\quad +(1+\mathbb{E}\|X^E_n(t_s)\|^2) \cdot \Delta t_s+(1+\mathbb{E}\|X^E_n(t_s)\|^2)\cdot\Delta t_s\Bigr)<+\infty.
\end{eqnarray}
Hence,
\begin{equation}
	\max\limits_{0\leq k\leq s+1}\mathbb{E}\|X^E_n(t_k)\|^2<+\infty.
\end{equation}
By the rules of induction, this completes the proof of \eqref{bound_mom_E_1}. 

From \eqref{euler_proc_def} it follows that for all 
$t\in [t_k,t_{k+1}]$, $k=0,1,\ldots,n-1$ 
\begin{eqnarray}
	\mathbb{E}\|\tilde X^E_n(t)\|^2 &\leq& C_2\Bigl(\mathbb{E}\|X^E_n(t_k)\|^2+ (t-t_k)^2\cdot (1+\mathbb{E}\|X^E_n(t_k)\|^2)\notag\\
&&\quad +(1+\mathbb{E}\|X^E_n(t_k)\|^2) \cdot (t-t_k)+(1+\mathbb{E}\|X^E_n(t_k)\|^2)\cdot (t-t_k)\Bigr)\notag\\
&&\leq C_3(1+\mathbb{E}\|X^E_n(t_k)\|^2),
\end{eqnarray}
and by \eqref{bound_mom_E_1}
\begin{equation}
	\label{bound_mom_E_2}
	\sup\limits_{0\leq t\leq T}\mathbb{E}\|\tilde X^E_n(t)\|^2=\max\limits_{0\leq k\leq n-1}\sup\limits_{t_k\leq t\leq t_{k+1}}\mathbb{E}\|\tilde X^E_n(t)\|^2\leq C_3(1+\max\limits_{0\leq k\leq n-1}\mathbb{E}\|X^E_n(t_k)\|^2)<+\infty.
\end{equation}
Note that for now the bound in \eqref{bound_mom_E_2} depends on $n$. In the second part of the proof we will show, with the help of Gronwall's lemma, that we can obtain the bound \eqref{err_euler_21}, with $C$ that is independent of $n$.

We can write for all $t\in [0,T]$ that
\begin{equation}
\label{euler_proc_1}
	\tilde X^E_n(t)=\xi+\int\limits_0^t \tilde a_n(s)ds+\int\limits_0^t \tilde b_n(s)dW(s)+\int\limits_0^t \tilde c_n(s)dN(s),
\end{equation}
where for $f\in \{a,b,c\}$
\begin{equation}
	\tilde f_n(s) = \sum\limits_{k=0}^{n-1}f(U_{k,n}^E)\cdot\mathbf{1}_{(t_k,t_{k+1}]}(s),
\end{equation}
and recall that $U_{k,n}^E=(t_k, X^E_n(t_{k}))=(t_k, \tilde X^E_n(t_{k}))$. Note that $(\tilde f_n(t))_{t\in [0,T]}\in\mathbb{L}$ for $f\in\{a,b,c\}$.  Furthermore, for $f\in \{a,b,c\}$ and all $t\in [0,T]$ it holds
\begin{eqnarray}
	\|\tilde f_n(t)\|^2&\leq& \Bigl(\sum\limits_{k=0}^{n-1}\|f(U_{k,n}^E)\|\cdot\mathbf{1}_{(t_k,t_{k+1}]}(t)\Bigr)^2\notag\\
	&&=\sum\limits_{k=0}^{n-1}\|f(U_{k,n}^E)\|^2\cdot\mathbf{1}_{(t_k,t_{k+1}]}(t),
\end{eqnarray}
which, together with Fact \ref{f_lin_gr}, implies
\begin{eqnarray}
\label{est_aux_1}
	&&\int\limits_0^t \|\tilde f_n(s)\|^2 ds\leq \int\limits_0^t \sum\limits_{k=0}^{n-1}\|f(U_{k,n}^E)\|^2\cdot\mathbf{1}_{(t_k,t_{k+1}]}(s)ds\notag\\
	&&\leq K_1^2 \int\limits_0^t \sum\limits_{k=0}^{n-1}(1+\|\tilde X^E_n(t_{k})\|)^2\cdot\mathbf{1}_{(t_k,t_{k+1}]}(s)ds\notag\\
	&&\leq 2K_1^2 \int\limits_0^t \sum\limits_{k=0}^{n-1}(1+\|\tilde X^E_n(t_{k})\|^2)\cdot\mathbf{1}_{(t_k,t_{k+1}]}(s)ds\notag\\
	&&\leq 2K_1^2\int\limits_0^T \sum\limits_{k=0}^{n-1}\mathbf{1}_{(t_k,t_{k+1}]}(s)ds +2K_1^2\int\limits_0^t \sum\limits_{k=0}^{n-1}\|\tilde X^E_n(t_{k})\|^2\cdot\mathbf{1}_{(t_k,t_{k+1}]}(s)ds\notag\\
	&&=2K_1^2T+2K_1^2\int\limits_0^t \sum\limits_{k=0}^{n-1}\|\tilde X^E_n(t_{k})\|^2\cdot\mathbf{1}_{(t_k,t_{k+1}]}(s)ds.
\end{eqnarray}
By the fact that $N(t)=\tilde N(t)+m(t)$, we have 
\begin{eqnarray}
	&&\mathbb{E}\|\tilde X^E_n(t)\|^2\leq C_1\Bigl(\mathbb{E}\|\xi\|^2+\mathbb{E}\Bigl\|\int\limits_0^t \tilde a_n(s)ds\Bigl\|^2+\mathbb{E}\Bigl\|\int\limits_0^t \tilde b_n(s)dW(s)\Bigl\|^2\notag\\
&&\quad +\mathbb{E}\Bigl\|\int\limits_0^t \tilde c_n(s)d\tilde N(s)\Bigl\|^2+\mathbb{E}\Bigl\|\int\limits_0^t \tilde c_n(s)dm(s)\Bigl\|^2\Bigr).
\end{eqnarray}
From the H\"older inequality, It\^o isometry for stochastic integrals driven by martingales, and  \eqref{est_aux_1} we get for $t\in [0,T]$
\begin{eqnarray}
	&&\mathbb{E}\Bigl\|\int\limits_0^t \tilde a_n(s)ds\Bigl\|^2\leq \mathbb{E}\Bigl(\int\limits_0^t \|\tilde a_n(s)\|ds\Bigr)^2\leq T\mathbb{E}\int\limits_0^t\|\tilde a_n(s)\|^2ds\notag\\
	&&\leq K_2+K_3\int\limits_0^t \sum\limits_{k=0}^{n-1}\|\tilde X^E_n(t_{k})\|^2\cdot\mathbf{1}_{(t_k,t_{k+1}]}(s)ds,
\end{eqnarray}
\begin{eqnarray}
	&&\mathbb{E}\Bigl\|\int\limits_0^t \tilde b_n(s)dW(s)\Bigl\|^2= \mathbb{E}\int\limits_0^t \|\tilde b_n(s)\|^2ds\notag\\
	&&\leq K_4+K_5\int\limits_0^t \sum\limits_{k=0}^{n-1}\|\tilde X^E_n(t_{k})\|^2\cdot\mathbf{1}_{(t_k,t_{k+1}]}(s)ds,
\end{eqnarray}
\begin{eqnarray}
	&&\mathbb{E}\Bigl\|\int\limits_0^t \tilde c_n(s)d\tilde N(s)\Bigl\|^2=\sum\limits_{j=1}^{m_N}\mathbb{E}\int\limits_0^t\|\tilde c^{(j)}_n(s)\|^2\cdot \lambda_j(s)ds\notag\\
	&&\leq\max\limits_{1\leq j\leq m_N}\|\lambda_j\|_{\infty}\cdot\mathbb{E}\int\limits_0^t\|\tilde c_n(s)\|^2ds\notag\\
	&&\leq K_6+K_7\int\limits_0^t \sum\limits_{k=0}^{n-1}\|\tilde X^E_n(t_{k})\|^2\cdot\mathbf{1}_{(t_k,t_{k+1}]}(s)ds,
\end{eqnarray}
and
\begin{eqnarray}
	&&\mathbb{E}\Bigl\|\int\limits_0^t \tilde c_n(s)dm(s)\Bigl\|^2=\sum\limits_{i=1}^d\mathbb{E}\Bigl|\sum\limits_{j=1}^{m_N}\int\limits_0^t \tilde c_{n,ij}(s)\cdot\lambda_j(s)ds\Bigl|^2\notag\\
	&\leq & m_N\cdot\sum\limits_{i=1}^d\sum\limits_{j=1}^{m_N}\mathbb{E}\Bigl|\int\limits_0^t \tilde c_{n,ij}(s)\cdot\lambda_j(s)ds\Bigl|^2\notag\\
	&\leq & m_N T\cdot\sum\limits_{i=1}^d\sum\limits_{j=1}^{m_N}\int\limits_0^t\mathbb{E}|\tilde c_{n,ij}(s)|^2\cdot\lambda^2_j(s) ds\notag\\
	&\leq &m_N T\cdot\max\limits_{1\leq j\leq m_N}\|\lambda_j\|^2_{\infty}\cdot\mathbb{E}\int\limits_0^t\|\tilde c_n(s)\|^2 ds\notag\\
	&&\leq K_8+K_9\int\limits_0^t \sum\limits_{k=0}^{n-1}\|\tilde X^E_n(t_{k})\|^2\cdot\mathbf{1}_{(t_k,t_{k+1}]}(s)ds.
\end{eqnarray}
Hence
\begin{eqnarray}
	\mathbb{E}\|\tilde X^E_n(t)\|^2&\leq& K_{10}(1+\mathbb{E}\|\xi\|^2)+K_{11}\int\limits_0^t \sum\limits_{k=0}^{n-1}\mathbb{E}\|\tilde X^E_n(t_{k})\|^2\cdot\mathbf{1}_{(t_k,t_{k+1}]}(s)ds\notag\\
	&\leq& K_{10}(1+\mathbb{E}\|\xi\|^2)+K_{11}\int\limits_0^t\sup\limits_{0\leq u\leq z}\mathbb{E}\|\tilde X^E_n(u)\|^2dz 
\end{eqnarray}
where $K_{10},K_{11}$ depend only on $K_1$, $T$, $m_N$, $\max\limits_{1\leq j\leq m_N}\|\lambda_j\|_{\infty}$. Therefore, for all $t \in [0,T]$
\begin{equation}
	\sup\limits_{0\leq s \leq t}\mathbb{E}\|\tilde X^E_n(s)\|^2\leq  K_{10}(1+\mathbb{E}\|\xi\|^2)+K_{11}\int\limits_0^t\sup\limits_{0\leq u\leq z}\mathbb{E}\|\tilde X^E_n(u)\|^2dz .
\end{equation}
Note that the function
\begin{equation}
	[0,T]\ni t \to \sup\limits_{0\leq s \leq t}\mathbb{E}\|\tilde X^E_n(s)\|^2
\end{equation}
is Borel measurable (as a nondecreasing function) and bounded, which follows from   \eqref{bound_mom_E_2}. By using Lemma \ref{GLCW} we get the thesis.\ \ \ $\blacksquare$
\\ \\
{\bf Proof of Theorem \ref{err_euler_1}.}  We have that for all $t\in [0,T]$
\begin{equation}
\label{sol_proc_1}	
	X(t)=\xi+\int\limits_0^t \hat a(s)ds+\int\limits_0^t \hat b(s)dW(s)+\int\limits_0^t \hat c(s)dN(s),
\end{equation}
where for $f\in\{a,b,c\}$
\begin{equation}
	\hat f(s)=\sum\limits_{k=0}^{n-1}f(s,X(s-))\cdot\mathbf{1}_{(t_k,t_{k+1}]}(s).
\end{equation}
(Note that for $f\in\{a,b,c\}$ the process  $(\hat f(t))_{t\in [0,T]}$ is in $\mathbb{L}$ so all Lebesgue and stochastic integrals appearing in  \eqref{sol_proc_1} are well-defined.) From the representations above we thus get
\begin{equation}
\label{err_est_1}
	\mathbb{E}\|X(t)-\tilde X^E_n(t)\|^2\leq 3\Bigl(A_{1,n}(t)+A_{2,n}(t)+A_{3,n}(t)\Bigr)
\end{equation}
where
\begin{eqnarray}
	&&A_{1,n}(t)=\mathbb{E}\Bigl\|\int\limits_0^t\Bigl(\hat a(s)-\tilde a_n(s)\Bigr)ds\Bigl\|^2,\notag\\
	&&A_{2,n}(t)=\mathbb{E}\Bigl\|\int\limits_0^t\Bigl(\hat b(s)-\tilde b_n(s)\Bigr)dW(s)\Bigl\|^2,\notag\\
	&&A_{3,n}(t)=\mathbb{E}\Bigl\|\int\limits_0^t\Bigl(\hat c(s)-\tilde c_n(s)\Bigr)dN(s)\Bigl\|^2.
\end{eqnarray}
By the Schwarz inequality we get
\begin{equation}
\label{err_est_2}
	A_{1,n}(t)\leq T\mathbb{E}\int\limits_0^t \|\hat a(s)-\tilde a_n(s)\|^2 ds,
\end{equation}
while from the It\^o isometry
\begin{equation}
\label{err_est_3}
	A_{2,n}(t)=\mathbb{E}\int\limits_0^t \|\hat b(s)-\tilde b_n(s)\|^2 ds.
\end{equation}
Furthermore
\begin{equation}
\label{a3_est_1}
	A_{3,n}(t)\leq 2 \Bigl(B_{1,n}(t)+B_{2,n}(t)\Bigr),
\end{equation}
where
\begin{eqnarray}
	&&B_{1,n}(t)=\mathbb{E}\Bigl\|\int\limits_0^t\Bigl(\hat c(s)-\tilde c_n(s)\Bigr)d\tilde N(s)\Bigl\|^2,\notag\\
	&&B_{2,n}(t)=\mathbb{E}\Bigl\|\int\limits_0^t\Bigl(\hat c(s)-\tilde c_n(s)\Bigr)dm(s)\Bigl\|^2.
\end{eqnarray}
From the isometry for martingale-driven stochastic integrals we get
\begin{eqnarray}
\label{a3_est_2}
	B_{1,n}(t)&=&\sum\limits_{j=1}^{m_N}\mathbb{E}\int\limits_0^t\|\hat c^{(j)}(s)-\tilde c^{(j)}_n(s)\|^2\cdot \lambda_j(s)ds\notag\\
	&\leq &\max\limits_{1\leq j\leq m_N}\|\lambda_j\|_{\infty}\cdot\mathbb{E}\int\limits_0^t \|\hat c(s)-\tilde c_n(s)\|^2 ds
\end{eqnarray}
and by the Schwarz inequality
\begin{eqnarray}
\label{a3_est_3}
	B_{2,n}(t)&= &\sum\limits_{i=1}^d\mathbb{E}\Bigl|\sum\limits_{j=1}^{m_N}\int\limits_0^t \Bigl(\hat c_{ij}(s)-\tilde c_{n,ij}(s)\Bigr)\cdot\lambda_j(s)ds\Bigl|^2\notag\\
	&\leq & m_N\cdot\sum\limits_{i=1}^d\sum\limits_{j=1}^{m_N}\mathbb{E}\Bigl|\int\limits_0^t \Bigl(\hat c_{ij}(s)-\tilde c_{n,ij}(s)\Bigr)\cdot\lambda_j(s)ds\Bigl|^2\notag\\
	&\leq & m_N T\cdot\sum\limits_{i=1}^d\sum\limits_{j=1}^{m_N}\int\limits_0^t\mathbb{E}|\hat c_{ij}(s)-\tilde c_{n,ij}(s)|^2\cdot\lambda^2_j(s) ds\notag\\
	&\leq &m_N T\cdot\max\limits_{1\leq j\leq m_N}\|\lambda_j\|^2_{\infty}\cdot\mathbb{E}\int\limits_0^t\|\hat c(s)-\tilde c_n(s)\|^2 ds.
\end{eqnarray}
Hence, by \eqref{a3_est_1}, \eqref{a3_est_2}, and  \eqref{a3_est_3} we get
\begin{equation}
\label{err_est_4}
	A_{3,n}(t)\leq C\cdot\mathbb{E}\int\limits_0^t\|\hat c(s)-\tilde c_n(s)\|^2 ds.
\end{equation}
Combining \eqref{err_est_1}, \eqref{err_est_2}, \eqref{err_est_3}, and \eqref{err_est_4} we arrive at
\begin{equation}
	\mathbb{E}\|X(t)-\tilde X^E_n(t)\|^2\leq C\sum\limits_{f\in\{a,b,c\}}\mathbb{E}\int\limits_0^t \|\hat f(s)-\tilde f_n(s)\|^2 ds.
\end{equation}
Now, for $f\in\{a,b,c\}$ and $s\in [0,T]$ we get
\begin{eqnarray}
	\|\hat f(s)-\tilde f_n(s)\|&=&\Bigl\|\sum\limits_{k=0}^{n-1}\Bigl(f(s,X(s-))-f(U_{k,n}^E)\Bigr)\cdot\mathbf{1}_{(t_k,t_{k+1}]}(s)\Bigl\|\notag\\
	&\leq& \sum\limits_{k=0}^{n-1}\|f(s,X(s-))-f(U_{k,n}^E)\|\cdot\mathbf{1}_{(t_k,t_{k+1}]}(s)\notag\\
	&\leq& K\sum\limits_{k=0}^{n-1}\Bigl((s-t_k)+\|X(s-)-\tilde X^E_n(t_k)\|\Bigr)\cdot\mathbf{1}_{(t_k,t_{k+1}]}(s)\notag\\
	&\leq& K\max\limits_{0\leq k \leq n-1}\Delta t_k+K\sum\limits_{k=0}^{n-1}\|X(s-)-X(t_k)\|\cdot\mathbf{1}_{(t_k,t_{k+1}]}(s)\notag\\
	&&\quad +K\sum\limits_{k=0}^{n-1}\|X(t_k)-\tilde X^E_n(t_k)\|\cdot\mathbf{1}_{(t_k,t_{k+1}]}(s)
\end{eqnarray}
and hence
\begin{eqnarray}
	\|\hat f(s)-\tilde f_n(s)\|^2 &\leq& 3K^2\max\limits_{0\leq k \leq n-1}(\Delta t_k)^2+3K^2\sum\limits_{k=0}^{n-1}\|X(s-)-X(t_k)\|^2\cdot\mathbf{1}_{(t_k,t_{k+1}]}(s)\notag\\
	&&\quad +3K^2\sum\limits_{k=0}^{n-1}\|X(t_k)-\tilde X^E_n(t_k)\|^2\cdot\mathbf{1}_{(t_k,t_{k+1}]}(s).
\end{eqnarray}
Since any trajectory of $X$ has only finite number of (jump) discontinuities in $[0,T]$, the functions
\begin{equation}
	[0,T]\ni s\to \sum\limits_{k=0}^{n-1}\|X(s-)-X(t_k)\|^2\cdot\mathbf{1}_{(t_k,t_{k+1}]}(s),
\end{equation}
\begin{equation}
	[0,T]\ni s\to \sum\limits_{k=0}^{n-1}\|X(s)-X(t_k)\|^2\cdot\mathbf{1}_{[t_k,t_{k+1})}(s)
\end{equation}
differs only on a finite subset of $[0,T]$ (the first is a c\`agl\`ad modification of the second one). This implies for $f\in\{a,b,c\}$ and  $t \in [0,T]$ that
\begin{eqnarray}
	\int\limits_0^t \|\hat f(s)-\tilde f_n(s)\|^2ds &\leq & 3K^2T\max\limits_{0\leq k \leq n-1}(\Delta t_k)^2+3K^2\int\limits_0^t\sum\limits_{k=0}^{n-1}\|X(s)-X(t_k)\|^2\cdot\mathbf{1}_{[t_k,t_{k+1})}(s)ds\notag\\
		&&\quad +3K^2\int\limits_0^t \sum\limits_{k=0}^{n-1}\|X(t_k)-\tilde X^E_n(t_k)\|^2\cdot\mathbf{1}_{(t_k,t_{k+1}]}(s)ds.
\end{eqnarray} 
By Lemma \ref{mean_sq_reg_X} we have that
\begin{eqnarray}
	\int\limits_0^t\sum\limits_{k=0}^{n-1}\mathbb{E}\|X(s)-X(t_k)\|^2\cdot\mathbf{1}_{[t_k,t_{k+1})}(s)ds&\leq& C(1+\mathbb{E}\|\xi\|^2)\int\limits_0^T\sum\limits_{k=0}^{n-1} (s-t_k)\cdot \mathbf{1}_{(t_k,t_{k+1}]}(s)ds\notag\\
	&\leq& C_1(1+\mathbb{E}\|\xi\|^2)\cdot\max\limits_{0\leq k\leq n-1}\Delta t_k.
\end{eqnarray}
Since $\max\limits_{0\leq k\leq n-1}(\Delta t_k)^2\leq T\max\limits_{0\leq k\leq n-1}\Delta t_k$, we get for all $t\in [0,T]$
\begin{eqnarray}
	\mathbb{E}\int\limits_0^t \|\hat f(s)-\tilde f_n(s)\|^2ds &\leq & K_1\cdot (1+\mathbb{E}\|\xi\|^2)\cdot \max\limits_{0\leq k\leq n-1}\Delta t_k\notag\\
&&\quad +K_2\int\limits_0^t \sum\limits_{k=0}^{n-1}\mathbb{E}\|X(t_k)-\tilde X^E_n(t_k)\|^2\cdot\mathbf{1}_{(t_k,t_{k+1}]}(s)ds.
\end{eqnarray} 
and hence
\begin{eqnarray}
	\mathbb{E}\|X(t)-\tilde X^E_n(t)\|^2 &\leq & K_3\cdot (1+\mathbb{E}\|\xi\|^2)\cdot \max\limits_{0\leq k\leq n-1}\Delta t_k\notag\\
&&\quad +K_4\int\limits_0^t \sum\limits_{k=0}^{n-1}\mathbb{E}\|X(t_k)-\tilde X^E_n(t_k)\|^2\cdot\mathbf{1}_{(t_k,t_{k+1}]}(s)ds\notag\\
&\leq& K_3\cdot (1+\mathbb{E}\|\xi\|^2)\cdot \max\limits_{0\leq k\leq n-1}\Delta t_k\notag\\
&&\quad +K_4\int\limits_0^t \sup\limits_{0\leq u\leq z}\mathbb{E}\|X(u)-\tilde X^E_n(u)\|^2 dz.
\end{eqnarray} 
Therefore, for all $t\in [0,T]$
\begin{eqnarray}
	\sup\limits_{0\leq s\leq t}\mathbb{E}\|X(s)-\tilde X^E_n(s)\|^2 &\leq & K_3\cdot (1+\mathbb{E}\|\xi\|^2)\cdot \max\limits_{0\leq k\leq n-1}\Delta t_k\notag\\
&&\quad +K_4\int\limits_0^t \sup\limits_{0\leq u\leq z}\mathbb{E}\|X(u)-\tilde X^E_n(u)\|^2 dz.
\end{eqnarray}
Note that the function
\begin{equation}
	[0,T]\ni t\to \sup\limits_{0\leq s\leq t}\mathbb{E}\|X(s)-\tilde X^E_n(s)\|^2 
\end{equation} 
is bounded (due to Lemmas \ref{mean_sq_reg_X}, \ref{err_euler_2}) and Borel measurable (as a nondecreasing function). Hence, by Gronwall's lemma \ref{GLCW} we get
\begin{equation}
\label{Euler_proc_est_1}
	\sup\limits_{0\leq t\leq T}\mathbb{E}\|X(t)-\tilde X^E_n(t)\|^2\leq K_5\cdot (1+\mathbb{E}\|\xi\|^2)\cdot \max\limits_{0\leq k\leq n-1}\Delta t_k,
\end{equation}
which completes the proof. \ \ \ $\blacksquare$ \\ \\
In order to have approximation of $X$ in between discretization points $(t_k)_{k=0,1,\ldots,n}$ we consider the following piecewise linear interpolation of the data $(t_k,X^E_n(t_k))_{k=0,1,\ldots}$
\begin{equation}
	\bar X^E_n(t)=\sum\limits_{k=0}^{n-1}\frac{(t-t_k)\cdot X^E_n(t_{k+1})+(t_{k+1}-t)\cdot X^E_n(t_k)}{t_{k+1}-t_k}\cdot\mathbf{1}_{[t_k,t_{k+1}]}(t), \ t\in [0,T].
\end{equation}
The process $\bar X^E_n=(\bar X^E_n(t))_{t\in [0,T]}$ has continuous trajectories, however, it is not adapted to the underlying filtration $(\Sigma_t)_{t\in [0,+\infty)}$. If needed, one can obtain an adapted approximation of $X$ by considering the following step process
\begin{equation}
	\hat X^E_n(t)=\sum\limits_{k=0}^{n-1}X^E_n(t_k)\cdot \mathbf{1}_{[t_k,t_{k+1})}(t), \  t\in [0,T),
\end{equation}
and $\hat X^E_n(T)=X^E_n(T)$. Of course $\hat X^E_n\in\mathbb{D}$. Such approximations, defined in the whole $[0,T]$, are useful when considering path-dependent option pricing, see Chapter 9.
\begin{thm}\label{err_euler_3}
Let us assume that $a,b,c$ satisfy $(B)$, $(C)$ and $\xi$ satisfies $(A0)$. Then there exists $C\in (0,+\infty)$ such that for all $n\in\mathbb{N}$ and $0=t_0<t_1<\ldots<t_n=T$ we have
\begin{equation}
\label{error_path_1}
	\Biggl(\mathbb{E}\int\limits_0^T\|X(t)-\bar X_n^E(t)\|^2 dt\Biggr)^{1/2}\leq C(1+(\mathbb{E}\|\xi\|^2)^{1/2})\max\limits_{0\leq k\leq n-1}(\Delta t_k)^{1/2}.
\end{equation}
and
\begin{equation}
\label{error_path_2}
	\Biggl(\mathbb{E}\int\limits_0^T\|X(t)-\hat X_n^E(t)\|^2 dt\Biggr)^{1/2}\leq C(1+(\mathbb{E}\|\xi\|^2)^{1/2})\max\limits_{0\leq k\leq n-1}(\Delta t_k)^{1/2}.
\end{equation}
\end{thm}
{\bf Proof.} We have for $t\in [t_k,t_{k+1}]$, $k=0,1,\ldots,n-1$, that
\begin{equation}
	X(t)-\bar X^E_n(t)=\frac{t_{k+1}-t}{t_{k+1}-t_k}\Bigl(X(t)-X^E_n(t_k)\Bigr)-\frac{t-t_k}{t_{k+1}-t_k}\Bigl(X^E_n(t_{k+1})-X(t)\Bigr).
\end{equation} 
Therefore, by Lemma \ref{mean_sq_reg_X} and \eqref{Euler_proc_est_1}
\begin{eqnarray}
	&&\mathbb{E}\int\limits_0^T\|X(t)-\bar X_n^E(t)\|^2 dt=\sum\limits_{k=0}^{n-1}\int\limits_{t_k}^{t_{k+1}}\mathbb{E}\|X(t)-\bar X^E_n(t)\|^2dt\notag\\
	&&\leq 4\sum\limits_{k=0}^{n-1}\int\limits_{t_k}^{t_{k+1}}\mathbb{E}\|X(t)-X(t_k)\|^2dt+4\sum\limits_{k=0}^{n-1}\int\limits_{t_k}^{t_{k+1}}\mathbb{E}\|X(t_{k+1})-X(t)\|^2dt\notag\\
	&&\quad +8T\sup\limits_{0\leq t\leq T}\mathbb{E}\|X(t)-\tilde X^E_n(t)\|^2\leq K(1+\mathbb{E}\|\xi\|^2)\max\limits_{0\leq k\leq n-1}\Delta t_k.
\end{eqnarray}	
By \eqref{Euler_proc_est_1} we have that
\begin{eqnarray}
	&&\mathbb{E}\int\limits_0^T\|X(t)-\hat X_n^E(t)\|^2 dt=\sum\limits_{k=0}^{n-1}\int\limits_{t_k}^{t_{k+1}}\mathbb{E}\|X(t)-X^E_n(t_k)\|^2dt\notag\\
	&&\leq 2\sum\limits_{k=0}^{n-1}\int\limits_{t_k}^{t_{k+1}}\mathbb{E}\|X(t)-X(t_k)\|^2dt+2\sum\limits_{k=0}^{n-1}\Delta t_k\cdot\mathbb{E}\|X(t_k)-X^E_n(t_k)\|^2\notag\\
	&&\leq C_1(1+\mathbb{E}\|\xi\|^2)\max\limits_{0\leq k\leq n-1}\Delta t_k+C_2\sup\limits_{0\leq t \leq T}\mathbb{E}\|X(t)-\tilde X^E_n(t)\|^2\notag\\
	&&\leq C_3(1+\mathbb{E}\|\xi\|^2)\max\limits_{0\leq k\leq n-1}\Delta t_k,
\end{eqnarray}
and this gives \eqref{error_path_2}. \ \ \ $\blacksquare$
\section{Optimality of the convergence rate $1/2$ for the Euler-Maruyama algorithm}
Let us consider the Euler-Maruyama algorithm $X^E_n$. In the case of equidistant discretization we show that the error bound $O(n^{-1/2})$ is sharp in a class of  coefficients $a,b,c,\xi$ satisfying the assumptions (B), (C) and (A0). 

Consider the geometric Brownian motion given by the SDE
\begin{equation}
    dX(t)=\sigma X(t) dW(t), \quad t\in [0,T],
\end{equation}
and with the initial-value $X(0)=1$, $\sigma>0$. The coefficients $a=c=0$, $b(t,x)=\sigma x$, $\xi=1$ satisfy the assumptions (B), (C) and (A0). Moreover, $X(t)=\exp\Bigl(-\frac{1}{2}\sigma^2 t+\sigma W(t)\Bigr)$ and $\mathbb{E}(X(t))^2=\exp(\sigma^2 t)\geq 1$ for all $t\in [0,T]$.  By \eqref{euler_proc_1}, \eqref{sol_proc_1} and It\^o isometry we have for all $n\in\mathbb{N}$ and any discretization $0=t_0<t_1<\ldots< t_n=T$ that
\begin{eqnarray}
\label{GBM_lower_b_1}
    &&\mathbb{E}|X(T)-X_n^E(T)|^2=\sigma^2\cdot\mathbb{E}\Biggl|\int\limits_0^T \Bigl[\sum\limits_{k=0}^{n-1}(X(s)-X^E_n(t_k))\cdot\mathbf{1}_{(t_k,t_{k+1}]}(s)\Bigr]dW(s)\Biggl|^2\notag\\
    &&=\sigma^2\cdot\mathbb{E}\int\limits_0^T\Bigl[\sum\limits_{k=0}^{n-1}(X(s)-X^E_n(t_k))\cdot\mathbf{1}_{(t_k,t_{k+1}]}(s)\Bigr]^2ds\notag\\
    &&=\sigma^2\cdot\mathbb{E}\int\limits_0^T\sum\limits_{k=0}^{n-1}(X(s)-X^E_n(t_k))^2\cdot\mathbf{1}_{(t_k,t_{k+1}]}(s)ds=\sigma^2\sum\limits_{k=0}^{n-1}\int\limits_{t_k}^{t_{k+1}}\mathbb{E}(X(s)-X^E_n(t_k))^2 ds.
\end{eqnarray}
Since $\displaystyle{X(s)=X(t_k)+\sigma\int\limits_{t_k}^s X(u)dW(u)}$ for $s\in [t_k,t_{k+1}]$, we have that
\begin{eqnarray}
\label{GBM_lower_b_2}
    &&\mathbb{E}(X(s)-X^E_n(t_k))^2=\mathbb{E}\Bigl[(X(t_k)-X^E_n(t_k))+\sigma\int\limits_{t_k}^sX(u)dW(u)\Bigr]^2\notag\\
    &&=\mathbb{E}(X(t_k)-X^E_n(t_k))^2+\sigma^2\mathbb{E}\Bigl(\int\limits_{t_k}^sX(u)dW(u)\Bigr)^2\notag\\
    &&\quad\quad+2\sigma\mathbb{E}\Bigl((X(t_k)-X^E_n(t_k))\cdot \int\limits_{t_k}^sX(u)dW(u)\Bigr)\notag\\
    &&\geq \sigma^2\mathbb{E}\Bigl(\int\limits_{t_k}^sX(u)dW(u)\Bigr)^2+2\sigma\mathbb{E}\Bigl((X(t_k)-X^E_n(t_k))\cdot \int\limits_{t_k}^sX(u)dW(u)\Bigr),
\end{eqnarray}
where 
\begin{equation}
    \label{GBM_lower_b_3}
    \mathbb{E}\Bigl(\int\limits_{t_k}^sX(u)dW(u)\Bigr)^2=\int\limits_{t_k}^s \mathbb{E}(X(u))^2du\geq s-t_k,
\end{equation}
\begin{eqnarray}
\label{GBM_lower_b_4}
    &&\mathbb{E}\Bigl((X(t_k)-X^E_n(t_k))\cdot \int\limits_{t_k}^sX(u)dW(u)\Bigr)\notag\\
    &&=\mathbb{E}\Bigl[(X(t_k)-X^E_n(t_k))\cdot \mathbb{E}\Bigl(\int\limits_{t_k}^sX(u)dW(u) \Bigl | \Sigma_{t_k}\Bigr)\Bigr]=0.
\end{eqnarray}
Therefore, by \eqref{GBM_lower_b_1}-\eqref{GBM_lower_b_4} and Jensen inequality we get that
\begin{equation}
    \mathbb{E}|X(T)-X_n^E(T)|^2\geq \frac{\sigma^4}{2}\sum\limits_{k=0}^{n-1}(\Delta t_k)^2\geq\frac{\sigma^4T^2}{2}n^{-1},
\end{equation}
which in turn implies the lower error bound $\Omega(n^{-1/2})$. Hence, under the assumptions of Theorem \ref{err_euler_1} the $L^2(\Omega)$-error $O(n^{-1/2})$ of the Euler-Maruyama algorithm, based on equidistant mesh, is sharp.
\section{Implementation issues}
Draw $\bar X^E_n(0)$ accordingly to the distribution of $\xi$ 
and for $k=0,1,\ldots,n-1$
\begin{equation}
	\bar X^E_n(t_{k+1})=\bar X^E_n(t_{k})+a(\bar U_{k,n}^E)\cdot\Delta t_k+(\Delta t_k)^{1/2}\cdot b(\bar U_{k,n}^E)\cdot Z_k+c(\bar U_{k,n}^E)\cdot U_k,
\end{equation}
where $\bar U_{k,n}^E=(t_k,\bar X^E_n(t_{k}))$ and
\begin{eqnarray}
	Z_k=[Z_k^1,\ldots,Z_k^{m_W}]^T,\\
	U_k=[U_k^1,\ldots,U_k^{m_N}]^T,
\end{eqnarray}
where 
\begin{itemize}
	\item $(Z^j_k)_{k=0,\ldots,n-1, j=1,\ldots,m_W}$ is an i.i.d sequence of random variables with $Z^1_1\sim N(0,1)$, 
	\item $(U_k^j)_{k=0,\ldots,n-1, j=1,\ldots,m_N}$ is a sequence of independent random variables, where $\displaystyle{U_k^j\sim \hbox{Poiss}\Bigl(\int\limits_{t_k}^{t_{k+1}}\lambda_j(s)ds\Bigr)}$, 
	\item the collection $\displaystyle{\{\xi,(Z^j_k)_{k=0,\ldots,n-1, j=1,\ldots,m_W},(U_k^j)_{k=0,\ldots,n-1, j=1,\ldots,m_N}\}}$ consists of independent random variables.
\end{itemize}
Then
\begin{equation}
	\bar X^E_{n,i}(0)=\xi_{i} \ (\hbox{in law}),
\end{equation}
\begin{eqnarray}
	\bar X^E_{n,i}(t_{k+1})=\bar X^E_{n,i}(t_{k})&+&a_i(\bar U_{k,n}^E)\cdot\Delta t_k\notag\\
	&+&(\Delta t_k)^{1/2}\cdot\sum\limits_{j=1}^{m_W}b_{ij}(\bar U_{k,n}^E)\cdot Z^j_k\notag\\
	&+&\sum\limits_{j=1}^{m_N}c_{ij}(\bar U_{k,n}^E)\cdot U^j_k,
\end{eqnarray}
for $k=0,1,\ldots,n-1$, $i=1,2,\ldots,d$.\\
\section{Exercises}
\begin{itemize}
	\item [1.] Let $U:\Omega\to\mathbb{R}^{d\times m}$ be $\Sigma_{\alpha}$-measurable where $0\leq \alpha<\beta$. Moreover, let $W$ and $N$ be $m$-dimensional Wiener process and $m$-dimensional non-homogeneous Poisson process, respectively. Show that for $Z\in \{N,W\}$
	\begin{equation}
		\int\limits_{\alpha}^{\beta} U dZ(t)=U\cdot (Z(\beta)-Z(\alpha)). 
	\end{equation}
	\item [2.] Show that there exist $c_W,c_N$ such that for all $0\leq \alpha\leq\beta$ we have
	\begin{eqnarray}
		&&\mathbb{E}\|W(\beta)-W(\alpha)\|^2\leq c_W \cdot (\beta-\alpha),\notag\\
		&&\mathbb{E}\|N(\beta)-N(\alpha)\|^2\leq c_N \cdot (\beta-\alpha).
	\end{eqnarray}
	\item [3.] For the Euler process
\begin{itemize}
	\item [(i)] show \eqref{interp_E_proc}.
	\item [(ii)] give a proof that $\tilde X^E_n=(\tilde X^E_n(t))_{t\in [0,T]}$  is adapted to $(\Sigma_t)_{t\in [0,+\infty)}$.
	\item [(iii)] show that
	\begin{equation}
	\label{xe_ind_1}
		X^E_n(t_k)=\eta+\sum\limits_{j=0}^{k-1}a(U_{j,n}^E)\cdot (t_{j+1}-t_j)+\sum\limits_{j=0}^{k-1}b(U_{j,n}^E)\cdot\Delta W_j+\sum\limits_{j=0}^{k-1}c(U_{j,n}^E)\cdot\Delta N_j
	\end{equation}
for $k=0,1,\ldots,n$. (We use the convention that $\displaystyle{\sum\limits_{j=0}^{-1} c_j:=0}$.) 
	\item [(iv)] give a proof that for all $t\in [0,T]$ we have
\begin{equation}
\label{euler_proc_11}
	\tilde X^E_n(t)=\eta+\int\limits_0^t \tilde a_n(s)ds+\int\limits_0^t \tilde b_n(s)dW(s)+\int\limits_0^t \tilde c_n(s)dN(s),
\end{equation}
where for $f\in \{a,b,c\}$
\begin{equation}
	\tilde f_n(t) = \sum\limits_{k=0}^{n-1}f(U_{k,n}^E)\cdot\mathbf{1}_{(t_k,t_{k+1}]}(t).
\end{equation}
	\item [4.] Give a proof of Fact \ref{f_np}.
\end{itemize}
\end{itemize}
\chapter{Randomized Euler scheme for jump-diffusion SDEs}
For the functions
\begin{eqnarray}
	&&a:[0,T]\times\mathbb{R}^d\to \mathbb{R}^d,\\
	&&b:[0,T]\times\mathbb{R}^d\to \mathbb{R}^{d\times m_W},\\
	&&c:[0,T]\times\mathbb{R}^d\to \mathbb{R}^{d\times m_N},
\end{eqnarray}
we assume that
\begin{itemize}
	\item [(B)] there exists $K\in (0,+\infty)$ such that for $f\in\{b,c\}$ and for all $s,t\in [0,T]$, $x,y\in\mathbb{R}^d$
	\begin{equation}
		\|f(t,x)-f(s,y)\|\leq K(|t-s|+\|x-y\|),
	\end{equation}
	\item [(C)] the function $a$ is Borel measurable and 
	there exists $L\in (0,+\infty)$ such that for all $t\in [0,T]$, $x,y\in\mathbb{R}^d$
	\begin{equation}
		\|a(t,x)-a(t,y)\|\leq L\|x-y\|,
	\end{equation}
	\item [(D)] there exists $D\in (0,+\infty)$ such that for all $t\in [0,T]$
	\begin{equation}
		\|a(t,0)\|\leq D.
	\end{equation}
	\item [(E)] the intensity functions $\lambda_j:[0,T]\to (0,+\infty)$, $j=1,2,\ldots,m_N$, are Borel measurable and bounded.
\end{itemize}
For $f=a$ the norm $\|\cdot\|$ is the euclidean norm $\|\cdot\|_2$, while for $f\in\{b,c\}$ the norm $\|\cdot\|$ is the Frobenius norm $\|\cdot\|_F$. From now we will denote the both norms by $\|\cdot\|$ and the notion will be clear from the context.

Note that $(C)$ and $(D)$ imply that for all $(t,y)\in [0,T]\times\mathbb{R}^d$
\begin{equation}
\label{lin_g_a}
	\|a(t,y)\|\leq \bar D (1+\|y\|),
\end{equation}
where $\bar D=\max\{L,D\}$.

The randomized Euler scheme is defined as follows. 

Let $n\in\mathbb{N}$ and $t_k=kT/n$, $k=0,1,\ldots,n$, be a uniform discretization of $[0,T]$. Let $(\tau_j)_{j\in\mathbb{N}_0}$ be a sequence of independent random variables on $(\Omega,\Sigma,\mathbb{P})$ that are identically uniformly distributed on $[0,1]$, and   such that the $\sigma$-fields $\Sigma_{\infty}$ and $\sigma\Bigl(\bigcup_{j\geq 0}\sigma(\tau_j)\Bigr)$ are independent. The classical Euler scheme for the jump-diffusion SDE \eqref{SDE_PROBLEM} is defined as follows:
\begin{equation}
	X^{RE}_n(0)=\xi,
\end{equation}
and for $k=0,1,\ldots,n-1$
\begin{equation}
	X^{RE}_n(t_{k+1})=X^{RE}_n(t_{k})+a(\theta_k,X^{RE}_n(t_{k}))\cdot\Delta t_k+b(U_{k,n}^{RE})\cdot\Delta W_k+c(U_{k,n}^{RE})\cdot\Delta N_k,
\end{equation}
where
\begin{eqnarray} 
	&&\theta_k=t_k+\tau_k\cdot \Delta t_k,\\
	&&U_{k,n}^{RE}=(t_k,X^{RE}_n(t_{k})),
\end{eqnarray}	
\begin{eqnarray}
	&&\Delta t_k = t_{k+1}-t_k=\frac{T}{n},\notag\\
	&&\Delta W_k=[\Delta W^1_k,\ldots,\Delta W^{m_W}_k]^T,\notag\\
	&&\Delta N_k=[\Delta N^1_k,\ldots,\Delta N^{m_N}_k]^T,
\end{eqnarray}
and
\begin{equation}
	\Delta Z_k^j=Z^j(t_{k+1})-Z^j(t_{k}), \ Z\in\{N,W\}.
\end{equation}
Since
\begin{equation}
	b(U_{k,n}^{RE})\cdot\Delta W_k=\Bigl(\sum\limits_{j=1}^{m_W}b_{ij}(U_{k,n}^{RE})\cdot\Delta W^j_k\Bigr)_{i=1,2,\ldots,d},
\end{equation}
\begin{equation}
	c(U_{k,n}^{RE})\cdot\Delta N_k=\Bigl(\sum\limits_{j=1}^{m_N}c_{ij}(U_{k,n}^{RE})\cdot\Delta N^j_k\Bigr)_{i=1,2,\ldots,d},
\end{equation}
we can write each component of the randomized Euler scheme as follows
\begin{equation}
	X^{RE}_{n,i}(0)=\xi_{i},
\end{equation}
\begin{eqnarray}
	X^{RE}_{n,i}(t_{k+1})=X^{RE}_{n,i}(t_{k})&+&a_i(\theta_k,X^{RE}_n(t_{k}))\cdot\Delta t_k\notag\\
	&+&\sum\limits_{j=1}^{m_W}b_{ij}(U_{k,n}^{RE})\cdot\Delta W^j_k\notag\\
	&+&\sum\limits_{j=1}^{m_N}c_{ij}(U_{k,n}^{RE})\cdot\Delta N^j_k,
\end{eqnarray}
for $k=0,1,\ldots,n-1$, $i=1,2,\ldots,d$.
The aim of this chapter is to prove the following result.
\begin{thm}
\label{err_euler_11}
Let us assume that $a,b,c$, $\{\lambda_j\}_{j=1,\ldots,m_N}$ satisfy $(B), (C), (D), (E)$ and $\xi$ satisfies $(A0)$. Then there exists $C\in (0,+\infty)$ such that for all $n\in\mathbb{N}$  we have
\begin{equation}
	\max\limits_{1\leq k\leq n}\Bigl(\mathbb{E}\|X(t_k)-X_n^{RE}(t_k)\|^2\Bigr)^{1/2}\leq C\cdot \Bigl(1+(\mathbb{E}\|\xi\|^2)^{1/2}\Bigr)\cdot n^{-1/2}.
\end{equation}
\end{thm}

For the error analysis of the randomized Euler scheme we need to consider the following extended filtration $(\tilde\Sigma^n_t)_{t\geq 0}$, defined as
\begin{equation}
	\tilde\Sigma^n_t=\Sigma_t\vee\sigma(\tau_0,\tau_1,\ldots,\tau_{n-1}), \quad t\geq 0.
\end{equation}
Note that $W$ is also a Wiener process with respect to $(\tilde\Sigma^n_t)_{t\geq 0}$, while $N$ is again Poisson process also with respect to $(\tilde\Sigma^n_t)_{t\geq 0}$. Hence, we can consider stochastic  integrals of processes from $\mathbb{L}$ that are  adapted to  $(\tilde\Sigma^n_t)_{t\geq 0}$.
Before we prove this fact we need several auxiliary lemmas concerning the so called {\it time-continuous Euler approximation} $\tilde X^{RE}_n=(\tilde X^{RE}_n(t))_{t\in [0,T]}$. It is defined as follows:
\begin{equation}
	\tilde X^{RE}_n(0)=\xi,
\end{equation}
and for $t\in (t_k,t_{k+1}]$, $k=0,1,\ldots,n-1$ we set
\begin{equation}
\label{reuler_proc_def1}
	\tilde X^{RE}_n(t)=X^{RE}_n(t_k)+a(\theta_k,X^{RE}_n(t_{k}))\cdot (t-t_k)+b(U_{k,n}^{RE})\cdot (W(t)-W(t_k))+c(U_{k,n}^{RE})\cdot (N(t)-N(t_k)).
\end{equation}
The approximation $\tilde X^{RE}_n$ is not implementable, since it requires the complete knowledge of trajectories of $N$ and $W$. 

By induction we have that
\begin{equation}
\label{interp_E_proc1}
	\tilde X^{RE}_n(t_k)=X^{RE}_n(t_k), \ k=0,1,\ldots,n.
\end{equation}
Since the trajectories of $[N,W]$ are right-continuous we get 
\begin{equation}
	\lim\limits_{t\to t_k+}\tilde X^{RE}_n(t)=X^{RE}_n(t_k)=\tilde X^{RE}_n(t_k),
\end{equation}
and
\begin{eqnarray}
	&&\lim\limits_{t\to t_{k+1}-}\tilde X^{RE}_n(t)=X^{RE}_n(t_k)+a(\theta_k,X^{RE}_n(t_{k}))\cdot (t_{k+1}-t_k)+b(U_{k,n}^{RE})\cdot (W(t_{k+1})-W(t_k))\notag\\
&&\quad\quad +c(U_{k,n}^{RE})\cdot (N(t_{k+1}-)-N(t_k))\neq \tilde X^{RE}_n(t_{k+1}) \ (=X^{RE}_n(t_{k+1}))
\end{eqnarray}
for $k=0,1,\ldots,n-1$. Moreover, it is easy to see that $\tilde X^{RE}_n=(\tilde X^{RE}_n(t))_{t\in [0,T]}$ is adapted to $(\tilde\Sigma^n_t)_{t\in [0,+\infty)}$ and has c\`adl\`ag paths. 
\begin{lem}
\label{err_euler_212}
Let us assume that $a,b,c$, $\{\lambda_j\}_{j=1,\ldots,m_N}$ satisfy $(B), (C), (D), (E)$ and $\xi$ satisfies $(A0)$. Then there exists $C\in (0,+\infty)$ such that for all $n\in\mathbb{N}$ we have
\begin{equation}
\label{err_euler_211}
	\sup\limits_{t\in [0,T]}\mathbb{E}\|\tilde X^{RE}_n(t)\|^2\leq C(1+\mathbb{E}\|\xi\|^2).
\end{equation}
\end{lem}
{\bf Proof.} Firstly we show that 
\begin{equation}
\label{bound_mom_E_112}
	\max\limits_{0\leq k \leq n}\mathbb{E}\|X^{RE}_n(t_k)\|^2<+\infty.
\end{equation}
We proceed by induction and let us assume that there exists $s\in \{0,1,\ldots,n-1\}$ with the property that 
\begin{equation}
	\max\limits_{0\leq k\leq s}\mathbb{E}\|X^{RE}_n(t_k)\|^2<+\infty.
\end{equation}
(This is of course satisfied for $s=0$.) Since $\{\sigma(b(U_{s,n}^{RE})),\sigma(\Delta W_s)\}$ and $\{\sigma(c(U_{s,n}^{RE})),\sigma(\Delta N_s)\}$ are two families of independent $\sigma$-fields, we get by the Fact \ref{f_np}
\begin{eqnarray}
	\mathbb{E}\|X^{RE}_n(t_{s+1})\|^2 &\leq& C_1\Bigl(\mathbb{E}\|X^{RE}_n(t_s)\|^2+(\Delta t_s)^2\cdot\mathbb{E}\|a(\theta_s,X^{RE}_n(t_{s}))\|^2\notag\\
&&\quad +\mathbb{E}\|b(U_{s,n}^{RE})\|^2 \cdot\mathbb{E}\|\Delta W_s\|^2+\mathbb{E}\|c(U_{s,n}^{RE})\|^2 \cdot\mathbb{E}\|\Delta N_s\|^2\Bigr).
\end{eqnarray}
Since there exist $c_W,c_N\in (0,+\infty)$ such that for all $t\in [t_k,t_{k+1}]$, $k=0,1,\ldots,n-1$ it holds
\begin{eqnarray}
	&&\mathbb{E}\|W(t)-W(t_k)\|^2\leq c_W (t-t_k),\notag\\
	&&\mathbb{E}\|N(t)-N(t_k)\|^2\leq c_N (t-t_k),
\end{eqnarray}
we get by Fact \ref{f_lin_gr} that
\begin{eqnarray}
	\mathbb{E}\|X^{RE}_n(t_{s+1})\|^2 &\leq& C_2\Bigl(\mathbb{E}\|X^{RE}_n(t_s)\|^2+(\Delta t_s)^2\cdot (1+\mathbb{E}\|X^{RE}_n(t_s)\|^2)\notag\\
&&\quad +(1+\mathbb{E}\|X^{RE}_n(t_s)\|^2) \cdot \Delta t_s+(1+\mathbb{E}\|X^{RE}_n(t_s)\|^2)\cdot\Delta t_s\Bigr)<+\infty.
\end{eqnarray}
Hence,
\begin{equation}
	\max\limits_{0\leq k\leq s+1}\mathbb{E}\|X^{RE}_n(t_k)\|^2<+\infty.
\end{equation}
By the rules of induction, this completes the proof of \eqref{bound_mom_E_112}. 

From \eqref{reuler_proc_def1} it follows that for all 
$t\in [t_k,t_{k+1}]$, $k=0,1,\ldots,n-1$ 
\begin{eqnarray}
	\mathbb{E}\|\tilde X^{RE}_n(t)\|^2 &\leq& C_2\Bigl(\mathbb{E}\|X^{RE}_n(t_k)\|^2+ (t-t_k)^2\cdot (1+\mathbb{E}\|X^{RE}_n(t_k)\|^2)\notag\\
&&\quad +(1+\mathbb{E}\|X^{RE}_n(t_k)\|^2) \cdot (t-t_k)+(1+\mathbb{E}\|X^{RE}_n(t_k)\|^2)\cdot (t-t_k)\Bigr)\notag\\
&&\leq C_3(1+\mathbb{E}\|X^{RE}_n(t_k)\|^2),
\end{eqnarray}
and by \eqref{bound_mom_E_112}
\begin{equation}
	\label{bound_mom_E_21}
	\sup\limits_{0\leq t\leq T}\mathbb{E}\|\tilde X^{RE}_n(t)\|^2=\max\limits_{0\leq k\leq n-1}\sup\limits_{t_k\leq t\leq t_{k+1}}\mathbb{E}\|\tilde X^{RE}_n(t)\|^2\leq C_3(1+\max\limits_{0\leq k\leq n-1}\mathbb{E}\|X^{RE}_n(t_k)\|^2)<+\infty.
\end{equation}
Note that for now the bound in \eqref{bound_mom_E_21} depends on $n$. In the second part of the proof we will show, with the help of Gronwall's lemma, that we can obtain the bound \eqref{err_euler_211}, with $C$ that is independent of $n$.

We can write for all $t\in [0,T]$ that
\begin{equation}
\label{euler_proc_12}
	\tilde X^{RE}_n(t)=\xi+\int\limits_0^t \tilde a_n(s)ds+\int\limits_0^t \tilde b_n(s)dW(s)+\int\limits_0^t \tilde c_n(s)dN(s),
\end{equation}
where 
\begin{equation}
	\tilde a_n(s) = \sum\limits_{k=0}^{n-1}a(\theta_k, \tilde X^{RE}_n(t_{k}))\cdot\mathbf{1}_{(t_k,t_{k+1}]}(s)
\end{equation}
and for $f\in \{b,c\}$
\begin{equation}
	\tilde f_n(s) = \sum\limits_{k=0}^{n-1}f(U_{k,n}^{RE})\cdot\mathbf{1}_{(t_k,t_{k+1}]}(s),
\end{equation}
and recall that $U_{k,n}^{RE}=(t_k, X^{RE}_n(t_{k}))=(t_k, \tilde X^{RE}_n(t_{k}))$. Note that $(\tilde f_n(t))_{t\in [0,T]}\in\mathbb{L}$ for $f\in\{a,b,c\}$.  Furthermore, for $f\in \{b,c\}$ and all $t\in [0,T]$ it holds
\begin{equation}
	\|\tilde f_n(t)\|^2\leq \Bigl(\sum\limits_{k=0}^{n-1}\|f(U_{k,n}^{RE})\|\cdot\mathbf{1}_{(t_k,t_{k+1}]}(t)\Bigr)^2=\sum\limits_{k=0}^{n-1}\|f(U_{k,n}^{RE})\|^2\cdot\mathbf{1}_{(t_k,t_{k+1}]}(t),
\end{equation}
which, together with Fact \ref{f_lin_gr}, implies
\begin{eqnarray}
\label{est_aux_11}
	&&\int\limits_0^t \|\tilde f_n(s)\|^2 ds\leq \int\limits_0^t \sum\limits_{k=0}^{n-1}\|f(U_{k,n}^{RE})\|^2\cdot\mathbf{1}_{(t_k,t_{k+1}]}(s)ds\notag\\
	&&\leq K_1^2 \int\limits_0^t \sum\limits_{k=0}^{n-1}(1+\|\tilde X^{RE}_n(t_{k})\|)^2\cdot\mathbf{1}_{(t_k,t_{k+1}]}(s)ds\notag\\
	&&\leq 2K_1^2 \int\limits_0^t \sum\limits_{k=0}^{n-1}(1+\|\tilde X^{RE}_n(t_{k})\|^2)\cdot\mathbf{1}_{(t_k,t_{k+1}]}(s)ds\notag\\
	&&\leq 2K_1^2\int\limits_0^T \sum\limits_{k=0}^{n-1}\mathbf{1}_{(t_k,t_{k+1}]}(s)ds +2K_1^2\int\limits_0^t \sum\limits_{k=0}^{n-1}\|\tilde X^{RE}_n(t_{k})\|^2\cdot\mathbf{1}_{(t_k,t_{k+1}]}(s)ds\notag\\
	&&=2K_1^2T+2K_1^2\int\limits_0^t \sum\limits_{k=0}^{n-1}\|\tilde X^{RE}_n(t_{k})\|^2\cdot\mathbf{1}_{(t_k,t_{k+1}]}(s)ds.
\end{eqnarray}
Analogously we get by \eqref{lin_g_a}
\begin{equation}
\label{t_a_est_1}
	\int\limits_0^t \|\tilde a_n(s)\|^2 ds\leq 2\bar D^2T+2\bar D^2\int\limits_0^t \sum\limits_{k=0}^{n-1}\|\tilde X^{RE}_n(t_{k})\|^2\cdot\mathbf{1}_{(t_k,t_{k+1}]}(s)ds.
\end{equation}
By the fact that $N(t)=\tilde N(t)+m(t)$, we have 
\begin{eqnarray}
	&&\mathbb{E}\|\tilde X^{RE}_n(t)\|^2\leq C_1\Bigl(\mathbb{E}\|\xi\|^2+\mathbb{E}\Bigl\|\int\limits_0^t \tilde a_n(s)ds\Bigl\|^2+\mathbb{E}\Bigl\|\int\limits_0^t \tilde b_n(s)dW(s)\Bigl\|^2\notag\\
&&\quad +\mathbb{E}\Bigl\|\int\limits_0^t \tilde c_n(s)d\tilde N(s)\Bigl\|^2+\mathbb{E}\Bigl\|\int\limits_0^t \tilde c_n(s)dm(s)\Bigl\|^2\Bigr).
\end{eqnarray}
From the H\"older inequality, It\^o isometry for stochastic integrals driven by martingales, and  \eqref{t_a_est_1} we get for $t\in [0,T]$
\begin{eqnarray}
	&&\mathbb{E}\Bigl\|\int\limits_0^t \tilde a_n(s)ds\Bigl\|^2\leq \mathbb{E}\Bigl(\int\limits_0^t \|\tilde a_n(s)\|ds\Bigr)^2\leq T\mathbb{E}\int\limits_0^t\|\tilde a_n(s)\|^2ds\notag\\
	&&\leq K_2+K_3\int\limits_0^t \sum\limits_{k=0}^{n-1}\|\tilde X^{RE}_n(t_{k})\|^2\cdot\mathbf{1}_{(t_k,t_{k+1}]}(s)ds,
\end{eqnarray}
\begin{equation}
	\mathbb{E}\Bigl\|\int\limits_0^t \tilde b_n(s)dW(s)\Bigl\|^2= \mathbb{E}\int\limits_0^t \|\tilde b_n(s)\|^2ds\leq K_4+K_5\int\limits_0^t \sum\limits_{k=0}^{n-1}\|\tilde X^{RE}_n(t_{k})\|^2\cdot\mathbf{1}_{(t_k,t_{k+1}]}(s)ds,
\end{equation}
\begin{eqnarray}
	&&\mathbb{E}\Bigl\|\int\limits_0^t \tilde c_n(s)d\tilde N(s)\Bigl\|^2=\sum\limits_{j=1}^{m_N}\mathbb{E}\int\limits_0^t\|\tilde c^{(j)}_n(s)\|^2\cdot \lambda_j(s)ds\leq\max\limits_{1\leq j\leq m_N}\|\lambda_j\|_{\infty}\cdot\mathbb{E}\int\limits_0^t\|\tilde c_n(s)\|^2ds\notag\\
	&&\leq K_6+K_7\int\limits_0^t \sum\limits_{k=0}^{n-1}\|\tilde X^{RE}_n(t_{k})\|^2\cdot\mathbf{1}_{(t_k,t_{k+1}]}(s)ds,
\end{eqnarray}
and
\begin{eqnarray}
	&&\mathbb{E}\Bigl\|\int\limits_0^t \tilde c_n(s)dm(s)\Bigl\|^2=\sum\limits_{i=1}^d\mathbb{E}\Bigl|\sum\limits_{j=1}^{m_N}\int\limits_0^t \tilde c_{n,ij}(s)\cdot\lambda_j(s)ds\Bigl|^2\notag\\
	&\leq & m_N\cdot\sum\limits_{i=1}^d\sum\limits_{j=1}^{m_N}\mathbb{E}\Bigl|\int\limits_0^t \tilde c_{n,ij}(s)\cdot\lambda_j(s)ds\Bigl|^2\notag\\
	&\leq & m_N T\cdot\sum\limits_{i=1}^d\sum\limits_{j=1}^{m_N}\int\limits_0^t\mathbb{E}|\tilde c_{n,ij}(s)|^2\cdot\lambda^2_j(s) ds\notag\\
	&\leq &m_N T\cdot\max\limits_{1\leq j\leq m_N}\|\lambda_j\|^2_{\infty}\cdot\mathbb{E}\int\limits_0^t\|\tilde c_n(s)\|^2 ds\notag\\
	&&\leq K_8+K_9\int\limits_0^t \sum\limits_{k=0}^{n-1}\|\tilde X^{RE}_n(t_{k})\|^2\cdot\mathbf{1}_{(t_k,t_{k+1}]}(s)ds.
\end{eqnarray}
Hence
\begin{eqnarray}
	\mathbb{E}\|\tilde X^{RE}_n(t)\|^2&\leq& K_{10}(1+\mathbb{E}\|\xi\|^2)+K_{11}\int\limits_0^t \sum\limits_{k=0}^{n-1}\mathbb{E}\|\tilde X^{RE}_n(t_{k})\|^2\cdot\mathbf{1}_{(t_k,t_{k+1}]}(s)ds\notag\\
	&\leq& K_{10}(1+\mathbb{E}\|\xi\|^2)+K_{11}\int\limits_0^t\sup\limits_{0\leq u\leq z}\mathbb{E}\|\tilde X^{RE}_n(u)\|^2dz 
\end{eqnarray}
where $K_{10},K_{11}$ depend only on $K_1$, $T$, $m_N$, $\max\limits_{1\leq j\leq m_N}\|\lambda_j\|_{\infty}$. Therefore, for all $t \in [0,T]$
\begin{equation}
	\sup\limits_{0\leq s \leq t}\mathbb{E}\|\tilde X^{RE}_n(s)\|^2\leq  K_{10}(1+\mathbb{E}\|\xi\|^2)+K_{11}\int\limits_0^t\sup\limits_{0\leq u\leq z}\mathbb{E}\|\tilde X^{RE}_n(u)\|^2dz .
\end{equation}
Note that the function
\begin{equation}
	[0,T]\ni t \to \sup\limits_{0\leq s \leq t}\mathbb{E}\|\tilde X^{RE}_n(s)\|^2
\end{equation}
is Borel measurable (as a nondecreasing function) and bounded, which follows from   \eqref{bound_mom_E_21}. By using Lemma \ref{GLCW} we get the thesis.\ \ \ $\blacksquare$
\\ \\
{\bf Proof of Theorem \ref{err_euler_1}.}  We have that for all $t\in [0,T]$
\begin{equation}
\label{sol_proc_11}	
	X(t)=\xi+\int\limits_0^t \hat a(s)ds+\int\limits_0^t \hat b(s)dW(s)+\int\limits_0^t \hat c(s)dN(s),
\end{equation}
where for $f\in\{a,b,c\}$
\begin{equation}
	\hat f(s)=\sum\limits_{k=0}^{n-1}f(s,X(s-))\cdot\mathbf{1}_{(t_k,t_{k+1}]}(s).
\end{equation}
(Note that for $f\in\{a,b,c\}$ the process  $(\hat f(t))_{t\in [0,T]}$ is in $\mathbb{L}$ so all Lebesgue and stochastic integrals appearing in  \eqref{sol_proc_11} are well-defined.) From the representations above we thus get
\begin{equation}
\label{err_est_11}
	\mathbb{E}\|X(t)-\tilde X^{RE}_n(t)\|^2\leq 3\Bigl(A_{1,n}(t)+A_{2,n}(t)+A_{3,n}(t)\Bigr)
\end{equation}
where
\begin{eqnarray}
	&&A_{1,n}(t)=\mathbb{E}\Bigl\|\int\limits_0^t\Bigl(\hat a(s)-\tilde a_n(s)\Bigr)ds\Bigl\|^2,\notag\\
	&&A_{2,n}(t)=\mathbb{E}\Bigl\|\int\limits_0^t\Bigl(\hat b(s)-\tilde b_n(s)\Bigr)dW(s)\Bigl\|^2,\notag\\
	&&A_{3,n}(t)=\mathbb{E}\Bigl\|\int\limits_0^t\Bigl(\hat c(s)-\tilde c_n(s)\Bigr)dN(s)\Bigl\|^2.
\end{eqnarray}
We have that
\begin{equation}
\label{err_est_21}
	A_{1,n}(t)\leq 3\Bigl(\mathbb{E}\|\tilde A^{RE}_{1,n}(t)\|^2+\mathbb{E}\|\tilde A^{RE}_{2,n}(t)\|^2+\mathbb{E}\|\tilde A^{RE}_{3,n}(t)\|^2\Bigr),
\end{equation}
where
\begin{eqnarray}
	\label{DEF_ARE_1}
	&&\mathbb{E}\|\tilde A^{RE}_{1,n}(t)\|^2=
\mathbb{E}\Biggl\|\int\limits_0^t\sum_{k=0}^{n-1}\Bigl(a(s,X(s-))-a(s,X(t_k))\Bigr)\mathbf{1}_{(t_k,t_{k+1}]}(s)ds\Biggl\|^2,\\
	&&\mathbb{E}\|\tilde A^{RE}_{2,n}(t)\|^2=\mathbb{E}\Biggl\|\int\limits_0^t\sum_{k=0}^{n-1}\Bigl(a(s,X(t_k))-a(\theta_k,X(t_k))\Bigr)\mathbf{1}_{(t_k,t_{k+1}]}(s)ds\Biggl\|^2,\\
	&&\mathbb{E}\|\tilde A^{RE}_{3,n}(t)\|^2=\mathbb{E}\Biggl\|\int\limits_0^t\sum_{k=0}^{n-1}\Bigl(a(\theta_k,X(t_k))-a(\theta_k,\tilde X^{RE}_n(t_k))\Bigr)\mathbf{1}_{(t_k,t_{k+1}]}(s)ds\Biggl\|^2.
\end{eqnarray}
From Lemma \ref{mean_sq_reg_X} and the H\"older inequality we get for all $t\in [0,T]$ that
\begin{eqnarray}
	\label{M1_EST_1}
		&&\mathbb{E}\|\tilde A^{RE}_{1,n}(t)\|^2\leq\mathbb{E}\Biggl(\int\limits_0^t\sum_{k=0}^{n-1}\|a(s,X(s-))-a(s,X(t_k))\|\cdot\mathbf{1}_{(t_k,t_{k+1}]}(s)ds\Biggr)^2\notag\\		
 &&\leq tL^2\mathbb{E}\int\limits_0^t\Biggl(\sum_{k=0}^{n-1}\|X(s-)-X(t_k)\|\cdot\mathbf{1}_{(t_k,t_{k+1}]}(s)\Biggr)^2ds\notag\\	
 &&\leq TL^2\mathbb{E}\int\limits_0^T	\sum_{k=0}^{n-1}\|X(s-)-X(t_k)\|^2\cdot\mathbf{1}_{(t_k,t_{k+1}]}(s)ds\notag\\
 &&=TL^2\mathbb{E}\sum_{k=0}^{n-1}\int\limits_{t_k}^{t_{k+1}}\|X(s-)-X(t_k)\|^2ds\notag\\
 &&=TL^2\mathbb{E}\sum_{k=0}^{n-1}\int\limits_{t_k}^{t_{k+1}}\|X(s)-X(t_k)\|^2ds\notag\\
 &&= TL^2\sum_{k=0}^{n-1}\int\limits_{t_k}^{t_{k+1}}\mathbb{E}\|X(s)-X(t_k)\|^{2} ds \leq C_2(1+\mathbb{E}\|\xi\|^2)n^{-1}.
\end{eqnarray}
We now estimate $\sup\limits_{t\in [0,T]}\mathbb{E}\|\tilde A^{RE}_{2,n}(t)\|^2$. To do that we use the following notation:
\begin{equation}
	i(t)=\sup\{i=0,1,\ldots,n \ | \ t_i\leq t\}, 
\end{equation}
\begin{equation}
	\zeta(t)=t_{i(t)},
\end{equation}
for $t\in [0,T]$. Now we can write that 
\begin{equation}
	\label{M2_EST_1}
	\mathbb{E}\|\tilde A^{RE}_{2,n}(t)\|^2\leq 2\Bigl(\mathbb{E}\|\tilde A^{RE}_{21,n}(t)\|^2+\mathbb{E}\|\tilde A^{RE}_{22,n}(t)\|^2\Bigr),
\end{equation}
where
\begin{eqnarray}
	&&\mathbb{E}\|\tilde A^{RE}_{21,n}(t)\|^2=\mathbb{E}\Biggl\|\sum_{k=0}^{i(t)-1}\int\limits_{t_k}^{t_{k+1}}\Bigl(a(s,X(t_k))-a(\theta_k,X(t_k))\Bigr)ds\Biggl\|^2,\\
	&&\mathbb{E}\|\tilde A^{RE}_{22,n}(t)\|^2=\mathbb{E}\Biggl\|\;\int\limits_{\zeta(t)}^t\Bigl(a(s,X(\zeta(t)))-a(\theta_{i(t)},X(\zeta(t)))\Bigr)ds\Biggl\|^2,
\end{eqnarray}
for all $t\in [0,T]$. Let for $k=0,1,\ldots,n-1$
\begin{equation}
	\label{DEF_T_Y_K}
	\tilde Y_k=\int\limits_{t_k}^{t_{k+1}}a(s,X(t_k))ds-a(\theta_k,X(t_k))\cdot\Delta t_{k}=\int\limits_{t_k}^{t_{k+1}}(a(s,X(t_k))-a(\theta_k,X(t_k)))ds, 
\end{equation}
\begin{equation}
	Z_j=\sum_{k=0}^j\tilde Y_k, \ j=0,1,\ldots,n-1,
\end{equation}
and $Z_{-1}:=0$. Then we can write that for $t\in [0,T]$ 
\begin{equation}
	\tilde A^{RE}_{21,n}(t)=Z_{i(t)-1}=\sum_{k=0}^{n-1}Z_{k-1}\cdot\mathbf{1}_{[t_k,t_{k+1})}(t)+Z_{n-1}\cdot\mathbf{1}_{\{T\}}(t).
\end{equation}
Therefore, we have for all $t\in [0,T]$ that
\begin{equation}
	\mathbb{E}\|\tilde A^{RE}_{21,n}(t)\|^2\leq \mathbb{E}\Biggl(\sup\limits_{t\in [0,T]}\|\tilde A^{RE}_{21,n}(t)\|^2\Biggr)=\mathbb{E}\Biggl(\max\limits_{0\leq j\leq n-1}\|Z_j\|^2\Biggr)=\mathbb{E}\Biggl[\mathbb{E}\Biggl(\max\limits_{0\leq j\leq n-1}\|Z_j\|^2 \ | \ \Sigma_{\infty}\Biggr)\Biggr].
\end{equation}
By Theorem \ref{f_lemma} we get
\begin{eqnarray}
	\mathbb{E}\Biggl(\max\limits_{0\leq j\leq n-1}\|Z_j\|^2 \ | \ \Sigma_{\infty}\Biggr)=\mathbb{E}\Biggl[\max\limits_{0\leq j\leq n-1}\Bigl\|\sum_{k=0}^j\int\limits_{t_k}^{t_{k+1}}(a(s,x_k)-a(\theta_k,x_k))ds\Bigl\|^2\Biggr]\Bigl|_{x_k=X(t_k), \ k=0,1,\ldots,n-1}.
\end{eqnarray}
Note that for all $x_0,x_1,\ldots,x_{n-1}\in\mathbb{R}^d$ the process
\begin{equation}
	\Biggl(\sum_{k=0}^j\int\limits_{t_k}^{t_{k+1}}(a(s,x_k)-a(\theta_k,x_k))ds, \ \sigma(\theta_0,\theta_1,\ldots,\theta_{j})\Biggr)_{j=0,1,\ldots,n-1}
\end{equation}
is a square integrable martingale. Hence, for all $x_0,x_1,\ldots,x_{n-1}\in\mathbb{R}^d$ we have by Theorem \ref{BDG_DISC}
\begin{eqnarray}
	&&\mathbb{E}\Biggl[\max\limits_{0\leq j\leq n-1}\Bigl\|\sum_{k=0}^j\int\limits_{t_k}^{t_{k+1}}(a(s,x_k)-a(\theta_k,x_k))ds\Bigl\|^2\Biggr]\leq c_2^2\sum_{k=0}^{n-1}\mathbb{E}\Bigl\|\int\limits_{t_k}^{t_{k+1}}(a(s,x_k)-a(\theta_k,x_k))ds\Bigl\|^2\notag\\
	&&\leq 2c_2^2\cdot n\cdot \Bigl(2\bar D (T/n)(1+\max\limits_{0\leq k\leq n-1}\|x_k\|)\Bigr)^2= C_3 n^{-1}(1+\max\limits_{0\leq k\leq n-1}\|x_k\|)^2, 
\end{eqnarray}
and therefore
\begin{eqnarray}
	&&\mathbb{E}\Biggl[\max\limits_{0\leq j\leq n-1}\Bigl\|\sum_{k=0}^j\int\limits_{t_k}^{t_{k+1}}(a(s,x_k)-a(\theta_k,x_k))ds\Bigl\|^2\Biggr]\Bigl|_{x_k=X(t_k), \ k=0,1,\ldots,n-1}\notag\\
	&&\leq C_3 n^{-1}(1+\max\limits_{0\leq k\leq n-1}\|X(t_k)\|)^2\leq C_3 n^{-1}(1+\sup\limits_{0\leq t\leq T}\|X(t)\|)^2.
\end{eqnarray}
Combining the above estimates we obtain
\begin{eqnarray}
	\label{TM_1_EST_4}
\mathbb{E}\|\tilde A^{RE}_{21,n}(t)\|^2 &\leq &\mathbb{E}\Biggl[\mathbb{E}\Biggl(\max\limits_{0\leq j\leq n-1}\|Z_j\|^2 \ | \ \Sigma_{\infty}\Biggr)\Biggr]\notag\\
&\leq& C_3 n^{-1}\cdot\mathbb{E}\Bigl(1+\sup\limits_{t\in [0,T]}\|X(t)\|\Bigr)^2 
\leq C_4 (1+\mathbb{E}\|\xi\|^2)n^{-1}.
\end{eqnarray}
Using  the H\"older inequality we get for all $t\in [0,T]$ that
\begin{eqnarray}
	\label{TM_2_EST1}
	\mathbb{E}\|\tilde A^{RE}_{22,n}(t)\|^2 &\leq& (t-\zeta(t))\cdot\int\limits_{\zeta(t)}^{t}\mathbb{E}\|a(s,X(\zeta(t)))-a(\theta_{i(t)},X(\zeta(t)))\|^2 ds\notag\\
	&\leq& C_5(t-\zeta(t))^2\cdot\Bigl(1+\sup\limits_{t\in [0,T]}\mathbb{E}\|X(t)\|^2\Bigr)\leq C_6(1+\mathbb{E}\|\xi\|^2)n^{-2}.
\end{eqnarray}
By (\ref{M2_EST_1}), (\ref{TM_1_EST_4}) and (\ref{TM_2_EST1}) we arrive at
\begin{equation}
	\label{M2_EST_2}
	\mathbb{E}\|\tilde A^{RE}_{2,n}(t)\|^2\leq C_7(1+\mathbb{E}\|\xi\|^2)n^{-1},
\end{equation}
for all $t\in [0,T]$. By the H\"older inequality we get for all $t\in [0,T]$ that
\begin{eqnarray}
	\label{M3_EST_1}
	\mathbb{E}\|\tilde A^{RE}_{3,n}(t)\|^2 &\leq& t\int\limits_0^t\mathbb{E}\Biggl(\sum_{i=0}^{n-1}\|a(\theta_i,X(t_i))-a(\theta_i,\tilde X^{RE}_n(t_i))\|\cdot\mathbf{1}_{(t_i,t_{i+1}]}(s)\Biggl)^2 ds\notag\\ 
	&\leq& C_8\int\limits_0^t\sum_{i=0}^{n-1}\mathbb{E}\|X(t_i)-\tilde X^{RE}_n(t_i)\|^2\cdot\mathbf{1}_{(t_i,t_{i+1}]}(s)ds.
\end{eqnarray}
Combining (\ref{err_est_21}) with (\ref{M1_EST_1}), (\ref{M2_EST_2}) and (\ref{M3_EST_1}) we obtain for all $t\in [0,T]$ that
\begin{equation}
	\label{RE_A_EST_2}
	A_{1,n}(t) \leq C_8\int\limits_0^t\sum_{i=0}^{n-1}\mathbb{E}\|X(t_i)-\tilde X^{RE}_n(t_i)\|^2\cdot\mathbf{1}_{(t_i,t_{i+1}]}(s)ds+C_9(1+\mathbb{E}\|\xi\|^2)n^{-1},
\end{equation}
From the It\^o isometry
\begin{equation}
\label{err_est_31}
	A_{2,n}(t)=\mathbb{E}\int\limits_0^t \|\hat b(s)-\tilde b_n(s)\|^2 ds.
\end{equation}
Furthermore
\begin{equation}
\label{a3_est_11}
	A_{3,n}(t)\leq 2 \Bigl(B_{1,n}(t)+B_{2,n}(t)\Bigr),
\end{equation}
where
\begin{eqnarray}
	&&B_{1,n}(t)=\mathbb{E}\Bigl\|\int\limits_0^t\Bigl(\hat c(s)-\tilde c_n(s)\Bigr)d\tilde N(s)\Bigl\|^2,\notag\\
	&&B_{2,n}(t)=\mathbb{E}\Bigl\|\int\limits_0^t\Bigl(\hat c(s)-\tilde c_n(s)\Bigr)dm(s)\Bigl\|^2.
\end{eqnarray}
From the isometry for martingale-driven stochastic integrals we get
\begin{eqnarray}
\label{a3_est_21}
	B_{1,n}(t)&=&\sum\limits_{j=1}^{m_N}\mathbb{E}\int\limits_0^t\|\hat c^{(j)}(s)-\tilde c^{(j)}_n(s)\|^2\cdot \lambda_j(s)ds\notag\\
	&\leq &\max\limits_{1\leq j\leq m_N}\|\lambda_j\|_{\infty}\cdot\mathbb{E}\int\limits_0^t \|\hat c(s)-\tilde c_n(s)\|^2 ds
\end{eqnarray}
and by the Schwarz inequality
\begin{eqnarray}
\label{a3_est_31}
	B_{2,n}(t)&= &\sum\limits_{i=1}^d\mathbb{E}\Bigl|\sum\limits_{j=1}^{m_N}\int\limits_0^t \Bigl(\hat c_{ij}(s)-\tilde c_{n,ij}(s)\Bigr)\cdot\lambda_j(s)ds\Bigl|^2\notag\\
	&\leq & m_N\cdot\sum\limits_{i=1}^d\sum\limits_{j=1}^{m_N}\mathbb{E}\Bigl|\int\limits_0^t \Bigl(\hat c_{ij}(s)-\tilde c_{n,ij}(s)\Bigr)\cdot\lambda_j(s)ds\Bigl|^2\notag\\
	&\leq & m_N T\cdot\sum\limits_{i=1}^d\sum\limits_{j=1}^{m_N}\int\limits_0^t\mathbb{E}|\hat c_{ij}(s)-\tilde c_{n,ij}(s)|^2\cdot\lambda^2_j(s) ds\notag\\
	&\leq &m_N T\cdot\max\limits_{1\leq j\leq m_N}\|\lambda_j\|^2_{\infty}\cdot\mathbb{E}\int\limits_0^t\|\hat c(s)-\tilde c_n(s)\|^2 ds.
\end{eqnarray}
Hence, by \eqref{a3_est_11}, \eqref{a3_est_21}, and  \eqref{a3_est_31} we get
\begin{equation}
\label{err_est_41}
	A_{3,n}(t)\leq C\cdot\mathbb{E}\int\limits_0^t\|\hat c(s)-\tilde c_n(s)\|^2 ds.
\end{equation}
Combining \eqref{err_est_11}, \eqref{err_est_21}, \eqref{err_est_31}, and \eqref{err_est_41} we arrive at
\begin{equation}
	\mathbb{E}\|X(t)-\tilde X^{RE}_n(t)\|^2\leq C\Biggl(A_{1,n}(t)+\sum\limits_{f\in\{b,c\}}\mathbb{E}\int\limits_0^t \|\hat f(s)-\tilde f_n(s)\|^2 ds\Biggr).
\end{equation}
Now, for $f\in\{b,c\}$ and $s\in [0,T]$ we get
\begin{eqnarray}
	\|\hat f(s)-\tilde f_n(s)\|&=&\Bigl\|\sum\limits_{k=0}^{n-1}\Bigl(f(s,X(s-))-f(U_{k,n}^E)\Bigr)\cdot\mathbf{1}_{(t_k,t_{k+1}]}(s)\Bigl\|\notag\\
	&\leq& \sum\limits_{k=0}^{n-1}\|f(s,X(s-))-f(U_{k,n}^{RE})\|\cdot\mathbf{1}_{(t_k,t_{k+1}]}(s)\notag\\
	&\leq& K\sum\limits_{k=0}^{n-1}\Bigl((s-t_k)+\|X(s-)-\tilde X^{RE}_n(t_k)\|\Bigr)\cdot\mathbf{1}_{(t_k,t_{k+1}]}(s)\notag\\
	&\leq& K\max\limits_{0\leq k \leq n-1}\Delta t_k+K\sum\limits_{k=0}^{n-1}\|X(s-)-X(t_k)\|\cdot\mathbf{1}_{(t_k,t_{k+1}]}(s)\notag\\
	&&\quad +K\sum\limits_{k=0}^{n-1}\|X(t_k)-\tilde X^{RE}_n(t_k)\|\cdot\mathbf{1}_{(t_k,t_{k+1}]}(s)
\end{eqnarray}
and hence
\begin{eqnarray}
	\|\hat f(s)-\tilde f_n(s)\|^2 &\leq& 3K^2\max\limits_{0\leq k \leq n-1}(\Delta t_k)^2+3K^2\sum\limits_{k=0}^{n-1}\|X(s-)-X(t_k)\|^2\cdot\mathbf{1}_{(t_k,t_{k+1}]}(s)\notag\\
	&&\quad +3K^2\sum\limits_{k=0}^{n-1}\|X(t_k)-\tilde X^{RE}_n(t_k)\|^2\cdot\mathbf{1}_{(t_k,t_{k+1}]}(s).
\end{eqnarray}
Since any trajectory of $X$ has only finite number of (jump) discontinuities in $[0,T]$, the functions
\begin{equation}
	[0,T]\ni s\to \sum\limits_{k=0}^{n-1}\|X(s-)-X(t_k)\|^2\cdot\mathbf{1}_{(t_k,t_{k+1}]}(s),
\end{equation}
\begin{equation}
	[0,T]\ni s\to \sum\limits_{k=0}^{n-1}\|X(s)-X(t_k)\|^2\cdot\mathbf{1}_{[t_k,t_{k+1})}(s)
\end{equation}
differ only on a finite subset of $[0,T]$ (the first is a c\`agl\`ad modification of the second one). This implies for $f\in\{a,b,c\}$ and  $t \in [0,T]$ that
\begin{eqnarray}
	\int\limits_0^t \|\hat f(s)-\tilde f_n(s)\|^2ds &\leq & 3K^2T\max\limits_{0\leq k \leq n-1}(\Delta t_k)^2+3K^2\int\limits_0^t\sum\limits_{k=0}^{n-1}\|X(s)-X(t_k)\|^2\cdot\mathbf{1}_{[t_k,t_{k+1})}(s)ds\notag\\
		&&\quad +3K^2\int\limits_0^t \sum\limits_{k=0}^{n-1}\|X(t_k)-\tilde X^{RE}_n(t_k)\|^2\cdot\mathbf{1}_{(t_k,t_{k+1}]}(s)ds.
\end{eqnarray} 
By Lemma \ref{mean_sq_reg_X} we have that
\begin{eqnarray}
	\int\limits_0^t\sum\limits_{k=0}^{n-1}\mathbb{E}\|X(s)-X(t_k)\|^2\cdot\mathbf{1}_{[t_k,t_{k+1})}(s)ds&\leq& C(1+\mathbb{E}\|\xi\|^2)\int\limits_0^T\sum\limits_{k=0}^{n-1} (s-t_k)\cdot \mathbf{1}_{(t_k,t_{k+1}]}(s)ds\notag\\
	&\leq& C_1(1+\mathbb{E}\|\xi\|^2)\cdot\max\limits_{0\leq k\leq n-1}\Delta t_k.
\end{eqnarray}
Since $\max\limits_{0\leq k\leq n-1}(\Delta t_k)^2\leq T\max\limits_{0\leq k\leq n-1}\Delta t_k$, we get for all $t\in [0,T]$
\begin{eqnarray}
	\mathbb{E}\int\limits_0^t \|\hat f(s)-\tilde f_n(s)\|^2ds &\leq & K_1\cdot (1+\mathbb{E}\|\xi\|^2)\cdot \max\limits_{0\leq k\leq n-1}\Delta t_k\notag\\
&&\quad +K_2\int\limits_0^t \sum\limits_{k=0}^{n-1}\mathbb{E}\|X(t_k)-\tilde X^{RE}_n(t_k)\|^2\cdot\mathbf{1}_{(t_k,t_{k+1}]}(s)ds.
\end{eqnarray} 
and hence
\begin{eqnarray}
	\mathbb{E}\|X(t)-\tilde X^{RE}_n(t)\|^2 &\leq & K_3\cdot (1+\mathbb{E}\|\xi\|^2)\cdot n^{-1}\notag\\
&&\quad +K_4\int\limits_0^t \sum\limits_{k=0}^{n-1}\mathbb{E}\|X(t_k)-\tilde X^{RE}_n(t_k)\|^2\cdot\mathbf{1}_{(t_k,t_{k+1}]}(s)ds\notag\\
&\leq& K_3\cdot (1+\mathbb{E}\|\xi\|^2)\cdot n^{-1}\notag\\
&&\quad +K_4\int\limits_0^t \sup\limits_{0\leq u\leq z}\mathbb{E}\|X(u)-\tilde X^{RE}_n(u)\|^2 dz.
\end{eqnarray} 
Therefore, for all $t\in [0,T]$
\begin{eqnarray}
	\sup\limits_{0\leq s\leq t}\mathbb{E}\|X(s)-\tilde X^{RE}_n(s)\|^2 &\leq & K_3\cdot (1+\mathbb{E}\|\xi\|^2)\cdot n^{-1}\notag\\
&&\quad +K_4\int\limits_0^t \sup\limits_{0\leq u\leq z}\mathbb{E}\|X(u)-\tilde X^{RE}_n(u)\|^2 dz.
\end{eqnarray}
Note that the function
\begin{equation}
	[0,T]\ni t\to \sup\limits_{0\leq s\leq t}\mathbb{E}\|X(s)-\tilde X^{RE}_n(s)\|^2 
\end{equation} 
is bounded (due to Lemmas \ref{mean_sq_reg_X}, \ref{err_euler_2}) and Borel measurable (as a nondecreasing function). Hence, by Gronwall's lemma \ref{GLCW} we get
\begin{equation}
\label{Euler_proc_est_11}
	\sup\limits_{0\leq t\leq T}\mathbb{E}\|X(t)-\tilde X^{RE}_n(t)\|^2\leq K_5\cdot (1+\mathbb{E}\|\xi\|^2)\cdot n^{-1},
\end{equation}
which completes the proof. \ \ \ $\blacksquare$ 
\subsection{Perturbed gradient descent method and Euler-Maruyama scheme}
Let us consider again the minimalization problem \eqref{opt_prob}. In order to handle plateus we can add some random noise to the GD algorithm \eqref{EU_GD} and obtain
\begin{eqnarray}
\label{EU_GD}
    &&y_0=x_0,\notag\\
    &&y_{k+1}=y_k-h\cdot\nabla f(y_k)+\sigma \Delta W_k+c\Delta N_k, \ k=0,1,\ldots
\end{eqnarray}
where $\sigma,c$ are additional hyperparameters, see \cite{SDE_PGD} where the case of fractional Wiener noise case was considered. This is of course the Euler-Maruyama scheme when applied to the following SDE
\begin{equation}
    dX(t)=-\nabla f(X(t))dt+\sigma dW(t)+c dN(t), t\geq 0,
\end{equation}
and $X(0)=x_0$, which is the so-called stochastic Langevin dynamic connected with the minimization problem
\eqref{opt_prob}.
\chapter{Introduction to optimal one-point approximation of solutions of SDEs}
\label{opt_one_points_sdes}
Let us consider the following SDE 
\begin{equation}
	\label{OU_T_SDE1}
		\left\{ \begin{array}{ll}
			dX(t)=\kappa(\mu- X(t))dt+\sigma dW(t), &t\in [0,T], \\
			X(0)=x_0, 
		\end{array}\right.
\end{equation}
with $T,\kappa,\sigma>0$ and $\mu\in\mathbb{R}$. Its solution defines the well-known Ornstein-Uhlenbeck process which is given by
\begin{equation}
    \label{sol_X_OU_1}
    X(t)=e^{-\kappa t}x_0+\mu(1-e^{-\kappa t})+\sigma e^{-\kappa t}\int\limits_0^t e^{\kappa s}dW(s), \quad t \in [0,T].
\end{equation}
 Using this process as an toy example, we will give an introduction to the broad theory of {\it optimal approximation} of solutions of SDEs. Namely, we assume that we have access only to the values of $W$ at finite number of time points, i.e., $W(t_1),W(t_2),\ldots, W(t_n)$, and, basing on this data, we are interested in optimal (in the $L^2(\Omega)$ sense) approximation of $X(T)$. Note that 
 \begin{equation}
 \label{sol_X_OU_2}
     \int\limits_0^t e^{\kappa s}dW(s)=e^{\kappa t}W(t)-\kappa\int\limits_0^t e^{\kappa u}W(u) du,
 \end{equation}
 so
 \begin{equation}
     \label{sol_X_OU_3}
     X(t)=e^{-\kappa t}x_0+\mu(1-e^{-\kappa t})+\sigma W(t)-\kappa\sigma e^{-\kappa t}\int\limits_0^t e^{\kappa u}W(u)du.
 \end{equation}
 Hence, approximation of $X(T)$ is equivalent to the problem of approximating the weighted Lebesgue integral $\displaystyle{\int\limits_0^T e^{\kappa u}W(u)du}$ having access only to discrete information about $W$. This is a more general case than \eqref{Leb_int_W1}, considered in Section \ref{Opt_approx_Leb_W}. However, we can use a similar approach in order to define an optimal method that approximates $X(T)$.
 
 Namely, an algorithm $\bar X=(\bar X_n)_{n\in\mathbb{N}}$ that approximates $X(T)$  is (a sequence) defined by a pair $(\bar\Delta,\bar\varphi)$ of two sequences $\bar\varphi=(\varphi_n)_{n\in\mathbb{N}}$, $\bar\Delta=(\Delta_n)_{n\in\mathbb{N}}$, where each
\begin{equation}
	\varphi_n:\mathbb{R}^n\to \mathbb{R}
\end{equation}
is a Borel function and
\begin{equation}
	\Delta_n=\{t_{0,n},t_{1,n},\ldots,t_{n,n}\}
\end{equation}
is a (possibly) non-uniform partition of $[0,T]$ such that
\begin{equation}
\label{disc_arb_1}
	0=t_{0,n}<t_{1,n}<\ldots<t_{n,n}=T.
\end{equation}
By
\begin{equation}
	\mathcal{\bar N}(W)=(\mathcal{N}_n(W))_{n\in\mathbb{N}}
\end{equation}
we mean a sequence of vectors $\mathcal{N}_n(W)$ which provides {\it standard information} with $n$
evaluations about  the Wiener process $W$ at the discrete points from $\Delta_n$, i.e.,
\begin{equation}
	\mathcal{N}_n(W)=[W(t_{1,n}),W(t_{2,n}),\ldots,W(t_{n,n})].
\end{equation}
(Recall that $W(0)=0$ almost surely.) After computing the information $\mathcal{N}_n(W)$, we apply the mapping $\varphi_n$ in order to obtain
the $n$th approximation $\bar X_n$ in the following way
\begin{equation}
	\bar X_n=\varphi_n(\mathcal{N}_n(W)).
\end{equation}
We also assume that for all $n\in\mathbb{N}$
\begin{equation}
\label{finite_norm_int_Xn_2}
	\mathbb{E}|\bar X_n|^2<+\infty.
\end{equation}
The informational cost of each $\bar X_n$ is equal to $n$ evaluations of $W$. The set of all methods $\bar X=(\bar X_n)_{n\in\mathbb{N}}$,  defined as above, is denoted by $\chi^{\rm noneq}$. We also consider the following subclass of $\chi^{\rm noneq}$
\begin{equation}
	\chi^{\rm eq}=\{\bar X\in \chi^{\rm noneq} \ | \ \forall_{n\in\mathbb{N}} \Delta_n=\Delta_n^{\rm eq}:=\{iT/n \ | \ i=0,1,\ldots,n\}\}.
\end{equation} 
The $n$th error of a method $\bar X$ is defined as
\begin{equation}
	e_n(\bar X)=\|X(T)-\bar X_n\|_{L^2(\Omega)}=\Bigl(\mathbb{E}|X(T)-\bar X_n|^2\Bigr)^{1/2}.
\end{equation}

Let us fix any method $\bar X\in\chi^{\rm noneq}$ that is based on sequences of discretizations $(\Delta_n)_{n\in\mathbb{N}}$ and functions $(\varphi_n)_{n\in\mathbb{N}}$.  By Theorem \ref{proj_wwo} we have  for  all $n\in\mathbb{N}$ that
\begin{equation}	
\label{ineq_wwo_intW_2}
	(e_n(\bar X))^2=\mathbb{E}|X(T)-\varphi_n(\mathcal{N}_n(W))|^2\geq \mathbb{E}|X(T)-\mathbb{E}(X(T) \ | \ \mathcal{N}_n(W))|^2,
\end{equation}
since $\mathbb{E}|X(T)|^2<+\infty$ and $\sigma(\bar X_n)\subset\sigma(\mathcal{N}_n(W))$. Let us denote by
\begin{equation}
\label{def_A_bar}
    \mathcal{\bar A}_n(T)=\mathbb{E}(X(T) \ | \ \mathcal{N}_n(W)),
\end{equation}
and, in particular, by
\begin{equation}
\mathcal{\bar A}_n^{\rm eq}(T)=\mathbb{E}(X(T) \ | \ W(jT/n), j=1,2,\ldots,n).    
\end{equation}
By \eqref{ineq_wwo_intW_2} we see that the error of $\bar X\in\chi^{\rm noneq}$, that is based on a sequence $(\Delta_n)_{n\in\mathbb{N}}$, cannot be smaller than the error of the algorithm $\mathcal{\bar A}(T)=(\mathcal{\bar A}_n(T))_{n\in\mathbb{N}}$ (which is also using $(\Delta_n)_{n\in\mathbb{N}}$). From the Doob-Dynkin representation theorem  we get that for every $n\in\mathbb{N}$ there exists a Borel function $\bar\varphi_n:\mathbb{R}^n\to\mathbb{R}$
such that
\begin{equation}
\label{WWO_DD}
    \mathcal{\bar A}_n(T)=\bar\varphi_n(\mathcal{N}_n(W)),
\end{equation}
since $\sigma(\mathcal{\bar A}_n(T))\subset\sigma(\mathcal{N}_n(W))$. Moreover, by the conditional Jensen inequality, it holds that $\mathbb{E}|\mathcal{\bar A}_n(T)|^2\leq \mathbb{E}|X(T)|^2<+\infty$. Hence, $\mathcal{\bar A}(T)=(\mathcal{\bar A}_n(T))_{n\in\mathbb{N}}\in\chi^{\rm noneq}$ for any sequence $\bar\Delta$ of the form \eqref{disc_arb_1} and $\mathcal{\bar A}^{\rm eq}(T)\in\chi^{\rm eq}$. In summary, for all $n\in\mathbb{N}$ we define
\begin{eqnarray}
\label{n_min_opt_noneq}
   &&e^{\rm noneq}(n):=\inf\limits_{\bar X\in\chi^{\rm noneq}}e_n(\bar X)=\inf\limits_{\Delta_n}e_n(\mathcal{\bar A}(T))\notag\\
   &&=\inf_{0=t_{0,n}<t_{1,n}<\ldots<t_{n,n}=T}\Bigl(\mathbb{E}|X(T)-\mathbb{E}(X(T) \ | \ W(t_{1,n}),W(t_{2,n}),\ldots, W(t_{n,n})))|^2\Bigr)^{1/2},
\end{eqnarray}
and
\begin{eqnarray}
\label{n_min_opt_eq}
   &&e ^{\rm eq}(n):=\inf\limits_{\bar X\in\chi^{\rm eq}}e_n(\bar X)=e_n(\mathcal{\bar A}^{\rm eq}(T))\notag\\
   &&=\Bigl(\mathbb{E}|X(T)-\mathbb{E}(X(T) \ | \ W(jT/n), j=1,2,\ldots,n)|^2\Bigr)^{1/2},
\end{eqnarray}
where it holds
\begin{equation}
    e^{\rm noneq}(n)\leq e^{\rm eq}(n),
\end{equation}
since $\chi^{\rm eq}\subset\chi^{\rm noneq}$. The equations \eqref{n_min_opt_noneq}, \eqref{n_min_opt_eq} define the so-called {\it $n$th minimal errors} in the classes $\chi^{\rm noneq}, \chi^{\rm eq}$, respectively.

As we will see, in the case of  the process $X$ given by \eqref{sol_X_OU_3} we are able to compute the conditional expectation \eqref{def_A_bar} (and therefore also the functions $\bar\varphi_n$) explicitly,  see however, Remark \ref{rem_cond_X}.  We also show the optimality of the algorithm $\mathcal{\bar A}(T)$.

From Theorem \ref{cond_fubini} we get 
\begin{equation}
\label{cond_exp_OU_1}
    \mathbb{E}(X(T) \ | \ \mathcal{N}_n(W))=e^{-\kappa T}x_0+\mu(1-e^{-\kappa T})+\sigma W(T)-\kappa\sigma e^{-\kappa T}\int\limits_0^Te^{\kappa u}\mathbb{E}(W(u) \ | \ \mathcal{N}_n(W))du,
\end{equation}
and from Lemma \ref{WN_bridge} we get the following explicit (and implementable) expression
\begin{eqnarray}
     &&\mathbb{E}(X(T) \ | \ \mathcal{N}_n(W))=e^{-\kappa T}x_0+\mu(1-e^{-\kappa T})+\sigma W(T)-\kappa\sigma e^{-\kappa T}\sum_{j=0}^{n-1}\int\limits_{t_j}^{t_{j+1}} e^{\kappa u}\mathbb{E}(W(u) \ | \ \mathcal{N}_n(W))du\notag\\
     &&=e^{-\kappa T}x_0+\mu(1-e^{-\kappa T})+\sigma W(T)-\kappa\sigma e^{-\kappa T}\sum_{j=0}^{n-1}\frac{W(t_{j+1})\cdot I_{1,j}+W(t_{j})\cdot I_{2,j}}{\Delta t_j},
\end{eqnarray}
where
\begin{equation}
    I_{1,j}=\int\limits_{t_j}^{t_{j+1}}e^{\kappa u}\cdot (u-t_j)du=\frac{e^{\kappa t_j}+e^{\kappa t_{j+1}}(\kappa \Delta t_j-1)}{\kappa^2},
\end{equation}
\begin{equation}
    I_{2,j}=\int\limits_{t_j}^{t_{j+1}}e^{\kappa u}\cdot (t_{j+1}-u)du=\frac{e^{\kappa t_{j+1}}-e^{\kappa t_{j}}(\kappa \Delta t_j+1)}{\kappa^2}.
\end{equation}
Hence, we have explicitly computed the conditional expectation $\mathbb{E}(X(T) \ | \ \mathcal{N}_n(W))$ and we see that it is of the form \eqref{WWO_DD}.

By \eqref{cond_exp_OU_1} we have
\begin{equation}
    X(T)-\mathbb{E}(X(T) \ | \ \mathcal{N}_n(W))=-\kappa\sigma e^{-\kappa T}\int\limits_0^Te^{\kappa u}\hat W_n(u)du,
\end{equation}
where $\hat W_n(\cdot)=W(\cdot)-\mathbb{E}(W(\cdot) \ | \ \mathcal{N}_n(W))$ consists of independent Brownian bridges, see \eqref{def_Z_bridges_1}. Let us denote by $\displaystyle{S_j=\int\limits_{t_j}^{t_{j+1}}e^{\kappa u}\hat W_n(u)du}$.  Then we get
\begin{equation}
\label{WWO_err_1}
    \mathbb{E}|X(T)-\mathbb{E}(X(T) \ | \ \mathcal{N}_n(W))|^2=(\kappa\sigma)^2 e^{-2\kappa T}\mathbb{E}\Bigl|\sum\limits_{j=0}^{n-1}S_j\Bigl|^2,
\end{equation}
and
\begin{equation}
    \mathbb{E}\Bigl|\sum\limits_{j=0}^{n-1}S_j\Bigl|^2=\sum\limits_{j=0}^{n-1}\mathbb{E}(S_j^2)+\sum\limits_{i\neq j}\mathbb{E}(S_iS_j),
\end{equation}
where for $i\neq j$ we get from \eqref{covar_indep_bb} that
\begin{equation}
    \mathbb{E}(S_iS_j)=\int\limits_{t_i}^{t_{i+1}}\int\limits_{t_j}^{t_{j+1}}e^{\kappa (t+s)}\mathbb{E}(\hat W_n(s)\hat W_n(t))dsdt=0,
\end{equation}
and
\begin{equation}
\label{def_Sj_lowb}
    \mathbb{E}(S_j^2)=\int\limits_{t_j}^{t_{j+1}}\int\limits_{t_j}^{t_{j+1}}e^{\kappa (t+s)}\mathbb{E}(\hat W_n(s)\hat W_n(t))dsdt,
\end{equation}
where, due to \eqref{covar_W},
\begin{equation}
\label{noneg_cov}
    \mathbb{E}(\hat W_n(s)\hat W_n(t))=Cov(\hat W_n(s)\hat W_n(t))\geq 0,
\end{equation}
for all $s,t\in [t_j,t_{j+1}]$. Since $\kappa>0$, $t,s\in [0,T]$, we get the (rough) estimate
\begin{equation}
\label{low_b_est_re}
    1\leq 1+\kappa(t+s)\leq e^{\kappa(t+s)}\leq e^{2\kappa T}.
\end{equation}
Hence, by \eqref{def_Sj_lowb}, \eqref{noneg_cov}, \eqref{low_b_est_re}, \eqref{int_covar}
\begin{equation}
    \frac{1}{12}\sum\limits_{j=0}^{n-1}(\Delta t_j)^3\leq \sum\limits_{j=0}^{n-1}\mathbb{E}(S_j^2)\leq \frac{e^{2\kappa T}}{12}\sum\limits_{j=0}^{n-1}(\Delta t_j)^3,
\end{equation}
which implies
\begin{equation}
\label{upp_low_b_WWO}
    (\kappa\sigma)^2\cdot\frac{e^{-2\kappa T}}{12}\cdot\sum\limits_{j=0}^{n-1}(\Delta t_j)^3\leq \mathbb{E}|X(T)-\mathbb{E}(X(T) \ | \ \mathcal{N}_n(W))|^2\leq (\kappa\sigma)^2\cdot\frac{1}{12}\cdot\sum\limits_{j=0}^{n-1}(\Delta t_j)^3.
\end{equation}
For the upper error bound we therefore get
\begin{equation}
    e_n(\mathcal{\bar A}(T))\leq \kappa\sigma\cdot\frac{T^{1/2}}{\sqrt{12}}\cdot\max\limits_{0\leq j\leq n-1}\Delta t_j,
\end{equation}
and, in particular, for $\mathcal{\bar A}_n^{\rm eq}(T)=\mathbb{E}(X(T) \ | \ W(jT/n), j=1,2,\ldots,n)$, based on the  equidistant discretization, we have
\begin{equation}
    \label{bar_A_eq_disc}
    e_n(\mathcal{\bar A}^{\rm eq}(T))\leq \kappa\sigma\cdot\frac{T^{3/2}}{\sqrt{12}}\cdot n^{-1}
\end{equation}
for all $n\in\mathbb{N}$. From the Jensen inequality we get for any discretization \eqref{disc_arb_1} that
\begin{equation}
\label{low_b_mesh_1}
    \sum\limits_{j=0}^{n-1}(\Delta t_j)^3\geq T^3\cdot n^{-2}.
\end{equation}
Combining \eqref{ineq_wwo_intW_2}, \eqref{upp_low_b_WWO}, and \eqref{low_b_mesh_1}  we obtain for any algorithm $\bar X\in\chi^{\rm noneq}$ and for all $n\in\mathbb{N}$ that
\begin{equation}
\label{low_err_b_noneq_1}
    e_n(\bar X)\geq e_n(\mathcal{\bar A}(T)) \geq \kappa\sigma \cdot T^{3/2}\cdot\frac{e^{-\kappa T}}{\sqrt{12}}\cdot n^{-1}.
\end{equation}
Since $\chi^{\rm eq}\subset\chi^{\rm noneq}$, we get by \eqref{bar_A_eq_disc} and \eqref{low_err_b_noneq_1} that
\begin{equation}
\label{upp_low_b_min_err_1}
      \kappa\sigma \cdot \frac{T^{3/2}}{\sqrt{12}}\cdot e^{-\kappa T}\cdot n^{-1}\leq \inf\limits_{\bar X\in\chi^{\rm noneq}}e_n(\bar X)\leq \inf\limits_{\bar X\in\chi^{\rm eq}}e_n(\bar X)\leq \kappa\sigma\cdot\frac{T^{3/2}}{\sqrt{12}}\cdot n^{-1},
\end{equation}      
for all $n\in\mathbb{N}$. Therefore, the $n$th minimal error in the both classes  is $O(n^{-1})$ and the algorithm $\mathcal{\bar A}^{\rm eq}(T)$ is the optimal one. The bounds in \eqref{upp_low_b_min_err_1} are sharp up to the constant $e^{-\kappa T}\in (0, 1]$ that might be small for large $\kappa T>0$. This suggests that  by using algorithms from $\chi^{\rm noneq}$ we can reduce constants in  the error bound but not the rate of convergence.

We now investigate how far we can (asymptotically) reduce value of  constants in error expressions when we use algorithms based on  nonequidistant discretizations.

Firstly, note that for any $0=t_0<t_1<\ldots<t_n=T$ and all $t,s\in [t_j,t_{j+1}]$, $j=0,1,\ldots,n-1$, we have
\begin{equation}
    e^{2\kappa t_j}\leq e^{\kappa (t+s)}\leq e^{2\kappa t_{j+1}},
\end{equation}
and  from \eqref{def_Sj_lowb} we get
\begin{equation}
\label{est_low_upp_sj_1}
    \frac{1}{12}\sum\limits_{j=0}^{n-1}e^{2\kappa t_j}(\Delta t_j)^3\leq \sum\limits_{j=0}^{n-1}\mathbb{E}(S_j^2)\leq \frac{1}{12}\sum\limits_{j=0}^{n-1}e^{2\kappa t_{j+1}}(\Delta t_j)^3.
\end{equation}
For any $\eta_j\in [t_j,t_{j+1}]$, $j=0,1,\ldots,n-1$ we have, in the case of equidistant discretization $\Delta_n=\{t_j=jT/n: j=0,1,\ldots,n\}$, that
\begin{equation}
    \lim\limits_{n\to +\infty} n^2\cdot\sum\limits_{j=0}^{n-1}e^{2\kappa \eta_j}(\Delta t_j)^3=T^2\cdot\lim\limits_{n\to +\infty}\sum\limits_{j=0}^{n-1}e^{2\kappa \eta_j}\cdot\frac{T}{n}=T^2\cdot\int\limits_0^T e^{2\kappa t}dt=\frac{T^2}{2\kappa}(e^{2\kappa T}-1).
\end{equation}
This and \eqref{est_low_upp_sj_1} imply 
\begin{equation}
    \lim\limits_{n\to +\infty}n^2\sum\limits_{j=0}^{n-1}\mathbb{E}(S_j^2)=\frac{T^2}{12}\cdot\int\limits_0^T e^{2\kappa t}dt,
\end{equation}
and hence by \eqref{WWO_err_1}
\begin{equation}
\label{asympt_err_WWO_1}
    \lim\limits_{n\to +\infty} n\cdot e_n(\mathcal{\bar A}^{\rm eq}(T))= C^{\rm eq}(\kappa,\sigma,T)
\end{equation}
with
\begin{equation}
    C^{\rm eq}(\kappa,\sigma,T)=\kappa\sigma\cdot e^{-\kappa T}\cdot\frac{T}{\sqrt{12}}\cdot\Bigl(\int\limits_0^T e^{2\kappa t}dt\Bigr)^{1/2}=\kappa^{1/2}\sigma\cdot\frac{T}{2\sqrt{6}}(1-e^{-2\kappa T})^{1/2}>0.
\end{equation}
Hence, we have from \eqref{n_min_opt_eq}, \eqref{asympt_err_WWO_1} what follows
\begin{equation}
    \lim\limits_{n\to +\infty} n\cdot\inf\limits_{\bar X\in\chi^{\rm eq}}e_n(\bar X)=\lim\limits_{n\to +\infty}n\cdot e_n(\mathcal{\bar A}^{\rm eq}(T))=C^{\rm eq}(\kappa,\sigma,T)>0
\end{equation}
and  the algorithm $\mathcal{\bar A}^{\rm eq}(T)$ is asymptotically optimal in the class $\chi^{\rm eq}$. 

We now turn to algorithms that use possibly non-equidistant discretizations of $[0,T]$. Recall that from \eqref{upp_low_b_min_err_1} we know that the rate $O(n^{-1})$ cannot be improved. We investigate, however, already mentioned  possibility of reducing the asymptotic constant $C^{\rm eq}$.

From \eqref{ineq_wwo_intW_2}, \eqref{WWO_err_1}, \eqref{est_low_upp_sj_1} for any algorithm $\bar X\in\chi^{\rm noneq}$  it holds
\begin{equation}
    n^2\cdot (e_n(\bar X))^2\geq n^2\cdot (e_n(\mathcal{\bar A}(T)))^2\geq\frac{(\kappa\sigma)^2}{12}e^{-2\kappa T}\cdot n^2\cdot\sum\limits_{j=0}^{n-1}e^{2\kappa t_j}(\Delta t_j)^3,
\end{equation}
whereas, by the Jensen inequality,
\begin{equation}
    \sum\limits_{j=0}^{n-1}e^{2\kappa t_j}(\Delta t_j)^3=\sum\limits_{j=0}^{n-1}\Bigl( e^{2\kappa t_j/3} \Delta t_j\Bigr)^3\geq n\cdot\Biggl(\frac{\sum\limits_{j=0}^{n-1}e^{2\kappa t_j/3} \Delta t_j}{n}\Biggr)^3=n^{-2}\cdot\Bigl(\sum\limits_{j=0}^{n-1}e^{2\kappa t_j/3} \Delta t_j\Bigr)^3,
\end{equation}
we obtain
\begin{equation}
\label{low_b_noneq_1}
    n^2\cdot (e_n(\bar X))^2\geq \frac{(\kappa\sigma)^2}{12}e^{-2\kappa T}\cdot\Bigl(\sum\limits_{j=0}^{n-1}e^{2\kappa t_j/3} \Delta t_j\Bigr)^3
\end{equation}
for all $n\in\mathbb{N}$.
Since the sequence $(\Delta_n)_{n\in\mathbb{N}}$  (used by $\bar X$) might be arbitrary, the partial sum
$\displaystyle{\sum\limits_{j=0}^{n-1}e^{2\kappa t_j/3} \Delta t_j}$ does not have to be a partial Riemann sum and we only know that
\begin{equation}
    \sum\limits_{j=0}^{n-1}e^{2\kappa t_j/3} \Delta t_j\geq T.
\end{equation}
For a meanwhile let us therefore consider the following subclass of $\chi^{\rm noneq}$
\begin{equation}
    \chi^{\rm noneq-qu}=\{\bar X\in\chi^{\rm noneq} \ | \ \lim\limits_{n\to +\infty}\max\limits_{0\leq j\leq n-1}\Delta t_j=0\},
\end{equation}
which is the class of all algorithms from $\chi^{\rm noneq}$ that are based on normal sequences of discretizations (likewise in Riemann integral). We call such a discretization a {\it quasi-uniform} one. In this case, by \eqref{low_b_noneq_1}, we get for any $\bar X\in\chi^{\rm noneq-qu}$
\begin{equation}
\label{lower_err_b_noneq_qu_1}
    \liminf\limits_{n\to +\infty}n\cdot e_n(\bar X)\geq C^{\rm noneq}(\kappa,\sigma,T),
\end{equation}
where
\begin{equation}
    C^{\rm noneq}(\kappa,\sigma,T)=\frac{\kappa\sigma}{\sqrt{12}}e^{-\kappa T}\cdot\Bigl(\int\limits_0^T e^{2\kappa t/3}dt\Bigr)^{3/2}=\frac{\sigma}{\sqrt{12\kappa}}\cdot e^{-\kappa T}\cdot\Bigl[\frac{3}{2}\Bigl(e^{2\kappa T/3}-1\Bigr)\Bigr]^{3/2}.
\end{equation}
The above considerations suggest the following technique of obtaining lower bound in  $\chi^{\rm noneq}$. Let us take any sequence $(m_n)_{n\in\mathbb{N}}$ of positive integers such that 
\begin{equation}
\label{def_m_n}
    \lim\limits_{n\to +\infty}m_n=+\infty, \quad \lim\limits_{n\to +\infty}\frac{m_n}{n}=0,
\end{equation}
and again let us consider arbitrary method $\bar X$ from $\chi^{\rm noneq}$ that is based on sequences $(\varphi_n)_{n\in\mathbb{N}}$ and $(\Delta_n)_{n\in\mathbb{N}}$. By $\hat\Delta=\{\hat\Delta_n\}_{n\geq 1}$ we denote the sequence of discretizations given by $\{\hat\Delta_n\}_{n\geq 1}=\{\Delta_n\cup\Delta_n^{\rm eq\it}\}_{n\geq 1}$, where  $\Delta_n^{\rm eq\it}=\{jT/m_n \ | \ j=0,1,\ldots,m_n\}$. Hence, for all $n\geq 1$,
\begin{equation}
	\hat\Delta_n=\Bigl\{0=\hat t_{0,n}<\hat t_{1,n}<\ldots<\hat t_{k_n,n}=T \ | \ \hat t_{j,n}\in\Delta_n\cup\Delta_n^{\rm eq\it}, j=0,1,\ldots,k_n \Bigr\},
\end{equation}
and, since $\{0,T\}\subseteq\Delta_n\cap\Delta_n^{\rm eq\it}$ for all $n\geq 1$,
\begin{equation}
	\label{EST_KN_UL}
	n\leq k_n\leq n+m_n-1.
\end{equation}
From (\ref{def_m_n}) and (\ref{EST_KN_UL}) we have that
\begin{equation}
	\label{EST_KN_UL_1}
		\lim_{n\to +\infty}\frac{n}{k_n}=1.
\end{equation}
Furthermore, since $\Delta^{\rm eq\it}_n\subset\hat\Delta_n$, we have that
\begin{equation}
	\label{DIAM_HAT_D}
	0\leq\max\limits_{0\leq i\leq k_n-1}(\hat t_{i+1,n}-\hat t_{i,n})\leq\frac{T}{m_n}
\end{equation}
and by \eqref{def_m_n}
\begin{equation}
\label{normal_mesh_mn}
    \lim\limits_{n\to +\infty} \max\limits_{0\leq i\leq k_n-1}(\hat t_{i+1,n}-\hat t_{i,n})=0.
\end{equation}
Hence, $(\hat\Delta_n)_{n\in\mathbb{N}}$ is normal. Let us denote by
\begin{equation}
    \mathcal{\hat N}_n(W)=[W(\hat t_{1,n}),W(\hat t_{2,n}),\ldots, W(\hat t_{k_n,n})].
\end{equation}
Then $\sigma(\bar X_n)\subset\sigma(\mathcal{N}_n(W))\subset\sigma(\mathcal{\hat N}_n(W))$ for all $n\in\mathbb{N}$. Therefore, for all $n\in\mathbb{N}$
\begin{equation}
    e_n(\bar X) \geq e_n(\mathcal{\hat A}(T)),
\end{equation}
where the algorithm $\mathcal{\hat A}(T)=(\mathcal{\hat A}_n(T))_{n\in\mathbb{N}}$ is defined by
\begin{equation}
    \mathcal{\hat A}_n(T)=\mathbb{E}(X(T) \ | \ \mathcal{\hat N}_n(W)).
\end{equation}
By performing analogous computations as for $\mathcal{\bar A}(T)$ we have that for all $n\in\mathbb{N}$
\begin{eqnarray}
 &&n^2\cdot (e_n(\bar X))^2\geq n^2\cdot (e_n(\mathcal{\hat A}(T)))^2\geq\frac{(\kappa\sigma)^2}{12}e^{-2\kappa T}\cdot n^2\cdot\sum\limits_{j=0}^{k_n-1}e^{2\kappa \hat t_j}(\Delta \hat t_j)^3\notag\\
 &&\geq \frac{(\kappa\sigma)^2}{12}e^{-2\kappa T}\cdot\Bigl(\frac{n}{k_n}\Bigr)^2\cdot\Bigl(\sum\limits_{j=0}^{k_n-1}e^{2\kappa \hat t_j/3} \Delta \hat t_j\Bigr)^3.
\end{eqnarray}
Hence, from \eqref{EST_KN_UL_1} and \eqref{normal_mesh_mn} we arrive at
\begin{equation}
\label{lower_err_noneq_1}
    \liminf\limits_{n\to +\infty} n\cdot e_n(\bar X)\geq C^{\rm noneq}(\kappa,\sigma,T).
\end{equation}
Hence, from \eqref{lower_err_b_noneq_qu_1}, \eqref{lower_err_noneq_1} we see that we have the same asymptotic lower bounds in the both classes $\chi^{\rm noneq}$, $\chi^{\rm noneq-qu}$. Moreover, note that
\begin{equation}
\label{ineq_const_1}
    C^{\rm noneq}(\kappa,\sigma,T)<C^{\rm eq}(\kappa,\sigma,T).
\end{equation}
We now show a construction of (asymptotically) optimal algorithm in the   class $\chi^{\rm noneq}$,   for which the error can be reduced by a factor $C^{\rm noneq}/C^{\rm eq}$ when comparing to  $\mathcal{\bar A}^{\rm eq}(T)$. We  consider so called {\it regular sequences of  discretizations} generated by a probability density function $h$, see \cite{SAYL}.  We will use the notation $\bar\Delta_{h}=\{\Delta_{h,n}\}_{n\geq 1}$ for a sequence of discretizations generated by a density $h:[0,T]\to (0,+\infty)$ that is assumed to be continuous and strictly positive. The knots
\begin{displaymath}
	\Delta_{h,n}=\{0=t_{0,n}<t_{1,n}<\ldots<t_{n,n}=T\},
\end{displaymath}
of the $n$th discretization are given by
\begin{equation}
	\label{DEST_H_DISC}
		\int\limits_{0}^{t_{i,n}}h(s)ds=\frac{iT}{n}, \ i=0,1,\ldots,n.
\end{equation}
Hence, by choosing such a density $h$ one gets a whole sequence of discretizations $\bar\Delta_{h}$. One can obtain the sequence of equidistant discretizations by taking $h\equiv 1$. Note that algorithms based on the discretizations generated by $h$ belong to the class $\chi^{\rm noneq\it}$. We will use the notation $\mathcal{\bar A}_h(T)=\{\mathcal{\bar A}_{h,n}(T)\}_{n\geq 1}$ for the algorithm $\mathcal{\bar A}(T)$ algorithm based on the sequence of discretizations $\bar\Delta_h$. 

Note that by \eqref{DEST_H_DISC} and the mean value theorem we have that for all $n\in\mathbb{N}$, $i=0,1,\ldots,n-1$ there exists $\xi_{j,n}\in [t_{j,n},t_{j+1,n}]$ such that
\begin{equation}
    \frac{T}{n}=\int\limits_{t_{j,n}}^{t_{j+1,n}} h(s)ds=h(\xi_{j,n}) \cdot\Delta t_{j,n}.
\end{equation}
This, gives, in particular, 
\begin{equation}
	\label{NORMAL_DISC_H}
		Tn^{-1}||h||_{\infty}\leq\max\limits_{0\leq j\leq n-1}(t_{j+1,n}-t_{j,n})\leq Tn^{-1}||1/h||_{\infty},
\end{equation}
since $\max\{\|h\|_{\infty},\|1/h\|_{\infty}\}<+\infty$. 
From \eqref{est_low_upp_sj_1}
we have that
\begin{equation}
    n^2\cdot\sum\limits_{j=0}^{n-1}\mathbb{E}(S_j^2)\leq\frac{1}{12}\sum_{j=0}^{n-1}e^{2\kappa t_{j+1}} \cdot (n \cdot \Delta t_j)^2 \cdot \Delta t_j=\frac{T^2}{12}\sum_{j=0}^{n-1}\frac{e^{2\kappa t_{j+1}}}{h^2(\xi_j)}\cdot\Delta t_j,
\end{equation}
and analogously
\begin{equation}
     n^2\cdot\sum\limits_{j=0}^{n-1}\mathbb{E}(S_j^2)\geq\frac{1}{12}\sum_{j=0}^{n-1}e^{2\kappa t_{j}} \cdot (n \cdot \Delta t_j)^2 \cdot \Delta t_j=\frac{T^2}{12}\sum_{j=0}^{n-1}\frac{e^{2\kappa t_{j}}}{h^2(\xi_j)}\cdot\Delta t_j.
\end{equation}
For each $n\in\mathbb{N}$ let us take  any sequence of intermediate points $\eta_{j,n}\in [t_{j,n},t_{j+1,n}]$, $j=0,1,\ldots,n-1$. Then
\begin{eqnarray}
    \sum_{j=0}^{n-1}\frac{e^{2\kappa \eta_j}}{h^2(\xi_j)}\cdot\Delta t_j=\sum_{j=0}^{n-1}\frac{e^{2\kappa \xi_{j}}}{h^2(\xi_j)}\cdot\Delta t_j+\Biggl(\sum_{j=0}^{n-1}\frac{e^{2\kappa \eta_j}}{h^2(\xi_j)}\cdot\Delta t_j-\sum_{j=0}^{n-1}\frac{e^{2\kappa \xi_{j}}}{h^2(\xi_j)}\cdot\Delta t_j\Biggr).
\end{eqnarray}
Since  for all $t,s\in [0,T]$
\begin{equation}
\label{Lip_e_k}
    |e^{2\kappa t}-e^{2\kappa s}|\leq 2\kappa e^{2\kappa T}|t-s|,
\end{equation}
we have by \eqref{NORMAL_DISC_H}
\begin{eqnarray}
  &&\Biggl|\sum_{j=0}^{n-1}\frac{e^{2\kappa \eta_j}}{h^2(\xi_j)}\cdot\Delta t_j-\sum_{j=0}^{n-1}\frac{e^{2\kappa \xi_{j}}}{h^2(\xi_j)}\cdot\Delta t_j\Biggl|  \leq \sum_{j=0}^{n-1}\frac{|e^{2\kappa\eta_j}-e^{2\kappa\xi_j}|}{h^2(\xi_j)}\cdot\Delta t_j\notag\\
  &&\leq 2\kappa e^{2\kappa T}\|1/h\|^2_{\infty}\sum_{j=0}^{n-1} |\eta_j-\xi_j| \cdot \Delta t_j\leq 2\kappa e^{2\kappa T}\|1/h\|^2_{\infty} T \cdot\max\limits_{0\leq j\leq n-1}\Delta t_j \leq 2\kappa e^{2\kappa T}\|1/h\|^3_{\infty} T^2 n^{-1}.
\end{eqnarray}
Hence, 
\begin{equation}
    \lim\limits_{n\to +\infty}\sum_{j=0}^{n-1}\frac{e^{2\kappa \eta_{j}}}{h^2(\xi_j)}\cdot\Delta t_j=\lim\limits_{n\to +\infty}\sum_{j=0}^{n-1}\frac{e^{2\kappa \xi_{j}}}{h^2(\xi_j)}\cdot\Delta t_j=\int\limits_0^T\frac{e^{2\kappa t}}{h^2(t)}dt,
\end{equation}
and
\begin{equation}
\label{asympt_err_WWO_h}
    \lim\limits_{n\to +\infty}n\cdot e_n(\mathcal{\bar A}_h(T))=\kappa\sigma e^{-\kappa T}\cdot\frac{T}{\sqrt{12}}\Bigl(\int\limits_0^T\frac{e^{2\kappa t}}{h^2(t)}dt\Bigr)^{1/2}.
\end{equation}
In order to find the optimal density $h_0$ we  show that for fixed $g\in C([0,T], (0+\infty))$ the function
\begin{equation}
    h_0(t)=T\cdot\frac{(g(t))^{1/3}}{\int\limits_0^T (g(t))^{1/3}dt}
\end{equation}
minimizes the functional
\begin{equation}
    F(h)=\int\limits_0^T \frac{g(t)}{h^2(t)}dt
\end{equation}
among all sampling densities $h\in C([0,T], (0,+\infty))$ such that $\displaystyle{\int\limits_0^T h(t)dt = T}$. By the H\"older inequality for any $h\in C([0,T], (0,+\infty))$ with $\displaystyle{\int\limits_0^T h(t)dt = T}$ we have that
\begin{eqnarray}
\label{Min_F_h_1}
    &&\int\limits_0^T (g(t))^{1/3}dt= \int\limits_0^T\frac{(g(t))^{1/3}}{h^{2/3}(t)}\cdot h^{2/3}(t)dt\notag\\
    &&\leq \Biggl(\int\limits_0^T\Bigl(\frac{(g(t))^{1/3}}{h^{2/3}(t)}\Bigr)^3 dt\Biggr)^{1/3}\cdot\Biggl(\int\limits_0^T (h^{2/3}(t))^{3/2}dt\Biggr)^{2/3} = T^{2/3}\cdot (F(h))^{1/3},
\end{eqnarray}
and
\begin{equation}
    F(h_0)=\frac{1}{T^2}\Biggl(\int\limits_0^T (g(t))^{1/3}dt\Biggr)^3.
\end{equation}
Therefore
\begin{equation}
    \Biggl(\int\limits_0^T (g(t))^{1/3}dt\Biggr)^3\leq T^2\cdot\inf\limits_{h\in C([0,T]), \ h>0, \ \int\limits_0^T h(t)dt=T}F(h)\leq T^2 F(h_0)=\Biggl(\int\limits_0^T (g(t))^{1/3}dt\Biggr)^3,
\end{equation}
and 
\begin{equation}
    \inf\limits_{h\in C([0,T]), \ h>0, \ \int\limits_0^T h(t)dt=T}F(h)=F(h_0)=\frac{1}{T^2} \Biggl(\int\limits_0^T (g(t))^{1/3}dt\Biggr)^3.
\end{equation}
Hence, the optimal sampling density, that minimized the asymptotic constant in \eqref{asympt_err_WWO_h}, is of the form
\begin{equation}
\label{def_opt_h0}
    h_0(t)=T\cdot\frac{e^{2\kappa t/3}}{\int\limits_0^T e^{2\kappa t/3}dt}=\frac{2\kappa T e^{2\kappa t/3}}{3(e^{2\kappa T/3}-1)}.
\end{equation}
Moreover, from the fact that equality in the H\"older inequality \eqref{Min_F_h_1} holds only in the case when $C_1 g/h^2$ and $C_2 h$ are equal almost everywhere for some $C_1,C_2\geq 0$ we get the uniqueness of $h_0$. 
\newline
We thus get
\begin{equation}
    \label{asympt_err_WWO_h_0}
    \lim\limits_{n\to +\infty}n\cdot e_n(\mathcal{\bar A}_{h_0}(T))=C^{\rm noneq}(\kappa,\sigma,T).
\end{equation}
Combining \eqref{lower_err_noneq_1}, \eqref{asympt_err_WWO_h_0} we obtain
\begin{equation}
    \lim\limits_{n\to +\infty} n\cdot\inf\limits_{\bar X\in\chi^{\rm noneq}}e_n(\bar X)=\lim\limits_{n\to +\infty}n\cdot e_n(\mathcal{\bar A}_{h_0}(T))=C^{\rm noneq}(\kappa,\sigma,T)>0,
\end{equation}
and the algorithm $\mathcal{\bar A}_{h_0}(T)$ is asymptotically optimal in $\chi^{\rm noneq}$ (and, in fact, also in the smaller class $\chi^{\rm noneq-qu}$). Moreover, from \eqref{DEST_H_DISC}, \eqref{def_opt_h0} we see that the optimal sampling points are of the form
\begin{equation}
    t_{j,n}=\frac{3}{2\kappa}\ln\Bigl[\frac{j}{n}(e^{2\kappa T/3}-1)+1\Bigr], \quad j=0,1,\ldots,n.
\end{equation}
\begin{rem}
    There is a vast, and still growing, literature on the strong optimal approximation of solutions of SDEs in various computational settings and under different regularity assumptions, see, for example, \cite{MGHAB}, \cite{PMPP1}, \cite{HMR3}, \cite{PP9}.
\end{rem}
\begin{rem}
    \label{rem_cond_X}
    In the general case the conditional expectation \eqref{def_A_bar} is hard, or even impossible, to compute explicitly. Hence, in the case of nonlinear SDEs other methods of construction of asymptotically optimal algorithms are used, see \cite{MGHAB} for the pure Wiener case and \cite{AKPP} where jump-diffusion SDEs are considered.
\end{rem}
\section{Exercises}
\begin{itemize}
    \item [1.] Show \eqref{sol_X_OU_1}.
    \item [2.] Give a proof of the  inequality \eqref{ineq_const_1}.
    \item  [3.] Justify the inequality \eqref{Lip_e_k}.
    \item [4.] (\textit{Optimal one-point approximation of Ornstein-Uhlenbeck process with jumps}) Using the approach and proof technique described in this chapter provide optimal rate of convergence in the classes $\chi^{\rm eq}$, $\chi^{\rm noneq}$ for the problem of approximating $X(T)$ where $X$ is the unique solution of
    \begin{equation}
        	\label{OU_T_SDE2}
		\left\{ \begin{array}{ll}
			dX(t)=\kappa(\mu- X(t))dt+\sigma dW(t)+\gamma dN(t), &t\in [0,T], \\
			X(0)=x_0, 
		\end{array}\right.
    \end{equation}
    under the assumption that we have access to  values of $W,N$ at finite number of discretization points, $\kappa>0$, $\mu\in\mathbb{R}$, $\sigma,\gamma\geq 0$,  $\sigma^2+\gamma^2\geq 0$, and with the positive  intensity $\lambda$ of the homogeneous Poisson process $N$.
    \item [5.] Compute for $t\in [0,T]$ the conditional expectation
    \begin{equation}
       \bar X_n(t)= \mathbb{E}(X(t) \ | \ W(t_1), W(t_2),\ldots, W(t_n)),
    \end{equation}
    where $X$ is 
    \begin{itemize}
        \item [(a)] the geometric Brownian motion \eqref{gBm_jumps_def} with $c=0$,
        \item [(b)] the Ornstein-Uhlenbeck process \eqref{sol_X_OU_3},
    \end{itemize}
    and $0=t_0<t_1<\ldots t_n=T$ is an arbitrary fixed discretization of $[0,T]$. Note that it can be shown that $\bar X_n$ provide optimal method (in a suitable sense) of reconstruction whole trajectories of $X$ based on the discrete data about $W$. \\
    Hint. Use Lemma \ref{WN_bridge}.
\end{itemize}
\chapter{Weak approximation of solutions of SDEs with application to option pricing}

For the functions
\begin{eqnarray}
	&&a:[0,T]\times\mathbb{R}^d\to \mathbb{R}^d,\\
	&&b:[0,T]\times\mathbb{R}^d\to \mathbb{R}^{d\times m_W},\\
	&&c:[0,T]\times\mathbb{R}^d\to \mathbb{R}^{d\times m_N},
\end{eqnarray}
we assume that
\begin{itemize}
	\item [(B)] there exists $K\in (0,+\infty)$ such that for $f\in\{a,b,c\}$ and for all $s,t\in [0,T]$, $x,y\in\mathbb{R}^d$
	\begin{equation}
		\|f(t,x)-f(s,y)\|\leq K(|t-s|+\|x-y\|),
	\end{equation}
	\item [(C)] the intensity functions $\lambda_j:[0,T]\to (0,+\infty)$, $j=1,2,\ldots,m_N$, are Borel measurable and bounded.
\end{itemize}
where for $f=a$ the norm $\|\cdot\|$ is the euclidean norm $\|\cdot\|_2$, while for $f\in\{b,c\}$ the norm $\|\cdot\|$ is the Frobenius norm $\|\cdot\|_F$. From now we will denote the both norms by $\|\cdot\|$ and the notion will be clear from the context.
We assume that $X$ is the unique strong solution of the following SDE
\begin{equation}
	\label{SDE_2}
		\left\{ \begin{array}{ll}
			dX(t)=a(t,X(t))dt+b(t,X(t))dW(t)+c(t,X(t-))dN(t), &t\in [0,T], \\
			X(0)=\xi.
		\end{array}\right.
\end{equation}
In many applications (for example, in option pricing) we need to compute
\begin{equation}
\label{efx_1}
	\mathbb{E}(F(X))
\end{equation}
where $F:D([0,T])\to\mathbb{R}$ is a given measurable functional. For example, in the case of european put option
\begin{equation}
	F(X)=h(X(T))=(K_0-X(T))^+,
\end{equation}
where $h(x)=(K_0-x)^+$ for $x\geq 0$ and $h(x)=K_0$ for $x<0$, while for Asian-type option
\begin{equation}
	F(X)=\phi\Bigl(\frac{1}{T}\int\limits_0^T \psi( X(t))dt\Bigr),
\end{equation}
where $\phi,\psi$ are given functions and, this time,
\begin{equation}
\label{asian_opt_1}
	F(w)=\phi\Bigl(\frac{1}{T}\int\limits_0^T \psi( w(t))dt\Bigr), \ w\in D([0,T]).
\end{equation}
In general, the value of \eqref{efx_1} cannot be computed explicitely. Hence, in order to have a good approximation of \eqref{efx_1} we can use Monte Carlo methods.

Firstly, we consider european option pricing, i.e., we consider approximation of 
\begin{equation}
	\mathbb{E}(h(X(T))),
\end{equation} 
where for $h:\mathbb{R}^d\to\mathbb{R}$ we assume that:
\begin{itemize}
	\item [(i)]there exists $D>0$ such that for all  $x\in\mathbb{R}^d$ it holds
	\begin{equation}
	\label{h_bound}
		|h(x)|\leq D,
	\end{equation}
	\item [(ii)] there exists $L>0$ such that for all $x,y\in\mathbb{R}^d$ it holds
	\begin{equation}
		|h(x)-h(y)|\leq L\|x-y\|.
	\end{equation}
\end{itemize}
Under the above assumptions the following integral is finite
\begin{equation}
	\mathbb{E}(h(X(T)))=\int\limits_{\mathbb{R}^d}h(x)d\mathbb{P}^{X(T)}(x),
\end{equation}
where $\mathbb{P}^{X(T)}$ is the law of random vector $X(T)$. Suppose that we know the law of $X(T)$, then we can sample  $M$ independent random vectors $\bar X_1(T),\ldots,\bar X_M(T)$ form  $\mathbb{P}^{X(T)}$ and take
\begin{equation}
	\frac{1}{M}\sum_{l=1}^M h(\bar X_l(T))
\end{equation}
in order to approximate  $\mathbb{E}(h(X(T)))$. However, the distribution of $X(T)$ is known only in some particular cases (for example, in the simplest Black-Scholes model). So in the case of general equation \eqref{SDE_2} we have to proceed in different way. Namely, we use the Euler-Maruyama approximation of $X$.

For the convenience of the reader we recall here the classical Euler-Maruyama scheme. Let $n\in\mathbb{N}$ and take an arbitrary discretization  $0=t_0<t_1<\ldots<t_{n}=T$ of the interval $[0,T]$. Set
\begin{equation}
	X^E_n(0)=\xi,
\end{equation}
and for $k=0,1,\ldots,n-1$
\begin{equation}
	X^E_n(t_{k+1})=X^E_n(t_{k})+a(U_{k,n}^E)\cdot\Delta t_k+b(U_{k,n}^E)\cdot\Delta W_k+c(U_{k,n}^E)\cdot\Delta N_k,
\end{equation}
where
\begin{equation} 
	U_{k,n}^E=(t_k,X^E_n(t_{k})),
\end{equation}	
\begin{eqnarray}
	&&\Delta t_k = t_{k+1}-t_k,\notag\\
	&&\Delta W_k=[\Delta W^1_k,\ldots,\Delta W^{m_W}_k]^T,\notag\\
	&&\Delta N_k=[\Delta N^1_k,\ldots,\Delta N^{m_N}_k]^T,
\end{eqnarray}
and
\begin{equation}
	\Delta Z_k^j=Z^j(t_{k+1})-Z^j(t_{k}), \ Z\in\{N,W\},
\end{equation}
\begin{equation}
	b(U_{k,n}^E)\cdot\Delta W_k=\Bigl(\sum\limits_{j=1}^{m_W}b_{ij}(U_{k,n}^E)\cdot\Delta W^j_k\Bigr)_{i=1,2,\ldots,d},
\end{equation}
\begin{equation}
	c(U_{k,n}^E)\cdot\Delta N_k=\Bigl(\sum\limits_{j=1}^{m_N}c_{ij}(U_{k,n}^E)\cdot\Delta N^j_k\Bigr)_{i=1,2,\ldots,d}.
\end{equation}
In addition let us fix $M\in\mathbb{N}$ - the number of simulated trajectories. Simulation of $M$ trajectories via Euler scheme $X_n^E$ goes as follows. Let $\xi_1,\ldots,\xi_M$ be independent random vectors with the same law as $\xi$.
For $k=0,1,\ldots,n-1$ and $l=1,2,\ldots,M$ we set
\begin{equation}
	\bar X^E_{n,l}(0)=\xi_l,
\end{equation}
and for $k=0,1,\ldots,n-1$
\begin{equation}
	\bar X^E_{n,l}(t_{k+1})=\bar X^E_{n,l}(t_{k})+a(\bar U_{k,n,l}^E)\cdot\Delta t_k+(\Delta t_k)^{1/2}\cdot b(\bar U_{k,n,l}^E)\cdot Z_{k,l}+c(\bar U_{k,n,l}^E)\cdot U_{k,l},
\end{equation}
with $\bar U_{k,n,l}^E=(t_k,\bar X^E_{n,l}(t_{k}))$ and
\begin{eqnarray}
	Z_{k,l}=[Z_{k,l}^1,\ldots,Z_{k,l}^{m_W}]^T,\\
	U_{k,l}=[U_{k,l}^1,\ldots,U_{k,l}^{m_N}]^T,
\end{eqnarray}
where 
\begin{itemize}
	\item $(Z^j_{k,l})_{k=0,\ldots,n-1, l=1,\ldots,M, j=1,\ldots,m_W}$ is an i.i.d sequence of random variables with $Z^1_{1,1}\sim N(0,1)$, 
	\item $(U_{k,l}^j)_{k=0,\ldots,n-1, l=1,\ldots,M, j=1,\ldots,m_N}$ is a sequence of independent random variables, where $\displaystyle{U_{k,l}^j\sim \hbox{Poiss}\Bigl(\int\limits_{t_k}^{t_{k+1}}\lambda_j(s)ds\Bigr)}$, 
	\item the collection $\displaystyle{\{\xi_1,\ldots,\xi_M,(Z^j_{k,l})_{k=0,\ldots,n-1, l=1,\ldots,M, j=1,\ldots,m_W},(U_{k,l}^j)_{k=0,\ldots,n-1, l=1,\ldots,M, j=1,\ldots,m_N}\}}$ consists of independent random variables.
\end{itemize}
Then, componentwise,
\begin{equation}
	\bar X^E_{n,i,l}(0)=\xi_{i,l} ,
\end{equation}
\begin{eqnarray}
	\bar X^E_{n,i,l}(t_{k+1})=\bar X^E_{n,i,l}(t_{k})&+&a_i(\bar U_{k,n,l}^E)\cdot\Delta t_k\notag\\
	&+&(\Delta t_k)^{1/2}\cdot\sum\limits_{j=1}^{m_W}b_{ij}(\bar U_{k,n,l}^E)\cdot Z^j_{k,l}\notag\\
	&+&\sum\limits_{j=1}^{m_N}c_{ij}(\bar U_{k,n,l}^E)\cdot U^j_{k,l},
\end{eqnarray}
for $k=0,1,\ldots,n-1$, $l=1,2,\ldots,M$, $i=1,2,\ldots,d$.
\\
Note that
\begin{equation}
	\bar X^E_{n,1}(T), \ \bar X^E_{n,2}(T),\ldots, \ \bar X^E_{n,M}(T)
\end{equation}
are independent random vectors with the same law as $X_n^E(T)$. Moreover, we can simulate them. Therefore, as an approximation of  $\displaystyle{\mathbb{E}(h(X(T)))}$ we take
\begin{equation}
	\frac{1}{M}\sum_{l=1}^M h(\bar X^E_{n,l}(T)).
\end{equation}
We now estimate the error 
\begin{equation}
	\Biggl(\mathbb{E}\Bigl|\mathbb{E}(h(X(T)))-\frac{1}{M}\sum_{l=1}^M h(\bar X^E_{n,l}(T))\Bigl|^2\Biggr)^{1/2}.
\end{equation}
By SLLN we get for all $n\in\mathbb{N}$
\begin{equation}
	\frac{1}{M}\sum_{l=1}^M h(\bar X^E_{n,l}(T))\to\mathbb{E}(h(X^E_n(T))))=\int\limits_{\mathbb{R}^d}h(x)d\mathbb{P}^{X^E_n(T)}(x) \ \hbox{when} \ M\to+\infty
\end{equation}
almost surely. By \eqref{wh_mcn_err} we get for any $n,K\in\mathbb{N}$ that
\begin{equation}
	\Biggl(\mathbb{E}\Bigl|\mathbb{E}(h(X^E_n(T))))-\frac{1}{M}\sum_{l=1}^M h(\bar X^E_{n,l}(T))\Bigl|^2\Biggr)^{1/2}=\frac{C_n(h)}{\sqrt{M}},
\end{equation}
where
\begin{equation}
0\leq C^2_n(h)=\int\limits_{\mathbb{R}^d}|h(x)|^2d\mathbb{P}^{X^E_n(T)}(x)-\Bigl(\int\limits_{\mathbb{R}^d}h(x)d\mathbb{P}^{X^E_n(T)}(x)\Bigr)^2\leq D^2.
\end{equation}
In addition, by \eqref{mnsqrt_err_2}
\begin{eqnarray}
	&&\mathbb{E}\Bigl|\mathbb{E}(h(X(T)))-\frac{1}{M}\sum_{l=1}^M h(\bar X^E_{n,l}(T))\Bigl|^2\notag\\
	&&= \Bigl|\mathbb{E}(h(X(T)))-\mathbb{E}(h(X^E_n(T)))\Bigl|^2+\mathbb{E}\Bigl|\mathbb{E}(h(X^E_n(T)))-\frac{1}{M}\sum_{l=1}^M h(\bar X^E_{n,l}(T))\Bigl|^2\notag\\
	&&\leq \mathbb{E}|h(X(T))-h(X^E_n(T))|^2+D^2/M\notag\\
	&&\leq L^2\mathbb{E}\|X(T)-X^E_n(T)\|^2+D^2/M\notag\\
	&&\leq C_1\max\limits_{0\leq k\leq n-1}\Delta t_k+D^2/M.
\end{eqnarray}
So on the uniform mesh we get
\begin{equation}
\Biggl(\mathbb{E}\Bigl|\mathbb{E}(h(X(T)))-\frac{1}{M}\sum_{l=1}^M h(\bar X^E_{n,l}(T))\Bigl|^2\Biggr)^{1/2}=O(n^{-1/2}+M^{-1/2}).
\end{equation}
We now turn to the Asian  options of the form \eqref{asian_opt_1}, where we assume that the functions  $\psi:\mathbb{R}^d\to\mathbb{R}^d$, $\phi:\mathbb{R}\to\mathbb{R}_+$ are globally Lipschitz and $\phi$ is bounded. Then we have that there exists $L_1\in (0,+\infty)$ such that for all $w_1,w_2\in D([0,T])$ 
\begin{equation}
	|F(w_1)-F(w_2)|\leq L_1\int\limits_0^T \|w_1(t)-w_2(t)\|dt.
\end{equation}
Since $\phi$ is bounded, the following integral is finite
\begin{equation}
	\mathbb{E}(F(X))=\int\limits_{D([0,T])}F(w)d\mathbb{P}^{X}(w),
\end{equation}
where $\mathbb{P}^{X}$ is the law of random element $X:\Omega\to D([0,T])$. If we could sample  $M$ independent random elements $\bar X_1,\ldots,\bar X_M$ form  $\mathbb{P}^{X}$ then we could take
\begin{equation}
	\frac{1}{M}\sum_{l=1}^M F(\bar X_l)
\end{equation}
in order to approximate  $\mathbb{E}(F(X))$. However, this is not possible even in the case $X=W$. Hence, again we have to use the Euler-Maruyama scheme for the global approximation of $X$. In this case we  take
\begin{equation}
	\hat X^E_{n}(t)=\sum\limits_{k=0}^{n-1} X^E_{n}(t_k)\cdot \mathbf{1}_{[t_k,t_{k+1})}(t), \  t\in [0,T),
\end{equation}
and $\hat X^E_{n}(T)=X^E_{n}(T)$, where $\hat X^E_n\approx X$. As an approximation of $M$ trajectories of $X$ we take
\begin{equation}
	\hat X^E_{n,l}(t)=\sum\limits_{k=0}^{n-1}\bar X^E_{n,l}(t_k)\cdot \mathbf{1}_{[t_k,t_{k+1})}(t), \  t\in [0,T),
\end{equation}
and $\hat X^E_{n,l}(T)=\bar X^E_{n,l}(T)$ for $l=1,2,\ldots,M$. Note that
\begin{equation}
	\hat X^E_{n,1}, \ \hat X^E_{n,2},\ldots, \ \hat X^E_{n,M}
\end{equation}
are independent $D([0,T])$-valued random elements with the same law as $\hat X_n^E$. As an approximation of $\displaystyle{\mathbb{E}(F(X))}$ we take
\begin{equation}
	\frac{1}{M}\sum_{l=1}^M F(\hat X^E_{n,l}),
\end{equation}
where
\begin{equation}
	F(\hat X^E_{n,l})=\phi\Bigl(\frac{1}{T}\int\limits_0^T \psi( \hat X^E_{n,l}(t))dt\Bigr)=\phi\Bigl(\frac{1}{T}\sum_{k=0}^n\psi(\bar X_{n,l}^E(t_k))\Delta t_k\Bigr).
\end{equation}
We now estimate the mean square error
\begin{equation}
	\Biggl(\mathbb{E}\Bigl|\mathbb{E}(F(X))-\frac{1}{M}\sum_{l=1}^M F(\hat X^E_{n,l})\Bigl|^2\Biggr)^{1/2}.
\end{equation}
Again, by SLLN we have for all $n\in\mathbb{N}$
\begin{equation}
	\frac{1}{M}\sum_{l=1}^M F(\hat X^E_{n,l})\to\mathbb{E}(F(\hat X^E_n)))=\int\limits_{D([0,T])}F(w)d\mathbb{P}^{\hat X^E_n}(w) \ \hbox{as} \ M\to+\infty
\end{equation}
almost surely. Furthemore, by  \eqref{wh_mcn_err} we get for any $M\in\mathbb{N}$ that
\begin{equation}
	\Biggl(\mathbb{E}\Bigl|\mathbb{E}(F(\hat X^E_n)))-\frac{1}{M}\sum_{l=1}^M F(\hat X^E_{n,l})\Bigl|^2\Biggr)^{1/2}=\frac{C_n(h)}{\sqrt{M}},
\end{equation}
where
\begin{equation}
0\leq C_n^2(h)=\int\limits_{D([0,T])}|F(w)|^2d\mathbb{P}^{\hat X^E_n}(w)-\Bigl(\int\limits_{D([0,T])}F(w)d\mathbb{P}^{\hat X^E_n}(w)\Bigr)^2\leq D^2.
\end{equation}
Hence, by Theorem \ref{err_euler_3}
\begin{eqnarray}
	&&\mathbb{E}\Bigl|\mathbb{E}(F(X))-\frac{1}{M}\sum_{l=1}^M F(\hat X^E_{n,l})\Bigl|^2\notag\\
	&&= \Bigl|\mathbb{E}(F(X))-\mathbb{E}(F(\hat X^E_n))\Bigl|^2+\mathbb{E}\Bigl|\mathbb{E}(F(\hat X^E_n))-\frac{1}{M}\sum_{l=1}^M F(\hat X^E_{n,l})\Bigl|^2\notag\\
	&&\leq \mathbb{E}|F(X)-F(\hat X^E_n)|^2+D^2 M^{-1}\leq L_1^2 \mathbb{E}\Biggl(\int\limits_0^T\|X(t)-\hat X^E_n(t)\| dt\Biggr)^{2}+D^2 M^{-1}\notag\\
&&\leq L_1^2 T\mathbb{E}\int\limits_0^T\|X(t)-\hat X^E_n(t)\|^2 dt+D^2 M^{-1}\leq C_1\max\limits_{0\leq k\leq n-1}(\Delta t_k)+D^2 M^{-1},
\end{eqnarray}
Again on the uniform mesh we obtain
\begin{equation}
\Biggl(\mathbb{E}\Bigl|\mathbb{E}(F(X))-\frac{1}{M}\sum_{l=1}^M F(\hat X^E_{n,l})\Bigl|^2\Biggr)^{1/2}=O(n^{-1/2}+M^{-1/2}).
\end{equation}
Therefore, in order to have a good balance between discretization and Monte  Carlo errors it is reasonable to take $M=O(n)$.
\section{Non-asymptotic (?) confidence intervals}
By the Hoeffding inequality we get for the European  option price $\mathbb{E}(h(X(T)))$, with $h$ satisfying \eqref{h_bound}, that
\begin{equation}
	\mathbb{P}\Biggl(\Bigl|\mathbb{E}(h(X^E_n(T)))-\frac{1}{M}\sum_{l=1}^M h(\bar X^E_{n,l}(T))\Bigl|>\varepsilon\Biggr)\leq 2\exp\Bigl(-\frac{\varepsilon^2 M}{2\|h\|^2_{\infty}}\Bigr).
\end{equation}
Using analogous argumentation as for \eqref{H_int_1} and \eqref{H_int_2} we obtain for all $\delta\in (0,1)$, $n,M\in\mathbb{N}$ that
\begin{equation}
	\mathbb{P}\Biggl(\Bigl|\mathbb{E}(h(X^E_n(T)))-\frac{1}{M}\sum_{l=1}^M h(\bar X^E_{n,l}(T))\Bigl|>C(\delta)\cdot M^{-1/2}\Biggr)\leq\delta,
\end{equation}
where
\begin{equation}
	C(\delta)=\sqrt{2\ln(2/\delta)}\cdot D.	
\end{equation}
(Note that in the case of European pay-off $h$ the bound $D$ might be known explicitely.) Hence, we get that for all $\delta\in (0,1)$, $n,M\in\mathbb{N}$
\begin{equation}
	\mathbb{E}(h(X^E_n(T)))\in\Bigl[\frac{1}{M}\sum_{l=1}^M h(\bar X^E_{n,l}(T))-\frac{C(\delta)}{\sqrt{M}},\frac{1}{M}\sum_{l=1}^M h(\bar X^E_{n,l}(T))+\frac{C(\delta)}{\sqrt{M}}\Bigr]
\end{equation}
with the probability at least $1-\delta$.

Recall that for non-negative random variables $X,Y$ we have for all $\varepsilon>0$ that
\begin{equation}
	\mathbb{P}(X+Y>\varepsilon)\leq \mathbb{P}(X>\varepsilon/2)+\mathbb{P}(Y>\varepsilon/2).
\end{equation}
Therefore, by the Hoeffding and Markov inequalities
\begin{eqnarray}
	&&\mathbb{P}\Biggl(\Bigl|\mathbb{E}(h(X(T)))-\frac{1}{M}\sum_{l=1}^M h(\bar X^E_{n,l}(T))\Bigl|>\varepsilon\Biggr)\notag\\
	&&\leq\mathbb{P}\Biggl(\Bigl|\mathbb{E}(h(X(T)))-\frac{1}{M}\sum_{l=1}^M h(\bar X_{l}(T))\Bigl|>\varepsilon/2\Biggr)\notag\\ 
	&&\quad\quad +\mathbb{P}\Biggl(\frac{L}{M}\sum_{l=1}^M \|\bar X_{l}(T)-\bar X^E_{n,l}(T)\|>\varepsilon/2\Biggr)\notag\\
	&&\leq 2\exp\Bigl(-\frac{(\varepsilon/2)^2 M}{2\|h\|^2_{\infty}}\Bigr)+\frac{2L}{M\varepsilon}\sum_{l=1}^M \mathbb{E}\|\bar X_{l}(T)-\bar X^E_{n,l}(T)\|\notag\\
	&&\leq 2\exp\Bigl(-\frac{(\varepsilon/2)^2 M}{2\|h\|^2_{\infty}}\Bigr)+\frac{2L}{\varepsilon} \Bigl(\mathbb{E}\|X(T)-X^E_{n}(T)\|^2\Bigr)^{1/2}\notag\\
	&&\leq 2\exp\Bigl(-\frac{(\varepsilon/2)^2 M}{2\|h\|^2_{\infty}}\Bigr)+\frac{2LC_1}{\varepsilon}\max\limits_{0\leq k\leq n-1}(\Delta t_k)^{1/2}.
\end{eqnarray}
On the uniform mesh we obtain
\begin{eqnarray}
	&&\mathbb{P}\Biggl(\Bigl|\mathbb{E}(h(X(T)))-\frac{1}{M}\sum_{l=1}^M h(\bar X^E_{n,l}(T))\Bigl|>2C(\delta)\cdot M^{-1/2}\Biggr)\notag\\
	&&\leq \delta+\frac{LC_1}{\sqrt{2}D}\frac{\sqrt{M}}{\sqrt{n\ln(2/\delta)}}.
\end{eqnarray}
\section{Implementation issues}
Below we present the implementation of the Euler scheme for the Merton model, together with European option pricing via Monte Carlo simulations performed on multiple threads.

\begin{lstlisting}[language=Python]
import numpy as np
import time
from multiprocessing import Pool
from multiprocessing import cpu_count

K = 80
T = 1
S_0 = 100
r = 0.08
sigma = 0.4
lam = 5.1012
c_0 = -0.12

def a(t, x):
    return (r - c_0 * lam) * x

def b(t, x):
    return sigma * x

def c(t, x):
    return c_0 * x


"""
European call
"""

def H(x, K):
    tmp = x - K
    tmp = np.where(tmp <0, 0.0, tmp)

    return tmp

"""
European put
"""

def P(x, K):
    tmp = K - x
    tmp = np.where(tmp < 0, 0.0, tmp)

    return tmp

"""
Euler method on uniform grid
"""

def Euler_path(params):
    np.random.seed(seed=None)
    eta = params["eta"]
    lam = params["lam"]
    T = params["T"]
    n = params["n"]

    h = float(T / n)
    sq_h = np.sqrt(h)
    E = float(eta)
    t = 0.0
    lh = lam * h

    for i in range(1, n+1):
        Normals = np.random.normal(0, 1)
        Poiss = np.random.poisson(lh)
        E += a( t, E ) * h + b( t, E ) * sq_h * Normals + c( t,  E ) * Poiss
        t = i * h

    return E

#######################
if __name__ == '__main__':
    # M - number of simulated trajectories
    # K - option strike

    #M = 50001
    # n = 10231

    M = 3901
    n = 3702

    val_E = []
    val_EV = []

    num_proc = cpu_count()
    print("CPUs' number: ", num_proc)

    params = []
    for i in range(M):
        params.append({"eta": S_0,
                       "lam": lam,
                       "T": T,
                       "n": n})

    threads = num_proc - 1

    p = Pool(threads)
    print("Number of CPUs' used:", threads)

    start = time.time()

    results = p.map_async(Euler_path, params)
    E_T = np.asarray(results.get())

    end = time.time()
    print("Computation time : ", end - start, " s")

    val_E = H( E_T, K )

    mean_E = np.mean(val_E)
    price_E = np.exp(-1.0 * r * T) * mean_E

    val_EV = np.power(val_E - mean_E, 2)
    EV = M * np.mean(val_EV) / (M-1)
    tmp_I = 2.0 * np.exp(-1.0 * r * T) * np.sqrt(EV / M)
    I1 = price_E - tmp_I
    I2 = price_E + tmp_I

    print("European call option price = ", price_E)
    print("Empirical variance:", EV)
    print("[",I1,I2,"]")
\end{lstlisting}

\chapter{Parameter estimation and forecasting for SDEs-based models}
The topic of parameter estimation for SDEs-based models is both very important and popular in the literature. This is due to the many applications in finances, see, for example, \cite{Iacus1}, \cite{IAYO}, \cite{saso1}.
%
%
\section{Introductory example - parameter estimation in the geometric Brownian motion model}
In this section we assume that the price $(S(t))_{t\in [0,T]}$ of the risky asset satisfies the Black-Scholes equation 
\begin{equation}
\label{BS_eq}
	dS(t)=\mu S(t)dt+\sigma S(t)dW(t), \ t\in [0,T], \ S(0)\in (0,+\infty),
\end{equation}
with $\mu\in\mathbb{R}$, $\sigma>0$, $S(0)>0$. Hence,
\begin{equation}
	S(t)=S(0)\exp\Bigl((\mu-\frac{1}{2}\sigma^2)t+\sigma W(t)\Bigr).
\end{equation}
\subsection{Quadratic variation based estimation}
Note that for the process $X(t)=\ln S(t)$ it holds for $t \in [0,T]$ that
\begin{equation}
\label{log_S}
	d(X(t))=\Bigl(\mu-\frac{1}{2}\sigma^2\Bigr)dt+\sigma dW(t),
\end{equation}
hence
\begin{equation}
	[X]_t=\sigma^2 t, \quad t \in [0,T],
\end{equation}
and, by the definition of the quadratic variation, we obtain
\begin{equation}
	\sigma^2=\frac{1}{T}[X]_T\approx\frac{1}{T}\sum\limits_{i=0}^{n-1}(X(t_{i+1})-X(t_i))^2=\frac{1}{T}\sum\limits_{i=0}^{n-1}\Bigl(\ln\frac{S(t_{i+1})}{S(t_i)}\Bigr)^2.
\end{equation}
Assuming knowledge of historical stock prices $(S(t_i))_{i=0,1,\ldots,n}$ at the moments $0=t_0<t_1<\ldots<t_n=T$ we can take
\begin{equation}
	\hat \sigma_n=\Biggl(\frac{1}{T}\sum\limits_{i=0}^{n-1}\Bigl(\ln\frac{S(t_{i+1})}{S(t_i)}\Bigr)^2\Biggr)^{1/2},
\end{equation}
and
\begin{equation}
	\hat \mu_n=\frac{1}{T}\sum\limits_{i=0}^{n-1}\ln\frac{S(t_{i+1})}{S(t_i)}+\frac{1}{2}(\hat \sigma_n)^2=\frac{1}{T}\ln\frac{S(T)}{S(0)}+\frac{1}{2}(\hat \sigma_n)^2
\end{equation}
as the simple estimators of $\mu$ and $\sigma$, respectively, under the assumed Black-Scholes model. We investigate the quality of these estimators.

We believe that the following results are known. However, we were unable to find suitable references in literature. Hence, for the convenience of the reader we present complete proofs.

The results below might be used in the context of high-frequency data, i.e. when $\displaystyle{\lim\limits_{n\to +\infty}\max\limits_{0\leq i\leq n-1}(t_{i+1}-t_i)=0}$. In this case by $\{t_0,t_1,\ldots,t_n\}$ we mean  $\{t^n_0,t^n_1,\ldots,t^n_n\}$ and when considering the limit we use the sequence of discretizations $(\{t^n_0,t^n_1,\ldots,t^n_n\})_{n\in\mathbb{N}}$. However, for the clarity of presentation we omit the upper script.
\begin{thm} There exists $C\in (0,+\infty)$ such that for all $n\in\mathbb{N}$ and any discretization $0=t_0<t_1<\ldots<t_n=T$ we have 
\begin{equation}
\label{mnsq_err_sigma_1}
		\Bigl(\mathbb{E}|(\hat\sigma_n)^2-\sigma^2|^2\Bigr)^{1/2}\leq C\max\limits_{0\leq i\leq n-1}(t_{i+1}-t_i)^{1/2}.
\end{equation}
Moreover, $(\hat \sigma_n)^2$ is an asymptotically unbiased estimator of $\sigma^2$, i.e.,
\begin{equation}
\label{mnsq_err_sigma_2}
	|\mathbb{E}(\hat \sigma_n)^2-\sigma^2|=\frac{1}{T}(\mu-\frac{1}{2}\sigma^2)^2\sum\limits_{i=0}^{n-1}(t_{i+1}-t_{i})^2\to 0
\end{equation}
as $\displaystyle{\lim\limits_{n\to +\infty}\max\limits_{0\leq i\leq n-1}(t_{i+1}-t_i)=0}$. In addition, if
\begin{equation}
\label{mesh_assumpt_2}
	\sum\limits_{n=1}^{+\infty}\max\limits_{0\leq i\leq n-1}(t_{i+1}-t_i)<+\infty,
\end{equation}
then
\begin{equation}
\label{as_sigma_conv}
	\lim\limits_{n\to+\infty}\hat\sigma_n= \sigma
\end{equation}
almost surely.
\end{thm}
{\bf Proof.} Note that
\begin{equation}
\label{log_ret_1}
	\ln\frac{S(t_{i+1})}{S(t_i)}=\Bigl(\mu-\frac{1}{2}\sigma^2\Bigr)\Delta t_i+\sigma\Delta W_i, \ i=0,1,\ldots,n-1,
\end{equation}
where $\Delta t_i=t_{i+1}-t_i$ and $\Delta W_i=W(t_{i+1})-W(t_i)$. Hence
\begin{equation}
\label{est_sig_sq_1}
	(\hat\sigma_n)^2=\frac{1}{T}(\mu-\frac{1}{2}\sigma^2)^2\sum\limits_{i=0}^{n-1}(\Delta t_i)^2+\frac{2}{T}\sigma (\mu-\frac{1}{2}\sigma^2)\sum\limits_{i=0}^{n-1}(\Delta t_i\cdot\Delta W_i)+\frac{\sigma^2}{T}\sum\limits_{i=0}^{n-1}(\Delta W_i)^2,
\end{equation}
and
\begin{equation}
	\mathbb{E}|(\hat\sigma_n)^2-\sigma^2|^2=\frac{1}{T^2}\cdot\mathbb{E}|A_n+B_n+C_n|^2=\frac{1}{T^2}\cdot\mathbb{E}\Bigl(A_n^2+B_n^2+C_n^2+2A_nB_n+2A_nC_n+2B_nC_n\Bigr)
\end{equation}
where
\begin{eqnarray}
	&&A_n=(\mu-\frac{1}{2}\sigma^2)^2\cdot\sum\limits_{i=0}^{n-1}(\Delta t_i)^2,\notag\\
	&&B_n=2\sigma (\mu-\frac{1}{2}\sigma^2)\cdot\sum\limits_{i=0}^{n-1}(\Delta t_i\cdot\Delta W_i),\notag\\
	&&C_n=\sigma^2\cdot\Bigl(\sum\limits_{i=0}^{n-1}(\Delta W_i)^2-T\Bigr).
\end{eqnarray}
Since $\mathbb{E}(\Delta W_i)=0$ and $\mathbb{E}(\Delta W_i)^2=\Delta t_i$, we have 
\begin{eqnarray}
	&&\mathbb{E}(A_n B_n)=A_n\cdot \mathbb{E}(B_n)=0,\notag\\
	&&\mathbb{E}(A_n C_n)=A_n\cdot \mathbb{E}(C_n)=0.
\end{eqnarray}
Moreover
\begin{equation}
	\mathbb{E}(A_n)^2=A_n^2=(\mu-\frac{1}{2}\sigma^2)^4\cdot\Bigl(\sum\limits_{i=0}^{n-1}(\Delta t_i)^2\Bigr)^2,
\end{equation}
and, since $\sigma(\Delta W_i)$ and $\sigma(\Delta W_j)$ are independent for $i\neq j$, we have
\begin{eqnarray}
	&&\mathbb{E}(B_n)^2=4\sigma^2 (\mu-\frac{1}{2}\sigma^2)^2\cdot\Bigl(\sum\limits_{i=0}^{n-1}(\Delta t_i)^2\cdot\mathbb{E}(\Delta W_i)^2+\sum\limits_{i\neq j}\Delta t_i\Delta t_j\mathbb{E}(\Delta W_i)\mathbb{E}(\Delta W_j)\Bigr)\notag\\
	&&=4\sigma^2 (\mu-\frac{1}{2}\sigma^2)^2\cdot\sum\limits_{i=0}^{n-1}(\Delta t_i)^3
\end{eqnarray}
\begin{eqnarray}	
	&&\mathbb{E}(C_n)^2=\sigma^4\cdot\mathbb{E}\Biggl(\sum\limits_{i=0}^{n-1}\Bigl( (\Delta W_i)^2-\Delta t_i\Bigr)\Biggr)^2\notag\\
	&&=\sigma^4\cdot\Biggl(\sum\limits_{i=0}^{n-1}\mathbb{E}\Bigl( (\Delta W_i)^2-\Delta t_i\Bigr)^2+\sum\limits_{i\neq j}\mathbb{E}\Bigl( (\Delta W_i)^2-\Delta t_i\Bigr)\cdot\mathbb{E}\Bigl( (\Delta W_j)^2-\Delta t_j\Bigr)\Biggr)\notag\\
	&&=\sigma^4\cdot\sum\limits_{i=0}^{n-1}(\mathbb{E}(\Delta W_i)^4-2\Delta t_i\mathbb{E}(\Delta W_i)^2+(\Delta t_i)^2)=2\sigma^4\cdot\sum\limits_{i=0}^{n-1}(\Delta t_i)^2,
\end{eqnarray}
and
\begin{eqnarray}
	&&\mathbb{E}(B_n C_n)=2\sigma^3 (\mu-\frac{1}{2}\sigma^2)\cdot\Biggl(\sum\limits_{i=0}^{n-1}\sum\limits_{j=0}^{n-1}\Delta t_i\cdot\mathbb{E}(\Delta W_i\cdot(\Delta W_j)^2)-T\sum\limits_{i=0}^{n-1}\Delta t_i\cdot\mathbb{E}(\Delta W_i)\Biggr)\notag\\
	&&=2\sigma^3 (\mu-\frac{1}{2}\sigma^2)\cdot\sum\limits_{i=0}^{n-1}\Delta t_i\cdot\mathbb{E}(\Delta W_i)^3=0.
\end{eqnarray}
Therefore
\begin{eqnarray}
\mathbb{E}|(\hat\sigma_n)^2-\sigma^2|^2&=&\frac{1}{T^2}(\mu-\frac{1}{2}\sigma^2)^4\cdot\Bigl(\sum\limits_{i=0}^{n-1}(\Delta t_i)^2\Bigr)^2+\frac{4\sigma^2}{T^2}(\mu-\frac{1}{2}\sigma^2)^2\cdot\sum\limits_{i=0}^{n-1}(\Delta t_i)^3\notag\\
&&+\frac{2\sigma^4}{T^2}\cdot\sum\limits_{i=0}^{n-1}(\Delta t_i)^2\leq C(T,\mu,\sigma)\cdot\max\limits_{0\leq i\leq n-1}\Delta t_i,
\end{eqnarray}
and this implies \eqref{mnsq_err_sigma_1}.

From \eqref{est_sig_sq_1} we get that
\begin{equation}
	\mathbb{E}(\hat\sigma_n)^2=\frac{1}{T}(\mu-\frac{1}{2}\sigma^2)^2\cdot\sum\limits_{i=0}^{n-1}(\Delta t_i)^2+\sigma^2,
\end{equation}
and the result \eqref{mnsq_err_sigma_2} follows.

By \eqref{est_sig_sq_1} we have that
\begin{equation}
\label{as_est_1}
	(\hat\sigma_n)^2=\frac{1}{T}(\mu-\frac{1}{2}\sigma^2)^2S_n^1+\frac{2}{T}\sigma (\mu-\frac{1}{2}\sigma^2)S_n^2+\frac{\sigma^2}{T}S_n^3
\end{equation}
where
\begin{eqnarray}
	&& S_n^1=\sum\limits_{i=0}^{n-1}(\Delta t_i)^2,\\
	&& S_n^2=\sum\limits_{i=0}^{n-1}(\Delta t_i\cdot\Delta W_i),\\
	&& S_n^3=\sum\limits_{i=0}^{n-1}(\Delta W_i)^2.
\end{eqnarray}
Now, we have that
\begin{equation}
\label{as_est_2}
	0\leq S_n^1\leq T\cdot\max\limits_{0\leq i\leq n-1}(t_{i+1}-t_i)\to 0,
\end{equation}
as $n\to+\infty$. By \eqref{mesh_assumpt_2} and Lemma 4.3, page 92 in \cite{LipShir_1} we get that
\begin{equation}
\label{as_est_3}
	S_n^3\to T
\end{equation}
almost surely as $n\to+\infty$, and hence
\begin{equation}
\label{as_est_4}
	|S_n^2|\leq \Bigl(T\cdot\max\limits_{0\leq i\leq n-1}(t_{i+1}-t_i)\Bigr)^{1/2}\cdot (S_n^3)^{1/2}\to 0
\end{equation}
almost surely as $n\to +\infty$. Combining \eqref{as_est_1}, \eqref{as_est_2}, \eqref{as_est_3}, \eqref{as_est_4} we get \eqref{as_sigma_conv}.
 \ \ \ $\blacksquare$\\ \\
Now we turn to the properties of $\hat \mu_n$. The first result is of positive nature.
\begin{thm}
	\label{mean_est_1}
	$\hat \mu_n$ is an asymptotically unbiased estimator of $\mu$, i.e.,
\begin{equation}
\label{mnsq_err_mean_2}
	|\mathbb{E}(\mu_n)-\mu|=\frac{1}{2T}(\mu-\frac{1}{2}\sigma^2)^2\sum\limits_{i=0}^{n-1}(t_{i+1}-t_{i})^2\to 0
\end{equation}
as $\displaystyle{\lim\limits_{n\to +\infty}\max\limits_{0\leq i\leq n-1}(t_{i+1}-t_i)=0}$. 
\end{thm}
Unfortunately, it turns out that  $\hat\mu_n$ converge, in the mean-square sense, not to $\mu$ but to the random variable $\displaystyle{\mu+\frac{\sigma}{T}W(T)}$, which has the law $N(\mu, \sigma^2/T)$.
\begin{thm} Let $\mu_T=\mu+\frac{\sigma}{T}W(T)$.  We have that
	\begin{equation}
\label{mnsq_err_mean_3}
	\Bigl(\mathbb{E}|\hat\mu_n-\mu|^2\Bigr)^{1/2}\to \sigma/\sqrt{T},
\end{equation}
and
\begin{equation}
\label{mnsq_err_mean_4}
	\mu_n\to \mu_T \ \hbox{in} \ L^2(\Omega)
\end{equation}
as $\displaystyle{\lim\limits_{n\to +\infty}\max\limits_{0\leq i\leq n-1}(t_{i+1}-t_i)=0}$, and
\begin{equation}
\label{mnsq_err_mean_5}
	\mu_n\to \mu_T \ \hbox{a.s.,} 
\end{equation}
if $\sum\limits_{n=1}^{+\infty}\max\limits_{0\leq i\leq n-1}(t_{i+1}-t_i)<+\infty$. Moreover,
\begin{equation}
\label{mnsq_err_mean_6}
	\lim\limits_{T\to+\infty}\mu_T=\mu \ \hbox{a.s.}
\end{equation}
\end{thm}
{\bf Proof.} By \eqref{log_ret_1} we obtain that
\begin{equation}
	\hat\mu_n-\mu=\frac{1}{2}\Bigl((\hat\sigma_n)^2-\sigma^2\Bigr)+\frac{\sigma}{T}W(T),
\end{equation}
and from the Minkowski inequality, and \eqref{mnsq_err_sigma_1} we have 
\begin{equation}
	\frac{\sigma}{\sqrt{T}}-\frac{1}{2}C\max\limits_{0\leq i\leq n-1}(t_{i+1}-t_i)^{1/2}\leq\|\hat\mu_n-\mu\|_{L^2(\Omega)}\leq \frac{\sigma}{\sqrt{T}}+\frac{1}{2}C\max\limits_{0\leq i\leq n-1}(t_{i+1}-t_i)^{1/2}.
\end{equation}
Taking both sides limits we get \eqref{mnsq_err_mean_3}, \eqref{mnsq_err_mean_4}, and \eqref{mnsq_err_mean_5}. By the SLLN for the Wiener process we get \eqref{mnsq_err_mean_6}. \ \ \ $\blacksquare$ \\ \\
Other accurate estimators of $\mu$ were constructed in literature such as, for example, quasi-maximum likelihood estimators (QMLE), see \cite{IAYO}. However, in the case of geometric Brownian motion the simple estimators $\mu_n$, $(\sigma_n)^2$ might give very close results to that obtained from QML estimation, see page 93. in \cite{IAYO}. Moreover, $\mu_n$ should give  quite good approximation to $\mu$ when $T$ is large enough. Hence, estimation of $\mu$ is possible for long time horizons. 
\begin{rem}\rm Note that we have
\begin{equation}
	[S]_t=\sigma\int\limits_0^tS^2(u)du, \quad t\in [0,T].
\end{equation}
Since
\begin{equation}
\frac{\sum\limits_{i=0}^{n-1}(S(t_{i+1})-S(t_{i}))^2}{\sum\limits_{i=0}^{n-1}(S(t_i))^2\Delta t_i}\to \frac{[S]_T}{\int\limits_0^T (S(t))^2 dt}=\sigma^2
\end{equation}
in probability as $\lim\limits_{n\to +\infty}\max\limits_{0\leq i\leq n-1}(t_{i+1}-t_i)=0$, we obtained an alternative estimator of $\sigma^2$. 
\end{rem}
\section{The maximum likelihood method (ML)}
Given $x_0,x_1,\ldots,x_n$, the  realisations $X_{\Theta}(t_0),X_{\Theta}(t_1),\ldots, X_{\Theta}(t_n)$ of a stochastic process $X_{\Theta}$, a very popular used procedure to estimate unknown parameters $\Theta\in\mathbb{R}^s$ of the process is the maximum
likelihood method (ML). In this method, the parameters are chosen to maximise the joint density $f_{\Theta}(x_0,x_1,\ldots,x_n)$ of $(X_{\Theta}(t_0),X_{\Theta}(t_1),\ldots, X_{\Theta}(t_n))$. For Markov processes (such as solutions of jump-diffusion SDEs), the joint density can be represented as a
product of transitional densities as follows
\begin{equation}
 f_{\Theta}(x_0,x_1,\ldots,x_n)=f_{X(t_0)}(x_0)\cdot f_{X(t_1)|X(t_0)}(x_1|x_0)\cdot\ldots\cdot f_{X(t_n)|X(t_{n-1})}(x_n|x_{n-1}),
\end{equation}
where the transitional density $f_{X(t_{i+1})|X(t_i)}(\cdot|x_{i})$ is a conditional density function of $X(t_{i+1})$ given $X(t_i)=x_i$. The aim is to minimize the negative log-likelihood defined as
\begin{equation}
\mathcal{L}(\Theta)=-\ln f_{\Theta}(x_0,x_1,\ldots,x_n)=-\sum\limits_{i=0}^{n-1}\ln f_{X(t_{i+1})|X(t_i)}(x_{i+1}|x_{i}),
\end{equation}
see \cite{saso1}. We now apply this idea to several SDEs for which the tranistion densities are known.
\subsection{Geometric Brownian motion}
First, we estimate the unknown parameters $\Theta=(\mu,\sigma)$ of the equation \eqref{log_S} or, equivalently, the equation \eqref{BS_eq}. In what follows we assume that $t_i=ih$, $i=0,1,\ldots,n$, $h=T/n$. 

Let $s_0,s_1,\ldots,s_n$ be discrete observations of $S$ at $t_0,t_1,\ldots,t_n$. We take $X(t)=\ln S(t)$ and let $x_k=\ln s_k$ be discrete realizations of $X$. By \eqref{log_S}, for $s,t\in [0,T]$, $s<t$
\begin{equation}
	X(t)=X(t)-X(s)+X(s)=\bar\mu (t-s)+X(s)+\sigma(W(t)-W(s)),
\end{equation}
where $\bar \mu = \mu-\frac{1}{2}\sigma^2$. Since $X(s)$ and $W(t)-W(s)$ are independent, we obtain
\begin{equation}
 	f_{X(t_{k+1})|X(t_k)}(x|x_{k})=\frac{1}{\sigma\sqrt{2\pi h}}\cdot\exp\Bigl(-\frac{(x-x_k-\bar \mu h)^2}{2\sigma^2 h}\Bigr).
\end{equation}
By taking $q=\sigma^2$, the negative log-likelihood is as follows
\begin{equation}
\mathcal{L}(\mu,q)=\frac{1}{2}n\ln(2\pi q h)+\frac{1}{2h}\sum_{i=0}^{n-1}\Bigl(\frac{\Delta x_k-(\mu-\frac{1}{2}q)h}{q^{1/2}}\Bigr)^2,
\end{equation}
with $\Delta x_k=x_{k+1}-x_k$. Note that
\begin{equation}
	\sum\limits_{k=0}^{n-1}\Delta x_k=\ln(s_n/s_0),
\end{equation}
and
\begin{equation}
	\frac{\partial \mathcal{L}}{\partial\mu}=-\frac{1}{q}\sum\limits_{k=0}^{n-1}\Bigl(\Delta x_k-(\mu-\frac{1}{2}q)h\Bigr),
\end{equation}
\begin{equation}
\frac{\partial \mathcal{L}}{\partial q}=\frac{N}{2q}+\frac{1}{hq^2}\sum\limits_{k=0}^{n-1}\Bigl(\Delta x_k-(\mu-\frac{1}{2}q)h\Bigr)\Bigl(\frac{qh}{2}-\frac{1}{2}\Bigl(\Delta x_k-(\mu-\frac{1}{2}q)h\Bigr)\Bigr).
\end{equation}
From $\frac{\partial \mathcal{L}}{\partial\mu}=0$ we get
\begin{equation}
\label{ml_est_1}
	\mu-\frac{1}{2}q=\frac{1}{T}\sum\limits_{k=0}^{n-1}\Delta x_k=\frac{1}{T}\ln(s_n/s_0),
\end{equation}
while from $\frac{\partial \mathcal{L}}{\partial q}=0$ we have
\begin{eqnarray}
\label{ml_est_2}
&&\frac{Tq}{2}+\frac{hq}{2}(x_n-x_0)-\frac{1}{2}\sum\limits_{k=0}^{n-1}(\Delta x_k)^2+\frac{1}{2}(\mu-\frac{1}{2}q)h(x_n-x_0)-(\mu-\frac{1}{2}q)\frac{Thq}{2}\notag\\
&&+	\frac{1}{2}(\mu-\frac{1}{2}q)h(x_n-x_0)-\frac{1}{2}(\mu-\frac{1}{2}q)^2hT=0.
\end{eqnarray}
Inserting \eqref{ml_est_1} to \eqref{ml_est_2} we get
\begin{equation}
	\frac{Tq}{2}+\frac{1}{2T}\Bigl(\ln(s_n/s_0)\Bigr)^2h=\frac{1}{2}\sum\limits_{k=0}^{n-1}\Bigl(\ln(s_{k+1}/s_k)\Bigr)^2,
\end{equation}
and hence
\begin{eqnarray}
	&&q=\frac{1}{T}\sum\limits_{k=0}^{n-1}\Bigl(\ln(s_{k+1}/s_k)\Bigr)^2-\frac{h}{T^2}\Bigl(\ln(s_n/s_0)\Bigr)^2\notag\\
	&&=\frac{1}{T}\sum\limits_{k=0}^{n-1}\Bigl(\ln(s_{k+1}/s_k)-\frac{h}{T}\ln(s_n/s_0)\Bigr)^2.
\end{eqnarray}
In summary we obtain the following ML-estimators of $(\mu,\sigma^2)$
\begin{equation}
	\widehat{\sigma^2_n}=\frac{1}{T}\sum\limits_{k=0}^{n-1}\Bigl(\ln(s_{k+1}/s_k)-\frac{h}{T}\ln(s_n/s_0)\Bigr)^2,
\end{equation}
\begin{equation}
	\widehat{\mu}_n=\frac{1}{T}\ln(s_n/s_0)+\frac{1}{2}\widehat{\sigma^2_n}.
\end{equation}
\subsection{Zero mean reverting OU-type process}
We start with the following SDE
\begin{equation}
	dX(t)=-aX(t)dt+\sigma dW(t), \quad t\in [0,T],
\end{equation}
where $X(0)=x_0$, $a>$ is a constant mean reversion rate and $\sigma>0$ is a deterministic volatility. Note that such a process is a base process for the Hull-White model and for the energy pricing models. Since for $t>s$
\begin{equation}
	X(t)=e^{-a(t-s)}X(s)+e^{-at}\int\limits_s^t e^{au}\sigma dW(u),
\end{equation}
the random variables $X(s)$ and $\displaystyle{\int\limits_s^t e^{au}\sigma dW(u)}$ are independent, and
\begin{equation}
	\sigma^2e^{-2at}\int\limits_s^te^{2au}du=\frac{\sigma^2}{2a}\Bigl(1-e^{-2a(t-s)}\Bigr),
\end{equation}
we get that
\begin{equation}
	f_{X(t_{k+1})|X(t_{k})}(x|x_k)=\frac{1}{\sqrt{2\pi}\Bigl(\frac{\sigma^2}{2a}\Bigl(1-e^{-2ah}\Bigr)\Bigr)^{1/2}}\exp\Biggl(-\frac{(x-e^{-ah}x_k)^2}{\frac{\sigma^2}{a}\Bigl(1-e^{-2ah}\Bigr)}\Biggr).
\end{equation}
The negative log-likelihood is as follows
\begin{eqnarray}
	&&\mathcal{L}(a,\sigma)=\sum\limits_{k=0}^{n-1}\Bigl[\frac{1}{2}\ln\Bigl(2\pi\cdot \frac{\sigma^2}{2a}\Bigl(1-e^{-2ah}\Bigr)\Bigr)\notag\\
	&&\quad\quad\quad\quad+\frac{a}{\sigma^2(1-e^{-2ah})}(x_{k+1}-e^{-ah}x_k)^2\Bigr].
\end{eqnarray}
We reparametrize it in terms of
\begin{eqnarray}
\label{reparam_1}
	&&A=e^{-ah},\\
	&&\Sigma=\frac{\sigma^2}{2a}\Bigl(1-e^{-2ah}\Bigr),
\end{eqnarray}
and obtain
\begin{equation}
	\mathcal{L}(A,\Sigma)=\frac{n}{2}\ln(2\pi\Sigma)+\frac{1}{2\Sigma}\sum\limits_{k=0}^{n-1}(x_{k+1}-Ax_k)^2.
\end{equation}
We have that
\begin{equation}
	\frac{\partial\mathcal{L}}{\partial A}=\frac{1}{\Sigma}\sum\limits_{k=0}^{n-1}(x_{k+1}-Ax_k)\cdot (-x_k),
\end{equation}
and
\begin{equation}
\frac{\partial\mathcal{L}}{\partial \Sigma}=\frac{n}{2\Sigma}-\frac{1}{2\Sigma^2}\sum\limits_{k=0}^{n-1}(x_{k+1}-Ax_k)^2.
\end{equation}
From $\frac{\partial\mathcal{L}}{\partial A}=0$ we have
\begin{equation}
	A=\frac{\sum\limits_{k=0}^{n-1}x_kx_{k+1}}{\sum\limits_{k=0}^{n-1}x_k^2},
\end{equation}
while from $\frac{\partial\mathcal{L}}{\partial \Sigma}=0$,
\begin{equation}
	\Sigma=\frac{1}{n}\sum\limits_{k=0}^{n-1}(x_{k+1}-A x_k)^2.
\end{equation}
Using \eqref{reparam_1} we arrive at the following ML estimators of $a$ and $\sigma^2$
\begin{equation}
	\widehat{a}_n=-\frac{1}{h}\ln\Biggl(\frac{\sum\limits_{k=0}^{n-1}x_kx_{k+1}}{\sum\limits_{k=0}^{n-1}x_k^2}\Biggr),
\end{equation}
\begin{equation}
	\widehat{\sigma^2_n}=\frac{1}{n}\cdot\Biggl(\frac{2\widehat{a}_n}{1-e^{-2\widehat{a}_n h}}\Biggr)\cdot\sum\limits_{k=0}^{n-1}\Bigl(x_{k+1}-e^{-2\widehat{a}_n h}x_k\Bigr)^2.
\end{equation}
\subsection{Mean reverting OU-type process}
We now deal with the following OU process
\begin{equation}
	dX(t)=\kappa(\mu-X(t))dt+\sigma dW(t), \quad t\in [0,T],
\end{equation}
which is also called the Vasicek process. For $t>s$ we have
\begin{equation}
	X(t)=e^{-\kappa (t-s)}X(s)+\mu(1-e^{-\kappa (t-s)})+\sigma e^{-\kappa t}\int\limits_s^t e^{\kappa u}dW(u),
\end{equation}
and hence
\begin{equation}
	f_{X(t_{k+1})|X(t_{k})}(x|x_k)=\frac{1}{\sqrt{2\pi}\Bigl(\frac{\sigma^2}{2\kappa}\Bigl(1-e^{-2\kappa h}\Bigr)\Bigr)^{1/2}}\exp\Biggl(-\frac{(x-\mu(1-e^{-\kappa h})-e^{-\kappa h}x_k)^2}{\frac{\sigma^2}{\kappa}\Bigl(1-e^{-2\kappa h}\Bigr)}\Biggr).
\end{equation}
In this case the negative log-likelihood is as follows
\begin{eqnarray}
	&&\mathcal{L}(\mu,\sigma,\kappa)=\frac{n}{2}\ln\Bigl(2\pi\cdot \frac{\sigma^2}{2\kappa}\Bigl(1-e^{-2\kappa h}\Bigr)\Bigr)\notag\\
	&&\quad\quad\quad\quad+\frac{\kappa}{\sigma^2(1-e^{-2\kappa h})}\sum\limits_{k=0}^{n-1}(x_{k+1}-\mu-e^{-\kappa h}(x_k-\mu))^2,
\end{eqnarray}
and we reparametrise it in the following way
\begin{eqnarray}
\label{reparam_2}
	&&A=e^{-\kappa h},\\
	&&\Sigma=\frac{\sigma^2}{2\kappa}\Bigl(1-e^{-2\kappa h}\Bigr).
\end{eqnarray}
This gives
\begin{equation}
\mathcal{L}(\mu,\Sigma,A)=\frac{n}{2}\ln(2\pi \Sigma)+\frac{1}{2\Sigma}\sum\limits_{k=0}^{n-1}(x_{k+1}-\mu-e^{-\kappa h}(x_k-\mu))^2.
\end{equation}
Minimalization of $\mathcal{L}$ gives the following ML estimators
\begin{eqnarray}
\label{est_p_1}
	&&\widehat{\kappa}_n=-h^{-1}\ln(\widehat{\beta}_{1,n}),\\
\label{est_p_2}	
	&&\widehat{\mu}_n=\widehat{\beta}_{2,n},\\
\label{est_p_3}
	&&\widehat{\sigma^2}_n=2\widehat{\kappa}_n\widehat{\beta}_{3,n}(1-\widehat{\beta}^2_{n,1})^{-1},
\end{eqnarray}
where
\begin{equation}
\widehat{\beta}_{1,n}=\frac{n^{-1}\sum\limits_{k=0}^{n-1}x_{k+1}x_k-n^{-2}\sum\limits_{k=0}^{n-1}x_{k+1}\sum\limits_{k=0}^{n-1}x_{k}}{n^{-1}\sum\limits_{k=0}^{n-1}x^2_{k}-n^{-2}\Bigl(\sum\limits_{k=0}^{n-1}x_{k}\Bigr)^2},
\end{equation}
\begin{equation}
	\widehat{\beta}_{2,n}=\frac{n^{-1}\sum\limits_{k=0}^{n-1}(x_{k+1}-\widehat{\beta}_{1,n}x_k)}{1-\widehat{\beta}_{1,n}},
\end{equation}
and
\begin{equation}
	\widehat{\beta}_{3,n}=n^{-1}\sum\limits_{k=0}^{n-1}\Bigl(x_{k+1}-\widehat{\beta}_{1,n}x_k-\widehat{\beta}_{2,n}(1-\widehat{\beta}_{1,n})\Bigr)^2,
\end{equation}
see \cite{TACHE2009}.
\subsection{Geometric mean reverting OU-type process}
Finally, we investigate the following OU-type process
\begin{equation}
	dX(t)=\kappa X(t)(\mu-\ln(X(t)))dt+\sigma X(t) dW(t), \quad t\in [0,T],
\end{equation}
with $X(0)=x_0>0$, which is called geometric Ornstein-Uhlenbeck process. It can be shown by the use of It\^o formula that $0<X(t)=\exp(U(t))$ where
\begin{equation}
	dU(t)=\kappa (\bar \mu-U(t))dt+\sigma dW(t), \quad t\in [0,T],
\end{equation}
and $\bar\mu=\mu-\frac{\sigma^2}{2\kappa}$. Having the discrete observations $x_0,\ldots,x_n$ of $X$ we take $u_k=\ln x_k$ as the observations of $U$. Now we can use \eqref{est_p_1}-\eqref{est_p_3} in order to obtain the estimators $\widehat{\kappa}_n, \widehat{\bar\mu}_n,\widehat{\sigma^2}_n$  that are based on the values $(u_k)_{k=0}^n$. The estimator of $\mu$ is defined by
\begin{equation}
	\widehat{\mu}_n=\widehat{\bar\mu}_n+\frac{\widehat{\sigma^2}_n}{2\widehat{\kappa}_n}.
\end{equation}
\section{Quasi maximum likelihood method (QML)}
Let us consider the general problem of estimating unknown parameters of the following scalar SDE
\begin{equation}
\label{eq_1}
	dX(t)=a(t,X(t),\theta)dt+b(t,X(t),\theta)dW(t), \quad t\in [0,T]
\end{equation}
where the solution $X$ is sampled at the discrete points $t_i=iT/n$, $i=0,1,\ldots,n$, and $x_0,x_1,\ldots,x_n$ are the realisations of $X$ at $t_i$'s. We also assume that $X(0)=x_0$ and $b(t,x,\theta)>0$ for all $(t,x,\theta)\in [0,T]\times\mathbb{R}\times\mathbb{R}^s$. Morever, $a,b:[0,T]\times\mathbb{R}\times\mathbb{R}^s\to\mathbb{R}$ are (at least) Borel measurable, $\theta\in\mathbb{R}^s$ and we assume that X is the unique strong solution of \eqref{eq_1}. For such a case the explicit form of the transition density is not known. However, for the Euler scheme, defined as follows
\begin{equation}
	X^E(t_0)=x_0,
\end{equation}
\begin{equation}
	X^E(t_{k+1})=X^E(t_{k})+a(t_k,X^E(t_{k}),\theta)h+b(t_k,X^E(t_{k}),\theta)\Delta W_k, \quad k =0,1,\ldots,n-1,
\end{equation}
we have the following conditional law
\begin{equation}
	X^E(t_{k+1})|X^E(t_{k})\sim N(X^E(t_{k})+a(t_k,X^E(t_{k}),\theta)h, h|b(t_k,X^E(t_{k}),\theta)|^2).
\end{equation}
Hence,  we approximate the transition density with the following
\begin{equation}
	f_{X^E(t_{k+1})|X^E(t_{k})}(x|x_k)=\frac{1}{(2\pi b^2(t_k,x_k,\theta))^{1/2}}\exp\Bigl(-\frac{(x-x_k-a(t_k,x_k,\theta)h)^2}{2hb^2(t_k,x_k,\theta)}\Bigr).
\end{equation}
Hence,  the negative log-quasi likelihood is defined by
\begin{eqnarray}
	&&\mathcal{L}(\theta)=-\sum\limits_{k=0}^{n-1}\ln f_{X^E(t_{k+1})|X^E(t_{k})}(x_{k+1}|x_k)\notag\\
	&&=\frac{1}{2}\sum\limits_{k=0}^{n-1}\Bigl[\ln (2\pi h b^2(t_k,x_k,\theta))+\frac{(x_{k+1}-x_k-a(t_k,x_k,\theta)h)^2}{hb^2(t_k,x_k,\theta)}\Bigr].
\end{eqnarray}
We obtain the quasi maximum likelihood estimator $\hat\theta_n$ of $\theta$  by minimizing $\mathcal{L}(\theta)$ with respect to $\theta\in\mathbb{R}^s$, i.e.
\begin{equation}
    \hat\theta_n=argmin_{\theta\in\mathbb{R}^s}\mathcal{L}(\theta).
\end{equation}
Mostly, such minimalisation has to be done numerically.
\begin{rem}\rm The approach described in this section can be extended in order to cover the multidimensional case.
\end{rem}
\begin{rem}
    We considered the case when the unknown estimated parameters $\theta$ do not depend on time. However, in more realistic models such parameters are often time-dependent. In such a case parameter estimation is much more involved and differs a lot from the time-independent case. See, for example, \cite{SDES_NN} where application of artificial neural network is discussed in the context of estimation of time-dependent $\theta$.
\end{rem}
\section{Monte Carlo forecasting in SDE-based models}
Having calibrated  SDE-based model on $[0,T]$ to the real data $x_0,x_1,\ldots,x_n$ we can make predictions for $t>T$. Namely, outside the interval $[0,T]$ we consider the following SDE
\begin{equation}
\label{hat_SDE}
	d\hat X(t)=a(t,\hat X(t),\hat\theta_n)dt+b(t,\hat X(t),\hat\theta_n)dW(t), \quad t\in [T,T+\Delta],
\end{equation}
where $\hat X(T)=x_n$, $\Delta>0$ is the length of the prediction horizon, $W=[W^1,\ldots,W^m]^T$ is $m$-dimensional Wiener process, and $\hat\theta_n$ is the estimated value of the true parameter $\theta$. W also assume that $a(t,x,\theta)$, $b(t,x,\theta)$ are defined for all $(t,x,\theta)\in [T,T+\Delta]\times\mathbb{R}^d\times\mathbb{R}^s$. We now describe (inductively) the forecasting procedure together with the construction of prediction intervals at confidence level $\alpha\in (0,1)$ (which we take in practice close to $1$).  Since in general we do not know the analytical solution $\hat X$, the prediction procedure is based on the Euler-Maruyama scheme that approximates $\hat X$. Note that there are many trajectories that we can obtain as predictions of future behavior of $\hat X$. We can handle with this problem by computing a (suitable) single trajectory together with prediction intervals (at confidence level $\alpha$) around it.

Let us take $h=\Delta/N$ for some $N\in\mathbb{N}$ and $t_k=T+kh$ for $k=0,1,\ldots,N.$
\subsection{Prediction intervals in the scalar case with $m=d=1$}
{\bf Approach 1:}  Assume that $\bar x_k$ is given prediction of the  value of $\hat X(t_k)$ (for $k=0$ we have $\bar x_0:=x_n$). The one step of Euler-Maruyama scheme from $\bar x_k$ at the time $t_k$ to $t_{k+1}=t_k+h$ is defined as follows
\begin{equation}
	\bar X^E_N(t_{k+1})=\bar x_k+a(t_k,\bar x_k,\hat\theta_n)\cdot h+b(t_k,\bar x_k,\hat\theta_n)\cdot \Delta W_k,
\end{equation} 
where $\Delta W_k=W(t_{k+1})-W(t_k)$. As in the Section \ref{subs_emp_conf_int}, let $q_{\alpha}$ be the  two-sided $\alpha$-quantile  of the normal distribution $N(0,1)$. We define the  prediction interval for (the unknown value) $\hat X(t_{k+1})$ as follows
\begin{equation}
\label{p_int_def}
	\mathcal{\bar P}_{k+1}(\alpha):=\Bigl[\bar\mu_{k+1}-q_{\alpha}\cdot\bar\sigma_{k+1},\bar\mu_{k+1}+q_{\alpha}\cdot\bar\sigma_{k+1}\Bigr],
\end{equation}
with
\begin{eqnarray}
	&&\bar\mu_{k+1}:=\bar x_k+a(t_k,\bar x_k,\hat\theta_n)\cdot h,\\
	&&\bar\sigma_{k+1}:=\sqrt{h}\cdot |b(t_k,\bar x_k,\hat\theta_n)|.
\end{eqnarray}
If $b(t_k,\bar x_k,\hat\theta_n)\neq 0$ then we have
\begin{equation}
	\mathbb{P}(\bar X^E_N(t_{k+1})\in\mathcal{\bar P}_{k+1}(\alpha))=\mathbb{P}(|\bar X^E_N(t_{k+1})-\bar\mu_{k+1}|\leq q_{\alpha}\cdot\bar\sigma_{k+1})=\mathbb{P}\Bigl(\frac{|\Delta W_k|}{\sqrt{h}}\leq q_{\alpha}\Bigr)=\alpha,
\end{equation}
while if $b(t_k,\bar x_k,\hat\theta_n)=0$ then $\bar X^E_N(t_{k+1})=\bar\mu_{k+1}$ with probability one. Since $\bar X^E_N(t_{k+1})\approx \hat X(t_{k+1})$ and the law of $\bar X^E_N(t_{k+1})$ is $N(\bar\mu_{k+1},\bar\sigma^2_{k+1})$ (if $\bar\sigma_{k+1}>0$), in the case when $\bar\sigma_{k+1}>0$  the prediction $\bar x_{k+1}$ of the value $\hat X(t_{k+1})$ we  take as a number drawn from the distribution $N(\bar\mu_{k+1},\bar\sigma^2_{k+1})$, while $\bar x_{k+1}:=\bar x_k+a(t_k,\bar x_k,\hat\theta_n)\cdot h$ if $\bar\sigma_{k+1}=0$. 

The drawback of this approach is that, with a small probability $1-\alpha$ but still, the drawn value $\bar x_{k+1}$ may lie outside the prediction interval $\mathcal{\bar P}_{k+1}(\alpha)$. We overcome this problem in the Approach 2. 
\newline\newline
{\bf Approach 2:}  Assume that we have given $\bar x_k$ and $\tilde x_k$, where $\bar x_k$ is the  prediction of the  value of $\hat X(t_k)$ and $\tilde x_k$ is equal to $\tilde X^E_N(t_k)$- the simulated value of Euler-Maruyama scheme at $t_k$ (for $k=0$ we take $\tilde x_0=\bar x_0:=x_n$). The one step of Euler-Maruyama scheme from  $t_k$ to $t_{k+1}=t_k+h$, starting from $\tilde x_k$ at $t_k$, is defined as follows
\begin{equation}
	\tilde X^E_N(t_{k+1})=\tilde x_k+a(t_k,\tilde x_k,\hat\theta_n)\cdot h+b(t_k,\tilde x_k,\hat\theta_n)\cdot \Delta W_k,
\end{equation} 
where $\Delta W_k=W(t_{k+1})-W(t_k)$ is the increment of the Wiener process that is independent of $\Sigma_{t_k}$ and, therefore, also of $\tilde X^E_N(t_k)$. Again, let $q_{\alpha}$ be the  two-sided $\alpha$-quantile  of the normal distribution $N(0,1)$. We define the  prediction interval for (the unknown value) $\hat X(t_{k+1})$ as follows
\begin{equation}
\label{p_int_def_2}
	\mathcal{\tilde P}_{k+1}(\alpha):=\Bigl[\tilde\mu_{k+1}-q_{\alpha}\cdot\tilde\sigma_{k+1},\tilde\mu_{k+1}+q_{\alpha}\cdot\tilde\sigma_{k+1}\Bigr],
\end{equation}
with
\begin{eqnarray}
	&&\tilde\mu_{k+1}:=\tilde x_k+a(t_k,\tilde x_k,\hat\theta_n)\cdot h,\\
	&&\tilde\sigma_{k+1}:=\sqrt{h}\cdot |b(t_k,\tilde x_k,\hat\theta_n)|.
\end{eqnarray}
Note that it holds $\tilde\mu_{k+1}=\mathbb{E}(\tilde X^E_N(t_{k+1}) \ | \ \Sigma_{t_k})$ and $\tilde\sigma^2_{k+1}=Var(\tilde X^E_N(t_{k+1}) \ | \ \Sigma_{t_k})$. If $b(t_k,\tilde x_k,\hat\theta_n)\neq 0$, then we have
\begin{eqnarray}
	&&\mathbb{P}(\tilde X^E_N(t_{k+1})\in\mathcal{\tilde P}_{k+1}(\alpha) \ | \ \Sigma_{t_k})=\mathbb{P}(|\tilde X^E_N(t_{k+1})-\tilde\mu_{k+1}|\leq q_{\alpha}\cdot\tilde\sigma_{k+1}\ | \ \Sigma_{t_k})\notag\\
 &&=\mathbb{P}\Bigl(\frac{|\Delta W_k|}{\sqrt{h}}\leq q_{\alpha}\ | \ \Sigma_{t_k}\Bigr)=\mathbb{P}\Bigl(\frac{|\Delta W_k|}{\sqrt{h}}\leq q_{\alpha}\Bigr)=\alpha,
\end{eqnarray}
and hence
\begin{equation}
    \mathbb{P}(\tilde X^E_N(t_{k+1})\in\mathcal{\tilde P}_{k+1}(\alpha) )=\alpha,
\end{equation}
while if $b(t_k,\tilde x_k,\hat\theta_n)=0$, then $\tilde X^E_N(t_{k+1})=\tilde\mu_{k+1}$ with probability one. Now, the main difference, when comparing to the  Approach 1, is that in the Approach 2 we take
\begin{equation}
	\bar x_{k+1}:=\tilde \mu_{k+1} 
\end{equation}
as the predicted value of $\hat X(t_{k+1})$, while for the next starting point for the one step of the Euler-Maruyama scheme we set
\begin{equation}
    \tilde x_{k+1}:= \tilde X^E_N(t_{k+1})   
\end{equation}
  if $\tilde\sigma^2_{k+1}>0$, and $\tilde x_{k+1}:=\tilde\mu_{k+1}$ if $\tilde\sigma_{k+1}=0$. (Hence, for fixed $\tilde x_k$, the value of  $\tilde x_{k+1}$ might be drawn from the distribution $N(\tilde\mu_{k+1},\tilde\sigma^2_{k+1})$.) In this approach the predicted value $\bar x_{k+1}$ is the mid point of $\mathcal{\tilde P}_{k+1}(\alpha)$, while the drawn next starting point $\tilde x_{k+1}$ may lie outside the prediction interval $\mathcal{\tilde P}_{k+1}(\alpha)$ but with  (small) probability $1-\alpha$. 

The procedure described in Approach 2  of Monte Carlo forecasting for SDEs might be implemented as follows:
\begin{itemize}
	\item [\bf 1)] draw independent random samples $Z_0,Z_1,\ldots,Z_{N-1}$ form the distribution $N(0,1)$,
		\item [\bf 2)] let $\tilde x_0:=x_n$ and for $k=0,1,\ldots,N-1$ compute
		\begin{equation}
			\tilde x_{k+1}=\tilde x_{k}+a(t_k,\tilde x_k,\hat\theta_n)\cdot h+				\sqrt{h}\cdot b(t_k,\tilde x_k,\hat\theta_n)\cdot Z_k,
			\end{equation}		 
		\item [\bf 3)] having $\{\tilde x_k\}_{k=0,1,\ldots,N}$, take $\bar x_0:=x_n$ and compute for $k=1,2,\ldots, N$ the predictions $\bar x_k$ and the prediction intervals $\mathcal{\tilde P}_k(\alpha)$ as follows
		\begin{eqnarray}
			&&\bar x_k = \tilde\mu_k=\tilde x_{k-1}+a(t_{k-1},\tilde x_{k-1},\hat\theta_n)\cdot h,\\
			&&\tilde\sigma_k=\sqrt{h}\cdot |b(t_{k-1},\tilde x_{k-1},\hat\theta_n)|,\\
			&&\mathcal{\tilde P}_k(\alpha)=[\tilde\mu_k-q_{\alpha}\cdot\tilde\sigma_k, \tilde\mu_k+q_{\alpha}\cdot\tilde\sigma_k].
		\end{eqnarray}			
\end{itemize}

\begin{rem} In the case when we have access to the real data (ground truth) at the time $t_k$ we can take it as the staring point $\tilde x_k$ for the one step of the Euler-Maruyama algorithm (from $t_k$ to $t_{k+1}$) in the Approach 2.
\end{rem}
\begin{rem} Note that if we take $\tilde x_k=\bar x_k=\tilde\mu_k$ for all $k$ (no matter if $\tilde\sigma^2_k=0$ or not)  in the Approach 2, then $(\tilde x_k)_{k=0,1,\ldots,N}$ corresponds the the Euler scheme when applied to the following ODE
\begin{equation}
	\hat x'(t)=a(t,\hat x(t),\hat\theta_n), \quad t\in [T,T+\Delta], \quad \hat x(T)=x_n.
\end{equation}
The drawback is that we lose in this case information about the noise modeled in \eqref{hat_SDE} - the prediction of $\hat X$ is oversmoothed. 
\end{rem}
We define the following functions
\begin{eqnarray}
    &&\tilde\mu_{k+1}(x)=x+h\cdot a(t_k,x,\hat\theta_n),\\
    &&\tilde\sigma^2_{k+1}(x)=h\cdot b^2(t_k,x,\hat\theta_n),
\end{eqnarray}
and let
\begin{equation}
    \mathcal{\tilde P}_{k+1}(\alpha,x)=[\tilde\mu_{k+1}(x)-q_{\alpha}\cdot\tilde\sigma_{k+1}(x),\tilde\mu_{k+1}(x)+q_{\alpha}\cdot\tilde\sigma_{k+1}(x)].
\end{equation}
Then, for $k=0,1,\ldots,N-1$
\begin{equation}
    A_{k+1}:=\{\tilde X^E_N(t_{k+1})\in\mathcal{\tilde P}_{k+1}(\alpha,\tilde X^E_N(t_{k}))\}=\{|\Delta W_k|\leq q_{\alpha}\cdot\sqrt{h}\},
\end{equation}
and hence, the sets $(A_k)_{k=1,\ldots,N}$ are independent. This implies that
\begin{equation}
    \mathbb{P}\Bigl(\bigcap_{k=0}^{N-1}A_{k+1}\Bigr)=\prod_{k=0}^{N-1}\mathbb{P}(A_{k+1})=\alpha^N,
\end{equation}
which gives the probability for the {\it simultaneous prediction interval} 
\begin{equation}
    \bigtimes_{k=0}^{N-1}\mathcal{\tilde P}_{k+1}(\alpha,\tilde X^E_N(t_{k})).
\end{equation}
In order to have an accurate prediction for a relatively large number of steps (for example, $N=20$) we should take $\alpha=0,998$. This corresponds to $q_{\alpha}\approx 3.09$.
\subsection{Ellipsoidal prediction regions in multidimensional case}
Let us take, as in the one-dimensional case, that
\begin{equation}
    \tilde\mu_{k+1}(x):= x+h\cdot a(t_k,x,\hat\theta_n), \quad x\in\mathbb{R}^d.
\end{equation}
Assume that we have given $\bar x_k$ and $\tilde x_k$, where $\bar x_k$ is the  prediction of the  value of $\hat X(t_k)$ and $\tilde x_k$ is equal to $\tilde X^E_N(t_k)$- the simulated value of Euler-Maruyama scheme at $t_k$ (for $k=0$ we take $\tilde x_0=\bar x_0:=x_n$). Then for the one step of Euler-Maruyama scheme from  $t_k$ to $t_{k+1}=t_k+h$, starting from $\tilde x_k$ at $t_k$, we have
\begin{equation}
    b^T(t_k,\tilde x_k,\hat\theta_n)\cdot\Bigl(\tilde X^E_N(t_{k+1})-\tilde\mu_{k+1}(\tilde x_k)\Bigr)=(b^Tb)(t_k,\tilde x_k,\hat\theta_n)\cdot\Delta W_k=\sqrt{h}\cdot (b^Tb)(t_k,\tilde x_k,\hat\theta_n)\cdot Z_k,
\end{equation}
where $Z_k:=\Delta W_k/\sqrt{h}$ has $m$-dimensional normal distribution $N(0,I_m)$. We impose the following assumption
\begin{equation*}
    (M) \ \  \hbox{for all} \ (t,x,z) \ \hbox{the matrix} \ b(t,x,z)\in\mathbb{R}^{d\times m} \ \hbox{has} \ m \ \hbox{linearly independent columns, i.e.}  \ rank(b(t,x,z))=m.
\end{equation*}
Since for all $(t,x,z)$ we have $rank(b(t,x,z))\leq\min\{d,m\}$, the assumption (M) implies that
\begin{equation}
\label{more_d_m}
    d\geq m.
\end{equation}
Then, for all $(t,x,z)$, the matrix $(b^Tb)(t,x,z)\in\mathbb{R}^{d\times d}$ is symmetric and positive definite. Hence,
\begin{equation}
    h^{-1/2}\cdot (b^Tb)^{-1}(t_k,\tilde x_k,\hat\theta_n)\cdot b^T(t_k,\tilde x_k,\hat\theta_n)\cdot\Bigl(\tilde X^E_N(t_{k+1})-\tilde\mu_{k+1}(\tilde x_k)\Bigr)=Z_k=\Delta W_k/\sqrt{h}
\end{equation}
has normal distribution $N(0,I_m)$. Therefore, for $B(0,R)=\{x\in\mathbb{R}^m \ | \ \|x\|\leq R\}$, $R\in [0,+\infty]$, we have
\begin{eqnarray}
    &&p(R)=\mathbb{P}\Biggl( h^{-1/2}\cdot (b^Tb)^{-1}(t_k,\tilde x_k,\hat\theta_n)\cdot b^T(t_k,\tilde x_k,\hat\theta_n)\cdot\Bigl(\tilde X^E_N(t_{k+1})-\tilde\mu_{k+1}(\tilde x_k)\Bigr)\in B(0,R)\Biggr)\notag\\
    &&=\int\limits_{B(0,R)}\frac{1}{(2\pi)^{m/2}}e^{-\|x\|^2/2}dx,
\end{eqnarray}
and we call $p(R)$ the {\it confidence level}. For all $y\in\mathbb{R}^d$ we define 
\begin{equation}
    \mathcal{P}_{k+1}(R,y)=\{z\in\mathbb{R}^d \ | \ h^{-1/2}\|b^+(t_k,y,\hat\theta_n)\cdot (z-\tilde\mu_{k+1}(y))\|\leq R\},
\end{equation}
where
\begin{equation}
    b^+(t,y,z):=(b^Tb)^{-1}(t,y,z)\cdot b^T(t,y,z)\in\mathbb{R}^{m\times d}
\end{equation}
is the Moore-Penrose inverse of $b(t,y,z)$, since we have assumed that $rank(b(t,y,z))=m$ and $d\geq m$, see \cite{StoBurl}. Note that $b^+(t,y,z)$ is also the left-inverse of $b(t,y,z)$, since
\begin{equation}
    b^+(t,y,z)\cdot b(t,y,z)=I_m.
\end{equation}
Moreover, if we denote by
\begin{equation}
    E_k(y):=(b^+(t_k,y,\hat\theta_n))^T\cdot b^+(t_k,y,\hat\theta_n),
\end{equation}
then we can write
\begin{equation}
    \mathcal{P}_{k+1}(R,y)=\{z\in\mathbb{R}^d \ | \ (z-\tilde\mu_{k+1}(y))^T\cdot E_k(y)\cdot (z-\tilde\mu_{k+1}(y))\leq h R^2\}.
\end{equation}
For all $y\in\mathbb{R}^d$ the matrix $E_k(y)\in\mathbb{R}^{d\times d}$ is symmetric and positive semi-definite, and for all $y,z\in\mathbb{R}^d$
\begin{equation}
    (z-\tilde\mu_{k+1}(y))^T\cdot E_k(y)\cdot (z-\tilde\mu_{k+1}(y))=\|b^+(t_k,y,\hat\theta_n)\cdot (z-\tilde\mu_{k+1}(y))\|^2\geq 0.
\end{equation}
We define the {\it prediction region}  for $\hat X(t_{k+1})$ as 
\begin{equation}
    \mathcal{\tilde P}_{k+1}(R):=\mathcal{P}_{k+1}(R,\tilde x_k).
\end{equation}
Then
\begin{equation}
\label{mult_dim_XE_P_1}
    \mathbb{P}\Bigl(\tilde X^E_N(t_{k+1})\in\mathcal{\tilde P}_{k+1}(R) \ | \ \Sigma_{t_k}\Bigr)=\mathbb{P}\Bigl(h^{-1/2}\|\Delta W_k\|\leq R \ | \ \Sigma_{t_k}\Bigr)=\mathbb{P}\Bigl(h^{-1/2}\|\Delta W_k\|\leq R \Bigr)=p(R),
\end{equation}
and therefore
\begin{equation}
    \mathbb{P}\Bigl(\tilde X^E_N(t_{k+1})\in\mathcal{\tilde P}_{k+1}(R) \Bigr)=p(R).
\end{equation}
We take
\begin{equation}
	\bar x_{k+1}:=\tilde \mu_{k+1}(\tilde x_k)=\tilde x_k+h\cdot a(t_k,\tilde x_k,\hat\theta_n) 
\end{equation}
as the predicted value of $\hat X(t_{k+1})$, while for the next starting point for the one step of the Euler-Maruyama scheme we set
\begin{equation}
    \tilde x_{k+1}:= \tilde X^E_N(t_{k+1}).  
\end{equation}
Hence, for fixed $\tilde x_k$, the value of  $\tilde x_{k+1}$ might be drawn from the distribution $N(\tilde\mu_{k+1}(\tilde x_k),\tilde V^2_{k+1}(\tilde x_k))$, where $\tilde V^2_{k+1}(y):=h\cdot (bb^T)(t_k,y,\hat\theta_n)$.

The procedure  of Monte Carlo forecasting for SDEs might be described as follows:
\begin{itemize}
	\item [\bf 1)] draw independent random samples $Z_0,Z_1,\ldots,Z_{N-1}$ form the distribution $N(0,I_m)$,
		\item [\bf 2)] let $\tilde x_0:=x_n$ and for $k=0,1,\ldots,N-1$ compute
		\begin{equation}
			\tilde x_{k+1}=\tilde x_{k}+a(t_k,\tilde x_k,\hat\theta_n)\cdot h+				\sqrt{h}\cdot b(t_k,\tilde x_k,\hat\theta_n)\cdot Z_k,
			\end{equation}		 
		\item [\bf 3)] having $\{\tilde x_k\}_{k=0,1,\ldots,N}$, take $\bar x_0:=x_n$, then, for $k=1,2,\ldots, N$, compute  the predictions $\bar x_k$ and set the  prediction regions $\mathcal{\tilde P}_k(R)$ as follows
		\begin{eqnarray}
			&&\bar x_k = \tilde\mu_k=\tilde x_{k-1}+a(t_{k-1},\tilde x_{k-1},\hat\theta_n)\cdot h,\\
			&&\mathcal{\tilde P}_k(R)=\{z\in\mathbb{R}^d \ | \ h^{-1/2}\|b^+(t_{k-1},\tilde x_{k-1},\hat\theta_n)\cdot (z-\tilde\mu_{k})\|\leq R\}.
		\end{eqnarray}			
\end{itemize}
For $k=0,1,\ldots,N-1$
\begin{equation}
    A_{k+1}:=\{\tilde X^E_N(t_{k+1})\in\mathcal{P}_{k+1}(R,\tilde X^E_N(t_{k}))\}=\{\|\Delta W_k\|\leq R\cdot\sqrt{h}\}=\{Z_k\in B(0,R)\},
\end{equation}
and hence, the sets $(A_k)_{k=1,\ldots,N}$ are independent. This implies that
\begin{equation}
\label{probab_wt_pred_reg}
    \mathbb{P}\Bigl((\tilde X^E_N(t_1),\tilde X^E_N(t_2),\ldots, \tilde X^E_N(t_N))\in \bigtimes_{k=0}^{N-1}\mathcal{P}_{k+1}(R,\tilde X^E_N(t_{k}))\Bigr)=\mathbb{P}\Bigl(\bigcap_{k=0}^{N-1}A_{k+1}\Bigr)=\prod_{k=0}^{N-1}\mathbb{P}(A_{k+1})=(p(R))^N,
\end{equation}
which gives the probability that the obtained $N$-step prediction $(\tilde X^E_N(t_1),\tilde X^E_N(t_2),\ldots, \tilde X^E_N(t_N))$ belongs to the {\it simultaneous prediction region} 
\begin{equation}
\label{ellip_pred_reg}
    \bigtimes_{k=0}^{N-1}\mathcal{P}_{k+1}(R,\tilde X^E_N(t_{k})).
\end{equation}

In a special case we can describe the geometry of $\mathcal{P}_{k+1}(R,y)$ in a more precise way. Namely, let us assume that for all $(t,y,z)$ 
\begin{equation}
    rank(b^+(t,y,z))=d.
\end{equation}
Since $rank(b^+(t,y,z))=d\leq \min\{d,m\}\leq m$, by \eqref{more_d_m} this assumption force the following
\begin{equation}
    d=m,
\end{equation}
and the matrix $b(t,y,z)$ has to be a square matrix in $\mathbb{R}^{d\times d}$. Moreover, it is of full rank (equal to $d$), and therefore it is also non-singular. In this case
for all $(t,y,z)$
\begin{equation}
    b^+(t,y,z)=b^{-1}(t,y,z),
\end{equation}
and the symmetric matrix
\begin{equation}
    E_k(y)=\Bigl(b(t_k,y,\hat\theta_n)\cdot b^T(t_k,y,\hat\theta_n)\Bigr)^{-1}
\end{equation}
is positive definite (In particular, this also implies that $\tilde V^2_{k+1}(y)$ is positive definite, since $\tilde V^2_{k+1}(y)=h\cdot E^{-1}_{k}(y)$.) Moreover,
\begin{eqnarray}
\label{pk_as_ellipsoid}
    &&\mathcal{P}_{k+1}(R,y)=\{z\in\mathbb{R}^d \ | \ h^{-1/2}\|b^{-1}(t_k,y,\hat\theta_n)\cdot (z-\tilde\mu_{k+1}(y))\|\leq R\}\notag\\
    &&=\tilde\mu_{k+1}(y)+h^{1/2}b(t_k,y,\hat\theta_n) B(0,R),
\end{eqnarray}
and, therefore, in the case when $E_k(y)$ is positive definite, the set $\mathcal{P}_{k+1}(R,y)$ is an ellipsoid centered at $\tilde\mu_{k+1}(y)$. The eigenvectors of $E_k(y)/(hR^2)$  are the principal axes of the ellipsoid, and length of each corresponding semi-axe is $1/(\lambda_j(y))^{1/2}$, where $\lambda_j(y)$'s are the (strictly positive) eigenvalues of $E_k(y)/(hR^2)$. 
\begin{rem}
    Note that we can write $\tilde V^2_{k+1}(y)=(\tilde V_{k+1}(y))^T\cdot \tilde V_{k+1}(y)$ where
    \begin{equation}
        \tilde V_{k+1}(y)=h^{1/2}\cdot b^T(t_k,y,\hat\theta_n),
    \end{equation}
    and
    \begin{equation}
        \mathcal{P}_{k+1}(R,y)=\{z\in\mathbb{R}^d \ | \ \|[\tilde V_{k+1}(y) \tilde V^T_{k+1}(y)]^{-1}\tilde V_{k+1}(y)(z-\tilde\mu_{k+1}(y))\|\leq R\}
    \end{equation}
\end{rem}
\begin{rem}
    Of course in the case $m=d=1$ we obtain the same prediction intervals as it was descibed in the one-dimensional case, since
    \begin{equation}
        \mathcal{P}_{k+1}(q_{\alpha},\tilde x_k)=\mathcal{\tilde P}_{k+1}(\alpha).
    \end{equation}
\end{rem}
\begin{rem}
    In this section we showed how to construct ellipsoidal prediction regions for uncertainty characterization when time-series under consideration are assumed to follow from SDEs-based models. See also \cite{gopi} where the authors use ellipsoidal prediction regions in different econometric fashioned approach to time-series uncertainty characterization.
\end{rem}
\subsection{Concluding remarks - why the geometry of prediction regions matters}
The  procedure described in the previous section of obtaining  prediction regions in multidimensional case takes into account the fact that coordinates of $\hat X$ and $\tilde X^E_N$ are correlated. Because of that it is possible to compute explicitly: 
\begin{itemize}
    \item   the probability \eqref{mult_dim_XE_P_1} that all of the coordinates $\tilde X^E_{N,j}$ are in $\mathcal{\tilde P}_{k+1}(R)$ simultaneously,
    \item  the probability \eqref{probab_wt_pred_reg} that the whole predicted $N$-steps forward belong to the simultaneous prediction region \eqref{ellip_pred_reg}.
\end{itemize}
One can however try an alternative (and, at first sight, more natural) approach in which one computes prediction intervals for each coordinate  $\hat X_j$ separately (by using an approach known from the one-dimensional case) and then takes a Cartesian product of them (which is in fact a hypercube in $\mathbb{R}^d$).  However, there is a pitfall in that approach, which we now describe in detail.

Note that for all $j=1,2,\ldots,d$, $k=0,1,\ldots,N-1$
\begin{equation}
\label{XE_j_law}
    \tilde X^E_{N,j}(t_{k+1})\ | \ \Sigma_{t_k}\sim N\Bigl(\tilde\mu_{k+1,j}(\tilde X^E_{N}(t_{k})), \tilde\sigma^2_{k+1,j}(\tilde X^E_{N}(t_{k}))\Bigr),
\end{equation}
where for all $x\in\mathbb{R}^d$
\begin{eqnarray}
    &&\tilde\mu_{k+1,j}(x)=x_j+h \cdot a_j(t_k,x,\hat\theta_n),\\
    &&\tilde\sigma^2_{k+1,j}(x)=h\cdot\sum_{l=1}^m (b(t_k,x,\hat\theta_n))^2.
\end{eqnarray}
Let us take
\begin{equation}
    \mathcal{\tilde P}_{k+1,j}(\alpha,x)=[\tilde\mu_{k+1,j}(x)-q_{\alpha}\cdot\tilde\sigma_{k+1,j}(x), \tilde\mu_{k+1,j}(x)+q_{\alpha}\cdot\tilde\sigma_{k+1,j}(x)],
\end{equation}
and
\begin{equation}
    B_{k+1,j}=\{\tilde X^E_{N,j}(t_{k+1})\in\mathcal{\tilde P}_{k+1,j}(\alpha,\tilde X^E_{N}(t_{k}))\},
\end{equation}
then, by \eqref{XE_j_law}, we have for all $j=1,2,\ldots,d$ that
\begin{equation}
    \mathbb{P}\Bigl(B_{k+1,j} \ | \ \Sigma_{t_k}\Bigr)=\mathbb{P}\Biggl(\frac{|\tilde X^E_{N,j}(t_{k+1})-\tilde\mu_{k+1,j}(\tilde X^E_{N}(t_{k}))|}{\tilde\sigma_{k+1,j}(\tilde X^E_{N}(t_{k}))}\leq q_{\alpha} \ \Bigl| \ \Sigma_{t_k}\Biggr)=\alpha,
\end{equation}
and hence for each coordinate of $\tilde X^E_{N}(t_{k+1})$ we have
\begin{equation}
     \mathbb{P}(B_{k+1,j})=\alpha.
\end{equation}
Note that for fixed $k=1,2,\ldots,d$ the sets
\begin{displaymath}
    B_{k+1,1},B_{k+1,2},\ldots, B_{k+1,d}
\end{displaymath}
are not independent (even conditionally wrt $\Sigma_{t_k}$). Therefore, we do not have the exact value (or at least good estimates) of the probability
\begin{equation}
    \mathbb{P}(B_{k+1})=\mathbb{P}(\tilde X^E_{N}(t_{k+1})\in\mathcal{\tilde B}_{k+1}(\alpha,\tilde X^E_N(t_k))),
\end{equation}
where
\begin{equation}
    B_{k+1}=\bigcap_{j=1}^d B_{k+1,j}
\end{equation}
and
\begin{equation}
    \mathcal{\tilde B}_{k+1}(\alpha,x)=\bigtimes_{j=1}^d\mathcal{\tilde P}_{k+1,j}(\alpha,x).
\end{equation}
Moreover, the sets $\{B_{k+1}\}_{k=0,1,\ldots,N-1}$ are also not independent. Hence, in general it is not known how to compute the probability
\begin{equation}
    \mathbb{P}\Bigl((\tilde X^E_N(t_1),\tilde X^E_N(t_2),\ldots,\tilde X^E_N(t_N))\in\bigtimes_{k=0}^{N-1}\mathcal{\tilde B}_{k+1}(\alpha,\tilde X^E_N(t_k))\Bigr)=\mathbb{P}\Bigl(\bigcap_{k=0}^{N-1}B_{k+1}\Bigr),
\end{equation}
and this method of constructing prediction regions do not allows us to control the probability that the whole predicted sequence belong to the simultaneous prediction region. This is an obvious drawback.
\section{Exercises}
\begin{itemize}
	\item [1.] Show for the equation \eqref{BS_eq} that the maximum likelihood estimators of $(\mu,\sigma^2)$ and quasi maximum likelihood estimators coincide.\\
	Hint. Apply QML procedure to \eqref{log_S}.
 \item [2.] Give a proof of \eqref{pk_as_ellipsoid}.
 \item [3.] Show for $j=1,2,\ldots,d$, $k=0,1,\ldots,N-1$ that
\begin{eqnarray}
    Var\Bigl(\sum_{l=1}^m b_{jl}(t_k,\tilde X^E_{N}(t_k),\hat\theta_n)\cdot\Delta W_k^l \ | \ \Sigma_{t_k}\Bigr)=\tilde\sigma^2_{k+1,j}(\tilde X^E_{N}(t_k)).
\end{eqnarray}
\item [4.] Justify that $(X^E(t_k))_{k=0,1,\ldots,n}$ is a discrete-time Markov process.
\end{itemize}
\chapter{Auxiliary results--part 1}
In this notes we use many results from  the theory of stochastic processes, including basics from measure and probability theory. We will now recall some of them. For a comprehensive exposition see, for example, \cite{Ash}, \cite{Cinlar}, \cite{cohn}, \cite{elcoh}, \cite{KALLEN}, \cite{MAKPOD}, \cite{MEDVEG}, \cite{prott}, \cite{ttao}.

By $\|\cdot\|_2$ we mean the Euclidean norm on $\mathbb{R}^d$.

Let $X$ be an arbitrary set. A collection $\mathcal{A}$ of subsets of $X$ is a {\it $\sigma$-field} on $X$ if
\begin{itemize}
	\item [(i)] $X\in\mathcal{A}$,
	\item [(ii)] for every $A\in\mathcal{A}$, the set $X\setminus A$ belongs to $\mathcal{A}$,
	\item [(iii)] for any infinite sequence $(A_n)_{n\in\mathbb{N}}\subset\mathcal{A}$ we have that $\bigcup_{n\in\mathbb{N}}A_n\in\mathcal{A}$,
	\item [(iv)]for any infinite sequence $(A_n)_{n\in\mathbb{N}}\subset\mathcal{A}$ we have that $\bigcap_{n\in\mathbb{N}}A_n\in\mathcal{A}$.
\end{itemize} 
As we can see  a $\sigma$-field on $X$ is a subclass of $2^X$\footnote{In fact, $2^X$ is the largest $\sigma$-field on $X$. The family $\{\emptyset,X\}$ is the smallest $\sigma$-field on $X$. For any $\sigma$-field on $X$ the following inclusion holds: $\{\emptyset,X\}\subset\mathcal{A}\subset 2^X$.} that contains $X$ and is closed with respect to complementation,  countable unions and intersections.  Note also that, in the definition of $\sigma$-field we could
have used only conditions $(i)$, $(ii)$, and $(iii)$, or only conditions $(i)$, $(ii)$, and $(iv)$, as it is sometimes in literature. Moreover, sometimes the authors use the term $\sigma$-algebra instead of $\sigma$-field.
\\
A measurable space is a pair  $(X,\mathcal{A})$ where $X$ is a set and $\mathcal{A}$ is a $\sigma$-field on $X$.

Let $X$ be a set and let $\mathcal{C}$ be an arbitrary class of subsets of $X$. By $\sigma(\mathcal{C})$ we denote the intersection of all those $\sigma$-fields that contain $\mathcal{C}$. It turns out that $\sigma(\mathcal{C})$ is  the smallest $\sigma$-field that contains $\mathcal{C}$, and it is called the $\sigma$-field {\it generated} by $\mathcal{C}$. Such $\sigma$-field always exists, see, for example, Corollary 1.1.3 in \cite{cohn}. If $\mathcal{C}$ is already a $\sigma$-field then $\sigma(\mathcal{C})=\mathcal{C}$.
Furthermore, the following holds (see, for example, \cite{Cinlar}).
\begin{prop} Let $\mathcal{C}$ and $\mathcal{D}$ be two families of subsets of $X$. Then:
	\begin{itemize}
		\item [(i)] If $\mathcal{C}\subset\mathcal{D}$ then $\sigma(\mathcal{C})\subset\sigma(\mathcal{D})$
		\item [(ii)] If $\mathcal{C}\subset\sigma(\mathcal{D})$ then $\sigma(\mathcal{C})\subset\sigma(\mathcal{D})$
		\item [(iii)] If $\mathcal{C}\subset\sigma(\mathcal{D})$ and $\mathcal{D}\subset\sigma(\mathcal{C})$, then $\sigma(\mathcal{C})=\sigma(\mathcal{D})$
		\item [(iv)] If $\mathcal{C}\subset\mathcal{D}\subset\sigma(\mathcal{C})$, $\sigma(\mathcal{C})=\sigma(\mathcal{D})$
	\end{itemize}
\end{prop}
The intersection of an arbitrary (even uncountable) class of $\sigma$-fields on $X$ is again a $\sigma$-field on $X$. This is not however true when considering union of $\sigma$-fields. Namely, except in some particular cases,  the union of two $\sigma$-fields $\mathcal{A}_1$, $\mathcal{A}_2$ on $X$ is not a $\sigma$-field. The $\sigma$-field generated by $\mathcal{A}_1\cup\mathcal{A}_2$ is denoted by $\mathcal{A}_1\vee \mathcal{A}_2$,
\begin{equation}
	\mathcal{A}_1\vee \mathcal{A}_2:=\sigma(\mathcal{A}_1\cup\mathcal{A}_2).
 \end{equation}
 In general, if $(\mathcal{A}_i)_{i\in I}$ is a family of $\sigma$-fields on $X$, where the index set $I$ might be even uncountable, then we define
 \begin{equation}
	\mathcal{A}_{I}:=\sigma\Bigl(\bigcup_{i\in I}\mathcal{A}_i\Bigr),
 \end{equation}
 the $\sigma$-field generated by $\bigcup_{i\in I}\mathcal{A}_i$. It is sometimes also denoted by $\bigvee_{i\in I}\mathcal{A}_i$. Note that the above definition of $\bigvee_{i\in I}\mathcal{A}_i$ is valid for a family $(\mathcal{A}_i)_{i\in I}$ of arbitrary classes of subsets of $X$ (i.e., some of $\mathcal{A}_i$'s do not have to be $\sigma$-fields on $X$). This follows from the following fact.
 \begin{prop}  
\label{sum_s_alg}
Let $(\mathcal{A}_i)_{i\in I}$ be a family of classes of subsets of $X$. Then
	\begin{equation}
		\sigma\Bigl(\bigcup_{i\in I}\mathcal{A}_i\Bigr)=\sigma\Bigl(\bigcup_{i\in I}\sigma(\mathcal{A}_i)\Bigr).
	\end{equation}
 \end{prop}
 {\bf Proof.}  Since for all $i\in I$ we have $\mathcal{A}_i\subset\sigma(\mathcal{A}_i)$, we obtain that $\displaystyle{\bigcup_{i\in I}\mathcal{A}_i\subset\bigcup_{i\in I}\sigma(\mathcal{A}_i)}$ and hence
 \begin{equation}
 \label{prop_inf_sum_1}
	\sigma\Bigl(\bigcup_{i\in I}\mathcal{A}_i\Bigr)\subset\sigma\Bigl(\bigcup_{i\in I}\sigma(\mathcal{A}_i)\Bigr).
 \end{equation}
 On the other hand,  $\displaystyle{\mathcal{A}_j\subset\bigcup_{i\in I}\mathcal{A}_i}$ for all $j\in I$ and therefore $\displaystyle{\sigma(\mathcal{A}_j)\subset\sigma\Bigl(\bigcup_{i\in I}\mathcal{A}_i\Bigr)}$ for all $j\in I$. Hence, $\displaystyle{\bigcup_{i\in I}\sigma(\mathcal{A}_i)\subset\sigma \Bigl(\bigcup_{i\in I}\mathcal{A}_i\Bigr)}$ and
 \begin{equation}
 \label{prop_inf_sum_2}
	\sigma\Bigl(\bigcup_{i\in I}\sigma(\mathcal{A}_i)\Bigr)\subset \sigma \Bigl(\bigcup_{i\in I}\mathcal{A}_i\Bigr).
 \end{equation}
 Combining \eqref{prop_inf_sum_1} and \eqref{prop_inf_sum_2} we get the thesis. \ \ \ $\blacksquare$ \\ \\
If $X$ is a topological space, then the $\sigma$-field generated by the collection $Top X$ of all open subsets of $X$ is called the {\it Borel $\sigma$-field} on $X$ and it is denoted by $\mathcal{B}(X)$ (i.e., $\mathcal{B}(X)=\sigma(Top X)$). Its elements are called {\it Borel sets}.

Let $(X_1,\mathcal{A}_1)$ and $(X_2,\mathcal{A}_2)$ be two measurable spaces. A function $f:X_1\to X_2$ is called $\mathcal{A}_1$-to-$\mathcal{A}_2$ measurable (or $\mathcal{A}_1 / \mathcal{A}_2$-measurable) if
\begin{equation}
	f^{-1}(B)\in \mathcal{A}_1 \ \hbox{for all} \ B\in\mathcal{A}_2,
\end{equation}
where $\displaystyle{f^{-1}(B)=\{x\in X_1 \ | \ f(x)\in B}\}$. If $A\in\mathcal{A}_1$ then a function $f:A\to X_2$ is $\mathcal{A}_1 / \mathcal{A}_2$-measurable if $f^{-1}(B)\in\mathcal{A}_1$ for any $B\in\mathcal{A}_2$. Moreover, we define
\begin{equation}
	\sigma(f)=\{f^{-1}(B) \ | \ B\in\mathcal{A}_2\}.
\end{equation}
It turns out that $\sigma(f)$ is a $\sigma$-field on $X_1$ and it is the smallest $\sigma$-field on $X_1$ such that $f$ is measurable relative to it and $\mathcal{A}_2$ (i.e., $\sigma(f) / \mathcal{A}_2$-measurable). We call $\sigma(f)$ the $\sigma$-field generated by $f$. It is easy to check that $f$ is $\mathcal{A}_1 / \mathcal{A}_2$-measurable iff $\sigma(f)\subset\mathcal{A}_1$. Moreover, if $(X_3,\mathcal{A}_3)$ is a measurable space then for any two functions $f_1:X_1\to X_2$ and $f_2:X_2\to X_3$, where $f_2$ is $\mathcal{A}_2 / \mathcal{A}_3$-measurable,  we have $\sigma(f_2\circ  f_1)\subset\sigma(f_1)$. This implies the following fact.
%
\begin{prop} 
\label{meas_comp_f}
 Let $(X_i,\mathcal{A}_i)$, $i=1,2,3$, be measurable spaces, let $f_1:X_1\to X_2$ be $\mathcal{A}_1 /\mathcal{A}_2$-measurable, while $f_2:X_2\to X_3$ is assumed to be $\mathcal{A}_2/\mathcal{A}_3$-measurable. The the composition
\begin{equation}
	f_2\circ f_1:X_1\to X_3
\end{equation} 
is $\mathcal{A}_1 / \mathcal{A}_3$-measurable.
\end{prop}
\begin{prop}Let $(X,\mathcal{A})$ be a measurable space, and let $A\subset X$. Then the function
\begin{equation}
	X\ni x \to \mathbf{1}_A(x)\in\{0,1\}
\end{equation}
is $\mathcal{A} / \mathcal{B}(\mathbb{R})$-measurable iff $A\in\mathcal{A}$.
\end{prop}
In the case when $X_2$ is topological space, equipped with the Borel $\sigma$-field $\mathcal{A}_2=\mathcal{B}(Top X_2)$, we refer to $\mathcal{A}_1 / \mathcal{B}(X_2)$-measurable function $f:X_1\to X_2$ as {\it Borel function}. 

Below we give important examples of Borel functions.
\begin{prop} 
\label{meas_fun_examples}
\begin{itemize}
	\item [(i)] If $X_1$ and $X_2$ are topological spaces, and if $f:X_1\to X_2$ is continuous function then it is Borel measurable, that is $\mathcal{B}(X_1) / \mathcal{B}(X_2)$-measurable. 
	\item [(ii)] Every right (left)-continuous function $f:\mathbb{R}\to\mathbb{R}$ is Borel. Hence, every c\`adl\`ag (c\`agl\`ad) function is Borel measurable.
	\item [(iii)] If $f:[a,b]\to\mathbb{R}$ is a bounded variation function then $f$ is Borel measurable. 
	\item [(iv)] Let $I$ be a subinterval of $\mathbb{R}$, and let $f:I\to\mathbb{R}$ be a monotone function. Then $f$ is Borel measurable.
	\item [(v)] If $f:\mathbb{R}^d\to\mathbb{R}$ is a convex function then it is continuous and hence, Borel measurable.
\end{itemize}
\end{prop}
Let us now assume that $(X_2,\|\cdot\|)$ is a real or complex Banach space and let $\mathcal{A}_2=\mathcal{B}(X_2)=\sigma(Top X_2)$. Then we have another notion of measurability. 
 A function $f:X_1\to X_2$ is {\it strongly measurable} (or {\it Bochner measurable}) if it is Borel measurable and has separable range (i.e., $f(X_1)\subset X_2$ is separable). If $X_2$ is separable, then every $X_2$-valued Borel measurable function is strongly measurable. (It follows from the fact that a subset of separable metric space is itself separable.) Moreover, if $f:X_1\to X_2$ is Borel measurable, then 
	\begin{equation}
		X_1 \ni x\to \|f(x)\|\in [0,+\infty)
	\end{equation}
	is $\mathcal{A}_1 / \mathcal{B}(\mathbb{R})$-measurable. Moreover, we have the following important result, see \cite{cohn}.
\begin{thm} Let $(X_1,\mathcal{A}_1)$ be a measurable space, and let $X_2$ be a real or complex Banach space. Then the set of all strongly measurable functions from $X_1$ to $X_2$ is a vector space.
\end{thm}
We want to stress here that the set of all Borel measurable functions from $X_1$ to $X_2$ can fail to be a vector space, see page 403 in \cite{cohn}. Therefore, separability of the range in the definition of strongly measurable function is crucial.

Let $(X,\mathcal{A},\mu)$ be a measure space and let $p\in [1,+\infty)$. Then we define
\begin{equation}
	L^p(X,\mathcal{A},\mu;\mathbb{R}^d)=\Bigl\{f:X\to\mathbb{R}^d \ \Bigl| \ f-\hbox{Borel function,} \int\limits_X \|f(x)\|_2^p\ \mu(dx)<+\infty  \Bigr\}.
\end{equation}
If we identify functions which are equal $\mu$-almost everywhere, then $L^p(X,\mathcal{A},\mu;\mathbb{R}^d)$ is a Banach space with norm
\begin{equation}
	\|f\|_{L^p(X;\mathbb{R}^d)}=\Bigl(\int\limits_X \|f(x)\|_2^p\ \mu(dx)\Bigr)^{1/p}.
\end{equation}
For the case $d=1$  we simply write $L^p(X)$.

The proof of the following result can be found in \cite{cohn}.
\begin{thm} (Change of measure in Lebesgue integral)
\label{COM} 
	Let $(X,\mathcal{A},\mu)$ be a measure space, let $(Y,\mathcal{B})$ be a measurable space, and let $f:(X,\mathcal{A})\to (Y,\mathcal{B})$ be measurable. Let $g$ be an extended real-valued $\mathcal{B}$-measurable function on $Y$. Then $g$ is $\mu f^{-1}$-integrable if and only if  $g\circ f$ is $\mu$-integrable. If these functions are integrable, then
	\begin{equation}
		\int\limits_Y g d(\mu f^{-1})=\int\limits_X (g\circ f) d\mu.
	\end{equation}
\end{thm}
For the proof of the mean value theorem we refer to, for example, page 167 in \cite{MAKPOD}.
\begin{thm} (Mean value theorem)
\label{MVT}
	Let $E\subset\mathbb{R}^d$ be a connected set with $\lambda_d(E)<+\infty$. If $f:E\to\mathbb{R}$ is continuous and bounded, then there exists $c\in E$ such that
	\begin{equation}
		\int\limits_E f(x)\ \lambda_d(dx)=f(c)\cdot \lambda_d(E).
	\end{equation}
\end{thm}
We also recall fundamental integral inequalities, see, for example, \cite{KALLEN}.
\begin{thm} (H\"older and Minkowski inequalities)
\label{HM_ineq}
	For any Borel measurable real-valued functions $f$ and $g$ on some measure space $(X,\mathcal{A},\mu)$, we have
	\begin{itemize}
		\item [(i)] for all $p,q,r>0$ with $p^{-1}+q^{-1}=r^{-1}$ that it holds
		\begin{equation}
			\Bigl(\int\limits_X |f(x)g(x)|^r\mu(dx)\Bigr)^{1/r}\leq \Bigl(\int\limits_X |f(x)|^p\mu(dx)\Bigr)^{1/p}\cdot\Bigl(\int\limits_X |g(x)|^q\mu(dx)\Bigr)^{1/q}, 
		\end{equation}
		\item [(ii)] for all $p>0$
		\begin{equation}
			\Bigl(\int\limits_X|f(x)+g(x)|^p\mu(dx)\Bigr)^{\min\{p,1\}/p}\leq\Bigl(\int\limits_X|f(x)|^p\mu(dx)\Bigr)^{\min\{p,1\}/p}+\Bigl(\int\limits_X|g(x)|^p\mu(dx)\Bigr)^{\min\{p,1\}/p}.
		\end{equation}
	\end{itemize}
\end{thm}
\begin{thm} (Jensen's inequality)
\label{JEN_ineq}
For any random vector $\xi:\Omega\to\mathbb{R}^d$, defined on a probability space $(\Omega,\Sigma,\mathbb{P})$, and convex function $f:\mathbb{R}^d\to\mathbb{R}$, we have that
\begin{equation}
	f(\mathbb{E}\xi)\leq \mathbb{E}(f(\xi)).
\end{equation}
\end{thm}
We also recall the very useful Gronwall's lemma for Lebesgue integrals, see, for example, \cite{xmao} or Annex C in  \cite{parras}, where more general version of the lemma can be found.
\begin{lem} (continuous version of the Gronwall's lemma)
\label{GLCW}
	Let $-\infty<a<b<+\infty$ and let $u:[a,b]\to\mathbb{R}$ be a Borel measurable integrable function. If there exist $\alpha\in\mathbb{R}$, $\beta\in [0,+\infty)$ such that for all $t\in [a,b]$ the function $u$ satisfies
	\begin{equation}
		u(t)\leq \alpha+\beta\int\limits_a^t u(s)ds,
	\end{equation}
	then for all $t\in [a,b]$ we have that
	\begin{equation}
		u(t)\leq \alpha e^{\beta(t-a)}.
	\end{equation}
\end{lem}
{\bf Proof.} Let us define the function $h:[a,b]\to\mathbb{R}$ as follows
\begin{equation}
\label{GL_1}
	h(t)=\alpha+\beta\int\limits_a^t u(s)ds.
\end{equation}
Then, by the assumption
\begin{equation}
\label{GL_2}
	u(t)\leq h(t).
\end{equation}
for all $t\in [a,b]$. Moreover, $h$ is absolutely continuous and 
\begin{equation}
\label{GL_3}
	h'(t)=\beta\cdot u(t)
\end{equation}
for almost all $t$ in $[0,T]$. Since $\beta\geq 0$, we get from \eqref{GL_2} and \eqref{GL_3} that
\begin{equation}
\label{GL_4}
	h'(t)\leq \beta h(t)
\end{equation}
for almost all $t$ in $[0,T]$. Furthermore, by \eqref{GL_4}
\begin{equation}
	\Bigl(e^{-\beta t}h(t)\Bigr)^{'}=-\beta e^{-\beta t}h(t)+e^{-\beta t}h'(t)\leq 0
\end{equation} 
for almost all $t$ in $[0,T]$. The functions $[a,b]\ni t\to e^{-\beta t}$ and $h$ are absolutely continuous, hence so it is   $[a,b]\ni t\to e^{-\beta t}h(t)$. Thereby
\begin{equation}
e^{-\beta t}h(t)-e^{-\beta a}h(a)=\int\limits_a^t \Bigl(e^{-\beta s}h(s)\Bigr)^{'}ds\leq 0
\end{equation}
for all $t\in [a,b]$. Hence, 
\begin{equation}
	u(t)\leq h(t)\leq e^{\beta (t-a)}h(a)=\alpha e^{\beta (t-a)}, \ t\in [a,b].
\end{equation}
This completes the proof. \ \ \ $\blacksquare$ \\ \\
Let $(X_i,\mathcal{A}_i)$, $i=1,2$, be measurable spaces. On the Cartesian product $X_1\times X_2$ of the sets $X_i$, $i=1,2$, we consider the so called {\it product} $\sigma$-field $\mathcal{A}_1\otimes \mathcal{A}_2$ defined by
\begin{equation}
	\mathcal{A}_1\otimes \mathcal{A}_2=\sigma\Bigl(\{A_1\times A_2 \ | \ A_i\in\mathcal{A}_i, i=1,2\}\Bigr).
\end{equation}
Then $(X_1\times X_2,\mathcal{A}_1\otimes \mathcal{A}_2)$ is a measurable space. Note that we also have
\begin{equation}
	\mathcal{A}_1\otimes \mathcal{A}_2=\sigma(\mathcal{\hat A}_1\cup \mathcal{\hat A}_2),
\end{equation}
where
\begin{eqnarray}
	\mathcal{\hat A}_1=\{A\times X_2 \ | \ A\in\mathcal{A}_1\},\\
	\mathcal{\hat A}_2=\{X_1\times B \ | \ B\in\mathcal{A}_2\},
\end{eqnarray}
see page. 5 in \cite{Cinlar}. Now, let $(X_i,\mathcal{A}_i,\mu_i)$, $i=1,2$, be $\sigma$-finite measure spaces. Then there exists a unique measure $\mu_1\times \mu_2$ on the $\sigma$-field $\mathcal{A}_1\otimes \mathcal{A}_2$ such that 
\begin{equation}
	(\mu_1\times \mu_2)(A_1\times A_2)=\mu_1(A_1)\cdot\mu_2(A_2)
\end{equation}
holds for any $A_i\in\mathcal{A}_i$, $i=1,2$. The measure $\mu_1\times \mu_2$ is called the {\it product measure}. Of course $(X_1\times X_2,\mathcal{A}_1\otimes \mathcal{A}_2,\mu_1\times \mu_2)$ is a measure space. 
\begin{prop} Let $(X,\mathcal{A})$, $(Y_i,\mathcal{C}_i)$ be measurable spaces and let $f_i:X\to Y_i$ for $i=1,2$. Define $H:X\to Y_1\times Y_2$ by
\begin{equation}
	H(x)=(f_1,f_2)(x)=(f_1(x),f_2(x)), \ x\in X.
\end{equation}
Then $H$ is $\mathcal{A} / \mathcal{C}_1\otimes\mathcal{C}_2$-measurable iff $f_i:X\to Y_i$ is $\mathcal{A} / \mathcal{C}_i$-measurable for  $i=1,2$.
\end{prop}
The proposition above has its multi-dimensional generalization, see, for example, Proposition 6.27, page 45 in \cite{Cinlar} or Lemma 1.8 in \cite{KALLEN}.

Below we state the measurability of sections.
\begin{prop} Let $(X_i,\mathcal{A}_i)$, $(Y,\mathcal{C})$, $i=1,2$, be measurable spaces. Let us assume  $f:X_1\times X_2\to Y$ to be $\mathcal{A}_1\otimes\mathcal{A}_2 / \mathcal{C}$-measurable. For any $x\in X_1$ the section 
\begin{equation}
	X_2\ni y\to f_x(y)=f(x,y)\in Y
\end{equation}
is $\mathcal{A}_2 /\mathcal{C}$-measurable, while for any $y\in X_2$ the section
\begin{equation}
	X_1\ni x\to f^y(x)=f(x,y)\in Y
\end{equation}
is $\mathcal{A}_1 /\mathcal{C}$-measurable.
\end{prop}
Hence, product measurability implies measurability of each section. We want to underline here that the converse is not true - without some additional requirements measurability of sections does not imply product measurability. We have however the following  result, that is very useful from a point of view of stochastic processes (see, for example, page 46 in \cite{Cinlar} or Remark 1.14 at page 5. in \cite{KARSHR}).
\begin{lem} 
\label{prod_meas_lem}
Let $(X,\mathcal{A})$ be a measurable space and let $f:X\times\mathbb{R}\to\mathbb{R}$ be such that the section $y\to f(x,y)$ is right-continuous (or left-continuous) for each $x\in X$ and that the section $x\to f(x,y)$ is $\mathcal{A} /\mathcal{B}(\mathbb{R})$-measurable for each $y\in\mathbb{R}$. Then, $f$ is $\mathcal{A}\otimes\mathcal{B}(\mathbb{R}) / \mathcal{B}(\mathbb{R})$-measurable.
\end{lem}
Note that $f$ may be a mapping with values in a separable metric space (such as the euclidean space $\mathbb{R}^d$) and the thesis of Lemma \ref{prod_meas_lem} still holds, see Lemma 6.4.6 in \cite{bogach2}.

\begin{thm} (Tonelli's Theorem)
\label{tonelli_thm}
Let $(X_i,\mathcal{A}_i,\mu_i)$, $i=1,2$, be $\sigma$-finite measure spaces, and let $f:X_1\times X_2\to [0,+\infty]$ be $\mathcal{A}_1\otimes\mathcal{A}_2 / \mathcal{B}(\mathbb{\bar R})$-measurable. Then
\begin{itemize}
	\item [(i)] the function
	\begin{equation}
		X_1\ni x_1\to \int\limits_{X_2}  f(x_1,x_2) \mu_2(dx_2) 
	\end{equation}
	is $\mathcal{A}_1 /\mathcal{B}(\mathbb{\bar R})$-measurable and the function
	\begin{equation}
		X_2\ni x_2\to \int\limits_{X_1}  f(x_1,x_2) \mu_1(dx_1) 
	\end{equation}
	is $\mathcal{A}_2 /\mathcal{B}(\mathbb{\bar R})$-measurable, and
	\item [(ii)] for the function $f$ the following holds
	\begin{eqnarray}
	\label{tonell_eq}
		\int\limits_{X_1\times X_2} f(x_1,x_2) \ \mu_1\times\mu_2(dx_1,dx_2)&=&\int\limits_{X_1}\Bigl(\int\limits_{X_2}f(x_1,x_2)\mu_2(dx_2)\Bigr)\mu_1(dx_1)\notag\\
&=&\int\limits_{X_2}\Bigl(\int\limits_{X_1}f(x_1,x_2)\mu_1(dx_1)\Bigr)\mu_2(dx_2).
	\end{eqnarray}
\end{itemize}
\end{thm}
We stress that we can use \eqref{tonell_eq} for every nonnegative and $\mathcal{A}_1\otimes\mathcal{A}_2 / \mathcal{B}(\mathbb{\bar R})$-measurable function $f$, no matter if it is $\mu_1\times \mu_2$-integrable or not. Hence, for an arbitrary $\mathcal{A}_1\otimes\mathcal{A}_2 / \mathcal{B}(\mathbb{\bar R})$-measurable function $f:X_1\times X_2\to \mathbb{\bar R}$ we can check if it is $\mu_1\times \mu_2$-integrable by applying \eqref{tonell_eq} to $\displaystyle{\int\limits_{X_1\times X_2} |f(x_1,x_2)| \ \mu_1\times\mu_2(dx_1,dx_2)}$. Such situation occurs many times in the theory of stochastic processes.
\begin{thm}
(Fubini's Theorem)
\label{fubini_thm}
Let $(X_i,\mathcal{A}_i,\mu_i)$, $i=1,2$, be $\sigma$-finite measure spaces, and let $f:X_1\times X_2\to [-\infty,+\infty]$ be $\mathcal{A}_1\otimes\mathcal{A}_2 / \mathcal{B}(\mathbb{\bar R})$-measurable and $\mu_1\times\mu_2$-integrable. Then
\begin{itemize}
	\item [(i)] for $\mu_1$-almost every $x_1\in X_1$ the section $f_{x_1}$ is $\mu_2$-integrable and for $\mu_2$-almost every $x_2\in X_2$ the section $f^{x_2}$ is $\mu_1$-integrable,
	\item [(ii)] the functions $I_f$ and $J_f$ defined by
	\begin{equation}
	\label{def_I_f}
		I_f(x_1)=\left\{ \begin{array}{ll}
			\int\limits_{X_2}f_{x_1}(x_2)\mu_2(dx_2) & \ \hbox{if} \ f_{x_1} \ \hbox{is} \ \mu_2-\hbox{integrable} , \\
			0 & \ \hbox{otherwise} 
		\end{array}\right.
	\end{equation}
	and
	\begin{equation}
	\label{def_J_f}
		J_f(x_2)=\left\{ \begin{array}{ll}
			\int\limits_{X_1}f^{x_2}(x_1)\mu_1(dx_1) & \ \hbox{if} \ f^{x_2} \ \hbox{is} \ \mu_1-\hbox{integrable} , \\
			0 & \ \hbox{otherwise}
		\end{array}\right.
	\end{equation}
	belong to $L^1(X_1,\mathcal{A}_1,\mu_1;\mathbb{R})$ and $L^1(X_2,\mathcal{A}_2,\mu_2;\mathbb{R})$, respectively, and
	\item [(iii)] the relation
	\begin{equation}
		\int\limits_{X_1\times X_2} f(x_1,x_2) \ \mu_1\times\mu_2(dx_1,dx_2)=\int\limits_{X_1}I_f(x_1)\mu_1(dx_1)=\int\limits_{X_2}J_f(x_2)\mu_2(dx_2)
	\end{equation}
	holds.
\end{itemize}
\end{thm}
Proofs of Theorems \ref{tonelli_thm} and \ref{fubini_thm} (and their generalizations to a product of finite number of $\sigma$-finite measure spaces) can be found in \cite{cohn}. 

Let $(\Omega,\Sigma,\mathbb{P})$ be a probability space, i.e., a measure space such that $\mathbb{P}(\Omega)=1$. A $\Sigma / \mathcal{B}(\mathbb{R}^d)$-measurable function $\xi:\Omega\to\mathbb{R}^d$ is called a random variable (or a random vector). If $(X,\mathcal{A})$ is a measurable space then we call a function $\xi:\Omega\to X$ a ($X$-valued) {\it random element} (or a {\it random function}) provided that it is $\Sigma / \mathcal{A}$-measurable. The law of $\xi$ is given by
\begin{equation}
\label{xi_law}
	\mu(A)=\mathbb{P}(\{\xi \in A\})=\mathbb{P}(\xi^{-1}(A)),
\end{equation}
for all $A\in\mathcal{A}$. Hence, the random element $\xi$ induces measure on $(X,\mathcal{A})$ and then the space $(X,\mathcal{A},\mu)$ is a measure space equipped with the probability measure $\mu$ ($\mu(X)=1$). The classical example is a Wiener space $(C([0,T]),\mathcal{B}(C([0,T])), w)$ equipped with the Wiener measure $w$. The measure $w$ is a law of the random element $\hat W:\Omega\to C([0,T])$, generated by a Wiener process $W$ (see Theorem \ref{gen_rand_el_1}).

Below we recall the Central Limit Theorem. For the proof see, for example, Theorem 10.3.16 in  \cite{cohn}.
\begin{thm} (Central Limit Theorem, CLT)
\label{CLT}
	Let $(X_j)_{j\in\mathbb{N}}$ be an i.i.d sequence of random variables with common mean $\mu$ and variance $\sigma^2$. Let $\displaystyle{S_n=\sum\limits_{j=1}^n X_j, n\in\mathbb{N}}$. Then the sequence
	\begin{equation}
		\Bigl(\frac{S_n-\mathbb{E}(S_n)}{\sqrt{Var(S_n)}}\Bigr)_{n\in\mathbb{N}}
	\end{equation}
	converges in distribution to a normal distribution with mean $0$ and variance $1$. 
\end{thm}
Note that under the assumptions of Theorem \ref{CLT}
\begin{equation}
\label{CLT_F1}
\frac{S_n-\mathbb{E}(S_n)}{\sqrt{Var(S_n)}}=\frac{S_n-n\mu}{\sigma\sqrt{n}}=\sqrt{n}\frac{\bar S_n-\mu}{\sigma},
\end{equation}
where $\displaystyle{\bar S_n=S_n/n}$.

The following result is a vary useful consequence of Borel-Cantelli lemma.
\begin{thm}(\cite{KLAN})
\label{BC_asconv}
 Let $\alpha>0$ and $K(p)\in [0,+\infty)$ for $p\geq 1$. In addition, let $Z_n$, $n\in\mathbb{N}$, be a sequence of random variables such that
\begin{equation}
	\Bigl(\mathbb{E}|Z_n|^p\Bigr)^{1/p}\leq K(p)\cdot n^{-\alpha}
\end{equation}
for all $p\geq 1$ and all $n\in\mathbb{N}$. Then for all $\varepsilon\in (0,\alpha)$ there exists a random variable $\eta_{\varepsilon}$ such that
\begin{equation}
	|Z_n|\leq \eta_{\varepsilon}\cdot n^{-\alpha+\varepsilon} \ \hbox{almost surely}
\end{equation}
for all $n\in\mathbb{N}$. Moreover, $\mathbb{E}|\eta_{\varepsilon}|^p<+\infty$ for all $p\geq 1$.
\end{thm}
Below we recall the famous Burkholder inequality, which is nowadays widely used when establishing error bounds for Monte Carlo algorithms (for references, see, \cite{BURKH}, \cite{RKYW1}). For the exposition of the inequality we need a notion of a filtration and martingale in the discrete-time case.

A {\it discrete filtration} on a probability space $(\Omega,\Sigma,\mathbb{P})$ is an increasing family $(\Sigma_n)_{n\in\mathbb{N}_0}$ of sub-$\sigma$-fields of $\Sigma$. A sequence $M=(M_n)_{n\in\mathbb{N}_0}$ of $\mathbb{R}^d$-valued random variables on $(\Omega,\Sigma,\mathbb{P})$ is said to be a {\it discrete-time martingale} with respect to $(\Sigma_n)_{n\in\mathbb{N}_0}$ provided that
\begin{itemize}
	\item [(i)] $M=(M_n)_{n\in\mathbb{N}_0}$ is adapted to $(\Sigma_n)_{n\in\mathbb{N}_0}$, i.e., $\sigma(M_n)\subset\Sigma_n$ for all $n\in\mathbb{N}_0$, 
	\item [(ii)] $\mathbb{E}\|M_n\|_2<+\infty$ for all $n\in\mathbb{N}_0$,
	\item [(iii)] $\mathbb{E}(M_{n+1} \ | \ \Sigma_n)=M_n$ for all $n\in\mathbb{N}_0$.
\end{itemize}  
We now recall two main inequalities for discrete-time martingales.
\begin{thm} (Hoeffding inequality - orginal version, \cite{Hoeff1})
\label{H_ineq}
 Let $(X_i)_{i=1,2,\ldots,n}$ be a sequence of independent random variables such that $a_i\leq X_i\leq b_i$ with probability one for $i=1,2,\ldots,n$ and for some real numbers $a_1,\ldots,a_n$, $b_1,\ldots,b_n$. Let $\displaystyle{S_n=\sum\limits_{k=1}^n X_k}$. Then for all $t>0$
\begin{equation}
	\mathbb{P}\Bigl(|S_n-\mathbb{E}(S_n)|>t\Bigr)\leq 2\exp\Bigl(-\frac{2t^2}{\sum\limits_{i=1}^n (b_i-a_i)^2}\Bigr).
\end{equation}
\end{thm}
\begin{thm} (Hoeffding-Azuma inequality, \cite{SLUO})
\label{HA_ineq}
 Let $(M_j)_{j=0,1,\ldots,n}$ be a real-valued martingale such that $|M_j-M_{j-1}|\leq a_j$ a.s. for all $j=1,2,\ldots,n$. Then, for all $r\geq 0$
 \begin{equation}
 \mathbb{P}\Bigl(|M_n-M_0|\geq r\Bigr)\leq 2 \exp\Biggl(-\frac{r^2}{2\sum\limits_{j=1}^na_j^2}\Biggr).
 \end{equation}
\end{thm}
\begin{thm}(Burkholder's inequality, \cite{BURKH}) 
\label{BDG_DISC}
For each $p\in (1,+\infty)$ there exist positive constants $c_p$ and $C_p$ such
that for every $\mathbb{R}^d$-valued discrete-time martingale $M=(M_n)_{n\in\mathbb{N}_0}$ and for every $n\in\mathbb{N}_0$ we have
\begin{equation}
	c_p\|[M]_n^{1/2}\|_{L^p(\Omega)}\leq \Bigl\|\max\limits_{0\leq j \leq n} \|M_j\|_2 \Bigl\|_{L^p(\Omega)}\leq C_p\|[M]_n^{1/2}\|_{L^p(\Omega)},
\end{equation}
where $\displaystyle{[M]_n=\|M_0\|^2_2+\sum\limits_{k=1}^n \|M_k-M_{k-1}\|^2_2}$ is the quadratic variation of $M$.
\end{thm}
 In the case when $p\in [2,+\infty)$ we show some useful upper estimates for $\|[M]_n^{1/2}\|_{L^p(\Omega)}$. Firstly, let us denote by $\Delta M_k:=M_k-M_{k-1}$ for $k=,1,2,\ldots,n$ and $\Delta M_0:=M_0$. Then, by the triangle inequality
\begin{displaymath}
   \|[M]_n^{1/2}\|^2_{L^p(\Omega)}=\|[M]_n\|_{L^{p/2}(\Omega)}=\Bigl\|\sum_{k=0}^n\|\Delta M_k\|_2^2\Bigl\|_{L^{p/2}(\Omega)}\leq \sum_{k=0}^n\Bigl\|\|\Delta M_k\|_2^2\Bigl\|_{L^{p/2}(\Omega)}=\sum_{k=0}^n \|\Delta M_k\|^2_{L^p(\Omega,\mathbb{R}^d)},
\end{displaymath}
hence
\begin{equation}
    \Bigl\|\max\limits_{0\leq j \leq n} \|M_j\|_2 \Bigl\|_{L^p(\Omega)}\leq C_p\Bigl(\sum_{k=0}^n \|\Delta M_k\|^2_{L^p(\Omega,\mathbb{R}^d)}\Bigr)^{1/2}.
\end{equation}
Next, by the H\"older inequality 
\begin{eqnarray}
    \sum_{k=0}^n\|\Delta M_k\|^2_{L^p(\Omega,\mathbb{R}^d)}\leq \Bigl(\sum_{k=0}^n \|\Delta M_k\|^p_{L^p(\Omega,\mathbb{R}^d)}\Bigr)^{2/p}\cdot (n+1)^{1-\frac{2}{p}}.
\end{eqnarray}
Combining the estimates above we obtain
\begin{equation}
\label{BDG_UPP_EST}
    \Bigl\|\max\limits_{0\leq j \leq n} \|M_j\|_2 \Bigl\|_{L^p(\Omega)}\leq C_p\Bigl(\sum_{k=0}^n \|\Delta M_k\|^2_{L^p(\Omega,\mathbb{R}^d)}\Bigr)^{1/2}\leq C_p\cdot (n+1)^{\frac{1}{2}-\frac{1}{p}}\cdot\Bigl(\sum_{k=0}^n \|\Delta M_k\|^p_{L^p(\Omega,\mathbb{R}^d)}\Bigr)^{1/p}.
\end{equation}

We also recall definition of a martingale and a local martingale in continuous time.  

An $\mathbb{R}^d$-valued stochastic process $M=(M(t))_{t\in [0,+\infty)}$  on $(\Omega,\Sigma,\mathbb{P})$ is said to be a {\it continuous-time martingale} with respect to the given filtration $(\Sigma_t)_{t\in [0,+\infty)}$ provided that
\begin{itemize}
	\item [(i)] $M=(M(t))_{t\geq 0}$ is adapted to $(\Sigma_t)_{t\in [0,+\infty)}$, i.e., $\sigma(M(t))\subset\Sigma_t$ for all $t\in [0,+\infty)$, 
	\item [(ii)] $\mathbb{E}\|M(t)\|_2<+\infty$ for all $t\in [0,+\infty)$,
	\item [(iii)] $\mathbb{E}(M(t) \ | \ \Sigma_s)=M(s)$ for all $0\leq s<t$.
\end{itemize}
A random variable $\tau:\Omega\to [0,+\infty]$ is said to be {\it stopping time} with respect to the filtration $(\Sigma_t)_{t\in [0,+\infty)}$ if, for every $t\in [0,+\infty)$,
\begin{equation}
	\{\tau\leq t\}\in\Sigma_t.
\end{equation}
We also set
\begin{equation}
	\Sigma_{\tau}=\{A\in\Sigma \ | \ \forall_{t\in [0,+\infty)} \ A\cap\{\tau\leq t\}\in\Sigma_t\}.
\end{equation}
An $\mathbb{R}^d$-valued stochastic process $M=(M(t))_{t\in [0,+\infty)}$  on $(\Omega,\Sigma,\mathbb{P})$ is said to be a {\it continuous-time local  martingale} with respect to the given filtration $(\Sigma_t)_{t\in [0,+\infty)}$ if it is adapted to $(\Sigma_t)_{t\in [0,+\infty)}$ and  if there exists an increasing sequence $(\tau_n)_{n\in\mathbb{N}}$ of  $(\Sigma_t)_{t\in [0,+\infty)}$-stopping times such that 
\begin{itemize}
	\item [(i)] $\mathbb{P}(\lim\limits_{n\to +\infty}\tau_n=+\infty)=1$,
	\item [(ii)] for every $n$ the process $(M(t\wedge\tau_n),\Sigma_t)_{t\in [0,+\infty)}$ is a martingale.
\end{itemize}
Let $\{\mathcal{F}_1,\ldots,\mathcal{F}_m\}$ be a family of sub-$\sigma$-fields of $\Sigma$. W call this collection {\it a family of independent $\sigma$-fields} or we say that the $\sigma$-fields $\mathcal{F}_1,\ldots,\mathcal{F}_m$ are {\it (mutually) independent} if, for every $A_1\in\mathcal{F}_1,\ldots,A_m\in\mathcal{F}_m$,
\begin{equation}
	\mathbb{P}(A_1\cap\ldots \cap A_m)=\mathbb{P}(A_1)\cdot\ldots\cdot\mathbb{P}(A_m).
\end{equation} 
Independence of events from $\Sigma$ can also be expressed in the term of independent $\sigma$-fields. Namely, the events $A_1,\ldots,A_m\in\Sigma$ are independent iff the $\sigma$-fields $\sigma(\{A_1\}),\ldots,\sigma(\{A_m\})$ are independent, where $\sigma(\{A_i\})=\{\emptyset,\Omega,A_i,\Omega\setminus A_i\}$.  

Random elements $\xi_j:\Omega\to X_j$, with values in arbitrary measurable spaces $(X_j,\mathcal{A}_j)$, $j=1,2,\ldots,m$, are {\it independent} if and only if the $\sigma$-fields $\sigma(\xi_1),\ldots,\sigma(\xi_m)$ are independent. Note that the independence of the events $A_1,\ldots,A_m\in\Sigma$ is equivalent  to the independence of the random variables $\mathbf{1}_{A_1},\ldots,\mathbf{1}_{A_m}$.

For a (possibly infinite) family $(\xi_i)_{i\in I}$ of random elements we say that they are {\it independent} if and only if  the random elements $\xi_{i_1},\ldots,\xi_{i_m}$ are independent for every choice of a finite number $i_1,\ldots,i_m$ of distinct indices in $I$. An analogous definition holds for an infinite family $(\mathcal{F}_i)_{i\in I}$ of  sub-$\sigma$-fields of $\Sigma$. 

Let $(\mathcal{A}_i)_{i\in I}$ be a family of (arbitrary) subsets of $\Sigma$. We say that $\mathcal{A}_i$'s are {\it (mutually) independent} iff for any finite subset $J$ of $I$ and for any $A_j\in\mathcal{A}_{j}$, $j\in J$, it holds
\begin{equation}
	\mathbb{P}\Bigl(\bigcap_{j\in J}A_j\Bigr)=\prod\limits_{j\in J}\mathbb{P}(A_j).
\end{equation}
\begin{lem}
\label{indep_3}
Let $(\mathcal{A}_i)_{i\in I}$ be a family of $\pi$-systems and $\mathcal{A}_i\subset\Sigma$ for $i\in I$. Then $(\sigma(\mathcal{A}_i))_{i\in I}$ is a family of independent $\sigma$-fields iff $(\mathcal{A}_i)_{i\in I}$ are independent. 
\end{lem}
W also say that the random element $\xi$ is independent of the sub-$\sigma$-field $\mathcal{G}\subset\Sigma$ if the $\sigma$-fields $\sigma(\xi)$ and $\mathcal{G}$ are independent. 

The following simple but useful lemma holds.
\begin{lem}
\label{indep_1}
  Let $\mathcal{A}_1,\ldots,\mathcal{A}_m$, $\mathcal{B}_1,\ldots,\mathcal{B}_m$, $\mathcal{C}_1,\ldots,\mathcal{C}_m$ be a sub-$\sigma$-fields of $\Sigma$. Moreover, we assume that
  \begin{itemize}
	\item [(i)] $\mathcal{A}_i,\mathcal{B}_i\subset\mathcal{C}_i$ for all $i=1,\ldots,m$,
	\item [(ii)] $(\mathcal{C}_i)_{i=1,\ldots,m}$ is a family of independent $\sigma$-fields,
	\item [(iii)] for all $i=1,\ldots,m$ the $\sigma$-fields $\mathcal{A}_i$ and $\mathcal{B}_i$ are independent.
  \end{itemize}
  Then $\{\mathcal{A}_1,\ldots,\mathcal{A}_m,\mathcal{B}_1,\ldots,\mathcal{B}_m\}$ is a family of independent $\sigma$-fields.
\end{lem}
{\bf Proof.} Let us take arbitrary $A_i\in\mathcal{A}_i$, $B_i\in\mathcal{B}_i$, $i=1,\ldots,m$. By the assumption (i) we have $A_i\cap B_i\in\mathcal{C}_i$ for $i=1,\ldots,m$, and by the assumptions (ii), (iii) we get
\begin{eqnarray}
	&&\mathbb{P}(A_1\cap\ldots\cap A_m\cap B_1\cap\ldots\cap B_m)=\mathbb{P}((A_1\cap B_1)\cap\ldots\cap(A_m\cap B_m))\notag\\
	&&=\prod\limits_{j=1}^m\mathbb{P}(A_j\cap B_j)=\prod\limits_{j=1}^m\mathbb{P}(A_j)\mathbb{P}(B_j).
\end{eqnarray}
This ends the proof. \ \ \ $\blacksquare$ \\ \\
Pairwise independence is much weaker than being mutually independent. However, independence can be checked by repeated checks for pairwise independence, see, for example, \cite{Cinlar} or \cite{KALLEN}.
\begin{prop}
\label{indep_2}
The sub-$\sigma$-fields $\mathcal{F}_1,\mathcal{F}_2,\ldots$ of $\Sigma$ are independent if and only if $\bigvee_{i=1,\ldots,n}\mathcal{F}_i$ and $\mathcal{F}_{n+1}$ are independent for all $n\in\mathbb{N}$.
\end{prop}
We also recall the grouping property, see, for example, \cite{KALLEN}.
\begin{prop} Let $(\mathcal{F})_{i\in I}$ be a family of independent sub-$\sigma$-fields of $\Sigma$, and let $\mathcal{T}$ be a disjoint partition of $I$. Then the $\sigma$-fields
\begin{equation}
	\mathcal{F}_{S}=\bigvee_{s\in S}\mathcal{F}_s, \ S\in\mathcal{T},
\end{equation}
are again independent.
\end{prop}
Independence has its consequence for the product of random variables.
\begin{lem} Let $X,Y$ be independent and integrable random variables. Then
\begin{equation}
	\mathbb{E}(X\cdot Y)=\mathbb{E}(X)\cdot\mathbb{E}(Y).
\end{equation}
The same holds true if $X,Y$ are nonnegative and not necessarily integrable.
\end{lem}
The proof of the following projection property of conditional expectation can be found, for example, in \cite{Cinlar}.
\begin{thm} 
\label{proj_wwo}
Let $\mathcal{G}$ be a sub-$\sigma$-field of $\Sigma$. 	For every $X\in L^2(\Omega,\Sigma,\mathbb{P})$ we have
\begin{equation}
	\mathbb{E}\Bigl(X-\mathbb{E}( X \ | \ \mathcal{G})\Bigr)^2=\inf\limits_{Y\in L^2(\Omega,\mathcal{G},\mathbb{P})}\mathbb{E}(X-Y)^2.
\end{equation}
\end{thm}
\begin{thm} (Conditional Fubini's Theorem, \cite{Brks})
	\label{cond_fubini}
	Suppose that $X\in L^1(\Omega\times [0,T],\Sigma\otimes \mathcal{B}([0,T]),\mathbb{P}\times\lambda_1)$ and let $\mathcal{G}$ be a sub-$\sigma$-field of $\Sigma$. Then
	\begin{equation}
		\mathbb{E}\Bigl(\int\limits_0^T X(t)dt \ | \ \mathcal{G}\Bigr)=\int\limits_0^T \mathbb{E}(X(t) \ | \ \mathcal{G}) \ dt \ a.s.
	\end{equation}
\end{thm}
\begin{thm} (The freezing lemma, \cite{parras})
\label{f_lemma}
 Let $X$ and $Y$ be respectively a $d$-dimensional and a $k$-dimensional random vector, and let $\psi:\mathbb{R}^{d+k}\to\mathbb{R}$ be Borel measurable such that $\psi(X,Y)\geq 0$ a.s., or $\mathbb{E}|\psi(X,Y)|<+\infty$. Let $\mathcal{G}$ be a sub-$\sigma$-field of $\Sigma$ such that $X$ is $\mathcal{G}$-measurable, and $\sigma(Y)$, $\mathcal{G}$ are independent $\sigma$-fields. Then
\begin{equation}
	\mathbb{E}\Bigl(\psi(X,Y) \ | \ \mathcal{G}\Bigr)=\mathbb{E}\Bigl(\psi(x, Y)\Bigr)\Bigl|_{x=X}=\int\limits_{\mathbb{R}^k}\psi(X,y)\mathbb{P}_Y(dy).
\end{equation}
\end{thm}
The following fact can be found for example in \cite{Applb}, page 238.
\begin{lem}
\label{itoint_det_f} Let $W=(W(t))_{t\in [0,T]}$ be the Wiener process with respect to the filtration $(\Sigma_t)_{t\in [0,T]}$. If $f\in L^2([0,T])$ is a deterministic function, then for all $t\in [0,T]$ the $\sigma$-fields $\displaystyle{\sigma\Bigl(\int\limits_t^T f(s)dW(s)\Bigr)}$ and $\Sigma_t$ are independent.
\end{lem}
\chapter{Auxiliary results--part 2}
In this section we recall basic facts from the theory of continuous-time stochastic processes.

Let $(\Omega,\Sigma,\mathbb{P})$ be a {\it complete probability space}, that is all subsets of measure zero sets are also measurable (and hence, also of measure zero). We consider the product space $(\Omega\times [0,+\infty),\Sigma\otimes \mathcal{B}([0,+\infty)),\mathbb{P}\times \lambda_1)$. By a stochastic process $X$, with the base probability space $(\Omega,\Sigma,\mathbb{P})$ and with values in $\mathbb{R}^d$, we mean a mapping
\begin{equation}
	X:\Omega\times [0,+\infty)\to\mathbb{R}^d,
\end{equation}
such that for each $t\in [0,+\infty)$ the function
\begin{equation}
	\Omega\ni \omega\to X(t,\omega)\in\mathbb{R}^d
\end{equation}
is $\Sigma / \mathcal{B}(\mathbb{R}^d)$-measurable, i.e., it is a random variable. We usually use the following notational convention $X=(X(t))_{t\in [0,+\infty)}$, which, in particular, underlines that a stochastic process might be seen as a collection of random variables. For a fixed $\omega\in\Omega$ the $\mathbb{R}^d$-valued function 
\begin{equation}
\label{def_sp}
	[0,+\infty)\ni t\to X(\omega,t)\in\mathbb{R}^d
\end{equation} 
is called a {\it trajectory} (or {\it realization} or {\it sample function}) corresponding to the elementary event $\omega$. Note that, without any additional assumptions imposed on $X$, such a trajectory might not be even measurable. We say that the stochastic process  $X$ is {\it product measurable} if it is $\Sigma\otimes \mathcal{B}([0,+\infty)) / \mathcal{B}(\mathbb{R}^d)$ - measurable.   The following widely used fact is a direct consequence of a product measurability and Fubini's theorem.
\begin{prop}
\label{meas_Fub_1} Let $X=(X(t))_{t\in [0,+\infty)}$ be a product measurable stochastic process. Then
\begin{itemize}
	\item [(i)] All trajectories of $X$ are $\mathcal{B}([0,+\infty) / \mathcal{B}(\mathbb{R}^d)$-measurable functions.
	\item [(ii)] For any $t\in [0,+\infty)$ the function 
	\begin{equation}
		\Omega\ni\omega\to X(\omega,t)\in \mathbb{R}^d
	\end{equation}
	is $\Sigma / \mathcal{B}(\mathbb{R}^d)$-measurable, so it is $\mathbb{R}^d$-valued random variable.
	\item [(iii)] if $\mathbb{E}(X(t))$ exists for all $t\in [0,+\infty)$ then
	\begin{equation}
		[0,+\infty)\ni t\to\mathbb{E}(X(t))\in \mathbb{R}^d
	\end{equation}
	is $\mathcal{B}([0,+\infty)) / \mathcal{B}(\mathbb{R}^d)$-measurable.
	\item [(iv)] if $A\in\mathcal{B}([0,+\infty))$ and $\displaystyle{\int\limits_{A}\mathbb{E}|X(t)|dt<+\infty}$, then
		\begin{equation}
			\int\limits_A|X(t)|dt<+\infty \ \hbox{almost surely},
		\end{equation}
		and
		\begin{equation}
			\int\limits_A \mathbb{E}(X(t))dt=\mathbb{E}\Bigl(\int\limits_A X(t) dt\Bigr).
		\end{equation}
\end{itemize}
\end{prop}
Let $(\Sigma_t)_{t\in [0,+\infty)}$ be a family of sub-$\sigma$-fields of $\Sigma$, such that 
\begin{equation}
	\Sigma_s\subset\Sigma_t\subset\Sigma, \ \hbox{for all} \ 0\leq s<t.
\end{equation}
Then we call $(\Sigma_t)_{t\in [0,+\infty)}$ the {\it filtration} on $(\Omega,\Sigma,\mathbb{P})$. Moreover, we set
\begin{equation}
	\Sigma_{\infty}=\sigma\Bigl(\bigcup_{t\geq 0}\Sigma_t\Bigr).
\end{equation}
Of course $\Sigma_{\infty}\subset\Sigma$. We say that the filtration $(\Sigma_t)_{t\in [0,+\infty)}$ is {\it complete} if 
\begin{equation}
	\mathcal{N}\subset\Sigma_0,
\end{equation}
where 
\begin{equation}
\label{null_sys}
	\mathcal{N}=\{A\in\Sigma \ | \ \mathbb{P}(A)=0\}.
\end{equation}
If $(\Sigma_t)_{t\in [0,+\infty)}$ is not complete then we can always take
\begin{equation}
	\bar\Sigma_t=\sigma(\Sigma_t\cup\mathcal{N}), \ t\in [0,+\infty).
\end{equation}
Of course the filtration $(\bar\Sigma_t)_{t\in [0,+\infty)}$ is complete and it is called {\it augmented} filtration.
The filtration $(\Sigma_t)_{t\in [0,+\infty)}$ is called {\it right-continuous} if $\Sigma_t=\Sigma_{t+}$ for all $t\in [0,+\infty)$, where $\displaystyle{\Sigma_{t+}=\bigcap_{s>t}\Sigma_s}$. (We have that $\Sigma_t\subset\Sigma_{t+}$ for all $t\in [0,+\infty)$ and the family $(\Sigma_{t+})_{t\in [0,+\infty)}$ is already a right-continuous filtration.) If $(\Sigma_t)_{t\in [0,+\infty)}$ is complete and right-continuous then we say that it satisfies the {\it usual conditions}. Note that the filtration $(\tilde\Sigma_t)_{t\in [0,+\infty)}$, where
\begin{equation}
	\tilde\Sigma_t=\bigcap_{s>t}\bar\Sigma_s \ (=\bar\Sigma_{t+}),
\end{equation}
always satisfies the usual conditions, even if the initial filtration $(\Sigma_t)_{t\in [0,+\infty)}$ does not.

If for all $t\in [0,+\infty)$
\begin{equation}
	\sigma(X(t))\subset \Sigma_t,
\end{equation} 
then we say that the stochastic process $X=(X(t))_{t\in [0,+\infty)}$ is adapted to $(\Sigma_t)_{t\in [0,+\infty)}$. Every stochastic process $X$ is adapted to $(\Sigma^X_t)_{t\in [0,+\infty)}$ where
\begin{equation}
	\Sigma_t^X=\sigma(X(s) \ | \ s\in [0,t])=\sigma\Bigl(\bigcup_{s\in [0,t]} \sigma(X(s))\Bigr).
\end{equation}
We call $(\Sigma^X_t)_{t\in [0,+\infty)}$ the {\it natural filtration} of the process $X$. This is the smallest filtration to which $X$ is adapted. We stress that the natural filtration of $X$ neither  have to be complete nor  right-continuous, while $(\tilde\Sigma^X_t)_{t\in [0,+\infty)}$ satisfies the usual conditions. Analogously we also define the last element of $(\Sigma_t^X)_{t\in [0,+\infty)}$ by
\begin{equation}
	\Sigma_{\infty}^X=\sigma\Bigl(\bigcup_{t\geq 0}\Sigma_t^X\Bigr).
\end{equation}
Note that we have 
\begin{equation}	
	\label{last_el_fX}
\Sigma_{\infty}^X=\sigma\Bigl(\bigcup_{t\geq 0}\sigma(X(t))\Bigr).
\end{equation}

\subsection{What does it mean that two processes are equivalent?}
\begin{defn} Suppose that $X=(X(t))_{t\in [0,+\infty)}$ and $Y=(Y(t))_{t\in [0,+\infty)}$, are two $\mathbb{R}^d$-valued stochastic processes defined on the same probability space $(\Omega,\Sigma,\mathbb{P})$. 
\begin{itemize}
	\item [(i)] We say that $X$ and $Y$ are {\bf stochastically equivalent} if
		\begin{equation}
			\forall_{t\in [0,+\infty)} \ \mathbb{P}\Bigl(\{\omega\in\Omega \ |\ X(\omega,t)=Y(\omega,t)\}\Bigr)=1.
		\end{equation} 
		We then say that $Y$ is a {\bf modification} (or a {\bf version}) of $X$.
	\item [(ii)] We say that $X$ and $Y$ are {\bf indistinguishable} if	
	\begin{equation}
			\mathbb{P}\Bigl(\{\omega\in\Omega \ | \ \forall_{t\in [0,+\infty)} \ X(\omega,t)=Y(\omega,t)\}\Bigr)=1.
		\end{equation}
\end{itemize}
\end{defn}
Of course if two processes are indistinguishable then they are stochastically equivalent, but the converse is in general not true. However, the following important result holds, see \cite{elcoh}.
\begin{thm} Suppose that $X=(X(t))_{t\in [0,+\infty)}$ and $Y=(Y(t))_{t\in [0,+\infty)}$ are stochastically equivalent, and that both processes are almost surely right-(or left-) continuous. Then they are indistinguishable. 
\end{thm}
In summary, we have four different ways in which we can say that two stochastic processes are equal:
\begin{itemize}
	\item [(i)] $X(\omega,t)=Y(\omega,t)$ for all $(\omega,t)\in \Omega\times [0,+\infty)$, but this situation occurs rarely,
	\item [(ii)] $X$ and $Y$ are stochastically equivalent,
	\item [(iii)] $X$ and $Y$ are indistinguishable,
	\item [(iv)] $X=Y$ $d\mathbb{P}\times d\lambda_1$-almost everywhere, i.e.,
	\begin{equation}
		(\mathbb{P}\times\lambda_1)\Bigl(\{(\omega,t)\in\Omega\times [0,+\infty) \ | \ X(\omega,t)\neq Y(\omega,t)\}\Bigr)=0.
	\end{equation}
	In order to use that notion of equivalence we have to assume that $X$ and $Y$ are product measurable processes.
\end{itemize}
\subsection{Three notions of measurability of stochastic processes}
Let $(\Omega,\Sigma,\mathbb{P})$ be a complete probability space with a filtration $(\Sigma_t)_{t\in [0,+\infty)}$ satisfying the usual conditions. We consider the following classes of stochastic processes
\begin{equation}
\label{def_proc_D}
	\mathbb{D}=\{X:\Omega\times [0,+\infty)\to\mathbb{R}^d \ | \ X \ \hbox{is adapted to} \ (\Sigma_t)_{t\in [0,+\infty)} \ \hbox{and has c\`adl\`ag paths}\},
\end{equation}
and
\begin{equation}
\label{def_proc_L}
	\mathbb{L}=\{X:\Omega\times [0,+\infty)\to\mathbb{R}^d \ | \ X \ \hbox{is adapted to} \ (\Sigma_t)_{t\in [0,+\infty)} \ \hbox{and has  c\`agl\`ad paths}\}.
\end{equation}
Then, the {\it predictable $\sigma$-field $\mathcal{P}$} on $\Omega\times [0,+\infty)$ is the smallest $\sigma$-field making all processes in $\mathbb{L}$ measurable, i.e.,
\begin{equation}
	\mathcal{P}=\sigma\Bigl(\bigcup\limits_{X\in \mathbb{L}}\sigma(X)\Bigr).
\end{equation}
The {\it optional $\sigma$-field $\mathcal{O}$} on $\Omega\times [0,+\infty)$ is the smallest $\sigma$-field making all processes in $\mathbb{D}$ measurable, i.e.,
\begin{equation}
	\mathcal{O}=\sigma\Bigl(\bigcup\limits_{X\in \mathbb{D}}\sigma(X)\Bigr).
\end{equation}
A stochastic process $X=(X(t))_{t\in [0,+\infty)}$ is called  {\it progressive} or {\it progressively measurable} if for each $t\in [0,+\infty)$ the mapping
\begin{equation}
	\Omega\times [0,t] \ni (\omega,s)\to X(\omega,s)\in\mathbb{R}^d
\end{equation}
is $\Sigma_t\otimes\mathcal{B}([0,t]) / \mathcal{B}(\mathbb{R}^d)$-measurable. The {\it progressive $\sigma$-field $\mathcal{A}$} on $\Omega\times [0,+\infty)$ is the smallest $\sigma$-field that makes all progressive processes measurable. We have that
\begin{equation}
	\mathcal{P}\subset\mathcal{O}\subset\mathcal{A}\subset \Sigma\otimes\mathcal{B}([0,+\infty)),
\end{equation}
see \cite{prott}.
For the proof of the fact below we refer to \cite{MEDVEG}.
\begin{prop} Let $f:[0,+\infty)\to \mathbb{R}^d$ be a Borel function and let us consider the (deterministic process) $X(\omega,t)=f(t)$ for all $(\omega,t)\in \Omega\times [0,+\infty)$. Then $X=(X(t))_{t\in [0,+\infty)}$ is predictable.
\end{prop}
\subsection{Stochastic process as a random element with values in functional space and vice versa}
Sometimes it is convenient to switch between two possible ways of looking at stochastic processes. The first possibility is to consider a stochastic process $X$ as a product measurable function \eqref{def_sp} with values in $\mathbb{R}^d$. However, if (almost) all trajectories of $X$ belong to some functional space $E$ equipped with $\sigma$-field $\mathcal{E}$, then we can consider a mapping $\hat X:\Omega\to E$  defined as
\begin{equation}	
\label{XHX}
	\hat X(\omega)=X(\omega,\cdot), \ \omega\in\Omega.
\end{equation}  
If $\hat X$ is $\Sigma / \mathcal{E}$-measurable then we say that the process $X$ generates the random element $\hat X$ in $(E,\mathcal{E})$. Moreover, the law $\mu$ of $\hat X$ is a probabilistic measure induced by $\hat X$ on the measure space $(E,\mathcal{E})$, see \eqref{xi_law}.
 
We listed below three most common cases. For the proofs see, for example, \cite{gusak}.
\begin{thm}
\label{gen_rand_el_1}
\begin{itemize}
	\item [(i)] Let $E=C([0,T])$, equipped with supremum norm $\|\cdot\|_{\infty}$, and $\mathcal{E}=\mathcal{B}(E)$. If $X$ is a stochastic process with continuous trajectories, then it generates a random element in $C([0,T])$.
	\item [(ii)] Let $E=D([0,T])$ be the Skorokhod space, equipped with the Skorokhod metric $d$, and $\mathcal{E}=\mathcal{B}(E)$. If $X$ is a stochastic process with c\`adl\`ag trajectories, then it generates a random element in $D([0,T])$.
	\item [(iii)] Let $E=L^2([0,T])$, equipped with the norm $\|\cdot\|_{L^2([0,T])}$, and $\mathcal{E}=\mathcal{B}(E)$. If $X\in L^2(\Omega\times [0,T],\Sigma\otimes\mathcal{B}([0,T]),\mathbb{P}\times\lambda_1)$, then it generates a random element in $L^2([0,T])$.
\end{itemize}
\end{thm}
Having a random element $\hat X$ in some functional space $E$ it is natural to ask if there is a product measurable stochastic process $X$ satisfying \eqref{XHX}. In the case when $E=L^2([0,T])$ the situation is more subtle since $L^2([0,T])$ consists of equivalence classes.
\begin{thm}
\label{from_B2prod}
	\item [(i)] Let $\hat X$ be a random element in $C([0,T])$. Then the process $X$ defined by
	\begin{equation}
	\label{prod_meas_rep}
			X(\omega,t)=(\hat X(\omega))(t), (\omega,t)\in \Omega\times [0,T],
	\end{equation}
	is product measurable and has continuous trajectories.
	\item [(ii)] Let $\hat X$ be a random element in $D([0,T])$. Then the process $X$ defined by \eqref{prod_meas_rep} 	is product measurable and has c\`adl\`ag trajectories.
	\item [(iii)] Let $\hat X$ be a random element in $L^2([0,T])$. Then there exists a product measurable process $X$ such that for almost all $\omega$, the equality $\displaystyle{X(\omega,t)=(\hat X(\omega))(t)}$ holds almost everywhere on $[0,T]$.
\end{thm}
For the proof of $(iii)$ see Proposition 2, page 741. in \cite{MAKPOD}. In order to show $(i)$ and $(ii)$ we recall after \cite{Parth} what follows.
\begin{lem} 
\label{cm_meas}
	\begin{itemize}
		\item [(a)] The coordinate mappings $\delta_t:C([0,T])\to \mathbb{R}$ defined by
			\begin{equation}
			\label{cm_1}
				\delta_t (f)=f(t), \ f\in C([0,T])
			\end{equation}
			are $\mathcal{B}(C([0,T])) / \mathcal{B}(\mathbb{R})$-measurable  for all $t\in [0,T]$.
		\item [(b)] The coordinate mappings $\delta_t:D([0,T])\to \mathbb{R}$ defined by
			\begin{equation}
			\label{cm_2}
				\delta_t (f)=f(t), \ f\in D([0,T])
			\end{equation}
			are $\mathcal{B}(D([0,T])) / \mathcal{B}(\mathbb{R})$-measurable  for all $t\in [0,T]$.
	\end{itemize}
\end{lem}
\noindent
{\bf Proof of Theorem \ref{from_B2prod} (i).} For any $t\in [0,T]$ we can write   \eqref{prod_meas_rep} as
\begin{equation}
 X(\omega,t)=\delta_t(\hat X(\omega)), \omega\in\Omega.
\end{equation}
By Lemma \ref{cm_meas} (a) and from the fact that the random element $\hat X$ is $\Sigma /\mathcal{B}(C([0,T]))$-measurable, we get for any $B\in\mathcal{B}(\mathbb{R})$ that
\begin{equation}
	(X(t))^{-1}(B)=\Bigl(\delta_t(\hat X)\Bigr)^{-1}(B)=\hat X^{-1}(\delta_t^{-1}(B)) \in\Sigma,
\end{equation}
so for any $t\in [0,T]$ the mapping
\begin{equation}
	\Omega\ni\omega\to X(\omega,t)\in\mathbb{R}
\end{equation}
is $\Sigma / \mathcal{B}(\mathbb{R})$-measurable. For any $\omega\in\Omega$ we have that  $\hat X(\omega)\in C([0,T])$. Hence, by \eqref{prod_meas_rep} we get that for any $\omega\in\Omega$ the function
\begin{equation}
	[0,T]\ni t\to X(\omega,t)\in\mathbb{R}
\end{equation}
is continuous. By Lemma \ref{prod_meas_lem} we obtain that $X$ is $\Sigma\otimes \mathcal{B}([0,T]) / \mathcal{B}(\mathbb{R})$-measurable. \ \ \ $\blacksquare$ \\ \\
The proof of Theorem \ref{from_B2prod} (ii) is analogous and therefore it is left as an exercise.
\begin{rem} The spaces $C([0,T])$, $L^2([0,T])$  considered in Theorem \ref{gen_rand_el_1} (i), (iii) are separable Banach spaces. Hence, the random elements in Theorem \ref{gen_rand_el_1} (i), (iii) are Bochner measurable. Since the space $D([0,T])$, under the Skorokhod metric $d$, is complete separable metric space, the random element in Theorem \ref{gen_rand_el_1} (ii) is Borel measurable. (Recall also that the vector space $D([0,T])$ under the supremum norm is Banach but not separable.)
\end{rem}
\section{Exercises}
\begin{itemize}
	\item [1.] Show that $\mathcal{N}$ in \eqref{null_sys} is a $\pi$-system. Moreover, prove that
\begin{equation}
	\mathcal{C}=\{A\in\Sigma \ | \  \mathbb{P}(A)\in\{0,1\}\}
\end{equation} 
is a $\sigma$-field on $\Omega$ and $\sigma(\mathcal{N})=\mathcal{C}$.
	\item [2.] Let $\mathcal{G}$ be a sub-$\sigma$-field of $\Sigma$. Show that the $\sigma$-fields $\sigma(\mathcal{N})$ and $\mathcal{G}$ are independent.
	\item [3.] Let $(\Sigma_t)_{t\in [0,+\infty)}$ be a complete filtration of a complete probability space $(\Omega,\Sigma,\mathbb{P})$. Show that for each $t\in [0,+\infty)$ the probability space $(\Omega,\Sigma_t,\mathbb{P})$ is complete.
	\item [4.] Give a proof of \eqref{last_el_fX}.\\
	Hint: Use Proposition \ref{sum_s_alg}.
	\item [5.] Let $f:[0,+\infty)\times\mathbb{R}^d\to\mathbb{R}^m$ be a Borel function and let $X=(X(t))_{t\in [0,+\infty)}$ be a $\mathbb{R}^d$-valued predictable (optional, progressive) process. Show that the process $(f(t,X(t)))_{t\in [0,+\infty)}$ is predictable (optional, progressive).
	\item [6.] Let $f:[0,+\infty)\times\mathbb{R}^d\to\mathbb{R}^m$ be a Borel function and let $X=(X(t))_{t\in [0,+\infty)}$ be a $\mathbb{R}^d$-valued adapted and product measurable process. Show that the process $(f(t,X(t)))_{t\in [0,+\infty)}$ is adapted and product measurable.
	\item [7.] Let us assume that $(\Omega,\Sigma,(\Sigma_t)_{t\in [0,+\infty)},\mathbb{P})$ is a complete probability space with a complete filtration $(\Sigma_t)_{t\in [0,+\infty)}$. Suppose that $X=(X(t))_{t\in [0,+\infty)}$ and $Y=(Y(t))_{t\in [0,+\infty)}$ are two stochastically equivalent processes, where $X$ is adapted to $(\Sigma_t)_{t\in [0,+\infty)}$. Show that $Y$ is also adapted.
	\item [8.] For any two product measurable stochastic processes $X,Y$ prove that the following are equivalent:
	\begin{itemize}
		\item [(i)] $\displaystyle{(\mathbb{P}\times\lambda_1)\Bigl(\{(\omega,t)\in\Omega\times [0,+\infty) \ | \ X(\omega,t)\neq Y(\omega,t)\}\Bigr)=0}$,
		\item [(ii)] $\displaystyle{\mathbb{P}\Bigl(\{\omega\in\Omega \ | \ X(\omega,t)\neq Y(\omega,t)\}\Bigr)=0}$ for $d\lambda_1$-almost all $t\in [0,+\infty)$,
		\item [(iii)]$\displaystyle{\lambda_1\Bigl(\{t\in [0,+\infty) \ | \ X(\omega,t)\neq Y(\omega,t)\}\Bigr)=0}$ for $d\mathbb{P}$-almost all $\omega\in\Omega$.
	\end{itemize}
	\item [9.] Show that for any $t\in [0,T]$ the coordinate mapping \eqref{cm_1} is continuous.
	\item [10.] Give a proof of Theorem \ref{from_B2prod} (ii).
\end{itemize}
\begingroup
\renewcommand{\section}[2]{}%

\endgroup

\begin{thebibliography}{22}

\bibitem{andresgorniewicz1}
{\sc J. Andres and L. G\'orniewicz}, {\it Topological Fixed Point Principles for Boundary Value Problems}, vol. I, Springer Science+Business Media Dordrecht, 2003.

\bibitem{Applb}
{\sc D. Applebaum}, {\it L\'evy Processes and Stochastic Calculus}, 2nd edition, Cambridge University Press, 2011.

\bibitem{Ash}
{\sc R. B. Ash}, {\it Real Analysis and Probability}, Academic Press, 1972.

\bibitem{VB1}
{\sc V. Barbu}, {\it Differential Equations}, Springer, 2016.

\bibitem{TBPhD}
{\sc T. Bochacik}, {\it Randomized algorithms approximating solutions of ordinary differential equations}, Ph.D. thesis, AGH University of Science and Technology, Krak\'ow, 2023,

\bibitem{randRK}
{\sc T. Bochacik, M. Goćwin, P. M. Morkisz, and P. Przybyłowicz}, Randomized Runge-Kutta method -- Stability and convergence under inexact information, J. Complex. 65 (2021), 101554

\bibitem{bogach1}
{\sc V. I. Bogachev}, {\it Measure Theory}, Vol. I, Springer, 2007.

\bibitem{bogach2}
{\sc V. I. Bogachev}, {\it Measure Theory}, Vol. II, Springer, 2007.

\bibitem{Brks}
{\sc R. A. Brooks}, Conditional expectations associated with stochastic processes, {\it Pacific J. of Mathem.} {\bf 41} (1972) 33--42.

\bibitem{BURKH}
{\sc D. L. Burkholder}, Explorations in martingale theory and its applications, In: B. Davis and R. Song (eds.), {\it Selected Works of Donald L. Burkholder}, Selected Works in Probability and Statistics, DOI 10.1007/978-1-4419-7245-334, Springer Science+Business Media, 2011.

\bibitem{TACHE2009}
{\sc C. Y. Tang, S. X. Chen}, Parameter estimation and bias correction for diffusion process, {\it J. Econom.} {\bf 149} (2009), 65--81.

\bibitem{Cinlar}
{\sc E. Cinlar}, {\it Probability and Stochastics}, Springer, 2011.

\bibitem{CLCAM}
{\sc J.M.C. Clark, R.J. Cameron}, The maximum rate of convergence of discrete approximations for stochastic differential equations, In {\it Stochastic Differential Systems} (B. Grigelionis, ed.) Lect. Notes Control
Inf. Sci. {\bf 25} (1980), Springer, Berlin, 162--171. 

\bibitem{cohn}
{\sc D. L. Cohn}, {\it Measure Theory}, 2nd Ed., Springer Science, 2013.

\bibitem{daun}
{\sc T. Daun}, On the randomized solution of initial value problems, {\it J. Complexity} {\bf 27} (2011), 300--311.

\bibitem{jdpp}
{\sc J. D\k{e}bowski, P. Przyby{\l}owicz}, Optimal approximation of stochastic integrals with respect to a homogeneous Poisson process, {\it Mediter. J. Math.} {\bf 13} (2016) 3713--3727.

\bibitem{elcoh}
{\sc S.N. Cohen, R.J. Elliot}, {\it Stochastic Calculus and Applications}, 2nd ed., Springer, 2015.

\bibitem{fried1}
{\sc A. Friedman}, {\it Stochastic Differential Equations and Applications}, two volumes bound as one, Dover Publications, Inc., 2006.

\bibitem{gar}
{\sc A. Gardo{\'n}}, The order of approximations for solutions of It\^o-type stochastic differential equations with jumps, {\it Stoch. Anal. Appl.} {\bf 22} (2004) 679--699.

\bibitem{gikskor}
{\sc I.I. Gikhman, A.V. Skorokhod}, {\it The Theory of Stochastic Processes III}, Springer-Verlag, 2007.

\bibitem{gopi}
{\sc F. Golestaneh, P. Pinson, R. Azizipanah-Abarghooee,  H. B. Gooi}, Ellipsoidal Prediction Regions for Multivariate Uncertainty Characterization, {\it IEEE Transactions on Power Systems} {\bf  33} (2018),  4519--4530.

\bibitem{GRTAL}
{\sc C. Graham, D. Talay}, {\it Stochastic Simulation and Monte Carlo Methods}, Springer, 2013.

\bibitem{gring}
{L. G\'orniewicz, R.S. Ingarden}, {\it Mathematical Analysis for Physicists} (in Polish), Wydawnictwo Naukowe UMK, 2012.

\bibitem{gusak}
{\sc D. Gusak, A. Kukush, A. Kulik, Y. Mishura, A. Pilipenko}, {\it Theory of Stochastic Porcesses with Applications to Financial Mathematics and Risk Theory}, Springer, 2009.

\bibitem{HeinMilla} 
\textsc{S. Heinrich and B. Milla}, \emph{The randomized complexity of initial value problems}, J. Complex. 24 (2008), 77--88.

\bibitem{hertl}
{\sc P. Hertling}, Nonlinear Lebesgue and It\^o integration problems of high complexity, {\it J.Complexity} {\bf 17} (2001), 366--387.

\bibitem{higham2000}
{\sc D.J. Higham}, Mean-square and asymptotic stability of the stochastic theta method, {\it Siam J. Numer. Anal.} {\bf 38} (2000), 753--769.

\bibitem{higpla1}
{\sc D.J. Higham, P.E. Kloeden}, Numerical methods for nonlinear stochastic differential equations
with jumps, {\it Numer. Math.} {\bf 101} (2005) 101--119.

\bibitem{higpla3}
{\sc D.J. Higham, P.E. Kloeden}, Strong convergence rates for backward Euler on a class of
nonlinear jump-diffusion problems, {\it J. Comput. Appl. Math.} {\bf 205} (2007) 949--956.

\bibitem{Hoeff1}
W. Hoeffding, Probability inequalities for sums of bounded random variables, {\it Journal of the American Statistical Association} (1963) {\bf 58}, 13--30.

\bibitem{HMR3} 
{\sc N. Hofmann, T. M\"uller--Gronbach, K. Ritter}, The optimal discretization 
of stochastic differential equations, {\it J. Complexity} {\bf 17} (2001) 117--153.

\bibitem{Iacus1}
{\sc S. M. Iacus}, 
{\it Simulation and Inference for Stochastic Differential Equations - With R Examples}, Springer, 2008.

\bibitem{IAYO}
{\sc S. M. Iacus, N. Yoshida}, 
{\it Simulation and Inference for Stochastic Processes with YUMA}, Springer, 2018.

\bibitem{ikwa}
{\sc N. Ikeda, S. Watanabe}, {\it Stochastic Differential Equations and Diffusion Processes}, 2nd
ed., North-Holland Math. Library 24, North-Holland, Amsterdam, 1989.

\bibitem{jankow}
{\sc J. Jankowska, M. Jankowski}, {\it A Survey of Numerical Methods and Algorithms I}, WNT, 1981 (in Polish).

\bibitem{JenNeuen}
\textsc{A. Jentzen and A. Neuenkirch}, \emph{A random Euler scheme for Carath\'eodory differential equations}, J. Comput. Appl. Math. 224 (2009), 346--359.

\bibitem{Kac1}
\textsc{B. Kacewicz}, \emph{Almost optimal solution of initial-value problems by randomized and quantum algorithms}, J. Complex. 22 (2006), 676--690.

\bibitem{Kac2}
B. Kacewicz, \emph{Improved bounds on the randomized and quantum complexity of initial-value problems}, J. Complex.
21 (2005), 740--756.

\bibitem{Kac3}
B. Kacewicz, \emph{Randomized and quantum algorithms yield a speed-up for initial-value problems}, J. Complex. 20
(2004) 821--834.

\bibitem{AKPhD}
{\sc A. Ka{\l u}{\.z}a}, {\it Optimal Algorithms for Solving Stochastic Initial-Value Problems with Jumps},
Ph.D. thesis, AGH University of Science and Technology, Krak\'ow, 2020, https://winntbg.
bg.agh.edu.pl/rozprawy2/11743/full11743.pdf.

\bibitem{AKPP}
{\sc A. Ka{\l u}{\.z}a, P. Przyby{\l}owicz},
 Optimal global approximation of jump-diffusion SDEs via path-independent step-size control, {\it Appl. Numer. Math.} {\bf 128} (2018) 24--42.

\bibitem{AKPMPP}
{\sc A. Ka{\l u}{\.z}a, P. M. Morkisz, P. Przyby{\l}owicz},
Optimal approximation of stochastic integrals in analytic noise model, {\it Appl. Math. and Comput.} {\bf 356} (2019),  74--91.

\bibitem{SDES_NN}
{\sc A. Kałuża, P. M. Morkisz, B. Mulewicz, P. Przybyłowicz, M. Wiącek}, Deep learning-based estimation of time-dependent parameters in Markov models with application to nonlinear regression and SDEs, https://arxiv.org/abs/2312.08493, 2023

\bibitem{KALLEN}
{\sc O. Kallenberg},
{\it Foundations of Modern Probability}, 2nd Ed., Springer-Verlag, 2002.

\bibitem{CLS1}
{\sc L. A. Klimko, P. I. Nelson}, On Conditional Least Squares Estimation for Stochastic Processes, {\it Ann. of Stat.}{\bf  6} (1978),  629--642.

\bibitem{KLAN}
{\sc P. E. Kloeden, A. Neuenkirch},
The pathwise convergence of approximation schemes for stochastic differential equations, {\it LMS J. Comput. Math.} {\bf 10} (2007), 235--253.

\bibitem{KruseWu_1} 
\textsc{R. Kruse and Y. Wu}, \emph{Error Analysis of Randomized Runge–Kutta Methods for Differential Equations with Time-Irregular Coefficients}, Comput. Methods Appl. Math. 17 (2017), 479--498.

\bibitem{LipShir_1}
{\sc R. S. Liptser, A. N. Shiryaev}, {\it Statistics of Random Processes I. General Theory}, Springer-Verlag Berlin Heidelberg 2001

\bibitem{loeve}
{\sc M. Lo\'eve}, {\it Probability Theory I}, 4th ed., Springer-Verlag, 1977.

\bibitem{SDE_PGD}
{\sc A. Lucchi, F. Proske, A. Orvieto, F. Bach, H. Kersting}, On the Theoretical Properties of Noise Correlation in
Stochastic Optimization, https://arxiv.org/abs/2209.09162

\bibitem{SLUO}
{\sc S. Luo}, On Azuma-Type Inequalities for Banach Space-Valued Martingales, {\it J. Theor. Probab.}  (2021), https://doi.org/10.1007/s10959-021-01086-5


\bibitem{xmao}
{\sc X. Mao}, {\it Stochastic Differential Equations and Applications}, 2nd ed., Woodhead Publishing, 2011.

\bibitem{MEDVEG}
{\sc P. Medvegyev},
{\it Stochastic Integration Theory}, Oxford University Press, 2009.

\bibitem{KRAO}
{\sc R. L. Karandikar, B.V. Rao},
{\it Introduction to Stochastic Calculus}, Springer, 2018.

\bibitem{RKYW1}
{\sc R. Kruse, Y. Wue},
A randomized Milstein method for stochastic differential equations with non-differentiable drift coefficients, {\it Discr. and Cont. Dyn. Sys. - Series B} {\bf 24} (2019), 3475--3502.

\bibitem{kumar}
{\sc C. Kumar}, Milstein-type schemes of SDE driven by L\'evy noise with super-linear diffusion coefficients, https://arxiv.org/pdf/1707.02343.pdf

\bibitem{kumsab}
{\sc  C. Kumar, S. Sabanis}, On tamed Milstein schemes of SDEs driven by L\'evy noise, {\it Discr. and Cont. Dyn. Sys. - Series B} {\bf 22} (2017) 421--463.

\bibitem{KARSHR}
{\sc I. Karatzas, S. E. Shreve},
\textit{Brownian Motion and Stochastic Calculus}, 2nd edition, Springer - Verlag, New York, 1991.


\bibitem{MAKPOD}
{\sc B. Makarov, A. Podkorytov},
{\it Real Analysis: Measures, Integrals and Applications}, Springer-Verlag, 2013.

\bibitem{mcshane}
{\sc E. J. McShane}, {\it Stochastic Calculus and Stochastic Models}, Academic Press, 1974.

\bibitem{MGHAB} 
{\sc T. M\"uller--Gronbach}, Strong approximation of systems of stochastic differential equations, {\it Habilitationsschrift}, TU Darmstadt (2002).


\bibitem{PMPP14}
{\sc P. M. Morkisz and P. Przyby{\l}owicz}, Strong approximation of solutions of stochastic differential equations with time-irregular coefficients via randomized Euler algorithm, Appl. Numer. Math., 78 (2014), pp. 80--94.

\bibitem{PMPP1}
{\sc P. M. Morkisz, P. Przyby{\l}owicz}, Optimal pointwise approximation of SDE's from inexact
information, {\it J. Comp. Appl. Math.} {\bf 324} \rm (2017) 85--100. 

\bibitem{Nov1}
{\sc E. Novak}, {\it Deterministic and Stochastic Error Bounds in Numerical Analysis}, Lecture Notes in Mathematics, vol.1349, Springer–Verlag, New York, 1988.

\bibitem{GPage}
{\sc G. Pag\`es}, {\it Numerical Probability - An Introduction with Applications to Finance}, Springer International Publishing AG, 2018.

\bibitem{Parth}
{\sc K. R. Parthasarathy}, {\it Probability Measures on Metric Spaces}, AMS Chelsea Publishing, 2005.

\bibitem{parras}
{\sc E. Pardoux, A. Rascanu}, {\it Stochastic Differential Equations, Backward SDEs, Partial Differential Equations}, Springer, 2014.

\bibitem{plat4}
{\sc E. Platen}, An approximation method for a class of It\^o processes with jump
component, {\it Liet. Mat. Rink.} {\bf 22} (1982) 124--136.


\bibitem{PBL} 
{\sc E. Platen, N. Bruti-Liberati}, {\it Numerical Solution of Stochastic Differential 
Equations with Jumps in Finance}, Springer Verlag, Berlin, Heidelberg, 2010.

\bibitem{prott}
{\sc P. E. Protter}, {\it Stochastic Integration and Differential Equations}, 2nd ed., Springer-Verlag Berlin Heidelberg, 2004.

\bibitem{PP7}
{\sc P. Przyby{\l}owicz}, Optimal global approximation of stochastic differential equations with additive Poisson noise, {\it Numer. Algor.} {\bf 73} \rm (2016) 323--348.

\bibitem{PP8}
{\sc P. Przyby{\l}owicz}, Optimal sampling design for global approximation of jump diffusion stochastic differential equations, {\it Stochastics} {\bf 91} (2019), 235--264.

\bibitem{PP9}
{\sc P. Przyby{\l}owicz}, Efficient approximate solution  of jump-diffusion SDEs via path-dependent adaptive step-size control, {\it J. Comput. and Appl. Math.} {\bf 350} (2019), 396-411. 

\bibitem{PPVSMS1}
{\sc P. Przyby{\l}owicz, V. Schwarz and M. Sz\"olgyenyi}, Randomized Milstein
algorithm for approximation of solutions of jump-diffusion SDEs, Journal of Computational and
Applied Mathematics (2023), https://doi.org/10.1016/j.cam.2023.115631

\bibitem{PPMSLS1}
{\sc P. Przyby{\l}owicz, M. Sobieraj, L. St\k{e}pie{\'n}}, Efficient approximation of SDEs driven by countably dimensional
Wiener process and Poisson random measure, SIAM J. Numer. Anal. 60 (2022), 824-855.


\bibitem{RIT} K. Ritter, \textit{Average--Case Analysis of Numerical Problems}, Lecture Notes in Mathematics, Vol. 1733, Springer-Verlag, Berlin 2000.

\bibitem{SAYL} J. Sacks, D. Ylvisaker, Design for regression problems with correlated errors III, Ann. Math. Stat. \bf 41 \rm (1970), 2057--2074

\bibitem{saso1}
{\sc S. S\"arkk\"a, A. Solin}, {\it Applied Stochastic Differential Equations}, Cambridge University Press, 2019.

\bibitem{situ}
{\sc R. Situ}, {\it Theory of Stochastic Differential Equations with Jumps and Applications}, Springer Science, 2005.

\bibitem{stengle1}
\textsc{G. Stengle}, \emph{Numerical methods for systems with measurable coefficients}, Appl. Math. Lett. 3 (1990), 25--29.

\bibitem{stengle2}
\textsc{G. Stengle}, \emph{Error analysis of a randomized numerical method}, Numer. Math. 70 (1995), 119--128.

\bibitem{ttao}
{\sc T. Tao}, {\it An Introduction to Measure Theory}, American Mathematical Society, 2011.

\bibitem{WW01} G. W. Wasilkowski, H. Wo\'zniakowski, On the complexity of stochastic integration, Math. Comp. \bf 70 \rm (2001), 685-698.

\bibitem{TWW88}
{\sc J.F. Traub, G.W. Wasilkowski, H. Wo{\'z}niakowski}, {\it Information--Based Complexity}, Academic Press, New York, 1988.

\bibitem{weizwin}
{\sc H. von Weizs\"acker, G. Winkler}, {\it Stochastic Integrals--An Introduction}, Springer, 1990.

\bibitem{YEH}
{\sc J. Yeh}, {\it Martingales and Stochastic Analysis}, World Scientific, 1995.

\end{thebibliography}
\end{document}